\documentclass[11pt]{amsart}
\usepackage[a4paper,left=25mm,right=25mm,top=25mm,bottom=25mm]{geometry}
\usepackage{microtype}
\usepackage{mathrsfs}
\usepackage{amsfonts}
\usepackage{amstext}
\usepackage{amssymb}
\usepackage[dvipsnames]{xcolor}
\usepackage{enumitem}
\setlist[enumerate]{label=\textup{(\roman*)},ref=\roman*,
  leftmargin=2em,labelsep=0.5em,itemsep=0.35\baselineskip,
  topsep=0.4\baselineskip,parsep=0pt}
\usepackage{tikz}
\usepackage{float}
\usepackage[
  colorlinks,
  linkcolor=blue,
  anchorcolor=green,
  citecolor=red,
  linktocpage=true
]{hyperref}
\hypersetup{hypertexnames=false}
\allowdisplaybreaks[2]

\newcommand{\Keywords}[1]{\par\noindent {\small{ Keywords\/}: #1}}

\theoremstyle{plain}
\newtheorem{theorem}{Theorem}[section]
\newtheorem{lemma}[theorem]{Lemma}
\newtheorem{proposition}[theorem]{Proposition}
\newtheorem{corollary}[theorem]{Corollary}
\newtheorem{criterion}[theorem]{Criterion}

\theoremstyle{definition}
\newtheorem{definition}[theorem]{Definition}

\numberwithin{equation}{section}

\newcommand{\DC}{\mathrm{DC}}
\newcommand{\eps}{\varepsilon}
\newcommand{\Id}{\mathrm{Id}}
\newcommand{\T}{\mathbb T}
\newcommand{\Z}{\mathbb Z}
\newcommand{\R}{\mathbb R}

\title[Acceleration Strata for Prevalent Perturbations]
{Positive-Measure Acceleration Strata for Prevalent Perturbations of
N-Fold Pullbacks of Supercritical Almost Mathieu Potentials}
\author[Liang Jinhao]{Liang Jinhao}
\address{(Liang Jinhao) Hohai University}
\email{jhliang\_hhu@163.com}

\author[Wang Yiqian]{Wang Yiqian}
\address{(Wang Yiqian) School of Mathematics, Nanjing University}
\email{yiqianw@nju.edu.cn}

\author[Xu Jiahao]{Xu Jiahao\textsuperscript{*}}
\address{(Xu Jiahao) School of Mathematics, Nanjing University}
\email{xjhao512@outlook.com}
\thanks{\textsuperscript{*}Corresponding author: Xu Jiahao.}
\date{}

\begin{document}
\raggedbottom

\begin{abstract}
Fix a strip width \(\rho>0\), a Diophantine frequency \(\alpha\),
an integer \(N\ge2\), and \(\lambda>1\).  For the family
\[
 V_{\varepsilon,\delta}(x)
 =2\lambda\cos(2\pi N x)+\varepsilon\delta(x),
\]
we prove that there is an open prevalent set
\(\mathcal P_{\rm strat}\) of real-analytic strip-\(\rho\) perturbations
such that every \(\delta\in\mathcal P_{\rm strat}\) admits
\(\varepsilon_0(\delta)>0\) for which
\[
\begin{gathered}
 \operatorname{Leb}\left\{
 E\in\Sigma(V_{\varepsilon,\delta},\alpha):
 \omega(\alpha,E;V_{\varepsilon,\delta})=s
 \right\}>0,\\
 0<\varepsilon<\varepsilon_0(\delta),
 \qquad s=1,\ldots,N.
\end{gathered}
\]
The proof relates acceleration to the total multiplicity of real zeros
of analytic overlap coefficients for first-return cocycles.  Under
quantitative control of the return products, Poisson integration and
Kac normalization identify acceleration with half this multiplicity.
A positive analytic dual state gives the finite-volume spectral gaps and Schur
curvature needed at the upper spectral edge.  Schur variations separate
the perturbed maxima, and parameter exclusion preserves \(s\) pairs of
simple zeros on a positive-measure spectral set.

\end{abstract}

\maketitle
\section{Introduction}

\noindent Acceleration records the one-sided variation of the Lyapunov exponent
under an imaginary phase shift.  We study how its spectral level sets
change under small analytic perturbations of supercritical almost
Mathieu potentials pulled back by an \(N\)-fold covering.

Write \(\mathbb T:=\mathbb R/\mathbb Z\) and
\(\|x\|_{\mathbb T}:=\min_{k\in\mathbb Z}|x-k|\).
For
\(r>0\), put
\[
 \mathbb T_r:=\{z\in\mathbb C/\mathbb Z:|\operatorname{Im}z|<r\}
\]
and let
\[
 C^\omega_r(\mathbb T,\mathbb R)
 :=\left\{f\in C(\overline{\mathbb T_r}):
 f|_{\mathbb T_r}\ \text{is holomorphic and }%
 f(\overline z)=\overline{f(z)}\right\}.
\]
This is a real Banach space with norm
\[
 \|f\|_r:=\sup_{z\in\overline{\mathbb T_r}}|f(z)|.
\]

We study the one-frequency Schr\"odinger operator
\begin{equation}\label{eq:Schrodinger-operator-definition}
 (H_{x,V,\alpha}u)_j
 =u_{j+1}+u_{j-1}+V(x+j\alpha)u_j
 \quad\text{on }\ell^2(\mathbb Z).
\end{equation}
In \eqref{eq:Schrodinger-operator-definition},
\(V\in C^\omega_{\rho_V}(\mathbb T,\mathbb R)\) for some
\(\rho_V>0\), and \(\alpha\in\mathbb R\setminus\mathbb Q\).  Minimality of the rotation makes
the spectrum independent of the phase; this common spectrum is denoted by
\(\Sigma(V,\alpha)\) (see, for example, \cite{ZhangUH2020}).  Put
\[
 A_E^V(x)=\begin{pmatrix}E-V(x)&-1\\1&0\end{pmatrix},
 \qquad
 A_{m,\beta}^V(x,E)
 :=A_E^V(x+(m-1)\beta)\cdots A_E^V(x).
\]
When \(\beta=\alpha\), write \(A_m^V:=A_{m,\alpha}^V\).
For \(E\in\mathbb R\) and \(|y|<\rho_V\), the complexified Lyapunov
exponent and acceleration are
\[
\begin{aligned}
 L(\alpha,E,y;V)
 &=\lim_{m\to\infty}\frac1m\int_{\mathbb T}
   \log\|A_m^V(x+iy,E)\|\,dx,\\
 \omega(\alpha,E;V)
 &=\frac1{2\pi}\partial_y^+L(\alpha,E,0;V).
\end{aligned}
\]

The function \(y\mapsto L(\alpha,E,y;V)\) is finite, even, and convex;
the acceleration is therefore well defined and nonnegative.  Avila's
quantization theorem gives
\(\omega(\alpha,E;V)\in\mathbb Z_{\ge0}\)
\cite{AvilaGlobal}.  The elementary regularity and measurability facts
used below are recorded at the end of this introduction.

Fix an integer \(N\ge2\), \(\lambda>1\), and a strip width \(\rho>0\), and put
\[
 X:=C^\omega_\rho(\mathbb T,\mathbb R)
\]
with its strip supremum norm \(\|\cdot\|_\rho\).

\begin{definition}
\label{def:prevalence}
A Borel set \(\mathcal A\subset X\) is prevalent in the sense of
Hunt--Sauer--Yorke \cite{HuntSauerYorke} if there exists a
compactly supported Borel probability measure \(\mu\) on \(X\) such that
\[
 \mu\{w\in X:x+w\in\mathcal A\}=1,
 \qquad \forall x\in X.
\]
\end{definition}

The perturbative family is
\begin{equation}\label{eq:perturbative-potential-family}
 V_{\varepsilon,\delta}(x)
 =2\lambda\cos(2\pi N x)+\varepsilon\delta(x),
 \qquad
 \delta\in X.
\end{equation}
For \(\kappa>0\) and \(\tau>1\), define
\[
 \DC(\kappa,\tau):=
 \left\{\vartheta\in\mathbb R\setminus\mathbb Q:
 \|q\vartheta\|_{\mathbb T}\ge
 \kappa|q|^{-\tau}\ \text{for every }q\in\mathbb Z\setminus\{0\}\right\}.
\]
For the family \eqref{eq:perturbative-potential-family}, assume
\(\alpha\in\DC(\kappa,\tau)\).

For the unperturbed potential \(2\lambda\cos(2\pi Nx)\), the acceleration
equals \(N\) throughout the spectrum when \(\lambda>1\).  Avila asked
whether a typical perturbation has, for each \(s=1,\ldots,N\), a set of
spectral energies of positive measure with acceleration \(s\)
\cite[Section~2.1.4, item~(5)]{AvilaGlobal}.
The following theorem establishes this conclusion for Diophantine
frequencies, for every \(\lambda>1\) and every fixed analytic strip.

\begin{theorem}
\label{theorem.complete-main}
Fix \(\rho>0\), an integer \(N\ge2\), \(\lambda>1\), and
\(\alpha\in\DC(\kappa,\tau)\), where \(\kappa>0\) and \(\tau>1\).
There is an open prevalent set
\(\mathcal P_{\mathrm{strat}}\subset
C^\omega_\rho(\mathbb T,\mathbb R)\) such that every
\(\delta\in\mathcal P_{\mathrm{strat}}\) admits
\(\varepsilon_0(\delta)>0\) with the following property.  If
\(0<\varepsilon<\varepsilon_0(\delta)\), then
\begin{equation}\label{eq:main-positive-measure-strata}
 \operatorname{Leb}\left\{E\in
 \Sigma(V_{\varepsilon,\delta},\alpha):
 \omega(\alpha,E;V_{\varepsilon,\delta})=s\right\}>0,
 \qquad s=1,\ldots,N.
\end{equation}
\end{theorem}

Phase complexification in quasiperiodic Schr\"odinger theory goes back
to Herman's subharmonic method.  Sorets--Spencer used this method to
prove positivity of the Lyapunov exponent at large coupling
\cite{Herman1983,SoretsSpencer1991}.  Avila's global theory establishes
quantization of acceleration and describes the resulting spectral
regimes; Avila--Jitomirskaya--Sadel develop the corresponding structure
theory for complex one-frequency cocycles
\cite{AvilaGlobal,AJSComplex}.
In the positive-Lyapunov-exponent regime and under a strong Diophantine
condition, Han--Schlag characterize acceleration by the asymptotic
number of zeros of finite-volume Dirichlet determinants
\cite{HanSchlag2026}.  Other descriptions relate acceleration to winding
numbers \cite{WangWangYouZhou2024}, or to dual Lyapunov exponents through
the multiplicative Jensen formula \cite{GeJitomirskayaYouZhou2026}.

\subsection*{Proof strategy and organization}

Our criterion uses the local zero divisors of overlap coefficients for
first-return cocycles.  Under the quantitative hypotheses of
Theorem~\ref{thm:persistent-divisor-acceleration}, the
multiplier--overlap approximation, Poisson integration, and exact Kac
normalization \cite[Lemma~2.2]{Knill1992} give
\[
 P_{\rm loc}=2s
 \quad\Longrightarrow\quad
 \omega(\alpha,E;V_{\varepsilon,\delta})
 =\frac{P_{\rm loc}}2=s.
\]
The zeros may move with the return scale; their total multiplicity
\(P_{\rm loc}\) remains fixed.  To construct such divisors, we use a
positive analytic dual state to obtain finite-volume spectral gaps and
positive Schur curvature at the upper edge.  Schur variations then
separate the perturbed maxima.  A return induction with parameter
exclusion, following the approach of Young and Wang--Zhang
\cite{Young1997,WangZhang2015}, preserves the required simple zeros.

Section~\ref{sec:persistent-return-acceleration} proves the acceleration
criterion, and Section~\ref{sec:black-box-main-reduction} reduces the main
theorem to its hypotheses.  The almost Mathieu (AMO) upper-edge estimates in
Section~\ref{sec:AMO-upper-edge-seed} provide the starting data.
Section~\ref{sec:multiwell-seed-construction} constructs \(2s\) simple
zeros at the initial scale.  Section~\ref{sec:real-CF-propagation}
preserves them on a positive-measure spectral set and constructs their
holomorphic extensions.  Appendix~\ref{app:AMO-upper-edge} proves the
AMO upper-edge estimates; Appendix~\ref{app:normalized-product-proof}
proves the normalized-product estimate used by the criterion.

\subsection*{Notation and preliminary facts}

For the model-specific estimates, the fixed background data are
\[
 \mathfrak b:=(N,\lambda,\rho,\alpha,\kappa,\tau).
\]
The letters \(c,C>0\) denote constants depending only on these data;
their values may change between estimates.  Once a return-window exponent
\(C_I\) has been fixed, it is included in the background data for that
application.  Further dependence is displayed in numbered constants,
such as \(C_i=C_i(R,r_*)\) or \(D_i=D_i(R,r_*)\), and the arguments
are then suppressed.  Constants in conclusions and smallness conditions,
and exponents in the scale inequalities, retain their chosen values.
Auxiliary coefficients in a proof may vary between estimates when this
is stated explicitly; their parameter dependence remains unchanged.
All choices are completed before the corresponding scale threshold is
fixed.  The abstract estimates state their parameter dependence separately.

We use \(|\cdot|\) for Lebesgue measure on \(\mathbb T\) or
\(\mathbb R\), and \(\mathbf1_A\) for the indicator of a measurable
set \(A\).  Write \(\mathbb N=\{1,2,\ldots\}\) and
\(\mathbb N_0=\{0,1,2,\ldots\}\); for \(d\in\mathbb N\),
\(A^d\) denotes the \(d\)-fold Cartesian product.  For compact \(K\),
\(C(K)\) denotes the complex-valued continuous functions on \(K\);
\(\ell^2(S)\), for \(S\subset\mathbb Z\), denotes the space of
square-summable complex-valued functions on \(S\).

Put
\[
 T_\beta x:=x+\beta\pmod1,
 \qquad \beta\in\mathbb R.
\]
For a finite interval \(\Gamma\subset\mathbb Z\), let
\(H_{V,\Gamma}(x)\) denote the Dirichlet restriction of
\(H_{x,V,\alpha}\) to \(\ell^2(\Gamma)\); explicitly,
\[
 (H_{V,\Gamma}(x)\psi)_j
 =\psi_{j+1}+\psi_{j-1}+V(x+j\alpha)\psi_j,
 \qquad j\in\Gamma,
\]
with \(\psi=0\) outside \(\Gamma\).  If the frequency must be displayed,
we write \(H_{x,V,\beta}^{[a,b]}\) for the corresponding restriction to
\([a,b]\cap\mathbb Z\).  For the standard almost Mathieu potential, use
the abbreviations
\[
 H_{\lambda,\beta,x}:=H_{x,2\lambda\cos(2\pi\,\cdot),\beta},
 \qquad
 H_{\lambda,\beta,x}^{[a,b]}
 :=H_{x,2\lambda\cos(2\pi\,\cdot),\beta}^{[a,b]}.
\]
For complex vectors and matrices, \({}^{\mathsf T}\) denotes transpose
without complex conjugation; vector norms are Hermitian Euclidean norms,
and matrix norms are the induced operator norms.
For a real domain \(\Omega\), our derivative norm convention is
\[
 \|f\|_{C^k(\Omega)}
 :=\max_{|\zeta|\le k}\sup_\Omega\|\partial^\zeta f\|,
\]
using the absolute value for scalar \(f\).
Throughout, \(e_1=(1,0)^{\mathsf T}\) and
\(e_2=(0,1)^{\mathsf T}\) denote the standard basis of \(\mathbb C^2\).
For \(\theta\in\mathbb R\), our rotation convention is
\[
 R_\theta:=
 \begin{pmatrix}\cos\theta&-\sin\theta\\
                 \sin\theta& \cos\theta\end{pmatrix}.
\]

We next record the regularity facts used in the statement of the theorem.
For fixed \(y\), subadditivity gives
\[
\begin{gathered}
 a_m(y):=\int_{\mathbb T}\log\|A_m^V(x+iy,E)\|\,dx,\\
 a_{m+n}(y)\le a_m(y)+a_n(y),\qquad
 L(\alpha,E,y;V)=\inf_{m\ge1}\frac{a_m(y)}m.
\end{gathered}
\]
Put \(F_m=a_m/m\). Subharmonicity, periodicity, and real symmetry give,
in the sense of distributions,
\[
 F_m''(y)=\frac1m\int_{\mathbb T}
 \Delta_{x,y}\log\|A_m^V(x+iy,E)\|\,dx\ge0,
 \qquad F_m(-y)=F_m(y).
\]
Thus, for \(0\le t\le1\),
\[
\begin{aligned}
 F_m(ty_1+(1-t)y_2)&\le tF_m(y_1)+(1-t)F_m(y_2),\\
 L(\alpha,E,ty_1+(1-t)y_2;V)
 &\le tL(\alpha,E,y_1;V)+(1-t)L(\alpha,E,y_2;V).
\end{aligned}
\]
The second inequality follows by passing to the pointwise limit.
Since \(0\le L(\alpha,E,y;V)\le\log\sup_x\|A_E^V(x+iy)\|\),
the limit is finite on every smaller strip and its right derivative
exists \cite{Herman1983,AvilaGlobal}.

For fixed \(y\), every \(F_m(E,y)\) is continuous in \(E\), so
\(L(\alpha,E,y;V)=\inf_mF_m(E,y)\) is Borel measurable. With
\(h_k=\rho_V/(k+1)\),
\[
 \omega(\alpha,E;V)
 =\lim_{k\to\infty}
 \frac{L(\alpha,E,h_k;V)-L(\alpha,E,0;V)}{2\pi h_k}.
\]
Thus the acceleration is Borel measurable in \(E\), and each spectral
acceleration stratum in Theorem~\ref{theorem.complete-main} is a Borel set.

The unperturbed profile is the pullback, by the \(N\)-fold covering, of
the standard almost Mathieu potential:
\begin{equation}\label{eq:N-fold-AMO-pullback}
\begin{gathered}
 \pi_N(x):=Nx\pmod1,
 \qquad
 2\lambda\cos(2\pi Nx)
 =2\lambda\cos(2\pi\,\cdot)\circ\pi_N,\\
 \pi_N\circ T_\alpha=T_{N\alpha}\circ\pi_N.
\end{gathered}
\end{equation}
Set \(\widehat\alpha:=N\alpha\pmod1\).  Then, for every phase \(x\),
\begin{equation}\label{eq:N-fold-AMO-operator-identity}
 H_{x,2\lambda\cos(2\pi N\,\cdot),\alpha}
 =H_{\lambda,\widehat\alpha,Nx},
 \qquad
 \Sigma(2\lambda\cos(2\pi N\,\cdot),\alpha)
 =\Sigma(2\lambda\cos(2\pi\,\cdot),\widehat\alpha).
\end{equation}
In the covering \eqref{eq:N-fold-AMO-pullback}, label the fiber over \(0\) by
\[
 p_a:=\frac aN,\qquad 0\le a<N.
\]
Put
\begin{equation}\label{eq:effective-frequency-data}
 E_+:=\max\Sigma\bigl(2\lambda\cos(2\pi\,\cdot),\widehat\alpha\bigr),
 \qquad \gamma:=\log\lambda.
\end{equation}
By \eqref{eq:N-fold-AMO-operator-identity}, \(E_+\) is also the upper
spectral edge of the unperturbed \(N\)-fold potential at frequency
\(\alpha\).
The data in \eqref{eq:effective-frequency-data} satisfy
\(\widehat\alpha\notin\mathbb Q\) and
\begin{equation}\label{eq:effective-frequency-Diophantine}
 \|q\widehat\alpha\|_{\mathbb T}
 =\|qN\alpha\|_{\mathbb T}
 \ge \kappa N^{-\tau}|q|^{-\tau},
 \qquad q\in\mathbb Z\setminus\{0\}.
\end{equation}
Thus \eqref{eq:effective-frequency-Diophantine} permits the upper-edge
initial-scale proposition to be applied with frequency
\(\widehat\alpha\), Diophantine
constant \(\kappa N^{-\tau}\), and upper edge \(E_+\).

\section{Scale-dependent zero divisors and acceleration}
\label{sec:persistent-return-acceleration}

\noindent We first define the first-return cocycle on a finite union of
intervals and establish the exact inducing identity for its Lyapunov
exponent.  Algebraic singular value decompositions then reduce the
complexified exponent to the multiplier and overlap terms.  A Poisson
estimate for scale-dependent zero divisors, together with the asymptotic
product estimate and Kac normalization, yields
Theorem~\ref{thm:persistent-divisor-acceleration}: the acceleration is
half the persistent total multiplicity.

Throughout,
\(\log^+t:=\max\{\log t,0\}\) for \(t>0\).
For positive \(b_j\), the notation \(a_j=o_j(b_j)\) means
\(a_j/b_j\to0\) as \(j\to+\infty\).

\subsection{First-return maps on finite unions of intervals}

\begin{definition}
\label{fr:def:regular-phase-family}
Let \(T_\vartheta x=x+\vartheta\) be an irrational rotation and let
\(\mathscr K\) be a finite or countable family of nonempty finite unions of
nondegenerate closed intervals in \(\mathbb T\).  Define
\begin{equation}\label{fr:eq:regular-phase-family}
 \mathscr N_\vartheta(\mathscr K)
 :=\bigcup_{K\in\mathscr K}\ \bigcup_{m\in\mathbb Z}
 T_\vartheta^m(\partial_{\mathbb T}K),
 \qquad
 \operatorname{Reg}_\vartheta(\mathscr K)
 :=\mathbb T\setminus\mathscr N_\vartheta(\mathscr K).
\end{equation}
When \(\vartheta\) is fixed, the subscript \(\vartheta\) is omitted.  If
only one return base \(K\) is under consideration, write
\[
 \operatorname{Reg}(K):=\operatorname{Reg}(\{K\}),
 \qquad K^{\rm reg}:=K\cap\operatorname{Reg}(K).
\]
For \(K\in\mathscr K\) and
\(x\in K\cap\operatorname{Reg}_\vartheta(\mathscr K)\), define the
positive forward and backward first-return times by
\[
 r_K^+(x):=\min\{r\ge1:T_\vartheta^rx\in K\},
 \qquad
 r_K^-(x):=\min\{r\ge1:T_\vartheta^{-r}x\in K\}.
\]
\end{definition}

\begin{lemma}
\label{fr:lem:regular-phases-induced-map}
Let \(\mathscr K\) be as in
Definition~\ref{fr:def:regular-phase-family}, and put
\[
 D_K:=K\cap\operatorname{Reg}_\vartheta(\mathscr K),
 \qquad K\in\mathscr K.
\]
Then the following hold.
\begin{enumerate}
\item
\(\mathscr N_\vartheta(\mathscr K)\) is countable, null, and invariant
under both \(T_\vartheta\) and \(T_\vartheta^{-1}\).
\item
For every \(K\in\mathscr K\), the forward and backward return times
are finite and locally constant on \(D_K\).  The first-return map
\[
 S_K:D_K\to D_K,
 \qquad S_Kx:=T_\vartheta^{r_K^+(x)}x,
\]
is a bimeasurable bijection with
\[
 S_K^{-1}x=T_\vartheta^{-r_K^-(x)}x.
\]
\item
For every \(K\in\mathscr K\), the map \(S_K\) preserves and is ergodic with respect to
\(d\mu_K=|K|^{-1}\mathbf1_{D_K}\,dx\), and
\begin{equation}\label{fr:eq:regular-Kac-identity}
 \int_{D_K}r_K^+(x)\,d\mu_K(x)=\frac1{|K|}.
\end{equation}
\end{enumerate}
\end{lemma}

\begin{proof}
By \eqref{fr:eq:regular-phase-family},
\[
 \#\mathscr N_\vartheta(\mathscr K)\le\aleph_0,\qquad
 |\mathscr N_\vartheta(\mathscr K)|=0,\qquad
 T_\vartheta^{\pm1}\mathscr N_\vartheta(\mathscr K)
 =\mathscr N_\vartheta(\mathscr K).
\]
Minimality and compactness give \(M_K<\infty\) such that
\[
 \mathbb T=\bigcup_{r=1}^{M_K}
 T_\vartheta^{\mp r}(\operatorname{int}_{\mathbb T}K),
 \qquad r_K^\pm\le M_K.
\]
If \(K=\mathbb T\), then \(r_K^\pm\equiv1\). Otherwise, for \(x\in D_K\), put
\[
 \delta_x:=\min_{|m|\le M_K}
 \operatorname{dist}(T_\vartheta^m x,\partial_{\mathbb T}K)>0.
\]
If \(x'\in D_K\) and \(\operatorname{dist}(x,x')<\delta_x\), then
\[
 \mathbf1_K(T_\vartheta^m x')=\mathbf1_K(T_\vartheta^m x)
 \quad(|m|\le M_K),
 \qquad r_K^\pm(x')=r_K^\pm(x).
\]
Thus the return times are locally constant in \(D_K\).  Invariance of the
regular set and the induced-transformation theorem
\cite{Kakutani1943} give the asserted bimeasurable inverse, preservation,
and ergodicity.  Since \(|D_K|=|K|\) and its saturation has full measure,
Kac's formula \cite{Kac1947} yields
\[
 \int_{D_K}r_K^+\,dx=1,\qquad
 \int_{D_K}r_K^+\,d\mu_K=|K|^{-1}.
\]
\end{proof}

\begin{definition}
\label{fr:def:common-return-partition}
Fix an irrational rotation \(T=T_\vartheta\) and let
\(I\subset\mathbb T\) be a nonempty
finite union of pairwise disjoint nondegenerate closed intervals, with
\(I\ne\mathbb T\).  A nonempty connected
open interval
\(\mathcal A\subset\operatorname{int}_{\mathbb T}I\) is a
\emph{two-sided return-time atom} if
\[
 r_I^+\equiv p_{\mathcal A},\qquad
 r_I^-\equiv q_{\mathcal A}
 \quad\text{on }\mathcal A\cap I^{\rm reg}
\]
for some \(p_{\mathcal A},q_{\mathcal A}\ge1\), and \(\mathcal A\) is
maximal with this property.  Put
\[
 \mathcal P_I
 :=\{\mathcal A:\mathcal A\text{ is a two-sided return-time atom of }I\},
 \qquad
 \mathscr B_I^{\rm ret}
 :=I\setminus\bigcup_{\mathcal A\in\mathcal P_I}\mathcal A.
\]
\end{definition}

\begin{lemma}
\label{fr:lem:finite-common-return-partition}
In the setting of Definition~\ref{fr:def:common-return-partition},
both \(\mathscr B_I^{\rm ret}\) and \(\mathcal P_I\) are finite.
\end{lemma}

\begin{proof}
Minimality and compactness \cite{Walters1982} give \(M_I<\infty\) with
\[
 r_I^\pm\le M_I,\qquad
 F:=I\cap\bigcup_{|m|\le M_I}T^{-m}(\partial_{\mathbb T}I),
 \qquad \#F<\infty.
\]
On each component of \(\operatorname{int}_{\mathbb T}I\setminus F\),
all functions \(\mathbf1_I\circ T^m\), \(|m|\le M_I\), are constant.
Hence each component lies in one maximal return atom, and
\[
 \mathscr B_I^{\rm ret}\subset F,\qquad
 \#\mathcal P_I
 \le\#\{\text{components of }\operatorname{int}_{\mathbb T}I\setminus F\}
 <\infty.
\]
\end{proof}

\subsection{Induced exponents and the exact inducing identity}
\label{subsec:induced-exponents}

Fix \(\rho>0\), \(\alpha\in\mathbb R\setminus\mathbb Q\),
\(V\in C^\omega_\rho(\mathbb T,\mathbb R)\), and \(E\in\mathbb R\).
Let \(T x=x+\alpha\) on \(\mathbb T\), and let \(K\) be a nonempty finite
union of nondegenerate closed intervals.  The return dynamics is used only on
the regular set of Definition~\ref{fr:def:regular-phase-family}.  For
\(x\in K^{\rm reg}\), let \(S_K\) be the first-return map from
Lemma~\ref{fr:lem:regular-phases-induced-map}.
For \(|y|<\rho\), use the real return times in the complexified cocycle
and define
\[
\begin{aligned}
 B_{K,y}^+(x,E)
 &:=A_{r_K^+(x)}^V(x+iy,E),\\
 B_{K,y}^-(x,E)
 &:=A_{r_K^-(x)}^V
 \bigl(T^{-r_K^-(x)}x+iy,E\bigr).
\end{aligned}
\]
Write \(B_{K,y}:=B_{K,y}^+\) for the induced cocycle.
With \(d\mu_K=|K|^{-1}\mathbf1_{K^{\rm reg}}dx\), define the \(n\)-th induced
cocycle iterate by
\[
 B_{K,y}^{(n)}(x,E)
 :=B_{K,y}(S_K^{n-1}x,E)\cdots B_{K,y}(x,E),
 \qquad n\ge1,
\]
and put
\[
\begin{aligned}
 L_{K,n}(E,y)&:=\frac1n\int_{K^{\rm reg}}
 \log\|B_{K,y}^{(n)}(x,E)\|\,d\mu_K(x),\\
 L_K(E,y)&:=\lim_{n\to\infty}L_{K,n}(E,y).
\end{aligned}
\]
The limit exists because the integrated logarithmic norms form a
subadditive sequence.
When \(K\ne\mathbb T\), for a two-sided return-time atom
\(\mathcal B\in\mathcal P_K\), analytic quantities mean the fixed-word
extensions on this atom.  Restrict \(B_K^\pm:=B_{K,0}^\pm\) to
\(\mathcal B\cap K^{\rm reg}\).  Whenever these
blocks have simple top singular values and coherent lifts of their standard
singular value decomposition (SVD) data are fixed on this atom, write
\((\lambda_K^\pm,u_K^\pm,s_K^\pm)\), with the convention
\[
 B_K^\pm
 =R_{u_K^\pm}
  \operatorname{diag}\bigl(\lambda_K^\pm,(\lambda_K^\pm)^{-1}\bigr)
  R_{\pi/2-s_K^\pm},
 \qquad \lambda_K^\pm>1.
\]
In the convention for the stable and unstable directions used in the return-induction
schemes of \cite{Young1997,WangZhang2015}, define the two-sided
first-return angle function by
\begin{equation}\label{fr:eq:complete-return-joining-convention}
 G_{K,\mathcal B}(x,E)
 :=\frac\pi2-s_K^+(x,E)+u_K^-(x,E)
 \quad(\operatorname{mod}2\pi).
\end{equation}
We call \(\cos G_{K,\mathcal B}\) the real overlap coefficient.
The next proposition restores the original Lyapunov exponent by Kac
normalization~\cite[Lemma~2.2]{Knill1992}.
For a return base \(\mathcal J\) of the same type, use the preceding
definitions with \(K=\mathcal J\).

\begin{proposition}
\label{prop:induced-global-LE-identity}
For every \(|y|<\rho\),
\begin{equation}
\label{eq:induced-global-LE-identity}
        L(\alpha,E,y;V)=|\mathcal J|L_{\mathcal J}(E,y).
\end{equation}
\end{proposition}

\begin{proof}
Fix \(|y|<\rho\), and put
\[
 R_k(x):=\sum_{\ell=0}^{k-1}
 r_{\mathcal J}^+(S_{\mathcal J}^\ell x),\qquad
 M_y:=\sup_{x\in\mathbb T}\|A_E^V(x+iy)\|.
\]
The cocycle identity and determinant one give
\[
\begin{aligned}
 B_{\mathcal J,y}^{(k)}(x,E)&=A_{R_k(x)}^V(x+iy,E),\\
 0\le\log\|A_E^V(x+iy)\|&\le\log M_y,\\
 0\le\log\|B_{\mathcal J,y}(x,E)\|
 &\le r_{\mathcal J}^+(x)\log M_y.
\end{aligned}
\]
By Lemma~\ref{fr:lem:regular-phases-induced-map}, the induced map is
ergodic and
\[
 \int r_{\mathcal J}^+\,d\mu_{\mathcal J}=|\mathcal J|^{-1};
 \qquad
 \frac{R_k(x)}k\longrightarrow|\mathcal J|^{-1}
 \quad(\mu_{\mathcal J}\text{-a.e.})
\]
by Birkhoff's theorem \cite{Birkhoff1931}.  The displayed bounds verify
Kingman's integrability hypotheses \cite{Kingman1973} for both cocycles.
More explicitly, on a common set \(\Omega_y\) of full
\(\mu_{\mathcal J}\)-measure,
\[
 R_k(x)\to\infty,\qquad
 \lim_{k\to\infty}\frac{\log\|A_{R_k(x)}^V(x+iy,E)\|}{R_k(x)}
 =L(\alpha,E,y;V),\qquad x\in\Omega_y,
\]
because \(\mu_{\mathcal J}\ll dx\) and the original Kingman limit
holds along the full sequence. Therefore
\[
\begin{aligned}
 L_{\mathcal J}(E,y)
 &=\lim_{k\to\infty}\frac1k
   \log\|B_{\mathcal J,y}^{(k)}(x,E)\|\\
 &=\lim_{k\to\infty}
   \frac{R_k(x)}k\,
   \frac{\log\|A_{R_k(x)}^V(x+iy,E)\|}{R_k(x)}\\
 &=|\mathcal J|^{-1}L(\alpha,E,y;V).
\end{aligned}
\]
\end{proof}

\subsection{Algebraic singular value decompositions and asymptotic approximation}

\begin{definition}
\label{def:complex-special-orthogonal-group}
Let \(M_2(\mathbb C)\) be the space of \(2\times2\) complex matrices,
let
\[
 I_2:=\begin{pmatrix}1&0\\0&1\end{pmatrix},
 \qquad
 \mathcal J_0:=\begin{pmatrix}0&-1\\1&0\end{pmatrix},
\]
and define
\[
 {\rm SO}(2,\mathbb C)
 :=\{Q\in M_2(\mathbb C):Q^{\mathsf T}Q=I_2,\ \det Q=1\}.
\]
\end{definition}

Equivalently,
\[
 Q\in{\rm SO}(2,\mathbb C)
 \quad\Longleftrightarrow\quad
 Q=(w,\mathcal J_0w)
 \ \text{for some }w\in\mathbb C^2
 \ \text{with }w^{\mathsf T}w=1.
\]

Thus \(Q^{-1}=Q^{\mathsf T}\).  A
holomorphic scalar- or matrix-valued function \(F\) on a
conjugation-invariant domain is called \emph{real-symmetric} if
\(F(\overline z)=\overline{F(z)}\), entrywise for matrices.
Write \(\mathbb C^\times:=\mathbb C\setminus\{0\}\).

The terminology below follows the algebraic singular value decomposition
of~\cite[Definition~3.1]{DrusvyatskiyLeeOttavianiThomas2017}, specialized
to \(2\times2\) matrices with oriented complex-orthogonal frames and
branches adapted to the cocycle.

\begin{definition}
\label{def:complex-orthogonal-rank-one-normal-form}
Let \(\Omega\) be a set, let \(S:\Omega\to\Omega\), and let
\(B:\Omega\to{\rm SL}(2,\mathbb C)\).  A representation of the form
\[
\begin{gathered}
 B(x)=U(Sx)
 \begin{pmatrix}\Lambda(x)&0\\0&\Lambda(x)^{-1}\end{pmatrix}
 V(x)^{\mathsf T},\\
 U,V:\Omega\to{\rm SO}(2,\mathbb C),
 \qquad \Lambda:\Omega\to\mathbb C^\times,
\end{gathered}
\]
is called an
\emph{algebraic singular value decomposition of \(B\) over \(S\)}.
The ordered multiplier branch and the oriented frames are
part of the normal-form data.  For \(k=1,2\), put
\[
 u_k(x):=U(x)e_k,
 \qquad
 v_k(x):=V(x)e_k,
 \qquad
 H(x):=v_1(x)^{\mathsf T}u_1(x).
\]
The matrix \(V(x)^{\mathsf T}U(x)\) is the transition matrix between the
two frames, and \(H\) is its leading \emph{overlap coefficient}.
\end{definition}

\paragraph{\textbf{Bilinear and Hermitian Gram matrices.}}
For \(N=(n_1,n_2)\in M_2(\mathbb C)\), set
\[
\begin{aligned}
 G_{\rm in}&:=N^{\mathsf T}N=(n_i^{\mathsf T}n_j)_{i,j=1}^2,
 &G_{\rm out}&:=NN^{\mathsf T},&
 G_{\diamond}^{\mathsf T}&=G_{\diamond},\\
 N^*&:=\overline N^{\mathsf T},
 &N^*N&\ge0,&NN^*&\ge0.
\end{aligned}
\]
The bilinear Gram matrices \(G_{\rm in},G_{\rm out}\) are holomorphic
for holomorphic \(N\), but need not be Hermitian or positive.  The Hermitian
Gram matrices generally fail to be holomorphic.  Both notions coincide
for real \(N\).  For invertible \(N\) and a positively oriented contour
\(\Gamma\) enclosing one isolated eigenvalue,
\[
\begin{aligned}
 G_{\rm out}&=NG_{\rm in}N^{-1},\\
 \Pi_G(\Gamma)&:=\frac1{2\pi i}
  \int_\Gamma(\xi I_2-G)^{-1}\,d\xi,
 &\Pi_G(\Gamma)^{\mathsf T}&=\Pi_G(\Gamma),\\
 \Pi_{G_{\rm out}}(\Gamma)&=
  N\Pi_{G_{\rm in}}(\Gamma)N^{-1}.
\end{aligned}
\]
These projections are holomorphic while \(\Gamma\) remains in the
resolvent set.  For real \(N\) with a simple top singular value, they are
the orthogonal projections onto the corresponding singular lines, with
largest Gram eigenvalue \(\|N\|^2\).

The following standard perturbation criterion for Riesz spectral projections
follows from
\cite[Chapter~II, \S\S~1.3--1.4]{Kato1995} and
\cite[Chapter~6]{HornJohnson2013}.

\begin{criterion}
\label{crit:rank-one-spectral-separation}
Put \(P:=e_1e_1^{\mathsf T}\).  If \(N\in{\rm GL}(2,\mathbb C)\) and
\(\|N-P\|\le1/10\), then each of \(N^{\mathsf T}N\) and
\(NN^{\mathsf T}\) has exactly one eigenvalue, counted with algebraic
multiplicity, in \(\{\mu:|\mu-1|<1/2\}\) and one in
\(\{\mu:|\mu|<1/2\}\); the corresponding Riesz spectral projections
have rank one.
\end{criterion}

\begin{proof}
Write \(N=P+E\), let \(G\) be either Gram matrix, and put
\(G_s=P+s(G-P)\). For \(j=0,1\), \(0\le s\le1\), and
\(|\xi-j|=1/2\),
\[
\begin{aligned}
 \|G-P\|&\le2\|E\|+\|E\|^2\le21/100,\\
 \|(\xi I_2-P)^{-1}s(G-P)\|&\le42/100<1,\\
 (\xi I_2-G_s)^{-1}
 &=\sum_{k=0}^{\infty}
 \bigl((\xi I_2-P)^{-1}s(G-P)\bigr)^k(\xi I_2-P)^{-1},\\
 \operatorname{rank}\Pi_{G_s}(\{|\xi-j|=1/2\})
 &=\operatorname{rank}\Pi_P(\{|\xi-j|=1/2\})=1.
\end{aligned}
\]
The last equality follows from continuity of the Riesz projections and
integer-valuedness of their ranks.
\end{proof}

We will use stability of the real singular directions.  Let
\(r\in\mathbb N\), let \(\Omega\subset\mathbb R^2\) be open, and,
for \(\sigma\in\{+,-\}\), let
\(\mathcal N_t^\sigma:\Omega\to{\rm GL}(2,\mathbb R)\),
\(0\le t\le1\), be continuous in \(t\) with values in \(C^r(\Omega)\).
Assume \(\det\mathcal N_t^\sigma>0\) throughout the parameter domain.
For \(1\le\nu\le r\), put
\[
 K_\nu:=1+\max_{\sigma,t}
 \|\mathcal N_t^\sigma\|_{C^\nu(\Omega)},
 \qquad
 \eta:=\max_{\sigma,t}
 \|\mathcal N_t^\sigma-\mathcal N_0^\sigma\|_{C^r(\Omega)}.
\]
Put
\[
 G_{t,\sigma}^{\rm in}:=(\mathcal N_t^\sigma)^{\mathsf T}\mathcal N_t^\sigma,
 \qquad
 G_{t,\sigma}^{\rm out}:=\mathcal N_t^\sigma(\mathcal N_t^\sigma)^{\mathsf T},
\]
and let \(\Pi_{t,\sigma}^{\rm in}\) and
\(\Pi_{t,\sigma}^{\rm out}\) be their spectral projections associated
with the larger eigenvalues.
Writing \(\mu_1(G)\ge\mu_2(G)\) for the eigenvalues of a real symmetric
matrix \(G\), we require
\begin{equation}\label{svd:eq:fixed-gap-input}
 \inf_{\sigma,t,x}\min_{\diamond\in\{{\rm in},{\rm out}\}}
 \bigl(\mu_1(G_{t,\sigma}^{\diamond}(x))
       -\mu_2(G_{t,\sigma}^{\diamond}(x))\bigr)\ge\frac34.
\end{equation}
\begin{lemma}
\label{svd:lem:fixed-gap-oriented-stability}
For the matrix families satisfying \eqref{svd:eq:fixed-gap-input},
there are polynomials \(\mathfrak P_r^{\rm svd}\) with nonnegative
coefficients and a constant \(C_1(r)>0\), depending only on \(r\), such that, if
\begin{equation}\label{svd:eq:fixed-gap-chart-smallness}
 C_1(r)\eta\mathfrak P_r^{\rm svd}(K_1,\ldots,K_r)\le\frac14,
\end{equation}
then the following estimates hold.
\begin{enumerate}
\item
The spectral projections satisfy
\begin{equation}\label{svd:eq:fixed-gap-projector-stability}
\begin{aligned}
 \max_{\sigma,t}\bigl(&
 \|\Pi_{t,\sigma}^{\rm in}-\Pi_{0,\sigma}^{\rm in}\|_{C^r}
 +\|\Pi_{t,\sigma}^{\rm out}-\Pi_{0,\sigma}^{\rm out}\|_{C^r}\bigr)
 \le C_1(r)\eta\mathfrak P_r^{\rm svd}(K_1,\ldots,K_r).
\end{aligned}
\end{equation}
\item
For \(1\le|\beta|\le r\),
\begin{equation}\label{svd:eq:fixed-gap-log-singular-derivatives}
 \max_{\sigma,t}
 \|\partial^\beta\log s_1(\mathcal N_t^\sigma)\|_{L^\infty(\Omega)}
 \le C_1(r)\mathfrak P_r^{\rm svd}(K_1,\ldots,K_r).
\end{equation}
\item
The logarithms of the dominant singular values satisfy
\begin{equation}\label{svd:eq:fixed-gap-log-singular-stability}
 \max_{\sigma,t}
 \|\log s_1(\mathcal N_t^\sigma)-
 \log s_1(\mathcal N_0^\sigma)\|_{C^r(\Omega)}
 \le C_1(r)\eta\mathfrak P_r^{\rm svd}(K_1,\ldots,K_r).
\end{equation}
\end{enumerate}
\end{lemma}

\begin{proof}
The product rule gives
\[
 \|G_{t,\sigma}^{\diamond}-G_{0,\sigma}^{\diamond}\|_{C^r}
 \le C_1(r)\eta\mathfrak P_r^{\rm svd}(K_1,\ldots,K_r),
 \qquad \diamond\in\{{\rm in},{\rm out}\}.
\]
For a real symmetric \(2\times2\) matrix with eigenvalues
\(\mu_1>\mu_2\), its principal projection is
\[
 \Pi_1=\frac{G-\mu_2I_2}{\mu_1-\mu_2},
 \qquad
 \mu_{1,2}=\frac{\operatorname{tr}G
 \pm\sqrt{(\operatorname{tr}G)^2-4\det G}}2.
\]
Use the positive square-root branch.  Enlarging \(C_1(r)\) once in
\eqref{svd:eq:fixed-gap-chart-smallness}, the segment
\(G_s:=(1-s)G_{0,\sigma}^{\diamond}+sG_{t,\sigma}^{\diamond}\) satisfies
\[
\begin{aligned}
 \mu_1(G_s)-\mu_2(G_s)
 &\ge\frac34-2\|G_{t,\sigma}^{\diamond}-G_{0,\sigma}^{\diamond}\|
 \ge\frac14,\qquad 0\le s\le1,\\
 \Pi_1(G_{t,\sigma}^{\diamond})-\Pi_1(G_{0,\sigma}^{\diamond})
 &=\int_0^1D\Pi_1(G_s)
   [G_{t,\sigma}^{\diamond}-G_{0,\sigma}^{\diamond}]\,ds.
\end{aligned}
\]
Differentiating the explicit projection formula through order \(r\)
therefore proves \eqref{svd:eq:fixed-gap-projector-stability}.
For the multiplier, use
\[
 \mu_1(G_{t,\sigma}^{\diamond})\ge\frac34,\qquad
 \log s_1(\mathcal N_t^\sigma)=\frac12\log\mu_1(G_{t,\sigma}^{\diamond}),
\]
and the same chain rule and segment identity to obtain
\eqref{svd:eq:fixed-gap-log-singular-derivatives} and
\eqref{svd:eq:fixed-gap-log-singular-stability}.

\end{proof}

Fix coherent \(C^r\) unit sections
\[
 \mathbf v_{0,\sigma}:\Omega\to\mathbb S^1,
 \qquad
 \Pi_{0,\sigma}^{\rm in}\mathbf v_{0,\sigma}
 =\mathbf v_{0,\sigma}.
\]

Where the denominator is nonzero, set
\begin{equation}\label{svd:eq:fixed-gap-oriented-axis-selection}
 \mathbf v_{t,\sigma}:=
 \frac{\Pi_{t,\sigma}^{\rm in}\mathbf v_{0,\sigma}}
      {\|\Pi_{t,\sigma}^{\rm in}\mathbf v_{0,\sigma}\|},
 \qquad
 \mathbf u_{t,\sigma}:=
 s_1(\mathcal N_t^\sigma)^{-1}
 \mathcal N_t^\sigma\mathbf v_{t,\sigma}.
\end{equation}

For these axes, write
\[
 H_t:=\mathbf v_{t,+}^{\mathsf T}\mathbf u_{t,-}.
\]
Whenever standard SVD angle lifts are used, our relative-angle charts are
\[
\begin{gathered}
 \mathbf u_{t,\sigma}=R_{u_t^\sigma}e_1,\qquad
 \mathbf v_{t,\sigma}=R_{s_t^\sigma-\pi/2}e_1,\\
 u_t^\sigma-u_0^\sigma,\ s_t^\sigma-s_0^\sigma\in(-\pi/2,\pi/2),
 \qquad H_t=\cos\!\left(\frac\pi2-s_t^++u_t^-\right).
\end{gathered}
\]

The constant and polynomial in
Lemma~\ref{svd:lem:fixed-gap-oriented-stability} are chosen to apply also
to the following consequence.

\begin{corollary}
\label{svd:cor:fixed-gap-oriented-axes}
Under the hypotheses and smallness condition of
Lemma~\ref{svd:lem:fixed-gap-oriented-stability}, the following estimates
hold with the same \(C_1(r)\) and \(\mathfrak P_r^{\rm svd}\).
\begin{enumerate}
\item
The axes in \eqref{svd:eq:fixed-gap-oriented-axis-selection} are defined
on the whole homotopy and satisfy
\begin{equation}\label{svd:eq:fixed-gap-axis-stability}
 \max_{\sigma,t}
 \bigl(\|\mathbf v_{t,\sigma}-\mathbf v_{0,\sigma}\|_{C^r}
      +\|\mathbf u_{t,\sigma}-\mathbf u_{0,\sigma}\|_{C^r}\bigr)
 \le C_1(r)\eta\mathfrak P_r^{\rm svd}(K_1,\ldots,K_r).
\end{equation}
For any angle lifts in the specified relative-angle charts, the left-hand
side of \eqref{svd:eq:fixed-gap-axis-stability} may also be replaced by
\[
 \max_{\sigma,t}
 \bigl(\|u_t^\sigma-u_0^\sigma\|_{C^r}
      +\|s_t^\sigma-s_0^\sigma\|_{C^r}\bigr).
\]
\item
For \(1\le|\beta|\le r\),
\begin{equation}\label{svd:eq:fixed-gap-angle-derivatives}
 \max_{\sigma,t}\left(
 \|\partial^\beta u_t^\sigma\|_{L^\infty(\Omega)}
 +\|\partial^\beta s_t^\sigma\|_{L^\infty(\Omega)}\right)
 \le C_1(r)\mathfrak P_r^{\rm svd}(K_1,\ldots,K_r).
\end{equation}
\item
The overlap coefficients satisfy
\begin{equation}\label{svd:eq:fixed-gap-overlap-stability}
 \|H_t-H_0\|_{C^r(\Omega)}
 \le C_1(r)\eta\mathfrak P_r^{\rm svd}(K_1,\ldots,K_r),
 \qquad0\le t\le1.
\end{equation}
\end{enumerate}
\end{corollary}

\begin{proof}
By \eqref{svd:eq:fixed-gap-chart-smallness},
\(\|\Pi_{t,\sigma}^{\rm in}-\Pi_{0,\sigma}^{\rm in}\|<1/2\), and hence
\[
 \mathbf v_{0,\sigma}^{\mathsf T}\Pi_{t,\sigma}^{\rm in}
 \mathbf v_{0,\sigma}
 =\|\Pi_{t,\sigma}^{\rm in}\mathbf v_{0,\sigma}\|^2>\frac12.
\]
Thus \eqref{svd:eq:fixed-gap-oriented-axis-selection} is defined for
every \(t\) and satisfies
\(\mathbf v_{t,\sigma}^{\mathsf T}\mathbf v_{0,\sigma}>0\).
For the initial section, writing \(\Pi=\Pi_{0,\sigma}^{\rm in}\) and
\(\mathbf v=\mathbf v_{0,\sigma}\),
\[
 \Pi=\mathbf v\mathbf v^{\mathsf T},\qquad
 \mathbf v^{\mathsf T}\mathbf v=1,\qquad
 \partial_i\mathbf v=(\partial_i\Pi)\mathbf v.
\]
For a multiindex \(\beta\) with \(\beta_i\ge1\),
\[
 \partial^\beta\mathbf v
 =\sum_{\gamma\le\beta-e_i}
 \binom{\beta-e_i}{\gamma}
 (\partial^{\gamma+e_i}\Pi)
 (\partial^{\beta-e_i-\gamma}\mathbf v),
\]
where \(e_i\) in this formula is the \(i\)-th coordinate multiindex.
Induction in \(|\beta|\) bounds these derivatives by
\(C_1(r)\mathfrak P_r^{\rm svd}(K_1,\ldots,K_r)\). In
\eqref{svd:eq:fixed-gap-oriented-axis-selection}, the chain rule and its
difference form now give \eqref{svd:eq:fixed-gap-axis-stability}, retaining
the factor \(\eta\).  Moreover, with \(s_1=s_1(\mathcal N_t^\sigma)\),
\[
\begin{aligned}
 (\mathcal N_t^\sigma)^{\mathsf T}\mathcal N_t^\sigma\mathbf v_{t,\sigma}
 &=s_1^2\mathbf v_{t,\sigma},&
 \|\mathbf u_{t,\sigma}\|^2
 &=s_1^{-2}\mathbf v_{t,\sigma}^{\mathsf T}
 (\mathcal N_t^\sigma)^{\mathsf T}\mathcal N_t^\sigma
 \mathbf v_{t,\sigma}=1,\\
 \mathbf v_{t,\sigma}^{\mathsf T}\mathbf v_{0,\sigma}&>0,&
 \mathbf u_{t,\sigma}^{\mathsf T}\mathbf u_{0,\sigma}&>0.
\end{aligned}
\]
The last inequality follows from the same smallness condition after
enlarging \(C_1(r)\).  Positive determinant gives the standard oriented SVD;
the positive inner products fix the relative-angle charts throughout the
homotopy.  Differentiating in these charts gives the angle version of
\eqref{svd:eq:fixed-gap-axis-stability} and
\eqref{svd:eq:fixed-gap-angle-derivatives}.  Finally,
\[
 H_t-H_0=
 (\mathbf v_{t,+}-\mathbf v_{0,+})^{\mathsf T}\mathbf u_{t,-}
 +\mathbf v_{0,+}^{\mathsf T}(\mathbf u_{t,-}-\mathbf u_{0,-})
\]
gives \eqref{svd:eq:fixed-gap-overlap-stability} by Leibniz's rule.
\end{proof}

For holomorphic complex families, the corresponding normalized axis is
\(\Pi v_0/(v_0^{\mathsf T}\Pi v_0)^{1/2}\), with the square root
continued from the real base point.  Hermitian normalization and positivity
are used only in the real results above.

We next consider finite products in diagonal form.  We first estimate
their growth, then their spectral data and canonical frames.  The proofs
of the next two lemmas and their corollary are given in
Appendix~\ref{app:normalized-product-proof}.
Let \(d,L\in\mathbb N\) and
\(\mathbb F\in\{\mathbb R,\mathbb C\}\), and
write \(\mathbb F^\times:=\mathbb F\setminus\{0\}\). Let
\(\Omega\subset\mathbb F^d\) be open and fix \(m\in\mathbb N_0\).
The data below are assumed to be \(C^m\) if
\(\mathbb F=\mathbb R\) and holomorphic if \(\mathbb F=\mathbb C\).
Here multiindices belong to \(\mathbb N_0^d\), and
\(\|\cdot\|_\infty\) is the supremum norm on
\(\Omega\).  For \(0\le t<L\), let
\(\epsilon_t:\Omega\to\mathbb F^\times\) be nonvanishing, and, for
\(1\le t<L\), let \(H_t,\Theta_t:\Omega\to\mathbb F\) satisfy
\[
 W_t:=\begin{pmatrix}H_t&-\Theta_t\\
                      \Theta_t&H_t\end{pmatrix}
 \in {\rm SO}(2,\mathbb C),
 \qquad H_t^2+\Theta_t^2=1.
\]
Put \(D_t:=\operatorname{diag}(1,\epsilon_t)\) and
\begin{equation}\label{B:eq:finite-two-state-product}
\begin{aligned}
 \mathscr M_1&:=D_0,\\
 \mathscr M_L&:=D_{L-1}W_{L-1}\cdots D_1W_1D_0\quad(L\ge2),\\
 \mathscr H_L&:=\prod_{t=1}^{L-1}H_t,
\end{aligned}
\end{equation}
where an empty product equals one.  Assume that there are
\(0<\epsilon,\kappa\le1\) and \(K_0,\ldots,K_m\ge0\) such that,
throughout \(\Omega\),
\begin{equation}\label{B:eq:finite-two-state-input-bounds}
\begin{aligned}
 0<|\epsilon_t|&\le\epsilon\le1&& (0\le t<L),\\
 |H_t|&\ge\kappa&& (1\le t<L),\\
 |\Theta_t|&\le K_0&& (1\le t<L).
\end{aligned}
\end{equation}

\begin{lemma}
\label{B:lem:finite-two-state-growth}
Under \eqref{B:eq:finite-two-state-input-bounds}, the condition
\begin{equation}\label{B:eq:finite-two-state-local-smallness}
 \epsilon\kappa^{-2}\le10^{-4}
\end{equation}
implies
\begin{equation}\label{B:eq:finite-two-state-uniform-log-product}
 \left\|
  \log\|\mathscr M_L\|-\sum_{t=1}^{L-1}\log|H_t|
 \right\|_\infty
 \le 10L\epsilon\kappa^{-2}.
\end{equation}
\end{lemma}

We next estimate derivatives of the spectral data.  For the next lemma
and corollary, assume \(m\ge1\).
For a nonvanishing scalar function
\(f\), \(\partial^\beta\log f\), \(|\beta|\ge1\), denotes the
branch-independent derivative of any local logarithm of \(f\).  Let
\(\mathscr L_0:\Omega\to\mathbb F^\times\) be \(C^m\) if
\(\mathbb F=\mathbb R\) and holomorphic if \(\mathbb F=\mathbb C\), and put
\(M:=\mathscr L_0\mathscr M_L\).  Suppose that, for \(1\le r\le m\),
\begin{equation}\label{B:eq:finite-two-state-jet-input}
 \max_{\substack{1\le t<L\\|\beta|=r}}
 \left(\|\partial^\beta H_t\|_\infty
       +\|\partial^\beta\Theta_t\|_\infty\right)
 \le K_r,
 \qquad
 \max_{\substack{0\le t<L\\|\beta|=r}}
 \|\partial^\beta\log\epsilon_t\|_\infty\le K_r,
\end{equation}
where the first maximum is zero when \(L=1\).

For the derivative conclusions, put
\[
 \widehat M:=(\mathscr L_0\mathscr H_L)^{-1}M,\qquad
 \epsilon_{\rm in}:=\|\epsilon_0\|_\infty,\qquad
 \epsilon_{\rm out}:=\|\epsilon_{L-1}\|_\infty.
\]
We use the following conventions for the spectral data.  Wherever
\(\|\widehat M-e_1e_1^{\mathsf T}\|\le1/10\),
Criterion~\ref{crit:rank-one-spectral-separation} defines rank-one Riesz
projections \(\Pi_M^{\rm in}\) and \(\Pi_M^{\rm out}\) for
\(\widehat M^{\mathsf T}\widehat M\) and
\(\widehat M\widehat M^{\mathsf T}\), respectively, selecting the
eigenvalues in \(\{|\mu-1|<1/2\}\).  The two Gram matrices are similar;
write \(\nu_M\) for their common selected eigenvalue and set
\[
 \Lambda_M:=(\mathscr L_0\mathscr H_L)\nu_M^{1/2},\qquad
 |\nu_M^{1/2}-1|<\frac12.
\]
The square root is the branch on \(\{|z-1|<1/2\}\) with \(\sqrt1=1\).
Likewise, \(\log(\Lambda_M/(\mathscr L_0\mathscr H_L))\) is the logarithm
of \(\nu_M^{1/2}\) with the branch normalized by \(\log1=0\).
\begin{lemma}
\label{B:lem:differentiated-finite-two-state-product}
Under \eqref{B:eq:finite-two-state-input-bounds} and
\eqref{B:eq:finite-two-state-jet-input}, there are constants
\(C_2=C_2(m)\ge2\) and \(0<c_1=c_1(m)\le10^{-4}\),
depending only on \(m\), and, for \(1\le r\le m\), polynomials \(\mathcal P_r\) with nonnegative
coefficients depending only on \(r\), such that the following holds.
The same constants are used simultaneously for every order
\(0\le r\le m\). Put
\[
 \mathcal B_0:=1+K_0,\qquad
 \mathcal B_r:=1+K_r+\mathcal P_r(K_1,\ldots,K_{r-1}).
\]
The condition
\begin{equation}\label{B:eq:finite-two-state-smallness}
 \epsilon L^{C_2}\kappa^{-C_2}
 \left(2+K_0+\sum_{r=1}^mK_r\right)^{C_2}
 \le c_1
\end{equation}
implies the following conclusions.
\begin{enumerate}
\item
The spectral data \(\Pi_M^{\rm in},\Pi_M^{\rm out},\nu_M\), and
\(\Lambda_M\) are defined throughout \(\Omega\), and
\[
 \widehat M:\Omega\to{\rm GL}(2,\mathbb C),\qquad
 \|\widehat M-e_1e_1^{\mathsf T}\|_\infty\le\frac1{10}.
\]
\item
For every \(|\beta|=r\le m\),
\begin{equation}\label{B:eq:finite-two-state-projector-output}
\begin{aligned}
 \|\partial^\beta(\Pi_M^{\rm in}-e_1e_1^{\mathsf T})\|_\infty
 &\le C_2\epsilon_{\rm in}L^{C_2}\kappa^{-C_2}\mathcal B_r,\\
 \|\partial^\beta(\Pi_M^{\rm out}-e_1e_1^{\mathsf T})\|_\infty
 &\le C_2\epsilon_{\rm out}L^{C_2}\kappa^{-C_2}\mathcal B_r,\\
 \left\|\partial^\beta
 \log\frac{\Lambda_M}{\mathscr L_0\mathscr H_L}\right\|_\infty
 &\le C_2\epsilon L^{C_2}\kappa^{-C_2}\mathcal B_r.
\end{aligned}
\end{equation}
\item
For \(1\le r\le m\), these order-\(r\) derivatives depend affinely on
\(\partial^\beta\log\epsilon_t,\partial^\beta H_t,
\partial^\beta\Theta_t\), with the lower-order input data fixed.
The linear and constant terms separately satisfy the corresponding bound
in \eqref{B:eq:finite-two-state-projector-output}.
\end{enumerate}
\end{lemma}

The stronger smallness condition also gives
\[
 \epsilon\kappa^{-2}\le L\epsilon\kappa^{-2}\le c_1\le 10^{-4},
 \qquad 10L\epsilon\kappa^{-2}\le10c_1<10^{-2},
\]
so Lemma~\ref{B:lem:finite-two-state-growth} applies as well.
The spectral projections determine canonical orthogonal frames as follows.
Where \(|e_1^{\mathsf T}\Pi_M^{\rm in}e_1-1|<1/2\), set
\[
\begin{aligned}
 v_M&:=\frac{\Pi_M^{\rm in}e_1}
 {(e_1^{\mathsf T}\Pi_M^{\rm in}e_1)^{1/2}},
 &u_M&:=\Lambda_M^{-1}Mv_M,\\
 V_M&:=(v_M,\mathcal J_0v_M),
 &U_M&:=(u_M,\mathcal J_0u_M).
\end{aligned}
\]
Here \(\mathcal J_0\) is from
Definition~\ref{def:complex-special-orthogonal-group}; the denominator
uses the square-root branch on \(\{|z-1|<1/2\}\) with \(\sqrt1=1\).

The constants and polynomials in
Lemma~\ref{B:lem:differentiated-finite-two-state-product} are chosen so
that the following corollary holds with the same \(c_1\), \(C_2\), and
\(\mathcal B_r\).

\begin{corollary}
\label{B:cor:finite-two-state-frames}
Under the hypotheses and smallness condition of
Lemma~\ref{B:lem:differentiated-finite-two-state-product}, the following
conclusions hold.
\begin{enumerate}
\item
The canonical frames are defined throughout \(\Omega\), and
\[
 |e_1^{\mathsf T}\Pi_M^{\rm in}e_1-1|<\frac12,\qquad
 U_M,V_M\in{\rm SO}(2,\mathbb C).
\]
They give the exact normal form
\begin{equation}\label{B:eq:finite-two-state-exact-bilinear-normal-form}
 M=U_M
 \begin{pmatrix}
  \Lambda_M&0\\
  0&\det(M)\Lambda_M^{-1}
 \end{pmatrix}
 V_M^{\mathsf T},
\end{equation}
with
\[
 \Pi_M^{\rm in}=V_Me_1e_1^{\mathsf T}V_M^{\mathsf T},\qquad
 \Pi_M^{\rm out}=U_Me_1e_1^{\mathsf T}U_M^{\mathsf T}.
\]
\item
For every \(|\beta|=r\le m\),
\begin{equation}\label{B:eq:finite-two-state-frame-output}
\begin{aligned}
 \|\partial^\beta(V_M-I_2)\|_\infty
 &\le C_2\epsilon_{\rm in}L^{C_2}\kappa^{-C_2}\mathcal B_r,\\
 \|\partial^\beta(U_M-I_2)\|_\infty
 &\le C_2\epsilon_{\rm out}L^{C_2}\kappa^{-C_2}\mathcal B_r.
\end{aligned}
\end{equation}
The affine highest-order assertion of
Lemma~\ref{B:lem:differentiated-finite-two-state-product}, including the
separate bounds for the linear and constant terms, holds for these
estimates as well.

\item
If \(\det M\equiv1\), equivalently
\(\mathscr L_0^2\prod_{t=0}^{L-1}\epsilon_t\equiv1\), then
\eqref{B:eq:finite-two-state-exact-bilinear-normal-form} is an algebraic
singular value decomposition over the identity base map, with diagonal
entries \(\Lambda_M,\Lambda_M^{-1}\).
\item
For real data, \(\Lambda_M\) is the oriented lift of the dominant
singular value and \(|\Lambda_M|=\|M\|\).  If \(\Omega\subset\mathbb C^d\)
is invariant under conjugation and all inputs are real-symmetric, then
the spectral projections, \(\Lambda_M\), and the frames are real-symmetric.
\end{enumerate}
\end{corollary}

We now pass from finite products to Lyapunov exponents.  Let \((\Omega,\mu,S)\) be an invertible probability-preserving system and
let \(B:\Omega\to {\rm SL}(2,\mathbb C)\) be a measurable cocycle with
\(\log^+\|B\|\in L^1(\mu)\), and put
\[
\begin{aligned}
 B_n(x)&:=B(S^{n-1}x)\cdots B(x),\\
 L_n&:=\frac1n\int_\Omega\log\|B_n(x)\|\,d\mu(x),
 & L_\infty&:=\lim_{n\to\infty}L_n.
\end{aligned}
\]
Assume that \(B\) admits the algebraic singular value decomposition over
the base map from
Definition~\ref{def:complex-orthogonal-rank-one-normal-form}, with
measurable data \(U,V,\Lambda\), overlap coefficient \(H\), and
\(\log|\Lambda|\in L^1(\mu)\):
\begin{equation}\label{eq:rank-one-exact-normal-form}
 B(x)=U(Sx)
 \begin{pmatrix}\Lambda(x)&0\\0&\Lambda(x)^{-1}\end{pmatrix}
 V(x)^{\mathsf T}.
\end{equation}
Suppose, for some \(D\ge4\), that
\begin{equation}\label{eq:rank-one-conditioning}
 \|U(x)\|+\|U(x)^{-1}\|
 +\|V(x)\|+\|V(x)^{-1}\|\le D
\end{equation}
almost everywhere.  Put
\[
 \Lambda_*:=\operatorname*{ess\,inf}_\Omega|\Lambda|,\qquad
 \eta_*:=\operatorname*{ess\,inf}_\Omega|H|.
\]

\begin{lemma}
\label{lem:rank-one-asymptotic-LE-approximation}
For the normal form \eqref{eq:rank-one-exact-normal-form} satisfying
\eqref{eq:rank-one-conditioning}, there are absolute constants
\(0<\delta_{\rm ro}\le10^{-4}\) and \(C>0\) with the following property.
If \(\Lambda_*>1\), \(\eta_*>0\), and
\begin{equation}\label{eq:rank-one-small-parameter}
 \mathfrak r:=D^{24}\Lambda_*^{-2}\eta_*^{-2}\le\delta_{\rm ro},
\end{equation}
then
\begin{equation}\label{eq:rank-one-asymptotic-LE-main-term}
 \left|L_\infty-
 \int_\Omega(\log|\Lambda|+\log|H|)\,d\mu\right|
 \le C\mathfrak r.
\end{equation}
\end{lemma}

\subsection{Scale-dependent zero divisors with fixed total multiplicity}

Fix \(\rho>0\),
\(V\in C^\omega_\rho(\mathbb T,\mathbb R)\),
\(\alpha\notin\mathbb Q\), and fix a level \(j\).  Let
\(K_j\subsetneq\mathbb T\) be a
nonempty finite union of nondegenerate closed intervals and let \(S_j\) be its
first-return map
under \(T_\alpha\), and put
\[
 d\mu_j:=|K_j|^{-1}\mathbf1_{K_j^{\rm reg}}\,dx.
\]
Fix \(E\in\mathbb R\).  Let \(\mathcal P_j\) be a finite refinement,
modulo a finite boundary set, of the two-sided return-time partition
\(\mathcal P_{K_j}\) from Definition~\ref{fr:def:common-return-partition};
its finiteness is Lemma~\ref{fr:lem:finite-common-return-partition}.  Thus both return words
are constant on the regular part of every nondegenerate open interval
atom \(\mathcal B\in\mathcal P_j\); denote the forward return time there
by \(R_{j,\mathcal B}\).  Fix real lifts of these atoms and retain the
same symbols for the lifts.  For \(z\) in the strip of analyticity of \(V\), set
\[
 B_{j,\mathcal B}(z,E):=A_{R_{j,\mathcal B}}^V(z,E),
 \qquad
 S_{j,\mathcal B}z:=z+R_{j,\mathcal B}\alpha.
\]
Suppress \(E\) from the notation below.

\begin{definition}
\label{def:moving-divisors-persistent-total}
A choice of the factorization
maps, distinguished atoms, zeros, and quotient functions occurring in
\textup{(\ref{item:divisor-factorization})--(\ref{item:divisor-nondistinguished})} is called a \emph{level-\(j\) first-return
zero-divisor datum} and is denoted by \(\mathscr D_j\).

Fix
\[
 P\in\mathbb N,\quad P\ge1,\qquad
 0<h<\rho,\qquad d>0,\qquad 0<\kappa\le1.
\]
We say that \(\mathscr D_j\) is \emph{admissible with parameters
\((P;h,d,\kappa)\)}, and write
\[
 \mathscr D_j\in\operatorname{DivAdm}(P;h,d,\kappa),
\]
if the following conditions hold.

\begin{enumerate}
\item\label{item:divisor-factorization} \emph{Holomorphic algebraic singular value decomposition of the
return matrices.}
There exist measurable maps
\[
 U_j,V_j:K_j^{\rm reg}\to{\rm SO}(2,\mathbb C),
 \qquad \Lambda_j:K_j^{\rm reg}\to\mathbb C^\times,
\]
together with atomwise holomorphic real-symmetric extensions, denoted by
the same symbols.  For
\[
 \mathcal B[h]:=\overline{\mathcal B}+i[-h,h],
\]
\(V_j\) and \(\Lambda_j\) are defined on a neighborhood of
\(\mathcal B[h]\), while \(U_j\) is defined on neighborhoods of
\(\mathcal B[h]\) and \(S_{j,\mathcal B}(\mathcal B[h])\).  On each atom,
\begin{equation}\label{eq:moving-divisor-rank-one-normal-form}
 B_{j,\mathcal B}(z,E)=U_j(S_{j,\mathcal B}z)
 \begin{pmatrix}\Lambda_j(z)&0\\0&\Lambda_j(z)^{-1}\end{pmatrix}
 V_j(z)^{\mathsf T},
 \qquad z\in\mathcal B[h].
\end{equation}
Put
\[
 H_j(z):=e_1^{\mathsf T}V_j(z)^{\mathsf T}U_j(z)e_1.
\]
For \(F_j\in\{U_j,V_j,\Lambda_j,H_j\}\), write
\[
 F_j(x,y):=F_j(x+iy).
\]
Only germs from the same fixed return word (including its temporary
refinements) and corresponding source--image branches must agree.  Distinct
genuine return-time atoms carry separate germs.
If a complex point belongs to thickenings of distinct atoms, every expression
is interpreted in the germ attached to the atom under discussion; boundary
values are understood atomwise.
More precisely, for every fixed \(|y|\le h\), the source and return-image
germs determine single-valued measurable maps, modulo the finite atom
boundary set,
\[
 U_{j,y},V_{j,y}:K_j^{\rm reg}\to{\rm SO}(2,\mathbb C),
 \qquad
 \Lambda_{j,y}:K_j^{\rm reg}\to\mathbb C^\times,
\]
where, in the preceding atomwise notation,
\[
 U_{j,y}(x):=U_j(x,y),\qquad
 V_{j,y}(x):=V_j(x,y),\qquad
 \Lambda_{j,y}(x):=\Lambda_j(x,y).
\]
These maps satisfy, for almost every \(x\in K_j\),
\[
 B_{K_j,y}(x,E)=U_{j,y}(S_jx)
 \begin{pmatrix}\Lambda_{j,y}(x)&0\\0&\Lambda_{j,y}(x)^{-1}\end{pmatrix}
 V_{j,y}(x)^{\mathsf T}.
\]
Here \(U_{j,y}(S_jx)\) is the return-image germ attached to the same
fixed word as \(x\), and it is the incoming germ used at \(S_jx\).
Consequently the overlap coefficient of this algebraic singular value
decomposition over \(S_j\) is exactly \(H_j(x+iy)\).

\item\label{item:divisor-zeros} \emph{Real zeros and total multiplicity.}
There exist an integer \(\nu_j\ge1\), pairwise distinct atoms
\(\mathcal B_{j,a}\in\mathcal P_j\), and integers \(P_{j,a}\ge1\),
\(1\le a\le\nu_j\), such that
\begin{equation}\label{eq:moving-divisor-total-multiplicity}
 \sum_{a=1}^{\nu_j}P_{j,a}=P,
 \qquad
 P_{j,\max}:=\max_{1\le a\le\nu_j}P_{j,a}.
\end{equation}
Set
\[
 \mathcal P_j^{\rm dist}
 :=\{\mathcal B_{j,a}:1\le a\le\nu_j\},
 \qquad
 \mathcal P_j^{\rm nd}:=
 \mathcal P_j\setminus\mathcal P_j^{\rm dist}.
\]
For each \(a\), there are real zeros
\[
 c_{j,a,\ell}\in\mathcal B_{j,a},
 \qquad 1\le\ell\le P_{j,a},
\]
listed with multiplicity, and a holomorphic real-symmetric nonvanishing
function \(\mathfrak a_{j,a}\), defined on a neighborhood of
\(\mathcal B_{j,a}[h]\), such that
\begin{equation}\label{eq:moving-divisor-local-factorization}
 H_j(z)=
 \left(\prod_{\ell=1}^{P_{j,a}}(z-c_{j,a,\ell})\right)
 \mathfrak a_{j,a}(z),
 \qquad z\in\mathcal B_{j,a}[h].
\end{equation}

\item\label{item:divisor-boundary} \emph{Boundary separation of the zeros.}
The zeros satisfy
\begin{equation}\label{eq:moving-divisor-root-clearance}
 \min_{\substack{1\le a\le\nu_j\\1\le\ell\le P_{j,a}}}
 \operatorname{dist}
 (c_{j,a,\ell},\partial\mathcal B_{j,a})
 \ge d.
\end{equation}

\item\label{item:divisor-nondistinguished} \emph{Uniform nonvanishing on nondistinguished atoms.}
For every
\(\mathcal B\in\mathcal P_j^{\rm nd}\),
\begin{equation}\label{eq:moving-divisor-safe-H}
 |H_j(x+iy)|\ge\kappa,
 \qquad x\in\mathcal B,\qquad |y|\le h.
\end{equation}
If \(\mathcal P_j^{\rm nd}=\varnothing\), this condition is void and we
use \(\kappa=1\).
\end{enumerate}

Fix \(j_0\in\mathbb N\), \(j_0\ge1\).  For
\(P_{\rm loc}\in\mathbb N\), \(P_{\rm loc}\ge1\), a sequence
\(\{\mathscr D_j\}_{j\ge j_0}\) of the preceding level-\(j\) data is
said to have \emph{persistent
zero-divisor data of total multiplicity \(P_{\rm loc}\)} if there are
sequences
\[
 0<h_j<\rho,\qquad d_j>0,\qquad 0<\kappa_j\le1
\]
such that
\begin{equation}\label{eq:moving-divisor-admissible-hierarchy}
 \mathscr D_j\in
 \operatorname{DivAdm}(P_{\rm loc};h_j,d_j,\kappa_j),
 \qquad j\ge j_0.
\end{equation}
Only \(P_{\rm loc}\) is fixed across levels.  The number \(\nu_j\), the
distinguished atoms, the individual multiplicities \(P_{j,a}\), the zero
multisets, the factorization maps, and the parameters
\(h_j,d_j,\kappa_j\) may depend on \(j\).
\end{definition}

The next lemma passes from a real zero divisor to a complex one.  Let \(I\subset\mathbb R\) be a bounded closed nondegenerate interval, let
\(P\in\mathbb N\), and let \(H\) be real-symmetric and holomorphic on a
neighborhood of \(I+i[-h_0,h_0]\), where \(h_0>0\).  Suppose that the
real zero divisor of \(H\) on \(I\), counted with multiplicity, is
\(r_1,\ldots,r_P\).  Put
\[
 \mathfrak a(z):=
 \frac{H(z)}{\prod_{\ell=1}^P(z-r_\ell)},\qquad
 J_k:=\sup_{I+i[-h_0,h_0]}
 |\partial_z^{P+k}H|\quad(k=1,2),
\]
where the quotient at repeated roots is defined by removable
continuation. Assume that, for some \(0<\mu_-\le\mu_+\),
\begin{equation}\label{eq:real-cluster-real-quotient-bounds}
 \mu_-\le|\mathfrak a(x)|\le\mu_+,\qquad x\in I.
\end{equation}
With the convention \(\mu_-/J_1=+\infty\) for \(J_1=0\), set
\begin{equation}\label{eq:real-cluster-reduced-height}
 \widehat h
 :=\frac14\min\left\{h_0,
       \frac{(P+1)!\mu_-}{J_1}\right\}.
\end{equation}

\begin{lemma}
\label{lem:real-cluster-complex-divisor}
If \eqref{eq:real-cluster-real-quotient-bounds} holds, then the following
conclusions hold.
\begin{enumerate}
\item
The quotient \(\mathfrak a\) is real-symmetric and nonvanishing on
\(I+i[-\widehat h,\widehat h]\), and
\begin{equation}\label{eq:real-cluster-complex-quotient-bounds}
 \frac34\mu_-
 \le|\mathfrak a(z)|
 \le\frac54\mu_+,
 \qquad z\in I+i[-\widehat h,\widehat h].
\end{equation}
\item
The second derivative of its logarithm satisfies
\begin{equation}\label{eq:real-cluster-logarithmic-bound}
 \sup_{I+i[-\widehat h,\widehat h]}
 |\partial_z^2\log\mathfrak a|
 \le \frac{8J_2}{3(P+2)!\mu_-}
      +\frac{16J_1^2}{9((P+1)!)^2\mu_-^2}.
\end{equation}
\end{enumerate}
\end{lemma}

\begin{proof}
The multiplicity convention makes the apparent singularities in this
quotient removable.  For
\[
 \Delta_P:=\left\{(t_1,\ldots,t_P)\in[0,1]^P:
                 \sum_{\ell=1}^Pt_\ell\le1\right\},
 \qquad
 t_0:=1-\sum_{\ell=1}^Pt_\ell,
\]
the vanishing conditions identify \(\mathfrak a(z)\) with the confluent
divided difference \(H[r_1,\ldots,r_P,z]\).  The Hermite--Genocchi formula
\cite[Section~9]{deBoor2005} therefore gives
\[
 \mathfrak a(z)
 =\int_{\Delta_P}
   H^{(P)}\!\left(t_0z+\sum_{\ell=1}^Pt_\ell r_\ell\right)
   \,d\mathbf t.
\]
This remains valid when roots coincide; the usual proof by iterated
integration along line segments applies verbatim to holomorphic \(H\).
The simplex is compact and every displayed argument remains in
\(I+i[-h_0,h_0]\), so differentiation under the integral gives, for
\(k=1,2\),
\[
 \mathfrak a^{(k)}(z)
 =\int_{\Delta_P}t_0^k
 H^{(P+k)}\!\left(t_0z+\sum_{\ell=1}^Pt_\ell r_\ell\right)
 \,d\mathbf t.
\]
The simplex integrals give
\[
\begin{aligned}
 |\mathfrak a'(z)|
 &\le J_1\int_{\Delta_P}t_0\,d\mathbf t
 =\frac{J_1}{(P+1)!},\\
 |\mathfrak a''(z)|
 &\le J_2\int_{\Delta_P}t_0^2\,d\mathbf t
 =\frac{2J_2}{(P+2)!}.
\end{aligned}
\]
Using \eqref{eq:real-cluster-real-quotient-bounds} and the height choice
\eqref{eq:real-cluster-reduced-height}, for
\(z=x+iy\), \(|y|\le\widehat h\),
\[
 |\mathfrak a(z)-\mathfrak a(x)|
 \le\frac{J_1|y|}{(P+1)!}\le\frac14\mu_-.
\]
Hence \eqref{eq:real-cluster-complex-quotient-bounds} holds and
\(\mathfrak a\) is nonvanishing.  Finally,
\[
 \partial_z^2\log\mathfrak a
 =\frac{\mathfrak a''}{\mathfrak a}
  -\left(\frac{\mathfrak a'}{\mathfrak a}\right)^2
\]
gives \eqref{eq:real-cluster-logarithmic-bound}.
\end{proof}

\subsection{Poisson estimate for zeros of the overlap coefficient}

Let
\[
 \mathscr D_j\in
 \operatorname{DivAdm}(P;h_j,d_j,\kappa_j)
\]
be a level-\(j\) first-return zero-divisor datum in the sense of
Definition~\ref{def:moving-divisors-persistent-total}, and retain
\(P_{j,\max}\) from \eqref{eq:moving-divisor-total-multiplicity}.
Fix \(0<y_j^-<y_j^+\le\min\{h_j,1\}\), set
\(\Delta y_j=y_j^+-y_j^-\), and suppose that, for some \(D_j\ge4\),
\begin{equation}\label{eq:moving-pair-frame-conditioning}
\begin{aligned}
 \max_{\sigma\in\{+,-\}}
 \operatorname*{ess\,sup}_{x\in K_j}
 \bigl(&\|U_j(x,y_j^\sigma)\|+\|U_j(x,y_j^\sigma)^{-1}\|\\
 &+\|V_j(x,y_j^\sigma)\|+\|V_j(x,y_j^\sigma)^{-1}\|\bigr)
 \le D_j.
\end{aligned}
\end{equation}
Suppose also that,
for some \(\mathscr A_j\ge2\), the first two bounds below hold for every
\(\mathcal B_{j,a}\in\mathcal P_j^{\rm dist}\), while the third holds for
every \(\mathcal B\in\mathcal P_j^{\rm nd}\):
\begin{equation}\label{eq:moving-pair-logarithmic-bounds}
\begin{aligned}
 |\mathfrak a_{j,a}(z)|&\ge\mathscr A_j^{-1}
 && z\in\mathcal B_{j,a}[h_j],\\
 |\partial_z^2\log\mathfrak a_{j,a}(z)|&\le\mathscr A_j
 && z\in\mathcal B_{j,a}[h_j],\\
 |\partial_z^2\log H_j(z)|&\le\mathscr A_j
 && z\in\mathcal B[h_j].
\end{aligned}
\end{equation}
For \(\sigma\in\{+,-\}\), define
\[
\begin{aligned}
 \Lambda_{j,*}^\sigma
 &:=\operatorname*{ess\,inf}_{x\in K_j}|\Lambda_j(x,y_j^\sigma)|,\\
 \eta_j^\sigma
 &:=\min\{\mathscr A_j^{-1}(y_j^\sigma)^{P_{j,\max}},\kappa_j\}.
\end{aligned}
\]
Assume, for both signs, that
\begin{equation}\label{eq:moving-pair-expansion-smallness}
 \Lambda_{j,*}^\sigma>1,
 \qquad
 \mathfrak r_j^\sigma
 :=D_j^{24}(\Lambda_{j,*}^\sigma)^{-2}(\eta_j^\sigma)^{-2}
 \le\delta_{\rm ro},
 \qquad \sigma\in\{+,-\},
\end{equation}
where \(\delta_{\rm ro}\) is the absolute threshold in
Lemma~\ref{lem:rank-one-asymptotic-LE-approximation}, and put
\[
 \mathcal V_j:=\sup_{y_j^-\le y\le y_j^+}
 \int_{K_j}|\partial_y\log|\Lambda_j(x,y)||\,dx.
\]

\begin{lemma}
\label{lem:moving-base-complete-pair-Poisson}
For \(\mathscr D_j\in\operatorname{DivAdm}(P;h_j,d_j,\kappa_j)\)
satisfying \eqref{eq:moving-pair-frame-conditioning},
\eqref{eq:moving-pair-logarithmic-bounds}, and
\eqref{eq:moving-pair-expansion-smallness}, there is an absolute constant
\(C>0\) such that
\begin{equation}\label{eq:moving-pair-Poisson-secant}
\begin{aligned}
 &\left|
 \frac{|K_j|}{2\pi}
 \frac{L_{K_j}(E,y_j^+)-L_{K_j}(E,y_j^-)}{\Delta y_j}
 -\frac12P\right|\\
 &\quad\le C\mathcal V_j
 +C\mathscr A_j|K_j|y_j^+
 +CP\frac{y_j^+}{d_j}
 +C\frac{|K_j|}{\Delta y_j}(\mathfrak r_j^++\mathfrak r_j^-).
\end{aligned}
\end{equation}
\end{lemma}

\begin{proof}
For \(\sigma\in\{+,-\}\) and
\(\mathcal B_{j,a}\in\mathcal P_j^{\rm dist}\), the zero-divisor
factorization \eqref{eq:moving-divisor-local-factorization} and
\eqref{eq:moving-pair-logarithmic-bounds} give
\[
\begin{aligned}
 |H_j(x+iy_j^\sigma)|
 &\ge \mathscr A_j^{-1}
      \prod_{\ell=1}^{P_{j,a}}
      |x-c_{j,a,\ell}+iy_j^\sigma|\\
 &\ge \mathscr A_j^{-1}(y_j^\sigma)^{P_{j,a}}
 \ge \mathscr A_j^{-1}(y_j^\sigma)^{P_{j,\max}}.
\end{aligned}
\]
Here we used \(0<y_j^\sigma\le1\) and
\(P_{j,a}\le P_{j,\max}\).  For every
\(\mathcal B\in\mathcal P_j^{\rm nd}\),
\eqref{eq:moving-divisor-safe-H} gives
\(|H_j(x+iy_j^\sigma)|\ge\kappa_j\) for \(x\in\mathcal B\).
Consequently,
\[
 \operatorname*{ess\,inf}_{x\in K_j}
 |H_j(x+iy_j^\sigma)|\ge\eta_j^\sigma>0.
\]
Condition \ref{def:moving-divisors-persistent-total}\textup{(\ref{item:divisor-factorization})} identifies the endpoint branches with the
algebraic SVD of \(B_{K_j,y_j^\sigma}\) over \(S_j\), with overlap
\(H_j(\,\cdot\,+iy_j^\sigma)\).  The induced system is invertible and
probability-preserving by Lemma~\ref{fr:lem:regular-phases-induced-map}.
Finitely many compact atom closures and holomorphy give
\[
 \log^+\|B_{K_j,y_j^\sigma}\|,\ \log|\Lambda_j(\,\cdot\,,y_j^\sigma)|
 \in L^1(\mu_j).
\]
Thus \eqref{eq:moving-pair-expansion-smallness} and the preceding bounds
allow Lemma~\ref{lem:rank-one-asymptotic-LE-approximation} to be applied
at both heights:
\begin{equation}\label{eq:moving-pair-induced-LE-expansion}
\begin{aligned}
 \biggl|L_{K_j}(E,y_j^\sigma)
 &-\int_{K_j}\bigl(
   \log|\Lambda_j(x,y_j^\sigma)|
   +\log|H_j(x+iy_j^\sigma)|
   \bigr)\,d\mu_j(x)\biggr|\\
 &\le C\mathfrak r_j^\sigma,
 \qquad \sigma\in\{+,-\}.
\end{aligned}
\end{equation}
Since \(|K_j|\,d\mu_j=dx\) on \(K_j\), the factor \(|K_j|\) removes
the normalization of \(\mu_j\).  In particular,
\[
\begin{aligned}
 &\frac{|K_j|}{2\pi\Delta y_j}
 \left|\int_{K_j}
 \bigl(\log|\Lambda_j(x,y_j^+)|
       -\log|\Lambda_j(x,y_j^-)|\bigr)\,d\mu_j(x)\right|\\
 &\qquad\le\frac{\mathcal V_j}{2\pi}.
\end{aligned}
\]
For \(\mathcal B_{j,a}\in\mathcal P_j^{\rm dist}\),
\[
 \partial_y\log|H_j(x+iy)|
 =\sum_{\ell=1}^{P_{j,a}}
   \frac{y}{(x-c_{j,a,\ell})^2+y^2}
 +\partial_y\log|\mathfrak a_{j,a}(x+iy)|.
\]
If \(F\) is nonvanishing, holomorphic, and real-symmetric on a neighborhood
of one of the closed atom rectangles, shrink that neighborhood to a
simply connected zero-free neighborhood and choose a holomorphic logarithm
of \(F\) there.  Then
\[
\begin{aligned}
 \partial_y\log|F(x+iy)|
 &=\operatorname{Re}\!\left(i(\log F)'(x+iy)\right),\\
 \partial_y^2\log|F(x+iy)|
 &=-\operatorname{Re}\!\left((\log F)''(x+iy)\right),
 \qquad
 \left.\partial_y\log|F(x+iy)|\right|_{y=0}=0.
\end{aligned}
\]
Consequently, if \(|\partial_z^2\log F|\le\mathscr A_j\), then
\begin{equation}\label{eq:moving-pair-regular-factor-vertical-bound}
 |\partial_y\log|F(x+iy)||
 \le\mathscr A_j|y|.
\end{equation}
Apply \eqref{eq:moving-pair-regular-factor-vertical-bound} atomwise with
\[
 F=
 \begin{cases}
  \mathfrak a_{j,a},&\mathcal B=\mathcal B_{j,a}\in\mathcal P_j^{\rm dist},\\
  H_j,&\mathcal B\in\mathcal P_j^{\rm nd},
 \end{cases}
\]
using \eqref{eq:moving-pair-logarithmic-bounds} and, in the second case,
the zero-free bound \eqref{eq:moving-divisor-safe-H}.
For \(\mathcal B_{j,a}=(a_{j,a},b_{j,a})\) and every root
\(c=c_{j,a,\ell}\),
\[
 \int_{a_{j,a}}^{b_{j,a}}
 \frac{y}{(x-c)^2+y^2}\,dx
 =\arctan\frac{b_{j,a}-c}{y}
  +\arctan\frac{c-a_{j,a}}{y},
\]
and \eqref{eq:moving-divisor-root-clearance} implies
\[
 0\le \pi-
 \int_{a_{j,a}}^{b_{j,a}}
 \frac{y}{(x-c)^2+y^2}\,dx
 \le \frac{y}{b_{j,a}-c}+\frac{y}{c-a_{j,a}}
 \le\frac{2y}{d_j}.
\]
Consequently, \eqref{eq:moving-divisor-total-multiplicity} gives
\[
 \left|
 \int_{K_j}\partial_y\log|H_j(x+iy)|\,dx-\pi P
 \right|
 \le \mathscr A_j|K_j|y+2P\frac{y}{d_j}.
\]
Integrating in \(y\) between the two heights and dividing by
\(\Delta y_j\) gives the explicit averaged estimate
\[
\begin{aligned}
\biggl|\frac1{\Delta y_j}\int_{K_j}
 &\bigl(\log|H_j(x+iy_j^+)|
       -\log|H_j(x+iy_j^-)|\bigr)\,dx-\pi P\biggr|\\
&\le \mathscr A_j|K_j|y_j^++2P\frac{y_j^+}{d_j}.
\end{aligned}
\]
Combining these estimates with
\eqref{eq:moving-pair-induced-LE-expansion} gives
\[
\begin{aligned}
 \left|\frac{|K_j|\,[L_{K_j}(E,y_j^+)-L_{K_j}(E,y_j^-)]}
 {2\pi\Delta y_j}-\frac P2\right|
 &\le \frac{\mathcal V_j+\mathscr A_j|K_j|y_j^+
                   +2Py_j^+/d_j}{2\pi}\\
 &\quad+\frac{C|K_j|}{\Delta y_j}
       (\mathfrak r_j^++\mathfrak r_j^-).
\end{aligned}
\]
This is \eqref{eq:moving-pair-Poisson-secant}.
\end{proof}

\subsection{The acceleration criterion}

Fix a total multiplicity \(P_{\rm loc}\in\mathbb N\), \(P_{\rm loc}\ge1\).
Let \(\rho>0\), \(V\in C^\omega_\rho(\mathbb T,\mathbb R)\),
\(\alpha\notin\mathbb Q\), and \(E\in\mathbb R\).  Suppose there are
an integer \(j_0\ge1\), scale-dependent return sets \(K_j\), positive scale parameters
\(m_j\to\infty\), and first-return zero-divisor data
\(\{\mathscr D_j\}_{j\ge j_0}\) with persistent total multiplicity
\(P_{\rm loc}\) and parameters \((h_j,d_j,\kappa_j)\), in the sense of
Definition~\ref{def:moving-divisors-persistent-total}.  Assume there exist
\(\chi>0\) and sequences
\[
 D_j\ge4,\qquad \mathscr A_j\ge2,
\]
such that
\begin{equation}\label{eq:moving-criterion-frame-conditioning}
 \sup_{|y|\le h_j}\operatorname*{ess\,sup}_{x\in K_j}
 \bigl(\|U_j(x,y)\|+\|U_j(x,y)^{-1}\|
 +\|V_j(x,y)\|+\|V_j(x,y)^{-1}\|\bigr)
 \le D_j,
\end{equation}
\begin{equation}\label{eq:moving-criterion-subexponential-data}
 \varepsilon_j
 :=\frac{\log D_j+\log\mathscr A_j+\log^+\kappa_j^{-1}
 +\log^+d_j^{-1}+\log^+h_j^{-1}}{m_j}
 \longrightarrow0,
\end{equation}
\begin{equation}\label{eq:moving-criterion-expansion}
 \operatorname*{ess\,inf}_{x\in K_j}|\Lambda_j(x,y)|
 \ge e^{\chi m_j},\qquad |y|\le h_j,
\end{equation}
and that the three bounds in
\eqref{eq:moving-pair-logarithmic-bounds} hold on their respective
distinguished and nondistinguished atoms.  Assume that there are constants
\(\beta_\pm\) satisfying
\begin{equation}\label{eq:moving-criterion-beta-budget}
 0<\beta_+<\beta_-,
 \qquad 2\chi-2P_{\rm loc}\beta_- -\beta_+>0.
\end{equation}
Define
\[
 y_j^+:=e^{-\beta_+m_j},\qquad y_j^-:=e^{-\beta_-m_j}.
\]
For all sufficiently large \(j\), one has \(y_j^+\le h_j\);
increase \(j_0\) accordingly.  Assume in addition that
\begin{equation}\label{eq:moving-criterion-modulus-smallness}
 \mathcal V_j
 :=\sup_{y_j^-\le y\le y_j^+}
 \int_{K_j}|\partial_y\log|\Lambda_j(x,y)||\,dx
 \longrightarrow0.
\end{equation}

\begin{theorem}
\label{thm:persistent-divisor-acceleration}
With the return sets and scales fixed above, let
\(\{\mathscr D_j\}_{j\ge j_0}\) have persistent total multiplicity
\(P_{\rm loc}\) as in
Definition~\ref{def:moving-divisors-persistent-total}.  If
\eqref{eq:moving-criterion-frame-conditioning},
\eqref{eq:moving-criterion-subexponential-data},
\eqref{eq:moving-criterion-expansion},
\eqref{eq:moving-pair-logarithmic-bounds},
\eqref{eq:moving-criterion-beta-budget}, and
\eqref{eq:moving-criterion-modulus-smallness} hold, then
\begin{equation}\label{eq:persistent-divisor-global-acceleration}
 \omega(\alpha,E;V)=\frac12P_{\rm loc}.
\end{equation}
\end{theorem}

\begin{proof}
Term by term, \eqref{eq:moving-criterion-subexponential-data} gives
\[
 D_j,\mathscr A_j,\kappa_j^{-1},d_j^{-1},h_j^{-1}
 \le e^{\varepsilon_jm_j}.
\]
Hence \(y_j^+\le h_j\) for all sufficiently large \(j\).
Use \(\Lambda_{j,*}^\sigma,\eta_j^\sigma,\mathfrak r_j^\sigma\), and
\(\Delta y_j\) from the setup of
Lemma~\ref{lem:moving-base-complete-pair-Poisson}, with \(P=P_{\rm loc}\).
Then
\[
 \Delta y_j
 =e^{-\beta_+m_j}
  \left(1-e^{-(\beta_--\beta_+)m_j}\right)
 \ge\frac12e^{-\beta_+m_j}
\]
for all sufficiently large \(j\).  Since
\(P_{j,\max}\le P_{\rm loc}\) by
\eqref{eq:moving-divisor-total-multiplicity},
\[
 \eta_j^-
 \ge \exp\{-(P_{\rm loc}\beta_-+\varepsilon_j)m_j\}.
\]
Since \(y_j^+>y_j^-\),
\[
 \eta_j^+\ge\eta_j^-
 \ge e^{-(P_{\rm loc}\beta_-+\varepsilon_j)m_j}.
\]
By \eqref{eq:moving-criterion-expansion},
\[
 \Lambda_{j,*}^\pm\ge e^{\chi m_j}>1.
\]
These bounds give
\[
 \mathfrak r_j^-+\mathfrak r_j^+
 \le2\exp\left\{-\left(
 2\chi-2P_{\rm loc}\beta_--26\varepsilon_j
 \right)m_j\right\}.
\]
By \eqref{eq:moving-criterion-beta-budget}, this tends to zero and
eventually \(\mathfrak r_j^\pm\le\delta_{\rm ro}\).
We may therefore apply
Lemma~\ref{lem:moving-base-complete-pair-Poisson} with \(P=P_{\rm loc}\).
Since \(|K_j|\le1\),
\begin{equation}\label{eq:moving-criterion-secant-error}
 \frac{|K_j|}{\Delta y_j}(\mathfrak r_j^-+\mathfrak r_j^+)
 \le4\exp\left\{-\left(
 2\chi-2P_{\rm loc}\beta_--\beta_+-26\varepsilon_j
 \right)m_j\right\}
 \longrightarrow0.
\end{equation}
The other errors in
Lemma~\ref{lem:moving-base-complete-pair-Poisson} satisfy
\[
 \mathscr A_j|K_j|y_j^+
 \le e^{-(\beta_+-\varepsilon_j)m_j}\longrightarrow0,
 \qquad
 \frac{y_j^+}{d_j}
 \le e^{-(\beta_+-\varepsilon_j)m_j}\longrightarrow0,
\]
and \(\mathcal V_j\to0\) by
\eqref{eq:moving-criterion-modulus-smallness}.  Thus
\eqref{eq:moving-criterion-secant-error} and the preceding estimates make
the right-hand side of \eqref{eq:moving-pair-Poisson-secant} tend to zero.
Therefore
Lemma~\ref{lem:moving-base-complete-pair-Poisson} and the exact inducing
identity \eqref{eq:induced-global-LE-identity} give
\begin{equation}\label{eq:moving-criterion-global-secant-limit}
\begin{aligned}
 \frac{L(\alpha,E,y_j^+;V)-L(\alpha,E,y_j^-;V)}
      {2\pi\Delta y_j}
 &=\frac{|K_j|}{2\pi}
   \frac{L_{K_j}(E,y_j^+)-L_{K_j}(E,y_j^-)}{\Delta y_j}\\
 &\longrightarrow\frac12P_{\rm loc}.
\end{aligned}
\end{equation}
Convexity of \(f(y):=L(\alpha,E,y;V)\) gives
\[
 f'_+(0)
 \le\frac{f(y_j^+)-f(y_j^-)}{y_j^+-y_j^-}
 \le f'_-(y_j^+)\longrightarrow f'_+(0).
\]
Thus \eqref{eq:moving-criterion-global-secant-limit} yields
\[
 \omega(\alpha,E;V)=\frac{f'_+(0)}{2\pi}=\frac12P_{\rm loc}.
\]

\end{proof}

\subsection{Proof of the direct asymptotic Lyapunov exponent approximation}

\begin{proof}[Proof of Lemma~\ref{lem:rank-one-asymptotic-LE-approximation}]
All constants below are absolute and uniform under
the hypotheses of the lemma.  Replacing \(\Omega\) by the intersection
of all forward and backward \(S\)-iterates of the full-measure set on
which the normal form and all bounds hold, we may assume that they hold
simultaneously along every finite orbit.  Since the column norms of
\(U\) and \(V\) are at most \(D\),
\[
 |H(x)|=|v_1(x)^{\mathsf T}u_1(x)|\le D^2,
 \qquad \eta_*\le D^2.
\]
Put
\[
 W(x):=V(x)^{\mathsf T}U(x)
 =\begin{pmatrix}H(x)&-\Theta(x)\\
                  \Theta(x)&H(x)\end{pmatrix}
 \in{\rm SO}(2,\mathbb C),
 \qquad
 \kappa_*:=\min\{1,\eta_*\}.
\]
The frame bound gives \(|\Theta|\le D^2\).  For fixed \(n\ge1\) and
almost every \(x\), write
\[
 \Lambda_k:=\Lambda(S^kx),\qquad
 \epsilon_k:=\Lambda_k^{-2},\qquad
 D_k:=\operatorname{diag}(1,\epsilon_k),\qquad
 W_k:=W(S^kx).
\]
The definition \eqref{eq:rank-one-small-parameter} of \(\mathfrak r\),
the inequalities \(\eta_*\le D^2\), \(D\ge4\), and
\(\mathfrak r\le\delta_{\rm ro}\le10^{-4}\) imply
\[
 \Lambda_*^{-2}\kappa_*^{-2}\le
 \begin{cases}
  D^{-24}\mathfrak r,&\eta_*\le1,\\
  D^{-20}\mathfrak r,&\eta_*>1,
 \end{cases}
 \le\mathfrak r\le10^{-4}.
\]
Thus \eqref{B:eq:finite-two-state-local-smallness} holds. For fixed
\(x\), regard these scalar data as constant functions on \(\Omega=\mathbb C\).  Take
\(d=1\), \(L=n\), \(m=0\), \(K_0=D^2\),
\(\epsilon=\Lambda_*^{-2}\), and \(\kappa=\kappa_*\).  Then
Lemma~\ref{B:lem:finite-two-state-growth} applies to
\[
 \mathscr M_n(x):=D_{n-1}W_{n-1}\cdots D_1W_1D_0.
\]
Direct multiplication of \eqref{eq:rank-one-exact-normal-form} gives
\[
 B_n(x)=U(S^nx)
 \left(\prod_{k=0}^{n-1}\Lambda_k\right)
 \mathscr M_n(x)V(x)^{\mathsf T}.
\]
By \eqref{B:eq:finite-two-state-uniform-log-product},
\[
 \left|\log\|\mathscr M_n\|-\sum_{k=1}^{n-1}\log|H(S^kx)|\right|
 \le10n\Lambda_*^{-2}\kappa_*^{-2}\le10n\mathfrak r.
\]
For invertible \(P,Q\),
\[
 |\log\|PNQ\|-\log\|N\||
 \le\max\{\log(\|P\|\|Q\|),
            \log(\|P^{-1}\|\|Q^{-1}\|)\},
\]
so the two exterior frames contribute at most \(2\log D\).  Consequently,
\begin{equation}\label{eq:rank-one-full-finite-product-estimate}
\begin{aligned}
 &\left|
 \log\|B_n(x)\|
 -\sum_{k=0}^{n-1}\log|\Lambda(S^kx)|
 -\sum_{k=1}^{n-1}\log|H(S^kx)|
 \right|\\
 &\hspace{55mm}\le Cn\mathfrak r+C\log D.
\end{aligned}
\end{equation}
Integrating \eqref{eq:rank-one-full-finite-product-estimate}, dividing by
\(n\), and using \(S\)-invariance gives
\[
 \left|L_n-
 \int_\Omega\log|\Lambda|\,d\mu
 -\left(1-\frac1n\right)\int_\Omega\log|H|\,d\mu
 \right|
 \le C\mathfrak r+\frac{C\log D}{n}.
\]
Here \(\log|H|\in L^1(\mu)\) because
\(\eta_*\le|H|\le D^2\). Letting \(n\to\infty\) proves
\eqref{eq:rank-one-asymptotic-LE-main-term}.
\end{proof}
\section{Reduction of the main theorem}
\label{sec:black-box-main-reduction}

\noindent We reduce Theorem~\ref{theorem.complete-main} to the construction
of first-return zero divisors with total multiplicity \(2s\) on
positive-measure spectral sets.  Proposition~\ref{prop:perturbative-hierarchy-package}
states the precise data required by the acceleration criterion, and its
application below gives the main theorem.  The construction of these data
occupies Sections~\ref{sec:AMO-upper-edge-seed}--\ref{sec:real-CF-propagation}.

Retain the fixed parameters of Theorem~\ref{theorem.complete-main} and
the family \(V_{\varepsilon,\delta}\) in
\eqref{eq:perturbative-potential-family}.
\begin{proposition}
\label{prop:perturbative-hierarchy-package}
There exists an open prevalent set
\(\mathcal P_{\mathrm{strat}}\subset
C^\omega_\rho(\mathbb T,\mathbb R)\) such that every
\(\delta\in\mathcal P_{\mathrm{strat}}\) admits
\(\varepsilon_{\mathrm{hier}}(\delta)>0\) with the following property.
For every
\[
 0<\varepsilon<\varepsilon_{\mathrm{hier}}(\delta),
 \qquad 1\le s\le N,
\]
there is a compact set
\(\mathcal Y_s^{\varepsilon,\delta}\subset\mathbb R\) with the following properties.
\begin{enumerate}
\item One has
\begin{equation}\label{eq:black-box-positive-spectral-set}
 \operatorname{Leb}(\mathcal Y_s^{\varepsilon,\delta})>0,
 \qquad
 \mathcal Y_s^{\varepsilon,\delta}
 \subset\Sigma(V_{\varepsilon,\delta},\alpha).
\end{equation}
\item For every \(E\in\mathcal Y_s^{\varepsilon,\delta}\), there are
\(j_0=j_0(E,\varepsilon,\delta,s)\ge1\), \(\chi>0\), and data
\[
 (K_j,m_j,\mathscr D_j,D_j,\mathscr A_j)_{j\ge j_0}
\]
satisfying all hypotheses of
Theorem~\ref{thm:persistent-divisor-acceleration} with
\(V=V_{\varepsilon,\delta}\), the fixed frequency \(\alpha\), and the
present energy \(E\).
Their level-\(j\) first-return zero-divisor data \(\mathscr D_j\),
\(j\ge j_0\), have persistent total multiplicity \(P_{\rm loc}\) in the
sense of Definition~\ref{def:moving-divisors-persistent-total}, where
\begin{equation}\label{eq:black-box-total-multiplicity}
 P_{\rm loc}=2s,
\end{equation}
and at every level there are exactly \(2s\) distinguished atoms, each
containing one simple real zero of its overlap coefficient; there are no other
local real zeros.  The height exponents in
\eqref{eq:moving-criterion-beta-budget} may be chosen as
\[
 \beta_+=\frac{\chi}{16s},\qquad
 \beta_-=\frac{\chi}{8s}.
\]
\end{enumerate}
\end{proposition}

The proof of Proposition~\ref{prop:perturbative-hierarchy-package} is
deferred to the end of Section~\ref{sec:real-CF-propagation}.

\begin{proof}[Proof of Theorem~\ref{theorem.complete-main}]
Let \(\mathcal P_{\mathrm{strat}}\) and
\(\varepsilon_{\mathrm{hier}}(\delta)\) be given by
Proposition~\ref{prop:perturbative-hierarchy-package}.  Fix
\(\delta\in\mathcal P_{\mathrm{strat}}\),
put \(\varepsilon_0(\delta):=\varepsilon_{\mathrm{hier}}(\delta)\), and
let \(0<\varepsilon<\varepsilon_0(\delta)\) and
\(1\le s\le N\). Theorem~\ref{thm:persistent-divisor-acceleration} and
\eqref{eq:black-box-total-multiplicity} give
\[
\begin{aligned}
 E\in\mathcal Y_s^{\varepsilon,\delta}
 &\Longrightarrow E\in\Sigma(V_{\varepsilon,\delta},\alpha),
 \quad\omega(\alpha,E;V_{\varepsilon,\delta})=\tfrac12(2s)=s,\\
 \operatorname{Leb}\{E\in\Sigma(V_{\varepsilon,\delta},\alpha):
 \omega(\alpha,E;V_{\varepsilon,\delta})=s\}
 &\ge\operatorname{Leb}(\mathcal Y_s^{\varepsilon,\delta})>0.
\end{aligned}
\]
The same proposition supplies the openness and prevalence of
\(\mathcal P_{\mathrm{strat}}\).
\end{proof}

\section{Upper-edge initial-scale analysis for the almost Mathieu operator}
\label{sec:AMO-upper-edge-seed}

\noindent We collect the almost Mathieu upper-edge estimates needed to
start the perturbative construction.  Beginning with the state at the
spectral edge, we state expansion and derivative bounds for the return
matrices, a quadratic normal form for the central overlap, and separation
estimates outside the central domain.  We then specify the uniform
initial-scale data used in the subsequent induction.  The upper-edge
estimates are proved in Appendix~\ref{app:AMO-upper-edge}.

Throughout this section, fix
\[
 \lambda>1,\qquad \beta\in\DC(\kappa_\beta,\tau),\qquad
 \kappa_\beta>0,\quad\tau>1,\quad C_I>0.
\]
In the AMO estimates below, \(c,C\) may depend only on
\((\lambda,\beta,\kappa_\beta,\tau,C_I)\).
Put
\[
 E_+(\beta):=\sup\Sigma(2\lambda\cos(2\pi\,\cdot),\beta),
 \qquad \gamma:=\log\lambda.
\]
All returns below are taken with respect to \(T_\beta\).
For integers \(n\ge2\), define
\begin{equation}\label{fr:eq:reciprocal-return-window}
 N_n:=\lceil n^{C_I}\rceil,\qquad
 I_n:=[-N_n^{-1},N_n^{-1}]\subset\mathbb T.
\end{equation}
The matrices \(B_{I_n}^\pm\) are those of
Subsection~\ref{subsec:induced-exponents} for
\(V(x)=2\lambda\cos(2\pi x)\) and frequency \(\beta\).
For exponents \(B_1>C_I\) and \(B_2>0\), to be chosen below, write
\begin{equation}\label{fr:eq:central-polynomial-bidomain}
 \widetilde I_n:=(-n^{-B_1},n^{-B_1}),\qquad
 \mathcal E_n:=[E_+(\beta)-n^{-B_2},E_+(\beta)+n^{-B_2}].
\end{equation}
For \(\mathcal A\in\mathcal P_{I_n}\) and
\(\zeta=(\zeta_x,\zeta_E)\in\mathbb N_0^2\), use the notation
\[
 r_+:=p_{\mathcal A},\qquad r_-:=q_{\mathcal A},\qquad
 \partial^\zeta_{x,E}:=\partial_x^{\zeta_x}\partial_E^{\zeta_E},
 \quad |\zeta|:=\zeta_x+\zeta_E.
\]

For candidate exponents
\(0<\sigma_{\rm cmp}\le(2\tau)^{-1}\) and integers \(B\ge4\), write
\begin{equation}\label{B:eq:Rtau-definition}
 R_\tau(m):=m^{1-\sigma_{\rm cmp}}
             (\log(m+2))^B,\qquad m\ge1.
\end{equation}
The first proposition gives the state at the spectral edge.

\begin{proposition}
\label{fr:prop:AMO-edge-state}
There are exponents \(\sigma_{\rm cmp},B\) in the stated ranges, a constant
\(C\ge1\), and a real-valued \(u\in\ell^2(\mathbb Z)\) such that
\begin{equation}\label{eq:AMO-seed-edge-state}
 H_{\lambda,\beta,0}u=E_+(\beta)u,
 \qquad u_0=1,
 \qquad u_{-j}=u_j,
\end{equation}
and
\begin{equation}\label{eq:AMO-seed-edge-state-decay}
 |u_j|+\|u\|_{\ell^2(\{k\in\mathbb Z:k>j\})}
 \le C\exp\{-\gamma j+CR_\tau(j+2)\},
 \qquad j\ge0.
\end{equation}
The pair \((\sigma_{\rm cmp},B)\) may be chosen so that
Propositions~\ref{fr:prop:deterministic-AMO-seed}--\ref{fr:prop:AMO-return-nondegeneracy}
hold for every finite derivative order \(r_\dagger\ge3\).
\end{proposition}

Fix this pair \((\sigma_{\rm cmp},B)\) and an order \(r_\dagger\ge3\).
All numbered constants, thresholds, and exponents in the next three
propositions may depend only on
\[
 (\lambda,\beta,\kappa_\beta,\tau,C_I,r_\dagger).
\]
They are chosen simultaneously, so the return windows, energy interval,
and central domain are the same in all three statements.
We first control the return matrices and their singular data.

\begin{proposition}
\label{fr:prop:deterministic-AMO-seed}
There are \(n_0=n_0(C_I,r_\dagger)\ge2\), \(c>0\), \(C\ge1\), and exponents
\[
 B_1>C_I,\qquad
 B_2>0,\qquad D_1=D_1(C_I,r_\dagger)>0,
\]
with
\begin{equation}\label{fr:eq:companion-edge-gap-hierarchy}
 B_2>D_1+2,
\end{equation}
such that the following conclusions and those of
Propositions~\ref{fr:prop:AMO-central-normal-form}
and~\ref{fr:prop:AMO-return-nondegeneracy} hold for every \(n\ge n_0\).
\begin{enumerate}
\item For each \(\mathcal A\in\mathcal P_{I_n}\), the matrices \(B_{I_n}^\pm\)
have simple top singular values. After fixing a real lift of
\(\mathcal A\), their standard SVD data
\((\lambda_{I_n}^\pm,u_{I_n}^\pm,s_{I_n}^\pm)\) admit continuous
coherent angle lifts on \(\mathcal A\times\mathcal E_n\).  For each sign, modulo independent additions of
\(2\pi\), the pair \((u_{I_n}^\pm,s_{I_n}^\pm)\) is unique up to
simultaneous addition of \(\pi\).

\item There are \(C_{10}=C_{10}(C_I,r_\dagger)>0\) and
\(D_{24}=D_{24}(C_I,r_\dagger)>0\) such that, for every atom
\(\mathcal A\), uniformly on \(\mathcal A\times\mathcal E_n\),
\(\sigma\in\{+,-\}\), and \(1\le|\zeta|\le r_\dagger\),
\[
\begin{aligned}
 \lambda_{I_n}^\sigma
 &\ge c\exp\{\gamma r_\sigma-CR_\tau(r_\sigma+2)\},\\
 |\partial^\zeta_{x,E}\lambda_{I_n}^\sigma|
 &\le C_{10}(r_\sigma+2)^{D_{24}}
       \lambda_{I_n}^\sigma,
\end{aligned}
\]
and
\begin{equation}\label{fr:eq:AMO-return-angle-derivatives}
 |\partial^\zeta_{x,E}u_{I_n}^\sigma|
 +|\partial^\zeta_{x,E}s_{I_n}^\sigma|
 \le C_{10}(r_\sigma+2)^{D_{24}}.
\end{equation}
The constants are independent of \(n,\mathcal A,\sigma\).
\end{enumerate}
\end{proposition}

We now fix these constants and domains.  The overlap on the central
atom has a quadratic normal form; its orientation is fixed by the
positive energy derivative below.  The SVD notation is that preceding
\eqref{fr:eq:complete-return-joining-convention}.

\begin{proposition}
\label{fr:prop:AMO-central-normal-form}
For every \(n\ge n_0\), the following statements hold.
\begin{enumerate}
\item The point \(0\) belongs to
\(\operatorname{Reg}(I_n)\), and
\[
 m_n:=r_{I_n}^+(0)=r_{I_n}^-(0),\qquad
 cN_n^{1/\tau}\le m_n\le N_n.
\]
There is a central atom \(\mathcal A_n^0\in\mathcal P_{I_n}\) with
\[
 \widetilde I_n\subset\mathcal A_n^0,\qquad
 p_{\mathcal A_n^0}=q_{\mathcal A_n^0}=m_n.
\]
\item On \(\widetilde I_n\times\mathcal E_n\), the standard SVD angles admit
continuous oriented lifts for which
\[
 \partial_E\cos G_{I_n,\mathcal A_n^0}(0,E_+(\beta))>0.
\]
There are real \(C^2\) functions
\(\widetilde C_{1,n},\widetilde C_{2,n},\widetilde R_n\) such that
\begin{equation}\label{fr:eq:central-polynomial-normal-form}
 \cos G_{I_n,\mathcal A_n^0}(x,E)
 =\widetilde C_{1,n}(x,E)x^2
  +\widetilde C_{2,n}(x,E)(E-E_+(\beta))
  +\widetilde R_n(x,E),
\end{equation}
and, on the same domain,
\[
\begin{gathered}
 c\le\widetilde C_{1,n},\widetilde C_{2,n}\le C,\\
 \|\widetilde C_{1,n}\|_{C^2}+\|\widetilde C_{2,n}\|_{C^2}
 \le n^{D_1},\qquad
 \|\widetilde R_n\|_{C^2}\le e^{-c\gamma m_n}.
\end{gathered}
\]
\item On \(\widetilde I_n\times\mathcal E_n\), one has
\begin{equation}\label{fr:eq:central-mismatch-transversality}
 c\le\partial_E\cos G_{I_n,\mathcal A_n^0}\le C,
 \qquad
 c\le\partial_x^2\cos G_{I_n,\mathcal A_n^0}\le C.
\end{equation}
\end{enumerate}
\end{proposition}

Every atom other than \(\mathcal A_n^0\) is called exterior.  The next
proposition controls the overlap away from the central domain and the
Dirichlet restrictions associated with every return atom.

\begin{proposition}
\label{fr:prop:AMO-return-nondegeneracy}
For the same constants and domains, fix \(n\ge n_0\) and
\(\mathcal A\in\mathcal P_{I_n}\). The following statements hold.
\begin{enumerate}
\item One has
\begin{equation}\label{fr:eq:exterior-cos-lower}
\begin{gathered}
 |\cos G_{I_n,\mathcal A}(x,E)|\ge c n^{-D_1},\\
 x\in\mathcal A\cap(I_n\setminus\widetilde I_n),\qquad E\in\mathcal E_n.
\end{gathered}
\end{equation}
\item If \(p,q\) are the forward and backward return-length labels of
\(\mathcal A\), then for either
\[
 (r,z)=(p,x)\quad\text{or}\quad(r,z)=(q,x-q\beta),
 \qquad x\in\mathcal A,
\]
one has
\begin{equation}\label{fr:eq:companion-stopped-edge-gap}
 E_+(\beta)-\max\operatorname{spec}
 H_{\lambda,\beta,z}^{[0,r-1]}
 \ge cn^{-D_1}.
\end{equation}
For each choice, the Dirichlet determinant
\(P_r(z,E):=\det(E-H_{\lambda,\beta,z}^{[0,r-1]})\) obeys
\begin{equation}\label{fr:eq:companion-return-determinant-persistence}
\begin{gathered}
 P_r(z,E)\ge\exp\{\gamma r-CR_\tau(r+2)\}>0,\\
 E\in\mathbb R,\qquad |E-E_+(\beta)|\le cn^{-D_1}.
\end{gathered}
\end{equation}
\end{enumerate}
\end{proposition}

In Appendix~\ref{app:AMO-upper-edge}, the positive dual state gives the
finite-volume gap in Lemma~\ref{amo:lem:single-resonance-gap} and the
central Schur estimates in Proposition~\ref{amo:prop:central-schur} and
Corollary~\ref{amo:cor:central-schur-curvature}. The positively oriented
Schur--SVD comparison of Lemma~\ref{amo:lem:schur-svd} then yields the
return-matrix estimates above.

Let \((q_j)_{j\ge1}\) be the continued-fraction denominators of
\(\alpha\).  Fix a return-window exponent \(C_I>0\).  Let
\(0<\sigma_{\rm cmp}\le(2\tau)^{-1}\) and \(B\ge4\) be supplied by
Proposition~\ref{fr:prop:AMO-edge-state}.  By
\eqref{eq:effective-frequency-Diophantine}, its parameters here are
\[
 \beta=\widehat\alpha,\qquad
 \kappa_\beta=\kappa N^{-\tau},\qquad
 E_+(\beta)=E_+.
\]
The scale in Proposition~\ref{fr:prop:deterministic-AMO-seed} is arbitrary;
we apply it at
\[
 n=Q=q_\ell(\alpha),\qquad T_\beta=T_{\widehat\alpha}.
\]
Thus \(Q\) need not be a continued-fraction denominator of \(\beta\).
Retain \(R_\tau\) from
\eqref{B:eq:Rtau-definition}.  For each finite \(d\ge3\), set
\begin{equation}\label{eq:fixed-background-data}
 \mathfrak D_{\rm AMO}^{(d)}
 :=(\mathfrak b,C_I,d).
\end{equation}
The derived frequency and the choices \(\sigma_{\rm cmp},B\) above are fixed by \(\mathfrak b\).
For each application, fix its largest required order \(d_{\max}\) and
denote the common output of
Propositions~\ref{fr:prop:deterministic-AMO-seed}--\ref{fr:prop:AMO-return-nondegeneracy} by
\(\operatorname{Data}^{(d_{\max})}\), and set
\[
 \operatorname{Data}^{(d)}
 :=\left.\operatorname{Data}^{(d_{\max})}\right|_{\text{derivative orders }\le d},
 \qquad 3\le d\le d_{\max}.
\]
These restrictions share their windows and structural exponents, and all
bounds depend only on \(\mathfrak D_{\rm AMO}^{(d_{\max})}\). Moreover,
\(R_\tau(m)=o(m)\).

\subsection{Uniform initial-scale data}

The following definition specifies the initial-scale data constructed in
Proposition~\ref{fr:prop:prevalent-multiwell-seed} and propagated by
Proposition~\ref{fr:prop:Young-CF-reinduction-exclusion}.
See Figure~\ref{fig:initial-scale-domain-hierarchy} for the domain hierarchy.

\begin{figure}[H]
\centering
\resizebox{\textwidth}{!}{%
\begin{tikzpicture}[
 x=0.47cm,y=0.47cm,
 every node/.style={font=\scriptsize},
 component/.style={draw=RoyalBlue,fill=RoyalBlue!12,thick,rounded corners=1pt},
 atom/.style={draw=Orange,very thick},
 smallset/.style={draw=ForestGreen,very thick},
 itinerary/.style={draw=gray!75,densely dashed},
 selected/.style={draw=Orange,thick}
]
% Parameter cell and fixed-energy phase fiber.
\node[font=\small\bfseries] at (7,6.55)
 {(i) Parameter cell and one lifted phase fiber};
\draw[gray!75] (0.75,5.75)--(13.25,5.75);
\draw[Orange,line width=4pt,line cap=round] (5.15,5.75)--(8.25,5.75);
\foreach \x in {0.75,2.15,3.70,5.15,8.25,10.10,11.65,13.25}
  \draw[gray!75] (\x,5.58)--(\x,5.92);
\node[Orange!80!black,above] at (6.70,5.84) {$\mathcal C$};
\node[anchor=east] at (0.65,5.75) {$\mathcal J_Q$};
\node at (7,5.12) {$\mathcal K_Q^0$};
\draw[gray!70] (0.45,4.45)--(13.55,4.45);
\path[component] (0.75,4.05) rectangle (3.05,4.85);
\path[component] (4.75,4.05) rectangle (9.25,4.85);
\path[component] (10.95,4.05) rectangle (13.25,4.85);
\node at (1.90,4.45) {$\mathcal U_{Q,1}$};
\node at (7.00,4.45) {$\mathcal U_{Q,a}$};
\node at (12.10,4.45) {$\mathcal U_{Q,2s}$};
\node at (3.90,4.45) {$\cdots$};
\node at (10.10,4.45) {$\cdots$};
\draw[itinerary] (5.55,4.05)--(5.55,4.85);
\draw[itinerary] (6.10,4.05)--(6.10,4.85);
\draw[itinerary] (8.15,4.05)--(8.15,4.85);
\draw[itinerary] (8.65,4.05)--(8.65,4.85);
\draw[atom] (6.10,3.55)--(8.15,3.55);
\draw[Orange] (6.10,3.39)--(6.10,3.71)
              (8.15,3.39)--(8.15,3.71);
\node[Orange!80!black,anchor=west] at (8.35,3.55) {$\nu_a$};
\draw[smallset] (6.65,2.55)--(7.60,2.55);
\draw[ForestGreen] (6.65,2.39)--(6.65,2.71)
                   (7.60,2.39)--(7.60,2.71);
\fill[BrickRed] (7.125,2.55) circle (2.1pt);
\node[ForestGreen!55!black,anchor=west] at (7.80,2.55)
 {$I_{q_\ell,a}^{(r)}(E)$};
\node[BrickRed] at (7.13,1.90) {$c_{\ell,a}(E)$};
\node at (7,0.85)
 {$I_{q_\ell,a}^{(r)}(E)\Subset\nu_a
       \subset\mathcal U_{Q,a}\subset\mathcal K_Q^0$};

% Phase--energy panel.
\node[font=\small\bfseries] at (23.25,6.55)
 {(ii) Fixed-itinerary cells over $\mathcal C'\subset\mathcal C$};
\draw[->] (16.10,1.25)--(16.10,6.05);
\node[left] at (16.10,5.75) {$E$};
\draw[->] (16.10,1.25)--(30.00,1.25) node[right] {$x$};
\draw[RoyalBlue,thick] (16.45,1.55) rectangle (29.55,5.85);
\node[anchor=west,RoyalBlue!70!black] at (16.65,5.55)
 {$\nu_a\times\mathcal C'$};
\path[fill=Orange!10]
 (21.00,1.55) .. controls (20.80,2.85) and (21.50,4.45) .. (21.25,5.85)
 -- (25.05,5.85)
 .. controls (25.30,4.35) and (24.65,2.80) .. (24.90,1.55)
 -- cycle;
\draw[itinerary]
 (18.10,1.55) .. controls (18.40,2.90) and (17.90,4.35) .. (18.45,5.85);
\draw[selected]
 (21.00,1.55) .. controls (20.80,2.85) and (21.50,4.45) .. (21.25,5.85);
\draw[selected]
 (24.90,1.55) .. controls (24.65,2.80) and (25.30,4.35) .. (25.05,5.85);
\draw[itinerary]
 (27.35,1.55) .. controls (27.00,2.95) and (27.70,4.45) .. (27.25,5.85);
\draw[ForestGreen!45,line width=10pt,opacity=.50]
 (23.05,1.55) .. controls (22.85,2.75) and (23.65,4.30) .. (23.45,5.85);
\draw[BrickRed,very thick]
 (23.05,1.55) .. controls (22.85,2.75) and (23.65,4.30) .. (23.45,5.85);
\node at (19.35,3.65) {$\mathscr B_{\ell,\eta_1}$};
\node[Orange!80!black] at (23.05,4.95) {$\mathscr B_{\ell,a}$};
\node at (28.25,4.80) {$\mathscr B_{\ell,\eta_2}$};
\draw[->,BrickRed]
 (26.15,4.25)--(23.45,4.15)
 node[pos=0,anchor=west] {$x=c_{\ell,a}(E)$};
\draw[->,ForestGreen!55!black]
 (26.40,2.85)--(23.20,2.75)
 node[pos=0,anchor=west] {$I_{q_\ell,a}^{(r)}(E)$};
\draw[<->] (23.15,2.05)--(24.80,2.05);
\node[above] at (23.98,2.05)
 {$\ge Q^{-(B_{11}+1)}$};
\end{tikzpicture}%
}
\caption{Domain hierarchy in
Definition~\ref{fr:def:Young-seed-package} (schematic, one real lift); the
double arrow is the fiberwise boundary margin.}
\label{fig:initial-scale-domain-hierarchy}
\end{figure}
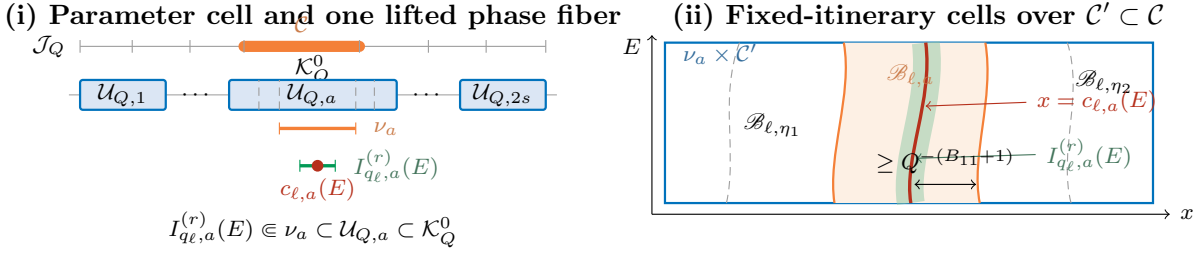

For the initial-scale construction, write
\[
\begin{aligned}
 \mathfrak D_{\rm seed}:=\bigl(&r_\dagger,B,\tau,\gamma;\,
 B_{10},B_{11},B_{13},a_{\rm seed},\eta_{\rm seed},\\
 &D_7,D_9,D_6,D_{11},
 D_8,M_*^{\rm seed},C_8\bigr),
\end{aligned}
\]
where \(r_\dagger,B\in\mathbb N\), \(r_\dagger\ge5\), \(B\ge4\),
\(\tau>1\), \(\gamma>0\), all remaining entries are positive,
\(a_{\rm seed}>1\), and
\(1\le C_8<\infty\).
Fix such initial-scale parameters \(\mathfrak D_{\rm seed}\), a real-analytic
potential \(V:\mathbb T\to\mathbb R\),
\(\alpha\in\DC(\kappa,\tau)\) with \(\kappa>0\), an integer \(s\ge1\),
a continued-fraction denominator \(Q=q_\ell\), and a nondegenerate compact
energy interval \(\mathcal J_Q\).
Put
\[
 w_\ell:=Q^{-B_{10}},
 \qquad
 \mathscr W_\ell:=[w_\ell,2w_\ell].
\]
\begin{definition}
\label{fr:def:Young-seed-package}
Scale-\(Q\) initial data with \(2s\) simple zeros and parameters
\(\mathfrak D_{\rm seed}\) consist of the following objects and estimates.
There are maps
\[
 c_{\ell,a}\in C^{r_\dagger-1}(\mathcal J_Q,\mathbb T),
 \qquad 1\le a\le2s,
\]
independent of \(r\), and fixed real lifts
\(\widetilde c_{\ell,a}\in
C^{r_\dagger-1}(\mathcal J_Q,\mathbb R)\).
For every \(r\in\mathscr W_\ell\), there is a partition of
\(\mathcal J_Q\) into at most \(C_8Q^{D_{11}}\) compact intervals
\(\mathcal C\) with disjoint interiors, on which all return lengths and
oriented lifts of the SVD angles are fixed.  On each cell, the zero
branches below are the restrictions of these global maps.
There are pairwise disjoint fixed nondegenerate closed intervals
\(\mathcal U_{Q,a}\) satisfying
\[
 \operatorname{dist}_{\mathbb T}
 \bigl(c_{\ell,a}(E),\mathbb T\setminus
       \operatorname{int}\mathcal U_{Q,a}\bigr)
 \ge Q^{-B_{13}},
 \qquad E\in\mathcal J_Q,\quad 1\le a\le2s.
\]
Put
\[
 \mathcal K_Q^0:=\bigcup_{a=1}^{2s}\mathcal U_{Q,a}.
\]
Let \(\pi_{\mathbb T}:\mathbb R\to\mathbb T\) be the quotient map.
For \(r\in\mathscr W_\ell\), define the energy-dependent smaller return set
\[
 I_{q_\ell,a}^{(r)}(E)
 :=\pi_{\mathbb T}
 [\widetilde c_{\ell,a}(E)-r,\widetilde c_{\ell,a}(E)+r],
 \qquad
 K_{q_\ell}^{(r)}(E)
 :=\bigcup_{a=1}^{2s}I_{q_\ell,a}^{(r)}(E).
\]
Require
\[
 K_{q_\ell}^{(r)}(E)\Subset\mathcal K_Q^0,
 \qquad (E,r)\in\mathcal J_Q\times\mathscr W_\ell.
\]
Conditions \textup{(\ref{item:seed-factorization})--(\ref{item:seed-velocities})} are uniform in the radius,
parameter cell, return-time atom, and component labels. Write
\(\partial^\zeta=\partial_{x,E}^\zeta\).

\begin{enumerate}
\item\label{item:seed-factorization} \emph{Exact return factorization and expanding initial matrices.}
For each fixed pair \((E,r)\), restrict the return dynamics to
\[
 \operatorname{Reg}\bigl(
   \{\mathcal K_Q^0,K_{q_\ell}^{(r)}(E)\}\bigr).
\]
On this set, let \(S_Q^0\) be the first-return map to
\(\mathcal K_Q^0\).  If
\(x\in K_{q_\ell}^{(r)}(E)\) belongs to this regular set, put
\[
\begin{aligned}
 L^+&:=\min\{L\ge1:(S_Q^0)^Lx\in K_{q_\ell}^{(r)}(E)\},\\
 L^-&:=\min\{L\ge1:(S_Q^0)^{-L}x\in K_{q_\ell}^{(r)}(E)\},
\end{aligned}
\]
and let \(x_i:=(S_Q^0)^ix\) for \(-L^-\le i\le L^+\).
The first-return matrices to the smaller set then have the exact chronological factorizations
\[
\begin{aligned}
 B_{K_{q_\ell}^{(r)}}^+(x,E)
 &=B_{\mathcal K_Q^0}^+(x_{L^+-1},E)\cdots
   B_{\mathcal K_Q^0}^+(x_0,E),\\
 B_{K_{q_\ell}^{(r)}}^-(x,E)
 &=B_{\mathcal K_Q^0}^+(x_{-1},E)\cdots
   B_{\mathcal K_Q^0}^+(x_{-L^-},E).
\end{aligned}
\]
Writing \(R_i:=r_{\mathcal K_Q^0}^+(x_i)\), every factor above is the
genuine fixed-word matrix \(B_{\mathcal K_Q^0}^+(x_i,E)=A_{R_i}^V(x_i,E)\),
and
\[
 r_{K_{q_\ell}^{(r)}(E)}^+(x)
   =\sum_{i=0}^{L^+-1}R_i,
 \qquad
 r_{K_{q_\ell}^{(r)}(E)}^-(x)
   =\sum_{i=-L^-}^{-1}R_i.
\]
On every subdomain on which the complete itinerary is fixed, these
factorizations extend continuously to its closure.
There exists
\[
 0\le\omega_Q^{\rm seed}
 \le C_8Q^{-a_{\rm seed}\eta_{\rm seed}}(\log(Q+2))^B
\]
such that every factor occurring in these two identities has lattice
length \(R\) satisfying
\[
 Q^{a_{\rm seed}}\le R\le C_8Q^{B_{13}\tau},
\]
and top singular value \(\lambda_{\rm seed}\) satisfying
\[
 \lambda_{\rm seed}
 \ge\exp\{(\gamma-\omega_Q^{\rm seed})R\}.
\]
For each such factor \(B_{\rm seed}\) of length \(R\), with SVD data
\((\lambda_{\rm seed},u_{\rm seed},s_{\rm seed})\), and every
\(1\le|\zeta|\le r_\dagger\),
\[
 |\partial^\zeta\log\lambda_{\rm seed}|
 +|\partial^\zeta u_{\rm seed}|+|\partial^\zeta s_{\rm seed}|
 +\lambda_{\rm seed}^{-1}\|\partial^\zeta B_{\rm seed}\|
 \le C_8(R+2)^{M_*^{\rm seed}}.
\]

For every two-sided return-time atom
\(\nu\in\mathcal P_{\mathcal K_Q^0}\), put
\begin{equation}\label{fr:eq:Young-seed-joining-definition}
 H_{{\rm seed},\nu}
 :=\cos G_{\mathcal K_Q^0,\nu}.
\end{equation}

\item\label{item:seed-transversality} \emph{Uniform transversality of the simple zeros.}
For every \(r\in\mathscr W_\ell\), every parameter cell \(\mathcal C\), and
every \(1\le a\le2s\), there is a return-time atom \(\nu_a\) such that
\[
 \mathscr U_{Q,a}^{\rm seed}
 :=\nu_a\times\mathcal C
 \subset \mathcal U_{Q,a}\times\mathcal C.
\]
Choose a real lift
\(\widetilde{\mathcal U}_{Q,a}\subset\mathbb R\) of
\(\mathcal U_{Q,a}\) and the contained lift
\(\widetilde\nu_a\subset\widetilde{\mathcal U}_{Q,a}\) of \(\nu_a\),
translating both by the same integer so that
\(\widetilde c_{\ell,a}|_{\mathcal C}\) takes values in
\(\widetilde\nu_a\), and put
\[
 \widetilde{\mathscr U}_{Q,a}^{\rm seed}
 :=\widetilde\nu_a\times\mathcal C
 \subset\widetilde{\mathcal U}_{Q,a}\times\mathcal C,
 \qquad
 \widetilde I_{q_\ell,a}^{(r)}(E)
 :=[\widetilde c_{\ell,a}(E)-r,\widetilde c_{\ell,a}(E)+r].
\]
For \(\widetilde{\mathscr U}\subset\mathbb R\times\mathcal C\), write
\(\widetilde{\mathscr U}_E
 :=\{x\in\mathbb R:(x,E)\in\widetilde{\mathscr U}\}\).  Then
\[
 \widetilde I_{q_\ell,a}^{(r)}(E)
 \Subset\widetilde\nu_a
 =(\widetilde{\mathscr U}_{Q,a}^{\rm seed})_E
 \quad(E\in\mathcal C).
\]
On this lifted domain, the initial-scale overlap coefficient in
\eqref{fr:eq:Young-seed-joining-definition} has the exact factorization
\[
 H_{{\rm seed},\nu_a}(\pi_{\mathbb T}x,E)
 =(x-\widetilde c_{\ell,a}(E))\mathfrak a_{{\rm seed},a}(x,E),
\]
where
\[
 Q^{-D_7}
 \le|\mathfrak a_{{\rm seed},a}(x,E)|
 \le Q^{D_7}.
\]
Moreover,
\[
 \|H_{{\rm seed},\nu_a}\circ(\pi_{\mathbb T},\operatorname{id})
   \|_{C^{r_\dagger}(\widetilde{\mathscr U}_{Q,a}^{\rm seed})}
 +\|\mathfrak a_{{\rm seed},a}\|_{C^{r_\dagger-1}}
 +\|\widetilde c_{\ell,a}\|_{C^{r_\dagger-1}}
 \le Q^{D_8}.
\]

\item\label{item:seed-endpoint-error} \emph{Fixed-order endpoint approximation error and domain compatibility.}
For each fixed \((r,\mathcal C)\), let
\(\{\mathscr B_{\ell,\eta}:\eta\in\mathfrak I(r,\mathcal C)\}\)
be the connected fixed-itinerary domains in
\[
 \{(x,E):E\in\mathcal C,\ x\in K_{q_\ell}^{(r)}(E)\}.
\]
Their energy fibers form a temporary refinement of the genuine two-sided
return atoms of \(K_{q_\ell}^{(r)}(E)\); both return lengths and the ordered
\(\mathcal K_Q^0\)-itinerary are fixed. Let
\(\nu(\eta)\in\mathcal P_{\mathcal K_Q^0}\) be the fixed old atom
containing the common base point of the backward and forward chains.
Orient the new forward input and backward output axes by positive inner
product with their respective old endpoint axes, which are continuous
on each child interval by \textup{(\ref{item:seed-transversality})}.
On pieces of the same genuine atom, the total return words agree;
these orientations therefore glue their singular data across the
temporary cuts. The overlap on the genuine atom restricts to
\[
 h_{\ell,\eta}^{(r)}
 :=\cos G_{K_{q_\ell}^{(r)},\mathscr B_{\ell,\eta}}.
\]
Then, for every \(\eta\in\mathfrak I(r,\mathcal C)\),
\[
 \|h_{\ell,\eta}^{(r)}-H_{{\rm seed},\nu(\eta)}\|
      _{C^{r_\dagger}(\mathscr B_{\ell,\eta})}
 \le C_8Q^{D_9}
       e^{-C_8^{-1}Q^{a_{\rm seed}}}.
\]
A parameter subcell \(\mathcal C'\subset\mathcal C\) is
\emph{boundary-admissible} if, for each \(a\), there is a first-return
phase--energy domain \(\mathscr B_{\ell,a}\) such that
\[
\begin{aligned}
 \mathscr B_{\ell,a}
 &\subset \mathscr U_{Q,a}^{\rm seed}|_{\mathcal C'},\\
 \{(c_{\ell,a}(E),E):E\in\mathcal C'\}
 &\subset\mathscr B_{\ell,a},\\
 \inf_{E\in\mathcal C'}
 \operatorname{dist}_{\mathbb T}
 \bigl(c_{\ell,a}(E),
       \partial\bigl((\mathscr B_{\ell,a})_E\bigr)\bigr)
 &\ge Q^{-(B_{11}+1)}.
\end{aligned}
\]
Here \((\mathscr B_{\ell,a})_E\) is a genuine return atom, whose boundary
excludes temporary cuts. On such a cell the domains
\(\mathscr B_{\ell,a}\) are called distinguished.

\item\label{item:seed-exterior} \emph{Uniform exterior lower bound.}
For every \(r\in\mathscr W_\ell\), every parameter cell \(\mathcal C\), and every
\(\mathcal A\in\mathcal P_{\mathcal K_Q^0}
\setminus\{\nu_a:1\le a\le2s\}\),
\[
 \inf_{(x,E)\in\mathcal A\times\mathcal C}
 |H_{{\rm seed},\mathcal A}(x,E)|\ge Q^{-D_7},
 \qquad
 \|H_{{\rm seed},\mathcal A}\|_{C^{r_\dagger}(\mathcal A\times\mathcal C)}
 \le Q^{D_8},
\]
and, for every \(1\le a\le2s\) and every
\((x,E)\in\mathscr U_{Q,a}^{\rm seed}\) with
\(x\notin I_{q_\ell,a}^{(r)}(E)\), require
\[
 |H_{{\rm seed},\nu_a}(x,E)|
 \ge Q^{-D_7}\operatorname{dist}_{\mathbb T}
      (x,c_{\ell,a}(E)).
\]

\item\label{item:seed-velocities} \emph{Separation of the energy derivatives of zero branches.}
For every pair \(a\ne b\), there is
\(\sigma_{ab}\in\{+1,-1\}\) such that the fixed relative lift
\[
 f_{\ell,a,b}:=\widetilde c_{\ell,a}-\widetilde c_{\ell,b}
\]
satisfies
\[
 2Q^{-D_6}
 \le \sigma_{ab}\partial_E f_{\ell,a,b}
 \le Q^{D_6},
 \quad E\in\mathcal J_Q,
 \qquad
 \operatorname{osc}_{\mathcal J_Q}f_{\ell,a,b}<\frac18.
\]
\end{enumerate}

Finally, require the uniform margin
\begin{equation}\label{fr:eq:Young-seed-angle-bootstrap-margin}
 \sup_{\nu}
 \|H_{{\rm seed},\nu}\|_{C^{r_\dagger}}
 \le \frac14 Q^{D_8},
\end{equation}
where the supremum is over the return-time atoms occurring on all
phase--energy domains, distinguished or otherwise, for all radii and
parameter cells, and each norm is taken on the corresponding atom--cell
domain (in its fixed real lift when required).
\end{definition}
\section{From the almost Mathieu initial scale to data with multiple
nondegenerate maxima}
\label{sec:multiwell-seed-construction}

\noindent We construct the admissible initial-scale data of
Definition~\ref{fr:def:Young-seed-package} from the almost Mathieu
upper-edge estimates.  We first establish the matrix estimates needed to
control the perturbed folds and their first-return induction.  The
limiting Schur-complement variations then separate the fold heights on an
open prevalent set of perturbations, yielding energy intervals with
exactly \(2s\) simple return zeros and the required pairwise transversality
for every \(1\le s\le N\).

Throughout this section, \(N,\lambda,\rho\),
\(X=C^\omega_\rho(\mathbb T,\mathbb R)\), and the order-indexed data
\(\mathfrak D_{\rm AMO}^{(d)}\) from \eqref{eq:fixed-background-data} are fixed; in
every occurrence of this notation, its superscript is the largest displayed
auxiliary order.
Write \(V_0(x):=2\lambda\cos(2\pi Nx)\).

\subsection{Internal return and matrix tools}
\paragraph{\textbf{Return times.}}
The following estimate bounds first returns in terms of interval size.
\begin{lemma}
\label{B:lem:return-time-poly-two-sided}
Let \(T_\alpha x=x+\alpha\), where \(\alpha\notin\mathbb Q\), and assume
that \(\kappa>0\) and \(\tau\ge1\) satisfy
\[
 \|m\alpha\|_{\mathbb T}\ge\kappa |m|^{-\tau},
 \qquad m\in\mathbb Z\setminus\{0\}.
\]
Put \(\ell_0:=\min\{1/4,\kappa/4\}\). There are \(c,C>0\), depending only on
\((\kappa,\tau)\), such that, for every interval
\(I\subset\mathbb T\) with \(0<|I|\le\ell_0\) and every
\[
 x\in I\setminus\bigcup_{m\in\mathbb Z}
 (T_\alpha)^m(\partial_{\mathbb T}I),
\]
the first positive forward and backward return times satisfy
\[
 c|I|^{-1/\tau}\le r_I^\pm(x)\le C|I|^{-\tau}.
\]
\end{lemma}

\begin{proof}
For \(r=r_I^+(x)\),
\[
 \kappa r^{-\tau}\le\|r\alpha\|_\T\le|I|
 \quad\Longrightarrow\quad r\ge(\kappa/|I|)^{1/\tau}.
\]
Let $q_s$ be the first continued-fraction denominator at least
$4|I|^{-1}$. For $1\le j\le q_s$,
\[
 \left\|j\alpha-\frac{jp_s}{q_s}\right\|_\T
 <q_{s+1}^{-1}\le q_s^{-1}.
\]
The orbit segment is $3/(2q_s)$-dense, so every gap in it has length
at most $3/q_s<|I|$. Every translate meets $I$, and
\[
 r_I^+(x)\le q_s\le\kappa^{-1}q_{s-1}^{\tau}
 \le C|I|^{-\tau}.
\]
Apply the same argument to $-\alpha$.
\end{proof}

For a matrix \(\mathcal Q\in M_2(\mathbb C)\) with distinct eigenvalues
\(\lambda_+\ne\lambda_-\), the Riesz spectral projection onto the
\(\lambda_+\)-eigenspace is
\begin{equation}\label{svd:eq:explicit-projector-formula}
 P_+(\mathcal Q)=
 \frac{\mathcal Q-\lambda_-I_2}{\lambda_+-\lambda_-}.
\end{equation}

\paragraph{\textbf{Products with separated singular directions.}}
For the next lemma and its proof, fix the following matrix chain.
Let \(d\ge1\), \(L\ge1\), and \(s\ge3\) be integers, let
\(\Omega\subset\mathbb R^d\) be open, and let
\(M_0,\ldots,M_{L-1}:\Omega\to\mathrm{SL}(2,\mathbb R)\) be \(C^s\)
and admit coherently oriented standard singular value decompositions with
\(C^s\) data
\[
 M_j=R_{u_j}\operatorname{diag}(\lambda_j,\lambda_j^{-1})
 R_{\pi/2-s_j},\qquad \lambda_j>1.
\]
Assume \(\lambda_j\ge\Lambda>1\) and \(0<\varkappa\le1\),
\[
 \left|\cos(\pi/2-s_j+u_{j-1})\right|\ge\varkappa
 \quad(1\le j<L),
\]
and define, for multiindices \(\beta\in\mathbb N_0^d\),
\[
 Q_s:=2+\max_{\substack{0\le j<L\\1\le|\beta|\le s}}
 \left(\|\partial^\beta\log\lambda_j\|_\infty
 +\|\partial^\beta u_j\|_\infty
 +\|\partial^\beta s_j\|_\infty\right).
\]
Put \(\mathcal M=M_{L-1}\cdots M_0\).

\begin{lemma}
\label{B:lem:polynomial-hyperbolic-chain}
For the matrix chain above, there are constants \(D_0(s),C_3(s)>0\), depending
only on \(s\), such that, if
\[
 L^{D_0(s)}Q_s^{D_0(s)}\varkappa^{-D_0(s)}\Lambda^{-2}
 \le10^{-3}\varkappa,
\]
the following conclusions hold.
\begin{enumerate}
\item One has
\[
 \|\mathcal M\|\ge
 (\varkappa/2)^{L-1}\prod_{j=0}^{L-1}\lambda_j.
\]
\item On any connected open subdomain \(\Omega_0\subset\Omega\) on which the
orientations are fixed, choose coherent lifts \(U,S\) of the expanding left
and contracting right singular directions of \(\mathcal M\), uniquely
normalized by
\[
 U-u_{L-1},\ S-s_0\in(-\pi/2,\pi/2).
\]
For a constant \(\chi\in\{+1,-1\}\) on \(\Omega_0\),
\[
 \mathcal M=\chi R_U
 \operatorname{diag}(\|\mathcal M\|,\|\mathcal M\|^{-1})
 R_{\pi/2-S}.
\]
If \(\chi=-1\), using \(U+\pi\) in place of \(U\) gives a standard SVD
without changing positive-order derivatives.  The normalized lifts
\(U,S\) satisfy, with all norms on the left taken over \(\Omega_0\),
\begin{equation}\label{fr:eq:chain-fixed-order-endpoint}
\begin{aligned}
 \|S-s_0\|_{C^s}
 &\le C_3(s)L^{D_0(s)}Q_s^{D_0(s)}\varkappa^{-D_0(s)}
       (\inf_\Omega\lambda_0)^{-2},\\
 \|U-u_{L-1}\|_{C^s}
 &\le C_3(s)L^{D_0(s)}Q_s^{D_0(s)}\varkappa^{-D_0(s)}
       (\inf_\Omega\lambda_{L-1})^{-2}.
\end{aligned}
\end{equation}
\item On the same subdomain \(\Omega_0\), the singular data satisfy
\[
 \max_{1\le|\beta|\le s}
 \bigl(\|\partial^\beta\log\|\mathcal M\|\|_\infty
 +\|\partial^\beta U\|_\infty+\|\partial^\beta S\|_\infty\bigr)
 \le C_3(s)L^{D_0(s)}Q_s^{D_0(s)}\varkappa^{-D_0(s)}.
\]
\end{enumerate}
\end{lemma}

\begin{proof}
If \(L=1\), then \(\mathcal M=M_0\), \(\|\mathcal M\|=\lambda_0\), and one
may take \(U=u_0\), \(S=s_0\); the endpoint errors vanish and all assertions
follow from the definition of \(Q_s\).  Hence assume \(L\ge2\) and fix one
connected open subdomain \(\Omega_0\) with fixed orientations.
For \(1\le t<L\), set
\[
 g_t:=\frac\pi2-s_t+u_{t-1},\qquad
 H_t:=\cos g_t,\qquad \Theta_t:=\sin g_t,
\]
and, for \(0\le t<L\), set
\[
 \epsilon_t:=\lambda_t^{-2},\qquad
 \Delta_t:=\operatorname{diag}(1,\epsilon_t),\qquad
 \mathscr L:=\prod_{t=0}^{L-1}\lambda_t.
\]
Then
\[
 \mathcal M
 =\mathscr L R_{u_{L-1}}\mathscr M_LR_{\pi/2-s_0},
 \qquad
 \mathscr M_L=\Delta_{L-1}R_{g_{L-1}}\cdots\Delta_1R_{g_1}\Delta_0.
\]
By the multivariable Fa\`a di Bruno formula~\cite{ConstantineSavits1996},
\[
 \max_{\substack{1\le t<L\\1\le|\beta|\le s}}
 \bigl(\|\partial^\beta H_t\|_\infty
      +\|\partial^\beta\Theta_t\|_\infty\bigr)
 +
 \max_{\substack{0\le t<L\\1\le|\beta|\le s}}
 \|\partial^\beta\log\epsilon_t\|_\infty
 \le C_3(s)Q_s^{C_3(s)}.
\]
Moreover, \(|H_t|\ge\varkappa\) and
\(\|\epsilon_t\|_{L^\infty(\Omega_0)}\le\Lambda^{-2}\).  Each \(H_t\) has constant
sign on \(\Omega_0\), and the quotient rule gives
\[
 \max_{\substack{1\le t<L\\1\le|\beta|\le s}}
 \|\partial^\beta\log|H_t|\|_\infty
 \le C_3(s)Q_s^{C_3(s)}\varkappa^{-s}.
\]
Thus one may take \(K_0=1\) and \(K_r=C_3(s)Q_s^{C_3(s)}\),
\(1\le r\le s\), in
Lemma~\ref{B:lem:differentiated-finite-two-state-product}.  Write
\(\Xi:=LQ_s\varkappa^{-1}\ge2\).  There is an exponent \(p_s>0\), depending
only on \(s\), such that the left-hand side of
\eqref{B:eq:finite-two-state-smallness} at order \(s\) is at most
\(\Lambda^{-2}\Xi^{p_s}\).  If \(c_1(s)\) is the corresponding constant in
that lemma, enlarge \(D_0(s)\ge p_s\) until
\(10^{-3}2^{p_s-D_0(s)}\le c_1(s)\).  The present hypothesis then gives
\[
 \Lambda^{-2}\Xi^{p_s}
 \le10^{-3}\varkappa \Xi^{p_s-D_0(s)}
 \le c_1(s),
\]
so \eqref{B:eq:finite-two-state-smallness} holds.  Enlarge \(D_0(s)\) once
more, if necessary, to dominate \(s\) and all fixed exponents below.
Apply Lemma~\ref{B:lem:differentiated-finite-two-state-product} and
Corollary~\ref{B:cor:finite-two-state-frames} with
\(W_t=R_{g_t}\) and \(\mathscr L_0=\mathscr L\).  Since
\[
 \mathscr L^2\prod_{t=0}^{L-1}\epsilon_t
 =\left(\prod_{t=0}^{L-1}\lambda_t\right)^2
  \prod_{t=0}^{L-1}\lambda_t^{-2}=1,
\]
its real determinant-one specialization applies.  Denote its internal input
and output projections by \(\widehat\Pi^{\rm in}\) and
\(\widehat\Pi^{\rm out}\), its canonical input and output frames by
\(\widehat V\) and \(\widehat U\), and its multiplier by
\(\Lambda_{\rm int}\).  Since the exterior rotations are orthogonal,
Corollary~\ref{B:cor:finite-two-state-frames} gives
\[
 |\Lambda_{\rm int}|=\|\mathscr L\mathscr M_L\|=\|\mathcal M\|.
\]
The zeroth-order multiplier estimate in
Lemma~\ref{B:lem:differentiated-finite-two-state-product}
and the smallness above give
\[
 \|\mathcal M\|
 \ge\frac12\mathscr L\prod_{t=1}^{L-1}|H_t|
 \ge\left(\frac{\varkappa}{2}\right)^{L-1}
      \prod_{t=0}^{L-1}\lambda_t.
\]

The expanding left and contracting right singular projections of
\(\mathcal M\) are, respectively,
\[
 R_{u_{L-1}}\widehat\Pi^{\rm out}R_{-u_{L-1}},
 \qquad
 I_2-R_{s_0-\pi/2}\widehat\Pi^{\rm in}R_{\pi/2-s_0}.
\]
The projection estimates in
\eqref{B:eq:finite-two-state-projector-output}, after differentiating
the exterior rotations and absorbing the resulting powers of \(Q_s\),
have endpoint factors bounded by
\((\inf_\Omega\lambda_{L-1})^{-2}\) and
\((\inf_\Omega\lambda_0)^{-2}\), respectively.  By the spectral separation and canonical
normalization in
Lemma~\ref{B:lem:differentiated-finite-two-state-product} and
Corollary~\ref{B:cor:finite-two-state-frames}, the real frames
\(\widehat U,\widehat V\) lie in the fixed relative-angle chart about
\(I_2\).  Thus there are unique \(C^s\) functions
\(\delta_U,\delta_S:\Omega_0\to(-\pi/2,\pi/2)\), and restoring the endpoint
rotations gives
\[
 \widehat U=R_{\delta_U},\quad \widehat V=R_{\delta_S},\qquad
 U=u_{L-1}+\delta_U,\qquad S=s_0+\delta_S.
\]
The canonical-frame estimates and the smoothness of this fixed chart give
\eqref{fr:eq:chain-fixed-order-endpoint}.  The relative-angle chart is
global on \(\Omega_0\), so no simple connectivity assumption is needed.
The multiplier bound in the same cited lemma gives
\[
 \log\|\mathcal M\|
 =\sum_{t=0}^{L-1}\log\lambda_t
  +\sum_{t=1}^{L-1}\log|H_t|+\mathcal E,
 \qquad
 \|\mathcal E\|_{C^s}
 \le C_3(s)L^{D_0(s)}Q_s^{D_0(s)}\varkappa^{-D_0(s)}\Lambda^{-2}.
\]
Differentiating this identity, using the displayed bounds for
\(\partial^\beta\log|H_t|\) and the endpoint angle estimates, proves the
last assertion.
\end{proof}

\paragraph{\textbf{Elementary identities for singular directions, spectral projections,
and transfer matrices.}}
We use the standard finite-dimensional spectral conventions of
\cite{HornJohnson2013} and the Dirichlet transfer-matrix and Green-function
conventions of \cite{Teschl2000}.
If \(\mathcal M\in\mathrm{SL}(2,\mathbb R)\) has top singular value
\(\lambda(\mathcal M)>1\), its contracting right-singular and expanding
left-singular projections
are
\begin{equation}\label{fr:eq:complete-K-return-projectors}
\begin{aligned}
 \mathscr P^s(\mathcal M)
 &=\frac{\lambda(\mathcal M)^2I_2-\mathcal M^{\mathsf T}\mathcal M}
 {\lambda(\mathcal M)^2-\lambda(\mathcal M)^{-2}},\\
 \mathscr P^u(\mathcal M)
 &=\frac{\mathcal M\mathcal M^{\mathsf T}-\lambda(\mathcal M)^{-2}I_2}
 {\lambda(\mathcal M)^2-\lambda(\mathcal M)^{-2}}.
\end{aligned}
\end{equation}
If \(P_{\rm ref}=vv^{\mathsf T}\), \(\|v\|=1\), and
\(P\) is a rank-one orthogonal projection with
\(\|P-P_{\rm ref}\|<1/2\), the compatible oriented representative of
\(\operatorname{ran}P\) is
\begin{equation}\label{svd:eq:signed-axis-from-projector}
 v(P):=\frac{Pv}{\|Pv\|}.
\end{equation}
If instead \(\Pi\) is a holomorphic rank-one projection satisfying
\[
 \Pi^2=\Pi=\Pi^{\mathsf T},\qquad \operatorname{tr}\Pi=1,
\]
\(v_0^{\mathsf T}v_0=1\), and
\(v_0^{\mathsf T}\Pi v_0\ne0\) on a simply connected parameter domain, then a
compatible holomorphic representative normalized by
\(v^{\mathsf T}v=1\) is
\begin{equation}\label{svd:eq:complex-signed-axis-from-projector}
 v_{\mathbb C}(\Pi;v_0)
 :=\frac{\Pi v_0}
 {(v_0^{\mathsf T}\Pi v_0)^{1/2}},
 \qquad
 v_{\mathbb C}(\Pi;v_0)^{\mathsf T}
 v_{\mathbb C}(\Pi;v_0)=1,
\end{equation}
where the square root is fixed by continuation from the real oriented
direction.

For the Dirichlet restriction
\(H_{y,V,\alpha}^{[0,k-1]}\) fixed in the Introduction, set
\[
 P_k^V(y,E):=\det\bigl(E-H_{y,V,\alpha}^{[0,k-1]}\bigr),
 \qquad P_0^V(y,E):=1.
\]
For a block of length \(m\), whenever \(D:=P_m^V(y,E)\ne0\), write
\(g_L=P_{m-1}^V(y+\alpha,E)/D\), and
\(g_R=P_{m-1}^V(y,E)/D\).  Then
\begin{equation}\label{svd:eq:normalized-transfer-green}
 D^{-1}A_m^V(y,E)=
 \begin{pmatrix}
 1&-g_L\\ g_R&-g_Lg_R+D^{-2}
 \end{pmatrix}.
\end{equation}
If \(M(\vartheta)\in M_2(\mathbb R)\), \(\vartheta\in[0,1/2]\), has singular values
\(s_1(\vartheta)\ge s_2(\vartheta)\), with \(s_1(\vartheta)\ge1\) and
\(\det M(\vartheta)=\vartheta\), then
\begin{equation}\label{svd:eq:artificial-uniform-singular-gap}
 s_1(\vartheta)^2-s_2(\vartheta)^2\ge1-\vartheta^2\ge\tfrac34.
\end{equation}
Finally, for real slope data define the positive algebraic factor
\begin{equation}\label{svd:eq:XiD-explicit}
 \Xi^D(h_L,h_R,a):=
 \frac1{\sqrt{1+h_L^2}\sqrt{1+(a-h_R)^2}}.
\end{equation}
The projection and axis identities follow from
\eqref{svd:eq:explicit-projector-formula}.  The continuant identity
\cite[Chapter~1]{Teschl2000} and \(\det A_m^V=1\) give
\eqref{svd:eq:normalized-transfer-green};
\eqref{svd:eq:artificial-uniform-singular-gap} follows from
\(s_1s_2=\vartheta\) and \(s_1\ge1\).

\subsection{Initial-scale estimates on a finite-union inducing domain}

The identity \eqref{eq:N-fold-AMO-pullback} transfers the unperturbed
one-peak estimates to \(N\) components.

For the next five results and their common proof, fix \(R>0\), the return-window
exponent \(C_I>0\), and an integer
\(r_*\ge3\).  Let
\(B,R_\tau,B_1,B_2,D_1\) be the data supplied
by Propositions~\ref{fr:prop:AMO-edge-state}--
\ref{fr:prop:AMO-return-nondegeneracy} with
\(\beta=\widehat\alpha\) at order \(r_*\).  Recall that
\(X=C^\omega_\rho(\mathbb T,\mathbb R)\) for the fixed strip width
\(\rho>0\), and recall \(\widehat\alpha,\pi_N,p_a,E_+\), and \(\gamma\)
from the Introduction.
Let
\(I_n,\widetilde I_n,\mathcal E_n\) be the windows
\eqref{fr:eq:reciprocal-return-window}--
\eqref{fr:eq:central-polynomial-bidomain} at frequency
\(\widehat\alpha\), and put
\[
 \mathcal J_n:=\pi_N^{-1}(I_n),\qquad
 \widetilde J_{a,n}:=p_a+N^{-1}\widetilde I_n,
 \qquad 0\le a<N.
\]
For the family \eqref{eq:perturbative-potential-family}, let
\(\delta\in X\) satisfy \(\|\delta\|_\rho\le R\).

For a return atom \(\mathcal A\in\mathcal P_{\mathcal J_n}\), write
\(p=p_{\mathcal A}\), \(q=q_{\mathcal A}\), and put
\[
 J_r^+:=\{0,\ldots,r-1\},\qquad
 J_r^-:=\{-r,\ldots,-1\},
\]
and let \(\mathfrak J_{\mathcal A}\) consist of the nonempty integer
subintervals of \(J_p^+\) or \(J_q^-\).  For
\(J\in\mathfrak J_{\mathcal A}\), define
the finite-volume homotopy and determinant by
\begin{equation}\label{eq:perturbed-finite-volume-homotopy-definition}
\begin{aligned}
 H_J^0(x)&:=H_{V_0,J}(x),\\
 H_{J,t}^{\eps,\delta}(x)
 &:=H_J^0(x)+t\eps\,
   \operatorname{diag}\bigl(\delta(x+j\alpha)\bigr)_{j\in J},\\
 P_{J,t}^{\eps,\delta}(x,E)
 &:=\det\bigl(E-H_{J,t}^{\eps,\delta}(x)\bigr),
 \qquad 0\le t\le1.
\end{aligned}
\end{equation}
Set
\(P_J^0(x,E):=P_{J,0}^{\eps,\delta}(x,E)=\det(E-H_J^0(x))\).  For a complete
forward or backward block of length \(r\), use \(J_r^+\) or \(J_r^-\),
respectively.

The two-sided return-time partition of \(\mathcal J_n\) is the pullback under
\(\pi_N\) of that of \(I_n\) and is independent of \((\eps,\delta)\).
For every \(\mathcal A\in\mathcal P_{\mathcal J_n}\), the labels
\(p_{\mathcal A},q_{\mathcal A}\) are common to the perturbed and
unperturbed cocycles and, on
\(\mathcal A\cap\operatorname{Reg}(\mathcal J_n)\), equal the actual
forward and backward return times.

\begin{proposition}
\label{prop:polynomial-starting-scale-perturbed-dichotomy}
With this notation, there are
\[
\begin{gathered}
 D_2=D_2(C_I,R,r_*)>0,\qquad
 C_4=C_4(C_I,R,r_*)\ge1,\\
 0<c\le1\le C<\infty,\qquad n_0=n_0(C_I,R,r_*)\ge2.
\end{gathered}
\]
The numbered constants, exponent, and threshold depend only on
\((\mathfrak D_{\rm AMO}^{(r_*)},R)\). Define
\begin{equation}
\label{eq:polynomial-perturbation-exponent-choice}
 \begin{aligned}
 D_3&:=2D_2
 +\max\{B_1,B_2,D_1\}+5,\\
 M_{r_*}&:=1+\frac{\tau r_*D_2}{C_I}.
 \end{aligned}
\end{equation}
This proposition and the four results below hold with one common choice
of these constants whenever
\begin{equation}
\label{eq:polynomial-starting-scale-smallness}
        n\in\mathbb N,\qquad n\ge n_0,
        \qquad
        0<\eps\le n^{-D_3}.
\end{equation}

The following hold uniformly for every
\(\mathcal A\in\mathcal P_{\mathcal J_n}\).
\begin{enumerate}
\item The return times satisfy
\begin{equation}\label{eq:perturbed-return-length-polynomial}
 c n^{C_I/\tau}\le\min\{p,q\},
 \qquad p+q\le Cn^{C_I\tau}.
\end{equation}
\item For every \(J\in\mathfrak J_{\mathcal A}\),
\(x\in\mathcal A\), \(E\in\mathcal E_n\), and \(t\in[0,1]\),
\begin{equation}\label{eq:perturbed-resolvent-gap}
 E-\max\operatorname{spec}H_{J,t}^{\eps,\delta}(x)
 \ge cn^{-D_1}.
\end{equation}
Under the same quantifiers, for every \(\beta\in\mathbb N_0^2\) with
\(1\le|\beta|\le r_*+1\),
\begin{equation}\label{eq:perturbed-resolvent-polynomial-derivatives}
 \left\|\partial_{x,E}^{\beta}
 (E-H_{J,t}^{\eps,\delta}(x))^{-1}\right\|_{\rm op}
 \le C_4n^{D_2}.
\end{equation}
\item For the same \(x,E,t\) and
\((\sigma,r)=(+,p)\) or \((-,q)\),
\begin{equation}\label{eq:perturbed-principal-determinant-lower}
 P_{J_r^\sigma,t}^{\eps,\delta}(x,E)
 \ge c\exp\{\gamma r-CR_\tau(r+2)\}>0.
\end{equation}
\end{enumerate}
\end{proposition}

On each \(\mathcal A\in\mathcal P_{\mathcal J_n}\), use the atomwise
SVD and angle-function convention
\eqref{fr:eq:complete-return-joining-convention}.  Denote the resulting
data by
\[
 (\lambda_{\mathcal A}^{\pm,\eps,\delta},
   u_{\mathcal A}^{\pm,\eps,\delta},
   s_{\mathcal A}^{\pm,\eps,\delta},
   G_{\mathcal A}^{\eps,\delta}),
\]
with oriented angle lifts continued along
\(t\mapsto V_{t\eps,\delta}\), \(0\le t\le1\), from \(t=0\).  A
superscript \(0\) denotes the corresponding unperturbed datum.

\begin{lemma}
\label{lem:polynomial-starting-return-growth}
With the constants and threshold in
Proposition~\ref{prop:polynomial-starting-scale-perturbed-dichotomy},
the following hold on \(\mathcal A\times\mathcal E_n\).
\begin{enumerate}
\item The singular values satisfy
\begin{align}
 \lambda_{\mathcal A}^{+,\eps,\delta}
 &\ge
 c\exp\{\gamma p-CR_\tau(p+2)\},
 \label{eq:perturbed-return-growth-plus}\\
 \lambda_{\mathcal A}^{-,\eps,\delta}
 &\ge
 c\exp\{\gamma q-CR_\tau(q+2)\}.
 \label{eq:perturbed-return-growth-minus}
\end{align}
\item For every \(\beta\in\mathbb N_0^2\) with \(1\le|\beta|\le r_*\),
\begin{align}
 \left|\partial_{x,E}^{\beta}
 \lambda_{\mathcal A}^{+,\eps,\delta}\right|
 &\le
 C_4(p+2)^{M_{r_*}}
 \lambda_{\mathcal A}^{+,\eps,\delta},
 \label{eq:perturbed-return-derivative-plus}\\
 \left|\partial_{x,E}^{\beta}
 \lambda_{\mathcal A}^{-,\eps,\delta}\right|
 &\le
 C_4(q+2)^{M_{r_*}}
 \lambda_{\mathcal A}^{-,\eps,\delta}.
 \label{eq:perturbed-return-derivative-minus}
\end{align}
\end{enumerate}
\end{lemma}

The same choice controls the singular directions and the overlap.

\begin{corollary}
\label{cor:polynomial-starting-overlap-stability}
For every \(\mathcal A\in\mathcal P_{\mathcal J_n}\), the following
hold.
\begin{enumerate}
\item On \(\mathcal A\times\mathcal E_n\), the coherently oriented
angle lifts satisfy, for \(\beta\in\mathbb N_0^2\) with
\(1\le|\beta|\le r_*\),
\begin{equation}
\label{eq:perturbed-axis-angle-derivatives}
\sum_{\sigma\in\{+,-\}}
\left(
\left|\partial_{x,E}^{\beta}
u_{\mathcal A}^{\sigma,\eps,\delta}\right|
+
\left|\partial_{x,E}^{\beta}
s_{\mathcal A}^{\sigma,\eps,\delta}\right|
\right)
\le C_4n^{D_2}.
\end{equation}
\item On the same \((x,E)\)-domain,
\begin{equation}
\label{eq:polynomial-Schur-SVD-stability}
 \left\|
 \cos G_{\mathcal A}^{\eps,\delta}
 -\cos G_{\mathcal A}^{0}
 \right\|_{C^{r_*}(\mathcal A\times\mathcal E_n)}
 \le C_4\eps n^{D_2}.
\end{equation}
\item If \(E_+^{\eps,\delta}\) denotes the perturbed upper spectral edge, then
\begin{equation}\label{eq:perturbed-edge-polynomial-margin}
 |E_+^{\eps,\delta}-E_+|
 \le\eps\|\delta\|_\infty
 \le\frac14n^{-B_2}.
\end{equation}

\item On the exterior portion of \(\mathcal A\),
\begin{equation}
\label{eq:perturbed-exterior-mismatch-separation}
 \left|\cos G_{\mathcal A}^{\eps,\delta}(x,E)\right|
 \ge\frac{c}{2}n^{-D_1}
\end{equation}
whenever
\[
 x\in
 \mathcal A\cap
 \left(
 \mathcal J_n\setminus
 \bigcup_{a=0}^{N-1}\widetilde J_{a,n}
 \right),
 \qquad E\in\mathcal E_n.
\]
\end{enumerate}
\end{corollary}

For \(0\le a<N\), let
\(\mathcal A_{a,n}\supset\widetilde J_{a,n}\) be the common return
atom containing \(p_a\).  Call these atoms central and every other atom
of \(\mathcal P_{\mathcal J_n}\) exterior.  Set
\[
 H_{a,n}^{\eps,\delta}
 :=\cos G_{\mathcal A_{a,n}}^{\eps,\delta},\qquad
 H_{a,n}^{0}:=\cos G_{\mathcal A_{a,n}}^{0},\qquad
 \widehat\Omega_{a,n}:=\widetilde J_{a,n}\times\mathcal E_n.
\]

\begin{proposition}
\label{prop:polynomial-starting-folds}
For every \(0\le a<N\), under the same scale condition
\eqref{eq:polynomial-starting-scale-smallness},
the following hold.
\begin{enumerate}
\item There is a unique
\((\xi_{a,n}^{\eps,\delta},E_{a,n}^{\eps,\delta})
\in\widehat\Omega_{a,n}\) satisfying
\[
 H_{a,n}^{\eps,\delta}=\partial_xH_{a,n}^{\eps,\delta}=0
 \quad\text{at }(\xi_{a,n}^{\eps,\delta},E_{a,n}^{\eps,\delta}).
\]
This is the fold point of \(H_{a,n}^{\eps,\delta}=0\) for projection
onto the energy coordinate.
With the first-return time \(m_n\) from
Proposition~\ref{fr:prop:AMO-central-normal-form},
\begin{equation}
\label{eq:perturbed-fold-displacement-polynomial}
 \left|\xi_{a,n}^{\eps,\delta}-p_a\right|
 +\left|E_{a,n}^{\eps,\delta}-E_+\right|
 \le
 C_4
 \left(
 \eps n^{D_2}+e^{-c\gamma m_n}
 \right).
\end{equation}
\item Throughout \(\widehat\Omega_{a,n}\),
\begin{equation}
\label{eq:perturbed-central-transversality}
 c\le
 \partial_E H_{a,n}^{\eps,\delta}
 \le C,
 \qquad
 c\le
 \partial_x^2 H_{a,n}^{\eps,\delta}
 \le C
\end{equation}
in the same upper-edge orientation.
\end{enumerate}
\end{proposition}

For the fold point above, set
\[
 \Omega_{a,n}^{\eps,\delta}
 :=
 \left\{
 |x-\xi_{a,n}^{\eps,\delta}|
 <\frac{1}{4N}n^{-B_1},
 \quad
 |E-E_{a,n}^{\eps,\delta}|
 <\frac14n^{-B_2}
 \right\}.
\]

\begin{corollary}
\label{cor:polynomial-starting-fold-normal-form}
For the same fold point, on \(\Omega_{a,n}^{\eps,\delta}\), there is
the exact recentered normal form
\begin{equation}
\label{eq:perturbed-recentered-normal-form}
 H_{a,n}^{\eps,\delta}(x,E)
 =
 C_{1,a,n}^{\eps,\delta}(x)
 (x-\xi_{a,n}^{\eps,\delta})^2
 +
 C_{2,a,n}^{\eps,\delta}(x,E)
 (E-E_{a,n}^{\eps,\delta}),
\end{equation}
where, in the common upper-edge orientation,
\begin{equation}
\label{eq:perturbed-recentered-coefficient-bounds}
 0<c\le C_{1,a,n}^{\eps,\delta}\le C,
 \qquad
 0<c\le C_{2,a,n}^{\eps,\delta}\le C.
\end{equation}
\end{corollary}

\begin{proof}[Proof of Proposition~\ref{prop:polynomial-starting-scale-perturbed-dichotomy}
and the four subsequent results]
\smallskip
\noindent\emph{Step (i): Parameters.}
Let \(D_{10}\) dominate \(D_1\) and the finitely many
polynomial losses obtained,
by repeated resolvent differentiation and Jacobi's formula, for resolvents
through order \(r_*+1\), determinant ratios, and normalized return matrices
through order \(r_*\).  Let \(d_{\rm svd}\) dominate the total degrees of
the finitely many polynomials
\(\mathfrak P_k^{\rm svd}\), \(1\le k\le r_*\), in
Lemma~\ref{svd:lem:fixed-gap-oriented-stability}.  Choose
\begin{equation}\label{eq:starting-scale-SVD-exponent-ledger}
 D_2\ge4(d_{\rm svd}+2)(D_{10}+1),
\end{equation}
and define \(D_3,M_{r_*}\) by
\eqref{eq:polynomial-perturbation-exponent-choice}.
By \eqref{eq:N-fold-AMO-pullback}, for every \(j\in\mathbb Z\),
\[
 (T_\alpha)^jx\in\mathcal J_n
 \quad\Longleftrightarrow\quad
 \pi_Nx+j\widehat\alpha\in I_n.
\]
Hence the return partition and both return lengths are unchanged.
Lemma~\ref{B:lem:return-time-poly-two-sided}, applied to \(I_n\), yields
\[
 c n^{C_I/\tau}\le\min\{p,q\},
 \qquad p+q\le Cn^{C_I\tau}.
\]
\smallskip
\noindent\emph{Step (ii): Finite-volume estimates.}\par
\noindent For \((\sigma,r)\in\{(+,p),(-,q)\}\), the forward and backward source
phases are \(\pi_Nx\) and \(\pi_Nx-r\widehat\alpha\), respectively.
After translating the lattice indices, the corresponding complete
unperturbed one-sided block is \(H_{J_r^\sigma}^0(x)\).  Hence
\eqref{fr:eq:companion-stopped-edge-gap} gives
\[
 E_+-\max\operatorname{spec}H_{J_r^\sigma}^0(x)
 \ge cn^{-D_1}.
\]
Thus, for \(E\in\mathcal E_n\),
\[
\begin{aligned}
 E-\max\operatorname{spec}H_{J_r^\sigma}^0(x)
 &\ge
 E_+-\max\operatorname{spec}H_{J_r^\sigma}^0(x)-n^{-B_2}\\
 &\ge cn^{-D_1}-n^{-B_2}
 \ge c'n^{-D_1},
\end{aligned}
\]
where the last inequality follows, after increasing \(n_0\), from
\eqref{fr:eq:companion-edge-gap-hierarchy}.  Cauchy interlacing
\cite{HornJohnson2013} passes this bound to every principal block in either
return matrix:
\begin{equation}\label{eq:perturbed-unperturbed-principal-block-gap}
 E-\max\operatorname{spec}H_J^0(x)
 \ge cn^{-D_1},
 \qquad E\in\mathcal E_n.
\end{equation}
Use the finite-volume homotopy
\eqref{eq:perturbed-finite-volume-homotopy-definition}.
Cauchy's estimate \cite{RudinRCA1987} and Weyl's eigenvalue perturbation
inequality \cite{HornJohnson2013} give, respectively,
\[
 \sup_x|\partial_x^k\delta(x)|
 \le C_4R,
 \qquad 0\le k\le r_*+2,
\]
and
\[
 \left|
 \max\operatorname{spec}H_{J,t}^{\eps,\delta}
 -\max\operatorname{spec}H_J^0
 \right|
 \le\eps R.
\]
By \eqref{eq:polynomial-perturbation-exponent-choice}--
\eqref{eq:polynomial-starting-scale-smallness},
\[
 \eps R\le Rn^{-D_3}\le\frac c4n^{-D_1}
\]
after increasing \(n_0\).  Weyl's inequality and
\eqref{eq:perturbed-unperturbed-principal-block-gap} therefore give
\begin{equation}\label{eq:perturbed-homotopy-principal-block-gap}
 E-\max\operatorname{spec}H_{J,t}^{\eps,\delta}(x)
 \ge cn^{-D_1},
 \qquad 0\le t\le1.
\end{equation}

With \eqref{eq:perturbed-homotopy-principal-block-gap}, put
\[
 R_{J,t}:=(E-H_{J,t}^{\eps,\delta})^{-1}.
\]
Writing \(K_{J,t}:=E-H_{J,t}^{\eps,\delta}\), Leibniz's rule gives,
for \(|\beta|\ge1\),
\[
 \partial^\beta R_{J,t}
 =-R_{J,t}\sum_{0<\nu\le\beta}
   \binom\beta\nu(\partial^\nu K_{J,t})
                    (\partial^{\beta-\nu}R_{J,t}).
\]
Thus induction on \(|\beta|\), using
\(\|R_{J,t}\|\le Cn^{D_1}\) and
\(\|\partial^\nu K_{J,t}\|\le C_4\), yields
\[
 \|\partial_{x,E}^\beta R_{J,t}\|_{\rm op}
 \le C_4n^{D_{10}},
 \qquad |\beta|\le r_*+1.
\]
The identity
\[
 R_{J,t}-R_{J,0}
 =t\eps R_{J,t}
   \operatorname{diag}_{j\in J}\delta(x+j\alpha)R_{J,0}
\]
and the same differentiation formula yield
\begin{equation}\label{eq:perturbed-resolvent-difference}
 \|\partial_{x,E}^\beta(R_{J,t}-R_{J,0})\|_{\rm op}
 \le C_4\eps n^{D_{10}},
 \qquad |\beta|\le r_*.
\end{equation}
Positivity and Jacobi's formula give
\[
 \partial_t\log P_{J,t}^{\eps,\delta}
 =-\eps\operatorname{tr}\!\left[
 R_{J,t}\operatorname{diag}(\delta(x+j\alpha))\right].
\]
Using the return-length and derivative bounds and integrating in \(t\),
\begin{equation}\label{eq:perturbed-proof-logdet-homotopy-bound}
 \left\|
 \log\frac{P_{J,t}^{\eps,\delta}}{P_J^0}
 \right\|_{C^{r_*}}
 \le C_4\eps n^{D_{10}}.
\end{equation}
For a complete forward or backward block, this and
\eqref{fr:eq:companion-return-determinant-persistence} imply
\[
 P_{J,t}^{\eps,\delta}(x,E)
 \ge
 \exp\{\gamma |J|-CR_\tau(|J|+2)-1\}.
\]
Here the common threshold is increased so that
\(C_4\eps n^{D_{10}}\le1\).

\smallskip
\noindent\emph{Step (iii): Singular-value estimates.}
For each \(\sigma\in\{+,-\}\), translate a backward block to forward
orientation when necessary, let \(A_{r,t}^{\sigma}\) be its length-\(r\)
fixed-word matrix for \(V_0+t\eps\delta\), and set
\[
 D_t^\sigma:=P_{J_r^\sigma,t}^{\eps,\delta},
 \qquad
 \mathcal N_t^\sigma:=(D_t^\sigma)^{-1}A_{r,t}^{\sigma},
 \qquad
 \lambda_{\mathcal A,t}^{\sigma}:=\|A_{r,t}^{\sigma}\|.
\]
For the homotopy notation used below, put
\[
 u_{\mathcal A,t}^{\sigma}
 :=u_{\mathcal A}^{\sigma,t\eps,\delta},\qquad
 s_{\mathcal A,t}^{\sigma}
 :=s_{\mathcal A}^{\sigma,t\eps,\delta},\qquad
 G_{\mathcal A,t}:=G_{\mathcal A}^{t\eps,\delta}.
\]
Positivity, \eqref{svd:eq:normalized-transfer-green},
\eqref{eq:perturbed-resolvent-difference}, and
\eqref{eq:perturbed-proof-logdet-homotopy-bound} give, after increasing
\(n_0\),
\begin{align}
 &(\mathcal N_t^\sigma)_{11}=1,\qquad
 \det\mathcal N_t^\sigma=(D_t^\sigma)^{-2}\le\tfrac12,\qquad
 \|\mathcal N_t^\sigma\|\ge1,
 \label{eq:perturbed-normalized-return-algebra}\\
 &\max_{\substack{\sigma\in\{+,-\},\ 0\le t\le1\\1\le k\le r_*}}
 \|\mathcal N_t^\sigma\|_{C^k}
 \le C_4n^{D_{10}},\notag\\
 &\qquad
 \max_{\sigma,t}
 \|\mathcal N_t^\sigma-\mathcal N_0^\sigma\|_{C^{r_*}}
 \le C_4\eps n^{D_{10}}.
 \label{eq:perturbed-normalized-return-raw-stability}\\
 &\max_{\substack{\sigma,t\\1\le|\beta|\le r_*}}
 \|\partial_{x,E}^{\beta}\log D_t^\sigma\|_{L^\infty}
 \le C_4n^{D_{10}}.
 \label{eq:perturbed-principal-logdet-raw-jets}
\end{align}
In particular,
\[
 s_1(\mathcal N_t^\sigma)\ge1,\qquad
 s_1(\mathcal N_t^\sigma)^2-s_2(\mathcal N_t^\sigma)^2
 \ge1-(D_t^\sigma)^{-4}\ge\tfrac34,
 \qquad 0\le t\le1.
\]

We now apply Lemma~\ref{svd:lem:fixed-gap-oriented-stability} and
Corollary~\ref{svd:cor:fixed-gap-oriented-axes} on
\(\operatorname{int}\mathcal A\times\operatorname{int}\mathcal E_n\)
to the pair
\((\mathcal N_t^+,\mathcal N_t^-)\), using the unperturbed oriented axes
as its initial sections; the resulting estimates extend to
\(\mathcal A\times\mathcal E_n\) by continuity.  With the common notation of those results,
\[
 K_k\le C_4n^{D_{10}},
 \qquad \eta\le C_4\eps n^{D_{10}}.
\]
By the definition of \(d_{\rm svd}\) and
\eqref{eq:starting-scale-SVD-exponent-ledger},
\begin{equation}\label{eq:perturbed-SVD-chart-budget}
 C_1(r_*)\eta\mathfrak P_{r_*}^{\rm svd}(K_1,\ldots,K_{r_*})
 \le C_4\eps n^{D_2/4}\le\frac14
\end{equation}
after increasing \(n_0\).  By
\eqref{eq:perturbed-SVD-chart-budget}, the formula
\[
 \mathbf v_{t,\sigma}
 =\frac{\Pi_{t,\sigma}^{\rm in}\mathbf v_{0,\sigma}}
 {\|\Pi_{t,\sigma}^{\rm in}\mathbf v_{0,\sigma}\|}
\]
is valid for every \(0\le t\le1\), and
\(\mathbf v_{t,\sigma}^{\mathsf T}\mathbf v_{0,\sigma}>0\).
The analogous left-axis formula also has positive inner product with
its initial axis; both lifts therefore remain in their initial charts.

The axis, overlap, and logarithmic singular-value conclusions
now give
\begin{align*}
 &\sum_{\sigma\in\{+,-\}}
 \|u_{\mathcal A,t}^{\sigma}-u_{\mathcal A,0}^{\sigma}\|_{C^{r_*}}
 +\sum_{\sigma\in\{+,-\}}
 \|s_{\mathcal A,t}^{\sigma}-s_{\mathcal A,0}^{\sigma}\|_{C^{r_*}}\\
 &\qquad
 +\|\cos G_{\mathcal A,t}-\cos G_{\mathcal A,0}\|_{C^{r_*}}
 \le C_4\eps n^{D_2/4},\\
 &\max_{\substack{t\\1\le|\beta|\le r_*}}
 \sum_{\sigma\in\{+,-\}}\left(
 |\partial_{x,E}^{\beta}u_{\mathcal A,t}^{\sigma}|
 +|\partial_{x,E}^{\beta}s_{\mathcal A,t}^{\sigma}|\right)
 \le C_4n^{D_2/4},\\
 &\max_{\substack{\sigma,t\\1\le|\beta|\le r_*}}
 \|\partial_{x,E}^{\beta}\log s_1(\mathcal N_t^\sigma)\|_{L^\infty}
 \le C_4n^{D_2/4}.
\end{align*}
The second line implies \eqref{eq:perturbed-axis-angle-derivatives}, and
the first gives \eqref{eq:polynomial-Schur-SVD-stability} at \(t=1\).
Since
\[
 \lambda_{\mathcal A,t}^{\sigma}=D_t^\sigma s_1(\mathcal N_t^\sigma),
 \qquad s_1(\mathcal N_t^\sigma)\ge1,
\]
the determinant bounds give
\eqref{eq:perturbed-return-growth-plus}--
\eqref{eq:perturbed-return-growth-minus}.  Moreover,
\(\log\lambda=\log D+\log s_1(\mathcal N)\) and
\eqref{eq:perturbed-principal-logdet-raw-jets} and
\eqref{svd:eq:fixed-gap-log-singular-derivatives},
\[
 \max_{\substack{\sigma,t\\1\le|\beta|\le r_*}}
 |\partial_{x,E}^{\beta}\log\lambda_{\mathcal A,t}^{\sigma}|
 \le C_4n^{D_2/4}.
\]
Fa\`a di Bruno applied to
\(\lambda=\exp(\log\lambda)\) therefore gives
\[
 |\partial_{x,E}^{\beta}\lambda_{\mathcal A,t}^{\sigma}|
 \le C_4n^{r_*D_2/4}
       \lambda_{\mathcal A,t}^{\sigma},
 \qquad1\le|\beta|\le r_*.
\]
The lower estimate in \eqref{eq:perturbed-return-length-polynomial}
absorbs this power into \(C_4(p+2)^{M_{r_*}}\), respectively
\(C_4(q+2)^{M_{r_*}}\), by the definition of \(M_{r_*}\). This proves
\eqref{eq:perturbed-return-derivative-plus}--
\eqref{eq:perturbed-return-derivative-minus}.
\eqref{eq:perturbed-edge-polynomial-margin} follows from
\(|E_+^{\eps,\delta}-E_+|\le\eps\|\delta\|_\infty\).

\smallskip
\noindent\emph{Step (iv): Fold points.}
The exact pullback of
\eqref{fr:eq:central-polynomial-normal-form}--
\eqref{fr:eq:central-mismatch-transversality} gives
\[
 \partial_EH_{a,n}^0\ge c,
 \qquad
 \partial_x^2H_{a,n}^0\ge c
\]
on the pulled-back central phase--energy domain.  By
\eqref{eq:polynomial-Schur-SVD-stability}, the coherently continued
upper-edge orientation satisfies
\[
        \partial_EH_{a,n}^{\eps,\delta}\ge c/2,
        \qquad
        \partial_x^2H_{a,n}^{\eps,\delta}\ge c/2
\]
throughout \(\widehat\Omega_{a,n}\).

At \((p_a,E_+)\), the unperturbed remainder estimate gives
\[
 |H_{a,n}^0(p_a,E_+)|
 +|\partial_xH_{a,n}^0(p_a,E_+)|
 \le Ce^{-c\gamma m_n}.
\]
Consequently, by the \(D_{10}\)-bounds and the choice of
\(D_2\),
\begin{equation*}
 \left|
 \begin{pmatrix}
 H_{a,n}^{\eps,\delta}\\
 \partial_xH_{a,n}^{\eps,\delta}
 \end{pmatrix}
 (p_a,E_+)
 \right|
 \le
 C_4
 \left(
 \eps n^{D_2/4}+e^{-c\gamma m_n}
 \right).
\end{equation*}
For the remainder of the proof, write
\[
 H:=H_{a,n}^{\eps,\delta}.
\]
The Jacobian of
\((x,E)\mapsto(H_{a,n}^{\eps,\delta},
\partial_xH_{a,n}^{\eps,\delta})\) has determinant
\[
 H_xH_{xE}-H_EH_{xx}.
\]
Put
\[
        \eta_n
        :=C_4
        \left(\eps n^{D_2/4}+e^{-c\gamma m_n}\right).
\]
Set
\[
 F_n(x,E):=(H(x,E),H_x(x,E)),
 \qquad z_0:=(p_a,E_+).
\]
The same derivative estimates give
\[
 |H_{xE}|+
 \sup_{\widehat\Omega_{a,n}}\|D^2F_n\|
 \le C_4n^{D_2/4}.
\]
Moreover,
\begin{equation*}
 \eta_n n^{3D_2/4}
 \le
 C_4\left(
 \eps n^{D_2}
 +n^{3D_2/4}e^{-c\gamma m_n}\right)
 =o_{n\to\infty}(1).
\end{equation*}
Here we used \eqref{eq:polynomial-perturbation-exponent-choice} and
\(m_n\ge c n^{C_I/\tau}\).  At the base point,
\[
 |H_xH_{xE}|\le C_4\eta_n n^{D_2/4}
 =o_{n\to\infty}(1),
 \qquad
 H_EH_{xx}\ge c^2/4.
\]
Thus the determinant stays away from zero and, allowing for \(H_{xE}\),
\[
 \|DF_n(z_0)^{-1}\|
 \le C_4n^{D_2/4}.
\]
Moreover,
\[
\begin{aligned}
 \|F_n(z_0)\|&\le\eta_n,\\
 \|DF_n(z_0)^{-1}\|^2
 \sup_{\widehat\Omega_{a,n}}\|D^2F_n\|\,\eta_n
 &\le C_4\eta_n n^{3D_2/4}\\
 &=o_{n\to\infty}(1).
\end{aligned}
\]
Put
\[
 A_n:=DF_n(z_0)^{-1},\qquad
 r_n:=2\|A_n\|\eta_n,\qquad
 \Phi_n(z):=z-A_nF_n(z).
\]
The exponent conditions give
\[
 \overline{B(z_0,r_n)}\subset\widehat\Omega_{a,n},\qquad
 \|\Phi_n(z_0)-z_0\|\le r_n/2,
\]
and
\[
 \sup_{z\in B(z_0,r_n)}\|D\Phi_n(z)\|
 \le \|A_n\|r_n\sup_{\widehat\Omega_{a,n}}\|D^2F_n\|
 =o(1)<\tfrac12.
\]
Hence \(\Phi_n\) maps \(\overline{B(z_0,r_n)}\) into itself and is a
contraction. Its fixed point satisfies \(F_n=0\), with displacement
\[
\begin{aligned}
 |\xi_{a,n}^{\eps,\delta}-p_a|
 +|E_{a,n}^{\eps,\delta}-E_+|
 &\le2r_n\le4C_4n^{D_2/4}\eta_n\\
 &=4C_4^2\left(
 \eps n^{D_2/2}
 +n^{D_2/4}e^{-c\gamma m_n}\right)\\
 &\le C_4\left(
 \eps n^{D_2}+e^{-c'\gamma m_n}\right),
\end{aligned}
\]
where \(0<c'<c\) and the last inequality holds after increasing \(n_0\).
Relabeling \(c'\) as \(c\) gives
\eqref{eq:perturbed-fold-displacement-polynomial}.

For uniqueness, strict convexity gives at most one critical point at
each energy.  Two folds \((x_i,E_i)\), \(E_1<E_2\), would imply
\[
\begin{aligned}
 H(x,E_1)&\ge H(x_1,E_1)+H_x(x_1,E_1)(x-x_1)=0,\\
 H(x,E_2)&=H(x,E_1)+\int_{E_1}^{E_2}H_E(x,t)\,dt
 \ge\tfrac c2(E_2-E_1)>0
 \qquad(x\in\widetilde J_{a,n}),
\end{aligned}
\]
contradicting \(H(x_2,E_2)=0\).

\smallskip
\noindent\emph{Step (v): Normal form and separation.}
Write
\[
        \xi=\xi_{a,n}^{\eps,\delta},
        \qquad E_*=E_{a,n}^{\eps,\delta}.
\]
Since \(H(\xi,E_*)=H_x(\xi,E_*)=0\), the fundamental theorem of calculus
gives the exact identity
\[
\begin{aligned}
 H(x,E)
 ={}&
 (x-\xi)^2
 \int_0^1(1-t)
 H_{xx}\bigl(\xi+t(x-\xi),E_*\bigr)\,dt\\
 &+(E-E_*)
 \int_0^1
 H_E\bigl(x,E_*+t(E-E_*)\bigr)\,dt.
\end{aligned}
\]
These integrals are \(C_{1,a,n}^{\eps,\delta}\) and
\(C_{2,a,n}^{\eps,\delta}\) in
\eqref{eq:perturbed-recentered-normal-form}.  The bounds for \(H_{xx}\)
and \(H_E\) give
\eqref{eq:perturbed-recentered-coefficient-bounds} and
\eqref{eq:perturbed-central-transversality}.

Finally, the unperturbed exterior estimate
\eqref{fr:eq:exterior-cos-lower} and
\eqref{eq:polynomial-Schur-SVD-stability} give
\[
 |\cos G_{\mathcal A}^{\eps,\delta}|
 \ge
 cn^{-D_1}
 -C_4\eps n^{D_2}.
\]
By \eqref{eq:polynomial-perturbation-exponent-choice}--
\eqref{eq:polynomial-starting-scale-smallness},
\[
 C_4\eps n^{D_2}\le\tfrac12cn^{-D_1},
\]
which proves \eqref{eq:perturbed-exterior-mismatch-separation}.
\end{proof}
\subsection{First-return induction with multiple simple zeros}

We now induce near the simple zeros created by distinct fold energies.
The following notation and assumptions apply to the next five results
and their common proof.  Fix an integer \(N\ge2\), \(R>0\), and an integer
\(r_*\ge3\).  Let
\(C_I\) be the fixed return-window exponent, and let
\[
 B_2,B_1,D_1,D_2,D_3,
 M_{r_*+1},n_0
\]
be the constants supplied by
Proposition~\ref{prop:polynomial-starting-scale-perturbed-dichotomy} at
order \(r_*+1\).  Use the fold points and normal forms from
Proposition~\ref{prop:polynomial-starting-folds} and
Corollary~\ref{cor:polynomial-starting-fold-normal-form} obtained with
that same choice of constants.  Recall \(\gamma=\log\lambda\) and
\((\sigma_{\rm cmp},B,R_\tau)\) from
Proposition~\ref{fr:prop:AMO-edge-state} and
\eqref{B:eq:Rtau-definition}.
Let
\[
 \delta\in X,\qquad \|\delta\|_\rho\le R,\qquad
 n\in\mathbb N,\qquad n\ge n_0,\qquad
 0<\eps\le n^{-D_3}.
\]
Suppose, in addition, that for some fixed exponents
\(B_3,B_4>0\),
\begin{equation}\label{eq:multiwell-fold-input}
\begin{aligned}
 \min_{0\le a<b<N}|E_{a,n}^{\eps,\delta}-E_{b,n}^{\eps,\delta}|
 &\ge n^{-B_3},\\
 \max_{0\le a<N}\bigl(
 |\xi_{a,n}^{\eps,\delta}-p_a|
 +|E_{a,n}^{\eps,\delta}-E_+^{\eps,\delta}|
 \bigr)&\le n^{-B_4}.
\end{aligned}
\end{equation}
Assume
\begin{equation}\label{eq:multiwell-seed-containment-order}
 B_4>\max\{B_2,2B_1+4\},
 \qquad B_3>B_4+1.
\end{equation}
Let
\[
 B_5,B_6,B_7,B_8,B_9>0
\]
satisfy
\begin{equation}\label{eq:multiwell-polynomial-exponent-order}
\begin{gathered}
 B_5>\max\{B_3,B_2,B_4\}+4,
 \qquad
 B_1+1<B_6<\tfrac12B_4-1,\\
 B_7>\max\{B_6,\tfrac12B_5\}+4,
 \qquad
 B_8>B_7(\tau+1)+4,\\
 B_9>
 \max\{B_3,B_8+\tfrac12B_5\}+4.
\end{gathered}
\end{equation}
Write
\[
 \mathbf b:=(B_3,B_4,B_5,B_6,
 B_7,B_8,B_9).
\]
The constants below may depend on \((R,r_*,\mathbf b)\), in addition to
the fixed background data; this dependence is displayed when the
constants are chosen and then suppressed.
Set
\begin{equation}\label{eq:multiwell-eta-mw-explicit}
 \eta_{\rm mw}:=\min\left\{
 \frac12,
 \frac{C_I\sigma_{\rm cmp}}{2B_7\tau^2}
 \right\}>0.
\end{equation}
Put \(\eta_0:=C_I/(2\tau)>0\).

Let \((a_1,\ldots,a_N)\) be the permutation of
\(\{0,\ldots,N-1\}\) satisfying
\[
 E_{a_1,n}^{\eps,\delta}>
 E_{a_2,n}^{\eps,\delta}>\cdots>
 E_{a_N,n}^{\eps,\delta}
\]
and put \(h_n=n^{-B_5}\).
All phase inequalities, intervals, and factorizations attached to \(a_s\)
are written in the real lift \(\widetilde J_{a_s,n}\).

\begin{corollary}
\label{cor:polynomial-multiwell-K-reinduction}
Under \eqref{eq:multiwell-fold-input}--
\eqref{eq:multiwell-polynomial-exponent-order}, there exist constants
\[
 c_3=c_3(R,r_*,\mathbf b),\qquad C_5=C_5(R,r_*,\mathbf b),
 \qquad 0<c_3\le1\le C_5,
\]
exponents \(B_{16},D_4>0\), and a threshold
\(n_{\rm mw}=n_{\rm mw}(R,r_*,\mathbf b)\ge n_0\), depending only on
\((\mathfrak D_{\rm AMO}^{(r_*+1)},R,\mathbf b)\), and not on
\((n,\eps,\delta,k)\), such that this corollary and the four results below
hold with one common choice of these constants whenever \(n\ge n_{\rm mw}\).

For each \(1\le k<N\), there exists a
closed interval \(\mathcal I_{n,k}\) of length \(n^{-B_9}\)
contained in
\begin{equation}\label{eq:multiwell-intermediate-layer-carrier}
 [E_{a_{k+1},n}^{\eps,\delta}+h_n,
   E_{a_k,n}^{\eps,\delta}-h_n],
\end{equation}
and there exists \(\mathcal I_{n,N}\) of the same length contained in
\begin{equation}\label{eq:multiwell-bottom-layer-carrier}
 [E_{a_N,n}^{\eps,\delta}-2h_n-n^{-B_9},
   E_{a_N,n}^{\eps,\delta}-h_n].
\end{equation}
All these intervals lie in \(\mathcal E_n\) and are common to the
four results below.
\end{corollary}

On these intervals, the short overlaps have the following simple zeros.

\begin{lemma}
\label{lem:polynomial-multiwell-short-roots}
The following hold on each interval \(\mathcal I_{n,k}\).
\begin{enumerate}
\item For every \(E\in\mathcal I_{n,k}\) and \(1\le s\le k\), the function
\(H_{a_s,n}^{\eps,\delta}(\,\cdot\,,E)\) has exactly two zeros in
\(\widetilde J_{a_s,n}\), denoted by
\[
 c_{n,k,s,+}(E)<\xi_{a_s,n}^{\eps,\delta}
 <c_{n,k,s,-}(E).
\]
For \(k<s\le N\), the function
\(H_{a_s,n}^{\eps,\delta}(\,\cdot\,,E)\) has no zero in
\(\widetilde J_{a_s,n}\).
\item For \(1\le s\le k\), in the common upper-edge orientation,
uniformly on \(\mathcal I_{n,k}\),
\begin{align}
 c_3 n^{-B_5/2}
 &\le
 |\partial_xH_{a_s,n}^{\eps,\delta}
   (c_{n,k,s,\pm}(E),E)|
 \le C_5 n^{-B_4/2},
 \label{eq:multiwell-short-root-slopes}\\
 c_3 n^{B_4/2}
 &\le|\partial_Ec_{n,k,s,\pm}(E)|
 \le C_5n^{B_5/2},
 \qquad
 \partial_Ec_{n,k,s,+}>0>\partial_Ec_{n,k,s,-}.
 \label{eq:multiwell-short-root-velocities}
\end{align}
\end{enumerate}
\end{lemma}

For the intervals and short zero branches just obtained, set
\(E_{n,k}^\circ:=\operatorname{mid}(\mathcal I_{n,k})\).
The next result chooses one inducing base for each of these intervals.
For a width \(w_{n,k}>0\), write
\begin{equation}\label{eq:multiwell-K-base}
\begin{aligned}
 U_{n,k,s,\pm}
 &:=[c_{n,k,s,\pm}(E_{n,k}^\circ)-w_{n,k},
     c_{n,k,s,\pm}(E_{n,k}^\circ)+w_{n,k}],\\
 \mathcal K_{n,k}
 &:=\bigcup_{s=1}^k(U_{n,k,s,-}\cup U_{n,k,s,+}).
\end{aligned}
\end{equation}

\begin{lemma}
\label{lem:polynomial-multiwell-common-base}
There exists
\(w_{n,k}\in[n^{-B_7},2n^{-B_7}]\) with the following
properties.
\begin{enumerate}
\item The \(2k\) intervals in \eqref{eq:multiwell-K-base} are disjoint.
\item For every \((s,\pm)\),
\(c_{n,k,s,\pm}(E)\in U_{n,k,s,\pm}\) for
\(E\in\mathcal I_{n,k}\), and the set
\(\{c_{n,k,s,\pm}(E):E\in\mathcal I_{n,k}\}\) is contained in the
interior of a single common \(\mathcal J_n\)-return atom, and
\begin{equation}\label{eq:multiwell-return-boundary-clearance}
 \operatorname{dist}_{\mathbb T}
 \bigl(c_{n,k,s,\pm}(E),
       \mathscr B_{\mathcal K_{n,k}}^{\rm ret}\bigr)
 \ge\tfrac12 n^{-B_8}.
\end{equation}
Here \(\mathscr B_{\mathcal K_{n,k}}^{\rm ret}\) is the boundary set from
Definition~\ref{fr:def:common-return-partition}; it is independent of
\(E\).

\item For every
\(x\in\mathcal K_{n,k}\cap
\operatorname{Reg}(\{\mathcal J_n,\mathcal K_{n,k}\})\), the
long first-return times satisfy
\begin{equation}\label{eq:multiwell-cross-return-bounds}
 c_3 n^{B_6/\tau}
 \le r_{\mathcal K_{n,k}}^\pm(x)
 \le C_5n^{B_7\tau}.
\end{equation}
\end{enumerate}
The widths can be chosen so that the two results below hold for these
same widths and bases.
\end{lemma}

For \(\mathcal B\in\mathcal P_{\mathcal K_{n,k}}\), set
\(R_{\mathcal B}^+:=p_{\mathcal B}\) and
\(R_{\mathcal B}^-:=q_{\mathcal B}\).  Use the atomwise SVD data and
angle function from \eqref{fr:eq:complete-return-joining-convention} for
the base \(\mathcal K_{n,k}\).

\begin{proposition}
\label{prop:polynomial-multiwell-long-returns}
For the common intervals and bases fixed above, the following hold.
\begin{enumerate}
\item On \(\mathcal B\times\mathcal I_{n,k}\),
\begin{equation}\label{eq:multiwell-long-growth}
 \lambda_{\mathcal K_{n,k}}^\pm(x,E)
 \ge\exp\!\left\{
 \gamma R_{\mathcal B}^\pm
 -C_5(R_{\mathcal B}^\pm)^{1-\eta_{\rm mw}}
  (\log(R_{\mathcal B}^\pm+2))^B\right\}.
\end{equation}
\item For every \(\beta\in\mathbb N_0^2\) with \(1\le|\beta|\le r_*\), on
\(\mathcal B\times\mathcal I_{n,k}\),
\begin{align}
 |\partial_{x,E}^\beta
   \lambda_{\mathcal K_{n,k}}^\pm|
 &\le C_5(R_{\mathcal B}^\pm+2)^{B_{16}}
       \lambda_{\mathcal K_{n,k}}^\pm,
 \label{eq:multiwell-long-modulus-derivatives}\\
 |\partial_{x,E}^\beta u_{\mathcal K_{n,k}}^\pm|
 +|\partial_{x,E}^\beta s_{\mathcal K_{n,k}}^\pm|
 &\le C_5(R_{\mathcal B}^\pm+2)^{B_{16}},
 \label{eq:multiwell-long-angle-derivatives}\\
 \|\partial_{x,E}^\beta B_{\mathcal K_{n,k}}^\pm\|
 &\le C_5(R_{\mathcal B}^\pm+2)^{B_{16}}
       \lambda_{\mathcal K_{n,k}}^\pm.
 \label{eq:multiwell-long-matrix-derivatives}
\end{align}

\item For every return atom \(\mathcal B\in\mathcal P_{\mathcal K_{n,k}}\)
contained in \(U_{n,k,s,\pm}\),
\begin{equation}\label{eq:multiwell-long-short-Cr-tail}
 \|\cos G_{\mathcal K_{n,k},\mathcal B}
      -H_{a_s,n}^{\eps,\delta}\|_{C^{r_*}}
 \le C_5n^{D_4}
 e^{-c_3 n^{\eta_0}}
\end{equation}
on \(\mathcal B\times\mathcal I_{n,k}\).
All mixed derivatives of every atomwise long-return overlap coefficient
through order \(r_*\) are bounded by \(C_5n^{D_4}\) on its
phase--energy domain.
\end{enumerate}
\end{proposition}

By \eqref{eq:multiwell-return-boundary-clearance}, for each
\((s,\pm)\) there is a unique common return atom containing its short
zero branch.  Write
\(\mathcal B_{n,k,s,\pm}\) for this atom, so that
\[
 c_{n,k,s,\pm}(E)\in\mathcal B_{n,k,s,\pm},
 \qquad E\in\mathcal I_{n,k}.
\]

\begin{corollary}
\label{cor:polynomial-multiwell-induced-zeros}
For the same intervals, bases, and oriented overlaps, the following hold.
\begin{enumerate}
\item If
\(\mathcal B\subset U_{n,k,s,\pm}\) is any other atom, then
\begin{equation}\label{eq:multiwell-nondistinguished-nozero}
 \inf_{\mathcal B\times\mathcal I_{n,k}}
 |\cos G_{\mathcal K_{n,k},\mathcal B}|
 \ge c_3 n^{-D_4}.
\end{equation}
\item The distinguished atoms collectively have exactly \(2k\) zeros
\(\widehat c_{n,k,s,\pm}(E)\), and
\[
 |\widehat c_{n,k,s,\pm}(E)-c_{n,k,s,\pm}(E)|
 \le C_5n^{D_4}e^{-c_3n^{\eta_0}}.
\]
Moreover,
\begin{equation}\label{eq:multiwell-long-root-slopes}
 c_3 n^{-B_5/2}
 \le
 |\partial_x\cos G_{\mathcal K_{n,k},\mathcal B_{n,k,s,\pm}}
   (\widehat c_{n,k,s,\pm}(E),E)|
 \le C_5n^{-B_4/2}.
\end{equation}
\item On each distinguished atom there is the exact transverse
factorization
\begin{equation}\label{eq:multiwell-long-transverse-factorization}
 \cos G_{\mathcal K_{n,k},\mathcal B_{n,k,s,\pm}}(x,E)
 =\bigl(x-\widehat c_{n,k,s,\pm}(E)\bigr)
   \mathfrak a_{n,k,s,\pm}(x,E),
\end{equation}
where, on
\(\mathcal B_{n,k,s,\pm}\times\mathcal I_{n,k}\),
\begin{equation}\label{eq:multiwell-long-transverse-coefficient}
 c_3 n^{-B_5/2}
 \le|\mathfrak a_{n,k,s,\pm}(x,E)|
 \le C_5n^{-B_4/2},
 \qquad
 \|\mathfrak a_{n,k,s,\pm}\|_{C^{r_*-1}}
 \le C_5n^{D_4}.
\end{equation}
\item The induced zero branches satisfy
\begin{equation}\label{eq:multiwell-long-root-velocities}
 c_3 n^{B_4/2}
 \le|\partial_E\widehat c_{n,k,s,\pm}(E)|
 \le C_5n^{B_5/2},
 \qquad
 \partial_E\widehat c_{n,k,s,+}>0>
 \partial_E\widehat c_{n,k,s,-}.
\end{equation}
The zero branches also satisfy
\begin{equation}\label{eq:multiwell-root-high-order-bound}
 \|\widehat c_{n,k,s,\pm}\|_{C^{r_*}(\mathcal I_{n,k})}
 \le C_5n^{D_4}.
\end{equation}
\end{enumerate}
\end{corollary}

\begin{proof}[Proof of Corollary~\ref{cor:polynomial-multiwell-K-reinduction}
and the four subsequent results]
In this proof, \(c=c(R,r_*,\mathbf b)\) and
\(C=C(R,r_*,\mathbf b)\) denote auxiliary bounds that may change from
one estimate to the next, with the fixed background dependence understood.
The constants \(c_3,C_5\) in the statements are chosen simultaneously
at the end.
\smallskip
\noindent\emph{Step (i): Energy intervals.}
Let \(B_{17}\) be the exponent supplied by
Lemma~\ref{B:lem:polynomial-hyperbolic-chain} at order \(r_*\).  Set
\[
\begin{gathered}
 D_{16}:=3+\max\{D_1,B_5,2B_7,
 B_8+\tfrac12B_5,r_*D_2\},\qquad
 \gamma_0:=\gamma/2,\\
 A_0>2+\max\{2\tau^2,2B_7\tau^2/C_I\},\qquad
 D_0>2+2\tau D_{16}/C_I,\\
 K_{\rm par}:=\max\{r_*M_{r_*+1},\tau D_2/C_I\},
\end{gathered}
\]
Choose
\[
\begin{gathered}
 B_{16}>r_*B_{17}[A_0(1+K_{\rm par})+D_0]+2,\\
 D_4>\max\left\{D_{16}+3,
 (2r_*-1)\left(r_*B_7\tau B_{16}+\tfrac12B_5\right)+1\right\},
\end{gathered}
\]
with \(D_4\) also larger than the finitely many fixed-order losses
in the chain, endpoint, coefficient, and overlap-coefficient
estimates below.

Write
\(T_s=E_{a_s,n}^{\eps,\delta}\).  For \(1\le k<N\),
\begin{equation}\label{eq:multiwell-intermediate-carrier-length}
\begin{aligned}
 |[T_{k+1}+h_n,T_k-h_n]|
 &=T_k-T_{k+1}-2h_n\\
 &\ge n^{-B_3}-2n^{-B_5}
 \ge\tfrac12n^{-B_3}
 >n^{-B_9}
\end{aligned}
\end{equation}
for all large \(n\), by
\(B_5>B_3\) and \(B_9>B_3\).
The bottom interval has length \(h_n+n^{-B_9}\), and adjacent
intervals have distance \(2h_n\). Thus choose
\[
 |\mathcal I_{n,k}|=n^{-B_9},\qquad
 \mathcal I_{n,k}\cap\mathcal I_{n,j}=\varnothing\quad(k\ne j)
\]
inside the corresponding intervals.

After increasing \(n_{\rm mw}\), the triangle inequality and
\eqref{eq:perturbed-edge-polynomial-margin} give, for every
\(E\in\mathcal I_{n,k}\),
\[
 |E-E_+|\le C(n^{-B_4}+n^{-B_5}+n^{-B_9})
              +\tfrac14n^{-B_2}<n^{-B_2},
\]
by \eqref{eq:multiwell-seed-containment-order}--
\eqref{eq:multiwell-polynomial-exponent-order}.  Hence
\(\mathcal I_{n,k}\subset\mathcal E_n\) before any central normal form is
used.  The same estimate places every \(T_s\) in the interval
\(\mathcal E_n\), so each segment joining \(E\) to \(T_s\) stays there.
The interval endpoints give \(T_s-E\ge h_n\) for \(s\le k\) and
\(E-T_s\ge h_n\) for \(s>k\), the latter being vacuous for \(k=N\).
Call \(a_s\) active if \(s\le k\) and inactive otherwise; its root
intervals, when present, are \(U_{n,k,s,+}\) and \(U_{n,k,s,-}\).

\smallskip
\noindent\emph{Step (ii): Zeros and separation.}
For the first part, use the exact perturbed normal form and its coefficient
bounds \eqref{eq:perturbed-recentered-normal-form}--
\eqref{eq:perturbed-recentered-coefficient-bounds}.  On an active branch set
\(D_{s}(E)=E_{a_s,n}^{\eps,\delta}-E\) and write
\(H=H_{a_s,n}^{\eps,\delta}\).  By construction of the energy intervals,
\begin{equation}\label{eq:multiwell-depth-range}
 h_n\le D_s(E)\le Cn^{-B_4},
\end{equation}
where the upper bound follows from
\eqref{eq:multiwell-fold-input},
\eqref{eq:multiwell-intermediate-layer-carrier}, and
\eqref{eq:multiwell-bottom-layer-carrier}.  By the choice of
\(B_4\),
\(D_s=o_{n\to\infty}(\min\{n^{-B_2},n^{-2B_1}\})\), and
\eqref{eq:perturbed-fold-displacement-polynomial}--
\eqref{eq:polynomial-starting-scale-smallness} give
\(|\xi_{a_s,n}^{\eps,\delta}-p_{a_s}|
\le(8N)^{-1}n^{-B_1}\).  Thus the interval
\[
 |x-\xi_{a_s,n}^{\eps,\delta}|
 \le r_0:=\frac1{8N}n^{-B_1}
\]
lies in both \(\Omega_{a_s,n}^{\eps,\delta}\) and
\(\widetilde J_{a_s,n}\), where the exact recentered form gives
\[
 H(\xi_{a_s,n}^{\eps,\delta},E)
 =-C_{2,a_s,n}^{\eps,\delta}D_s(E)<0,
 \qquad
 H(\xi_{a_s,n}^{\eps,\delta}\pm r_0,E)
 \ge cr_0^2-CD_s(E)>0.
\]
Since \(\partial_x^2H>0\) on \(\widehat\Omega_{a_s,n}\), there is exactly
one root on each side and no other root in \(\widetilde J_{a_s,n}\).
Comparing the two terms in the root equation and using
\eqref{eq:multiwell-depth-range} gives
\[
 c\sqrt{D_s(E)}\le
 |c_{n,k,s,\pm}(E)-\xi_{a_s,n}^{\eps,\delta}|
 \le C\sqrt{D_s(E)}.
\]
Put
\[
 c_L=c_{n,k,s,+}(E),\qquad c_R=c_{n,k,s,-}(E),
\]
and let \(x_E\in(c_L,c_R)\) be the unique critical point of
\(x\mapsto H_{a_s,n}^{\eps,\delta}(x,E)\).  If
\(r_L=x_E-c_L\) and \(r_R=c_R-x_E\), then
\(H(c_L,E)=H(c_R,E)=0\), \(H_x(x_E,E)=0\), and
\eqref{eq:perturbed-central-transversality} gives
\[
 \frac{c}{2}r_L^2\le-H(x_E,E)\le\frac{C}{2}r_L^2,
 \qquad
 \frac{c}{2}r_R^2\le-H(x_E,E)\le\frac{C}{2}r_R^2.
\]
Thus \(r_L\asymp r_R\).  Since \(\xi_{a_s,n}^{\eps,\delta}\) lies
between the two roots and both root distances from it are comparable to
\(\sqrt{D_s(E)}\), one has
\[
 r_L\asymp r_R\asymp c_R-c_L\asymp\sqrt{D_s(E)}.
\]
Integrating \(H_{xx}\) from \(x_E\) to either root now yields, separately,
\[
 |\partial_xH_{a_s,n}^{\eps,\delta}(c_L,E)|
 \asymp\sqrt{D_s(E)},
 \qquad
 |\partial_xH_{a_s,n}^{\eps,\delta}(c_R,E)|
 \asymp\sqrt{D_s(E)}.
\]
Together with \eqref{eq:multiwell-depth-range}, this proves
\eqref{eq:multiwell-short-root-slopes}.  Finally,
\eqref{eq:perturbed-central-transversality} gives
\[
 c\le
 \partial_EH_{a_s,n}^{\eps,\delta}
 (c_{n,k,s,\pm}(E),E)
 \le C,
\]
and the implicit formula
\[
 \partial_Ec_{n,k,s,\pm}(E)
 =-\frac{\partial_EH_{a_s,n}^{\eps,\delta}}
         {\partial_xH_{a_s,n}^{\eps,\delta}}
 \quad\text{at the root}
\]
gives \eqref{eq:multiwell-short-root-velocities}, including the signs.
The same formula and the preceding energy-derivative bound show that,
across one energy interval, every zero branch present there moves by at most
\[
 Cn^{B_5/2-B_9}
 \le\tfrac14n^{-B_8}.
\]
Since \(B_8>B_7\), this displacement is smaller than
\(w_{n,k}\), and hence
\(c_{n,k,s,\pm}(E)\in U_{n,k,s,\pm}\) throughout the energy interval.
For \(f''\ge c\) and \(f(c_L)=f(c_R)=0\), convexity of
\(f-\frac{c}{2}(x-c_L)(x-c_R)\), between the roots, and integration of
\(f'\), outside them, give
\begin{equation}\label{eq:multiwell-convex-product-bound}
 |f(x)|\ge\frac{c}{2}|x-c_L|\,|x-c_R|.
\end{equation}
Fix \(\mathcal A\in\mathcal P_{\mathcal J_n}\),
\(E\in\mathcal I_{n,k}\), and abbreviate
\(H_s=H_{a_s,n}^{\eps,\delta}\).  For the three parts of
\(\mathcal A\cap(\mathcal J_n\setminus\mathcal K_{n,k})\),
\[
\begin{aligned}
 &s>k,\quad x\in\widetilde J_{a_s,n}:\\
 &\qquad H_s(x,E)\ge H_s(x,T_s)+c(E-T_s)\ge ch_n;\\
 &s\le k,\quad
 x\in\widetilde J_{a_s,n}\setminus(U_{n,k,s,+}\cup U_{n,k,s,-}):\\
 &\qquad \min_{\varsigma=\pm}|x-c_{n,k,s,\varsigma}(E)|\ge w_{n,k}/2,\\
 &\qquad |H_s(x,E)|\ge(c/2)|x-c_L|\,|x-c_R|\ge(c/8)w_{n,k}^2;\\
 &x\notin\bigcup_s\widetilde J_{a_s,n}:\\
 &\qquad |\cos G_{\mathcal A}^{\eps,\delta}(x,E)|
 \ge(c/2)n^{-D_1}.
\end{aligned}
\]
The first case uses the fold equations and convexity at \(T_s\);
the second uses
\(Cn^{B_5/2-B_9}\le\frac14n^{-B_8}
=o(w_{n,k})\); the last is
\eqref{eq:perturbed-exterior-mismatch-separation}.
Each component of \(\mathcal K_{n,k}\) is compactly contained in its
central parent atom.  The same bounds hold on \(\partial\mathcal K_{n,k}\)
by continuity.
Since \(D_{16}>\max\{B_5,2B_7,D_1\}\), uniformly on every parent atom,
\begin{equation}\label{eq:multiwell-preliminary-safe-region}
 |\cos G_{\mathcal A}^{\eps,\delta}(x,E)|
 \ge c n^{-D_{16}},
 \quad
 x\in\mathcal A\cap
 \bigl(\mathcal J_n\setminus
       \operatorname{int}\mathcal K_{n,k}\bigr),
 \quad E\in\mathcal I_{n,k}.
\end{equation}
\smallskip
\noindent\emph{Step (iii): Return geometry.}
The within-peak root separation is at least
\(cn^{-B_5/2}\), whereas the radius is
\(O(n^{-B_7})\).  For \(a_s\ne a_t\) and
\(\varsigma,\varsigma'\in\{+,-\}\),
\[
 \operatorname{dist}_{\mathbb T}
 (U_{n,k,s,\varsigma},U_{n,k,t,\varsigma'})
 \ge N^{-1}-o_{n\to\infty}(1).
\]
Thus the intervals in
\eqref{eq:multiwell-K-base} are disjoint.  The same exponent ordering
places every \(U_{n,k,s,\pm}\) strictly inside the distinguished old
return atom \(\mathcal A_{a_s,n}\).  Fix a bijection
\[
 \nu\longleftrightarrow(s(\nu),\varsigma(\nu))
 \quad\text{between}\quad
 \{1,\ldots,2k\}\ \text{and}\ \{1,\ldots,k\}\times\{+,-\},
\]
and, for \(r\in[n^{-B_7},2n^{-B_7}]\), put
\[
\begin{aligned}
 z_\nu&:=c_{n,k,s(\nu),\varsigma(\nu)}(E_{n,k}^\circ),
 & a(\nu)&:=a_{s(\nu)},\\
 U_\nu(r)&:=[z_\nu-r,z_\nu+r],
 & K(r)&:=\bigcup_{\nu=1}^{2k}U_\nu(r).
\end{aligned}
\]
For \(1\le\nu\le2k\) and
\(x\in U_\nu(r)\cap\operatorname{Reg}(K(r))\), since
\(\partial U_\nu(r)\subset\partial K(r)\),
Lemma~\ref{B:lem:return-time-poly-two-sided} gives
\[
 r_{K(r)}^\pm(x)\le r_{U_\nu(r)}^\pm(x)
 \le\frac{4^\tau}{\kappa}|U_\nu(r)|^{-\tau}
 \le\frac{4^\tau}{\kappa}n^{B_7\tau}.
\]
Thus the common bound
\[
 R_n:=\left\lceil\frac{4^\tau}{\kappa}n^{B_7\tau}\right\rceil
\]
holds for both return times, uniformly in the radius interval.
Since a return time can change only when an orbit
point meets an endpoint,
\[
 \mathscr B_{K(r)}^{\rm ret}
 \subset
 K(r)\cap
 \bigcup_{|q|\le R_n}T_\alpha^{-q}\bigl(\partial K(r)\bigr).
\]
Consequently, a fixed center \(z_\nu\) can meet the return boundary only
at a radius satisfying
\[
 z_\nu+q\alpha
 \equiv z_\mu+\sigma r\pmod1,
 \qquad
 1\le\nu,\mu\le2k,\quad
 |q|\le R_n,\quad \sigma\in\{+1,-1\}.
\]
For each fixed \((\nu,\mu,q,\sigma)\), this congruence has at most one
solution in
\[
 [n^{-B_7}-2n^{-B_8},
  2n^{-B_7}+2n^{-B_8}]
\]
for all large \(n\).  Hence there are at most \(24N^2R_n\) relevant
candidate radii.  From the original radius interval remove the
intersection with the \(2n^{-B_8}\)-neighborhood of every such
candidate.  The total removed length is at most
\[
 Cn^{B_7\tau-B_8}
 =o_{n\to\infty}(n^{-B_7})
\]
by \eqref{eq:multiwell-polynomial-exponent-order}.  Thus one may choose
\(w_{n,k}\) outside all the removed neighborhoods and set
\[
 U_\nu:=U_\nu(w_{n,k}),\qquad
 \mathcal K_{n,k}:=K(w_{n,k}).
\]
For this choice,
\[
 \operatorname{dist}_{\mathbb T}
 (z_\nu,\mathscr B_{\mathcal K_{n,k}}^{\rm ret})
 \ge2n^{-B_8}.
\]
The preceding zero-displacement estimate is at most
\(\frac14n^{-B_8}\); therefore
\eqref{eq:multiwell-return-boundary-clearance} holds uniformly for
\(E\in\mathcal I_{n,k}\).
Connectedness of each zero graph and the clearance give its unique
atom \(\mathcal B_{n,k,s,\pm}\).

For the lower return bound, write
\(c_\nu:=z_\nu\).  The root estimate above
and \eqref{eq:multiwell-fold-input} imply
\begin{equation}\label{eq:multiwell-component-center-peak-alignment}
 \left\|c_\mu-c_\nu-
 \frac{a(\mu)-a(\nu)}N\right\|_{\mathbb T}
 \le Cn^{-B_6}.
\end{equation}
For centers satisfying
\eqref{eq:multiwell-component-center-peak-alignment}, if
\[
 x\in U_\nu\cap
 \operatorname{Reg}(\{\mathcal J_n,\mathcal K_{n,k}\}),
 \qquad x+q\alpha\in U_\mu,\qquad q\ne0,
\]
then
\[
 \left\|q\alpha-\frac{a(\mu)-a(\nu)}N\right\|_{\mathbb T}
 \le\tfrac12 C\bigl(n^{-B_6}+w_{n,k}\bigr)
 \le Cn^{-B_6}.
\]
Multiplication by the fixed integer \(N\) and the Diophantine inequality at
the integer \(Nq\) give
\[
 \kappa|Nq|^{-\tau}
 \le\|Nq\alpha\|_{\mathbb T}
 \le CNn^{-B_6},
\]
Thus
\[
 |q|\ge cn^{B_6/\tau}
 \quad(q\ne0),
\]
for same-peak and cross-peak returns in either direction.  Together with
the upper return bound \(R_n\), this proves
\eqref{eq:multiwell-cross-return-bounds}.

\smallskip
\noindent\emph{Step (iv): Induced matrices.}
Work first on the common regular set
\(\operatorname{Reg}(\{\mathcal J_n,\mathcal K_{n,k}\})\) and decompose a
first return to \(\mathcal K_{n,k}\) into consecutive returns to
\(\mathcal J_n\).  On each fixed itinerary atom, if the successive
\(\mathcal J_n\)-return lengths are \(p_1,\ldots,p_L\), then the cocycle
identity gives both the exact product decomposition and
\(R_{\mathcal B}^\pm=\sum_{j=1}^Lp_j\).  Every internal concatenation point
lies in the noncritical region covered by
\eqref{eq:multiwell-preliminary-safe-region}.  The short blocks satisfy
\eqref{eq:perturbed-return-growth-plus}--
\eqref{eq:perturbed-return-growth-minus} and
\eqref{eq:perturbed-return-derivative-plus}--
\eqref{eq:perturbed-return-derivative-minus}.
For either oriented chain, denote its chronological short-return matrices
by \(M_1,\ldots,M_L\), so that the complete long-return matrix is
\(\mathcal M:=M_L\cdots M_1\), and choose the coherent standard SVDs
\[
 M_i=R_{u_i}\operatorname{diag}(\lambda_i,\lambda_i^{-1})
 R_{\pi/2-s_i},
 \qquad \lambda_i>1.
\]
For \(1\le i<L\), define the internal overlap coefficients by
\[
 h_i:=\cos(\pi/2-s_{i+1}+u_i).
\]
If \(y_i\) is the intermediate base point after the first \(i\) short
returns in this chronological chain, let
\(\mathcal A_i\in\mathcal P_{\mathcal J_n}\) be the unique parent atom
containing \(y_i\).  Then
\[
 M_i=B_{\mathcal J_n}^+(y_{i-1},E)
     =B_{\mathcal J_n}^-(y_i,E).
\]
Thus, with the atomwise orientations fixed above, the scalar \(h_i\)
agrees up to one fixed sign with the genuine parent overlap at \(y_i\):
\[
 |h_i|=
 |\cos G_{\mathcal J_n,\mathcal A_i}^{\eps,\delta}(y_i,E)|.
\]
The backward chain uses the same chronological identity; hence
\eqref{eq:multiwell-preliminary-safe-region} bounds every internal overlap.
Write \(U,S\) for the expanding left and contracting right singular
directions of
\(\mathcal M\), and define the derivative bound for the chain
\[
 Q_{r_*}:=2+
 \max_{\substack{1\le i\le L\\1\le|\beta|\le r_*}}
 \left(
  \|\partial^\beta\log\lambda_i\|_\infty
  +\|\partial^\beta u_i\|_\infty
  +\|\partial^\beta s_i\|_\infty
 \right).
\]

Recall the auxiliary
\((A_0,D_0,\gamma_0)\) fixed before \(D_4\), and put
\[
 q_{\rm ch}:=n^{C_I/(2\tau)}.
\]
After increasing the fixed starting threshold, the \(\mathcal J_n\)-return bounds,
\eqref{eq:multiwell-cross-return-bounds}, and
\eqref{eq:multiwell-preliminary-safe-region} give
\[
 q_{\rm ch}\le p_i\le q_{\rm ch}^{A_0},\qquad
 L\le q_{\rm ch}^{A_0},\qquad
 \lambda_i\ge e^{\gamma_0p_i},
\]
while \eqref{eq:perturbed-axis-angle-derivatives} and
\(D_{16}\ge r_*D_2\), followed by Fa\`a di Bruno, give,
for every internal overlap coefficient,
\[
 |h_i|\ge q_{\rm ch}^{-D_0},\qquad
 \|h_i\|_{C^{r_*}}\le q_{\rm ch}^{D_0}.
\]
The parent modulus and singular-direction estimates, together with
\(p_i\ge cn^{C_I/\tau}\), give the order-\(r_*\) derivative bound
\begin{equation}\label{eq:multiwell-chain-Qr-budget}
 Q_{r_*}\le Cq_{\rm ch}^{A_0K_{\rm par}}.
\end{equation}
Set
\[
 \varkappa:=q_{\rm ch}^{-D_0},
 \qquad
 \Lambda:=e^{\gamma_0q_{\rm ch}}.
\]
With \(B_{17}\) from the chain lemma, the left-hand side of its
smallness condition is bounded by
\[
\begin{aligned}
 L^{B_{17}}Q_{r_*}^{B_{17}}
 \varkappa^{-B_{17}}\Lambda^{-2}
 &\le C q_{\rm ch}^{B_{17}
 [A_0(1+K_{\rm par})+D_0]}
 e^{-2\gamma_0q_{\rm ch}}\\
 &=o_{q_{\rm ch}\to\infty}(q_{\rm ch}^{-D_0}) =o_{q_{\rm ch}\to\infty}(\varkappa).
\end{aligned}
\]
Thus Lemma~\ref{B:lem:polynomial-hyperbolic-chain} applies after
increasing the threshold.  On each fixed-itinerary domain put
\(\mathsf R=\sum_{j=1}^Lp_j\); the parent return estimates give
\[
\begin{gathered}
 p_j\ge cn^{C_I/\tau},
 \qquad L\le C\mathsf R n^{-C_I/\tau},\\
 \frac{R_\tau(p_j+2)}{p_j}
 \le Cn^{-C_I\sigma_{\rm cmp}/\tau}(\log(n+2))^B.
\end{gathered}
\]
The norm conclusion of the chain lemma first gives the explicit
intermediate estimate
\[
\begin{aligned}
 \log\|\mathcal M\|
 &\ge\sum_{j=1}^L\log\lambda_j
      +(L-1)\log(\varkappa/2)\\
 &\ge\gamma\mathsf R
   -C\sum_{j=1}^LR_\tau(p_j+2)-CL\log(n+2).
\end{aligned}
\]
Consequently the total logarithmic loss is bounded by
\begin{equation}\label{eq:multiwell-explicit-accumulated-loss}
 C\mathsf R\left[
 n^{-C_I\sigma_{\rm cmp}/\tau}(\log(n+2))^B
 +n^{-C_I/\tau}\log(n+2)
 \right].
\end{equation}
On the other hand,
\[
 cn^{C_I/\tau}\le\mathsf R\le Cn^{B_7\tau}.
\]
Thus \(\log(n+2)\le C\log(\mathsf R+2)\).  With, for example,
the exponent \(\eta_{\rm mw}\) fixed in
\eqref{eq:multiwell-eta-mw-explicit},
both powers of \(n\) in
\eqref{eq:multiwell-explicit-accumulated-loss} are at most
\(C\mathsf R^{-\eta_{\rm mw}}\).  This proves
\eqref{eq:multiwell-long-growth} for both return directions.
The fixed-order log-norm and endpoint estimates in
Lemma~\ref{B:lem:polynomial-hyperbolic-chain}, together with
\eqref{eq:multiwell-chain-Qr-budget}, give, for
\(1\le|\beta|\le r_*\),
\[
 |\partial^\beta\log\|\mathcal M\||
 +|\partial^\beta U|+|\partial^\beta S|
 \le
 Cq_{\rm ch}^{B_{17}[A_0(1+K_{\rm par})+D_0]}
 \le C(\mathsf R+2)^{B_{16}/r_*}.
\]
Fa\`a di Bruno and differentiation of the singular value decomposition
of \(\mathcal M\) now give
\eqref{eq:multiwell-long-modulus-derivatives}--
\eqref{eq:multiwell-long-matrix-derivatives}.
By \eqref{fr:eq:complete-return-joining-convention},
\eqref{eq:multiwell-long-angle-derivatives},
\eqref{eq:multiwell-cross-return-bounds}, and Fa\`a di Bruno,
\begin{equation}\label{eq:multiwell-long-overlap-mixed-bound}
 \|\cos G_{\mathcal K_{n,k},\mathcal B}\|_{C^{r_*}}
 \le Cn^{r_*B_7\tau B_{16}}
 \le Cn^{D_4}.
\end{equation}
All constants are independent of \(n\), the itinerary atom, and the
component label.

Fix any \(\mathcal K_{n,k}\)-return atom \(\mathcal B\) contained in
\(U_{n,k,s,\pm}\).  In the forward chain, the rightmost factor adjacent
to \(x\) is the parent forward block on \(\mathcal A_{a_s,n}\); in the
backward chain, the leftmost factor adjacent to \(x\) is the parent backward
block on the same atom.  Apply the fixed-order endpoint contraction
\eqref{fr:eq:chain-fixed-order-endpoint} to these two chains.  The
derivative factors \(L^{B_{17}}Q_{r_*}^{B_{17}}
\varkappa^{-B_{17}}\) are polynomial in \(n\), whereas the inverse
squares of the adjacent parent singular values are bounded by
\[
 \exp\{-2\gamma m_n+CR_\tau(m_n+2)\}
 \le e^{-c n^{\eta_0}}
\]
because the adjacent \(\mathcal J_n\)-return length
\(m_n=r_{\mathcal J_n}^+(p_{a_s})
=r_{\mathcal J_n}^-(p_{a_s})\), independent of \(s\), has a fixed-power
lower bound.  Orient the long endpoint
directions by positive inner product with the corresponding parent
directions.  The
fixed-order product rule for their scalar products then gives
\eqref{eq:multiwell-long-short-Cr-tail}, uniformly for every such
\(\mathcal B\).  Across an internal boundary of the parent-itinerary partition, the
complete fixed-word matrices are analytic.  Their projections
\eqref{fr:eq:complete-K-return-projectors} satisfy
\[
 \|\mathscr P^{s,u}_{\rm long}-\mathscr P^{s,u}_{\rm parent}\|<1/10,
 \qquad
 v_{\rm long}
 =\frac{\mathscr P_{\rm long}v_{\rm parent}}
        {\|\mathscr P_{\rm long}v_{\rm parent}\|},
 \qquad v_{\rm long}^{\mathsf T}v_{\rm parent}>0.
\]
These unique analytic lifts agree on adjacent itinerary interiors.
Continuity of every fixed derivative and density of those interiors
extend \eqref{eq:multiwell-long-short-Cr-tail} to
\(\mathcal B\times\mathcal I_{n,k}\).

\smallskip
\noindent\emph{Step (v): Induced zeros.}
On \(U_{n,k,s,\pm}\), the distance to the other short
root of the same peak is at least \(cn^{-B_5/2}\), and
\[
 w_{n,k}+
 \sup_{E\in\mathcal I_{n,k}}
 |c_{n,k,s,\pm}(E)-c_{n,k,s,\pm}(E_{n,k}^\circ)|
 =o_{n\to\infty}(n^{-B_5/2}).
\]
Strict convexity therefore gives
\[
 |\partial_xH_{a_s,n}^{\eps,\delta}|
 \ge cn^{-B_5/2}
\]
with a constant sign throughout that component.  On the distinguished
atom, the separation from the return boundary and from the other short
root gives, at both atom endpoints,
\[
 |H_{a_s,n}^{\eps,\delta}|
 \ge c n^{-B_8-B_5/2}.
\]
Writing \(\overline{\mathcal B_{n,k,s,\pm}}=[b_-,b_+]\), the error in
\eqref{eq:multiwell-long-short-Cr-tail} is
\[
 o\!\left(\min\{n^{-B_8-B_5/2},
                       n^{-B_5/2}\}\right).
\]
It preserves \(H(b_-,E)H(b_+,E)<0\) and the constant derivative sign,
so the induced overlap has exactly one zero.  Moreover,
\[
\begin{aligned}
 |\widehat c_{n,k,s,\pm}(E)-c_{n,k,s,\pm}(E)|
 &\le Cn^{D_4}e^{-cn^{\eta_0}},\\
 c
 &\le
 \partial_E\cos G_{\mathcal K_{n,k},\mathcal B_{n,k,s,\pm}}
 (\widehat c_{n,k,s,\pm}(E),E)
 \le C.
\end{aligned}
\]
Write
\[
 F(x,E):=\cos
 G_{\mathcal K_{n,k},\mathcal B_{n,k,s,\pm}}(x,E),
 \qquad \widehat c(E):=\widehat c_{n,k,s,\pm}(E).
\]
More explicitly, the \(C^1\) comparison and the root-displacement estimate
give
\[
\begin{aligned}
 \left|F_x(\widehat c(E),E)
 -\partial_xH_{a_s,n}^{\eps,\delta}
   (c_{n,k,s,\pm}(E),E)\right|
 &\le Cn^{D_4}e^{-cn^{\eta_0}}\\
 &=o_{n\to\infty}(n^{-B_5/2}).
\end{aligned}
\]
The fundamental theorem of calculus gives the exact formula
\[
 \mathfrak a_{n,k,s,\pm}(x,E)
 :=\int_0^1
 F_x\bigl(\widehat c(E)+t(x-\widehat c(E)),E\bigr)\,dt.
\]
Both \(x\) and \(\widehat c(E)\) belong to the same distinguished return
atom \(\mathcal B_{n,k,s,\pm}\subset U_{n,k,s,\pm}\).  The \(C^2\)
comparison \eqref{eq:multiwell-long-short-Cr-tail} and
\eqref{eq:perturbed-central-transversality} give
\[
 |x-\widehat c(E)|\le 2w_{n,k},
 \qquad \|F_{xx}\|_{L^\infty}\le C.
\]
Consequently,
\[
 |F_x(x,E)-F_x(\widehat c(E),E)|
 \le\tfrac12 Cw_{n,k}\le Cn^{-B_7}
 =o_{n\to\infty}(n^{-B_5/2}),
\]
where \(B_7>B_5/2+4\).  Thus \(F_x\) keeps a constant
sign on the segment from \(\widehat c(E)\) to \(x\), and uniformly there
\[
 c n^{-B_5/2}\le |F_x|\le Cn^{-B_4/2}.
\]
The integral formula proves
\eqref{eq:multiwell-long-transverse-factorization}--
\eqref{eq:multiwell-long-transverse-coefficient}, including the claimed
decaying upper bound on the whole distinguished atom.  Differentiating the
root equation in \(E\) and using the preceding phase and energy derivative
bounds proves
\eqref{eq:multiwell-long-root-velocities}.
Put \(\mathscr Q_1:=F_E(\widehat c(E),E)\), so that
\(F_x\widehat c^{\,\prime}=-\mathscr Q_1\).
The higher derivatives of the root curve are
polynomially controlled.  For \(2\le j\le r_*\), let
\(\mathscr Q_j\) denote the universal Fa\`a di Bruno polynomial formed
by all terms not containing \(\partial_E^j\widehat c\).  Differentiating
\(F(\widehat c(E),E)=0\) exactly \(j\) times gives
\begin{equation}\label{eq:multiwell-root-derivative-recursion}
\begin{aligned}
 F_x(\widehat c(E),E)\,\partial_E^j\widehat c(E)
 &=-\mathscr Q_j\biggl(
  \partial_x^a\partial_E^bF(\widehat c(E),E),\\
 &\hspace{35mm}
  \partial_E^i\widehat c(E):
  a+b\le j,\ 1\le i<j
 \biggr).
\end{aligned}
\end{equation}
Use the first bound in \eqref{eq:multiwell-long-overlap-mixed-bound}
and put \(d:=r_*B_7\tau B_{16}\).  Then
\[
 \|F\|_{C^{r_*}}\le Cn^d,\qquad |F_x|\ge cn^{-B_5/2}.
\]
The recursion and induction on \(j\) give
\[
 |\partial_E^j\widehat c|
 \le Cn^{(2j-1)(d+B_5/2)},\qquad 1\le j\le r_*.
\]
Differentiating the integral formula for \(\mathfrak a\) gives, in turn,
\[
 \|\widehat c\|_{C^{r_*}}
 +\|\mathfrak a_{n,k,s,\pm}\|_{C^{r_*-1}}
 \le Cn^{(2r_*-1)(d+B_5/2)}.
\]
The initial choice of \(D_4\) therefore proves
\eqref{eq:multiwell-root-high-order-bound} and
\eqref{eq:multiwell-long-transverse-coefficient}.

If \(\mathcal B\) is any other atom in the same component
\(U_{n,k,s,\pm}\), then
\eqref{eq:multiwell-return-boundary-clearance} gives
\[
 \operatorname{dist}\bigl(\mathcal B,
 c_{n,k,s,\pm}(E)\bigr)\ge\tfrac12n^{-B_8}.
\]
The whole component \(U_{n,k,s,\pm}\) is separated from the other root
\(c_{n,k,s,\mp}(E)\) of the same fold by at least
 \(c n^{-B_5/2}\), as proved above.  Applying
\eqref{eq:multiwell-convex-product-bound} to these two roots gives
\[
 \inf_{\mathcal B\times\mathcal I_{n,k}}
 |H_{a_s,n}^{\eps,\delta}|
 \ge c n^{-B_8-B_5/2}.
\]
This polynomial quantity dominates the exponentially small right-hand
side of \eqref{eq:multiwell-long-short-Cr-tail}, proving
\eqref{eq:multiwell-nondistinguished-nozero}.
Choose \(c_3>0\) below all the finitely many lower-bound and decay
coefficients, and \(C_5\) above all the upper-bound coefficients needed
in the five statements.  Then choose \(n_{\rm mw}\ge n_0\) to dominate
all thresholds used above.  The return cutoff \(R_n\) and the phase
radius interval are independent of these coefficient choices.  This
proves the five statements with one common set of constants.
\end{proof}
\subsection{Limiting Schur-complement first variations}

Recall the covering points \(p_a=a/N\) from the Introduction and the even
upper-edge eigenfunction \(u\) from
Proposition~\ref{fr:prop:AMO-edge-state}.
Estimate~\eqref{eq:AMO-seed-edge-state-decay} makes the following series
absolutely convergent.  Define
\begin{align}
 \mathscr L_a^{\rm Sch}(\delta)
 &:=-\delta(p_a)
   -\sum_{j=1}^\infty u_j^2
    [\delta(p_a+j\alpha)+\delta(p_a-j\alpha)]
 \label{eq:polynomial-limiting-Schur-response}\\
 \mathfrak b_\infty^{\rm Sch}
 &:=1+2\sum_{j=1}^\infty u_j^2>0,
 \label{eq:polynomial-limiting-Schur-slope}\\
 \mathscr D_{a,b}^{\rm Sch}
 &:=\mathscr L_a^{\rm Sch}-\mathscr L_b^{\rm Sch}.
 \label{eq:polynomial-Schur-difference}
\end{align}
For every \(\delta\in X\),
\[
 |\mathscr L_a^{\rm Sch}(\delta)|
 \le \mathfrak b_\infty^{\rm Sch}\|\delta\|_\infty.
\]
Here \(X^*\) denotes the continuous dual of \(X\).  In particular,
\(\mathscr L_a^{\rm Sch}\in X^*\).  For
\(x\in\mathbb T\), let \(\operatorname{ev}_x\in X^*\) be the evaluation
functional \(\operatorname{ev}_x(\delta):=\delta(x)\).  Cauchy's estimate
in the fixed strip gives
\begin{equation}\label{eq:polynomial-evaluation-functional-Lipschitz}
 \|\operatorname{ev}_x-\operatorname{ev}_y\|_{X^*}
 \le C\|x-y\|_{\mathbb T}.
\end{equation}

\begin{definition}
\label{def:quantitative-Schur-classes}
Set
\begin{equation}\label{eq:explicit-Pstrat}
 \mathcal P_{\mathrm{strat}}
 :=\{\delta\in X:
 \mathscr D_{a,b}^{\rm Sch}(\delta)\ne0
 \text{ for every }0\le a<b<N\}
\end{equation}
and
\[
 \sigma_{\rm Sch}(\delta)
 :=\frac1{\mathfrak b_\infty^{\rm Sch}}
   \min_{0\le a<b<N}|\mathscr D_{a,b}^{\rm Sch}(\delta)|.
\]
For integers \(R,q\ge1\), put
\begin{equation*}
 \mathcal P_{\mathrm{strat}}(R,q)
 :=\{\delta\in X:\|\delta\|_\rho\le R,
                    \ \sigma_{\rm Sch}(\delta)\ge q^{-1}\}.
\end{equation*}
\end{definition}
\begin{proposition}
\label{prop:standard-prevalence-Schur}
The set \(\mathcal P_{\mathrm{strat}}\) of
Definition~\ref{def:quantitative-Schur-classes} is open and prevalent in \(X\).
\end{proposition}

\begin{proof}
For \(x\in\mathbb T\), let \(\delta_x\) denote the unit Dirac point
mass at \(x\).  For \(0\le a<N\), define the finite positive measure
\[
 \mu_a^{\rm Sch}:=\delta_{p_a}
 +\sum_{j=1}^\infty u_j^2
   (\delta_{p_a+j\alpha}+\delta_{p_a-j\alpha}).
\]
Then
\[
 \mathscr L_a^{\rm Sch}(\delta)
 =-\int_{\mathbb T}\delta\,d\mu_a^{\rm Sch},
 \qquad
 \mu_a^{\rm Sch}(\mathbb T)
 =\mathfrak b_\infty^{\rm Sch}.
\]
For \(a\ne b\), irrationality gives
\[
\begin{aligned}
 p_a+j\alpha=p_b+\ell\alpha\pmod1
 &\Longrightarrow N(j-\ell)\alpha\in\mathbb Z\\
 &\Longrightarrow j=\ell,\quad a-b\in N\mathbb Z,
\end{aligned}
\]
a contradiction.  Hence
\[
 \mu_a^{\rm Sch}(\{p_a\})=1,\qquad
 \mu_b^{\rm Sch}(\{p_a\})=0
 \quad\Longrightarrow\quad
 \mu_a^{\rm Sch}\ne\mu_b^{\rm Sch}
 \quad\Longrightarrow\quad
 \mathscr D_{a,b}^{\rm Sch}\ne0,
\]
since trigonometric polynomials belong to \(X\) and determine finite
Borel measures.

There are only finitely many functionals
\(\mathscr D_{a,b}^{\rm Sch}\), and each is nonzero.  Hence their kernels,
being proper linear subspaces of the real vector space \(X\), do not cover
\(X\).  Choose
\[
 w\in X\setminus
 \bigcup_{0\le a<b<N}\ker\mathscr D_{a,b}^{\rm Sch}.
\]
Let \(\nu\) be the pushforward of normalized Lebesgue measure on
\([-1,1]\) under \(t\mapsto tw\).  Its support is the compact segment
\([-w,w]\).  For any \(\delta_0\in X\) and every \(a<b\),
\[
 \mathscr D_{a,b}^{\rm Sch}(\delta_0+tw)
 =\mathscr D_{a,b}^{\rm Sch}(\delta_0)
  +t\mathscr D_{a,b}^{\rm Sch}(w),
\]
so
\[
\begin{gathered}
 \#\{t\in[-1,1]:
       \mathscr D_{a,b}^{\rm Sch}(\delta_0+tw)=0\}\le1,\\
 \nu\{v:\delta_0+v\notin\mathcal P_{\mathrm{strat}}\}=0
 \qquad(\delta_0\in X),\\
 X\setminus\mathcal P_{\mathrm{strat}}
 =\bigcup_{a<b}\ker\mathscr D_{a,b}^{\rm Sch}
 \quad\text{is closed}.
\end{gathered}
\]
This proves prevalence as in Definition~\ref{def:prevalence} and openness.
\end{proof}

Since the number of pairs \((a,b)\) is finite, the quantitative classes
exhaust \(\mathcal P_{\mathrm{strat}}\):
\begin{equation}\label{eq:Pstrat-quantitative-exhaustion}
 \mathcal P_{\mathrm{strat}}
 =\bigcup_{R=1}^\infty\bigcup_{q=1}^\infty
   \mathcal P_{\mathrm{strat}}(R,q).
\end{equation}

For the next two lemmas, fix \(R>0\) and a finite derivative order \(r_*\ge3\), and recall the
fixed return-window exponent \(C_I\) from
Proposition~\ref{fr:prop:deterministic-AMO-seed}.  Let
\(B_2,B_1,D_3,n_0\) and
\((\xi_{a,n}^{\eps,\delta},E_{a,n}^{\eps,\delta})\) be supplied by
the common choice in Propositions~\ref{prop:polynomial-starting-scale-perturbed-dichotomy}
and~\ref{prop:polynomial-starting-folds}
with profile bound \(2R\) and derivative order \(r_*+2\).
Let \(m_n\) be as in
Proposition~\ref{fr:prop:AMO-central-normal-form}, and let \(H_{a,n}^0\) be as defined before
Proposition~\ref{prop:polynomial-starting-folds}.  Set
\begin{equation}\label{eq:polynomial-response-common-bidomain}
 \Omega_{a,n}^{\rm resp}:=
 \left\{
 |x-p_a|\le\frac1{8N}n^{-B_1},\quad
 |E-E_+|\le\frac18n^{-B_2}
 \right\}.
\end{equation}

The unperturbed folds are common reference points for the response estimates.

\begin{lemma}
\label{lem:unperturbed-tip-fold}
There are \(n_{\rm resp}=n_{\rm resp}(R,r_*)\ge\max\{2,n_0\}\),
\(c>0\), and \(C_6=C_6(R,r_*)>0\), with the fixed background
dependence understood, for which the following holds.
For every integer \(n\ge n_{\rm resp}\), there is a common energy \(E_n^0\) and,
for each \(0\le a<N\), a unique phase \(\xi_{a,n}^0\) with the following properties.
\begin{enumerate}
\item The point \((\xi_{a,n}^0,E_n^0)\) is the unique unperturbed fold in
\(\Omega_{a,n}^{\rm resp}\) and satisfies
\begin{equation}\label{eq:polynomial-unperturbed-fold-equations}
 H_{a,n}^0(\xi_{a,n}^0,E_n^0)=0,
 \qquad
 \partial_xH_{a,n}^0(\xi_{a,n}^0,E_n^0)=0.
\end{equation}
\item The folds satisfy
\begin{equation}\label{eq:polynomial-unperturbed-fold-location}
 \xi_{a,n}^0=p_a+\xi_{0,n}^0,
 \qquad
 |\xi_{a,n}^0-p_a|+|E_n^0-E_+|
 \le C_6e^{-c\gamma m_n}.
\end{equation}
\item At each fold,
\begin{equation}\label{eq:polynomial-unperturbed-fold-nondegeneracy}
 c\le\partial_EH_{a,n}^0(\xi_{a,n}^0,E_n^0)\le C_6,
 \qquad
 c\le\partial_x^2H_{a,n}^0(\xi_{a,n}^0,E_n^0)\le C_6.
\end{equation}
\end{enumerate}
\end{lemma}

Recall \(\mathfrak b_\infty^{\rm Sch}\) and
\(\mathscr D_{a,b}^{\rm Sch}\) from
\eqref{eq:polynomial-limiting-Schur-slope} and
\eqref{eq:polynomial-Schur-difference}.  The next lemma estimates the
perturbed folds supplied by the same initial-scale construction.

\begin{lemma}
\label{lem:differentiated-tip-response}
The constants in Lemma~\ref{lem:unperturbed-tip-fold} can be chosen
simultaneously with
\[
 D_5>0,\qquad 0<\eta_{\rm resp}<C_I/\tau,
\]
depending only on \((\mathfrak D_{\rm AMO}^{(r_*+2)},R)\), so that,
for every
\[
 \|\delta\|_\rho\le R,\qquad n\ge n_{\rm resp},\quad n\in\mathbb N,
 \qquad 0<\eps\le n^{-D_3},
\]
the following estimates hold.
\begin{enumerate}
\item The perturbed folds satisfy
\begin{equation}\label{eq:polynomial-fold-cluster}
 \max_{0\le a<N}\left(
 |\xi_{a,n}^{\eps,\delta}-p_a|
 +|E_{a,n}^{\eps,\delta}-E_+^{\eps,\delta}|
 \right)
 \le C_6\bigl(n^{D_5}\eps
              +e^{-cn^{\eta_{\rm resp}}}\bigr).
\end{equation}
\item For \(0\le a<b<N\),
\begin{equation}\label{eq:polynomial-fold-response}
 E_{a,n}^{\eps,\delta}-E_{b,n}^{\eps,\delta}
 =-\frac{\eps}{\mathfrak b_\infty^{\rm Sch}}
   \mathscr D_{a,b}^{\rm Sch}(\delta)
   +\mathcal R_{a,b,n}^{\eps,\delta},
\end{equation}
where
\begin{equation}\label{eq:polynomial-fold-response-error}
 |\mathcal R_{a,b,n}^{\eps,\delta}|
 \le C_6\left(
 n^{2D_5}\eps^2
 +n^{D_5}\eps e^{-cn^{\eta_{\rm resp}}}
 \right).
\end{equation}
\end{enumerate}
\end{lemma}
\begin{proof}[Proof of Lemmas~\ref{lem:unperturbed-tip-fold} and
\ref{lem:differentiated-tip-response}]
\medskip\noindent
\emph{Step (i).} We estimate profile derivatives on a common phase--energy domain.

Proposition~\ref{fr:prop:AMO-central-normal-form} gives
\(c n^{C_I/\tau}\le m_n\le n^{C_I}+1\).
Consequently, for every fixed
\(D>0\) and \(0<\eta<C_I/\tau\), after increasing the threshold
for \(n\) depending on \((D,\eta)\),
\begin{equation}\label{eq:polynomial-response-tail-scale}
 n^D e^{-c\gamma m_n}\le e^{-n^\eta}.
\end{equation}
In this proof, \(C=C(R,r_*)\) and \(D=D(R,r_*)\) denote auxiliary
bounds that may increase from one estimate to the next; the fixed
background dependence is understood.  They are independent of
\(n,a,\delta,t,\eps\).  The constants \(C_6\) and \(D_5\) in the
statements will be fixed after these estimates.  Enlarge \(n_{\rm resp}\)
so that all the required profile and resolvent estimates hold.
On
\[
 \{V\in X:\|V-V_0\|_\rho<2Rn^{-D_3}\}
 \times\Omega_{a,n}^{\rm resp},
\]
all
proper Dirichlet subblocks in the two fixed distinguished return matrices
are pole-free, and each matrix has a simple largest singular value.  Choose their
oriented singular directions by continuation from \(V_0\), write their
contracting right- and expanding left-singular angles as
\(s_{a,n}^+(V;x,E)\) and
\(u_{a,n}^-(V;x,E)\), and define
\[
\begin{aligned}
 G_{a,n}(V;x,E)
 &:=\frac\pi2-s_{a,n}^+(V;x,E)+u_{a,n}^-(V;x,E)
 \quad(\operatorname{mod}2\pi),\\
 H_{a,n}(V;x,E)
 &:=\cos G_{a,n}(V;x,E).
\end{aligned}
\]
Thus \(H_{a,n}(V_0;x,E)=H_{a,n}^0(x,E)\).
Here \(C^k_{x,E}:=C^k(\Omega_{a,n}^{\rm resp})\).  For a Banach space
\(Y\), let \(\mathcal L^k(X,Y)\) denote the bounded \(k\)-linear maps
\(X^k\to Y\), and put \(\mathcal L(X,Y):=\mathcal L^1(X,Y)\).
The map \(V\mapsto H_{a,n}(V)\) is twice Fr\'echet differentiable on
this neighborhood.  Uniformly for \(\|\delta\|_\rho\le R\) and
\(0<\eps\le n^{-D_3}\),
\begin{align}
 \sup_{0\le t\le1}
 \|D_VH_{a,n}(V_0+t\eps\delta)\|_{\mathcal L(X,C^2_{x,E})}
 &\le Cn^{D},
 \label{eq:polynomial-first-profile-map-bound}\\
 \sup_{0\le t\le1}
 \|D_V^2H_{a,n}(V_0+t\eps\delta)\|_{\mathcal L^2(X,C^1_{x,E})}
 &\le Cn^{2D}.
 \label{eq:polynomial-second-profile-map-bound}
\end{align}
For every proper Dirichlet subblock \(\Gamma\) occurring in a
return matrix, let \(\Id\) denote the identity on \(\ell^2(\Gamma)\) and
put
\[
 R_\Gamma(V):=(E\Id-H_{V,\Gamma}(x))^{-1},\qquad
 \mathcal D_{\Gamma,x}(h):=
 \operatorname{diag}_{j\in\Gamma}h(x+j\alpha).
\]
The required profile derivatives are
\begin{equation}\label{eq:polynomial-profile-resolvent-first}
 D_VR_\Gamma[h]
 =R_\Gamma\mathcal D_{\Gamma,x}(h)R_\Gamma,
\end{equation}
and
\begin{equation}\label{eq:polynomial-profile-resolvent-second}
\begin{aligned}
 D_V^2R_\Gamma[h_1,h_2]
 &={}R_\Gamma\mathcal D_{\Gamma,x}(h_1)R_\Gamma
       \mathcal D_{\Gamma,x}(h_2)R_\Gamma\\
 &\quad+R_\Gamma\mathcal D_{\Gamma,x}(h_2)R_\Gamma
       \mathcal D_{\Gamma,x}(h_1)R_\Gamma.
\end{aligned}
\end{equation}
Differentiating
\((E\Id-H_{V,\Gamma}(x))R_\Gamma=\Id\) gives
\eqref{eq:polynomial-profile-resolvent-first}--
\eqref{eq:polynomial-profile-resolvent-second}.  Together with
\eqref{eq:perturbed-resolvent-gap},
\eqref{eq:perturbed-resolvent-polynomial-derivatives}, and
\eqref{eq:perturbed-return-length-polynomial}, these identities give
polynomial bounds for both profile derivatives of every boundary Green
function.  For either complete return matrix, put
\(\widetilde B=P^{-1}B\), where \(P>0\) is its principal determinant.
The normalized Green identity and Jacobi's formula give
\[
 \|D_V\widetilde B\|_{\mathcal L(X,C^2_{x,E})}
 +\|D_V^2\widetilde B\|_{\mathcal L^2(X,C^1_{x,E})}
 \le Cn^{D}.
\]
For the determinant entry, Jacobi's formula gives
\[
\begin{aligned}
 D_V(P^{-2})[h]
 &=-2P^{-2}D_V\log P[h],\\
 D_V^2(P^{-2})[h_1,h_2]
 &=P^{-2}\bigl(
 4D_V\log P[h_1]D_V\log P[h_2]
 -2D_V^2\log P[h_1,h_2]\bigr).
\end{aligned}
\]
Since \(P>0\), the singular lines of \(B\) and \(\widetilde B\) agree.
The Gram gap \(3/4\), the explicit projector formula, and
\eqref{svd:eq:signed-axis-from-projector} transfer these bounds to the
axes; the chain rule for
\(H=\cos(\frac\pi2-s^++u^-)\) proves
\eqref{eq:polynomial-first-profile-map-bound}--
\eqref{eq:polynomial-second-profile-map-bound}.

Finally, \eqref{eq:perturbed-fold-displacement-polynomial}, uniformly for
\(0\le t\le1\) after increasing \(n_{\rm resp}\), gives
\[
\begin{aligned}
 |\xi_{a,n}^{t\eps,\delta}-p_a|
 &\le\frac1{16N}n^{-B_1},\\
 |E_{a,n}^{t\eps,\delta}-E_+|
 &\le\frac1{16}n^{-B_2},\\
 \|V_0+t\eps\delta-V_0\|_\rho
 &\le Rn^{-D_3}<2Rn^{-D_3}.
\end{aligned}
\]
Hence the \(1/8\)-widths in \eqref{eq:polynomial-response-common-bidomain}
fit inside the \(1/4\)-widths of every homotopy domain, where all preceding
resolvents and singular axes are defined.

We next locate the unperturbed fold.
For \(\delta=0\), \(V_{\eps,0}=V_0\).  Proposition
\ref{prop:polynomial-starting-folds}
therefore gives the same unique nondegenerate fold for every
\(0<\eps\le n^{-D_3}\), and
\[
 |\xi_{a,n}^0-p_a|+|E_{a,n}^0-E_+|
 \le C(\eps n^{D_2}+e^{-c\gamma m_n})
 \xrightarrow[\eps\downarrow0]{}
 Ce^{-c\gamma m_n}.
\]
Translation symmetry and uniqueness give
\[
 H_{a,n}^0(p_a+x,E)=H_{0,n}^0(x,E)
 \quad\Longrightarrow\quad
 (\xi_{a,n}^0,E_{a,n}^0)=(p_a+\xi_{0,n}^0,E_n^0).
\]
This proves \eqref{eq:polynomial-unperturbed-fold-equations}--
\eqref{eq:polynomial-unperturbed-fold-nondegeneracy}.

\medskip\noindent
\emph{Step (ii).} We compute the finite-volume Schur derivatives and pass to the half-line limit.

Reindex the distinguished two-sided return word as
\([-m_n,m_n-1]\cap\mathbb Z\), with the central cut at \(0\).  Its two
strict sides are
\[
 \Gamma_{a,n}^-=[-m_n,-1]\cap\mathbb Z,
 \qquad
 \Gamma_{a,n}^+=[1,m_n-1]\cap\mathbb Z.
\]
The indices are lattice displacements from the cut phase \(x\).  Put
\(\partial_-:=-1\) and \(\partial_+:=1\).  For
\(\sigma\in\{+,-\}\), let
\[
 e_{\partial_-}:=e_{-1}\in\ell^2(\Gamma_{a,n}^-),
 \qquad
 e_{\partial_+}:=e_{1}\in\ell^2(\Gamma_{a,n}^+)
\]
be the standard basis vectors adjacent to the cut.  On the pole-free
neighborhood over \(\Omega_{a,n}^{\rm resp}\) from Step~(i), set
\[
\begin{aligned}
 R_{a,n}^\sigma(V;x,E)
 &:=(E\Id-H_{V,\Gamma_{a,n}^\sigma}(x))^{-1},\\
 g_{a,n}^\sigma(V;x,E)
 &:=R_{a,n}^\sigma(V;x,E)e_{\partial_\sigma},\\
 h_{a,n}^\sigma(V;x,E)
 &:=\langle e_{\partial_\sigma},
 R_{a,n}^\sigma(V;x,E)e_{\partial_\sigma}\rangle
 =g_{a,n}^\sigma(V;x,E)(\partial_\sigma).
\end{aligned}
\]
The central Schur complement is
\begin{equation}\label{eq:polynomial-central-Schur-definition}
 \mathcal S_{a,n}(V;x,E)
 :=E-V(x)-h_{a,n}^-(V;x,E)-h_{a,n}^+(V;x,E).
\end{equation}
At the unperturbed fold we suppress the arguments and write
\[
\begin{aligned}
 g_{a,n}^\sigma
 &:=g_{a,n}^\sigma(V_0;\xi_{a,n}^0,E_n^0),\\
 h_{a,n}^\sigma
 &:=h_{a,n}^\sigma(V_0;\xi_{a,n}^0,E_n^0),
 \qquad \sigma\in\{+,-\},\\
 \mathcal S_{a,n}^0
 &:=\mathcal S_{a,n}(V_0;\xi_{a,n}^0,E_n^0).
\end{aligned}
\]
For \(L\ge1\), set
\(\Gamma_L^+:=[1,L]\) and \(\Gamma_L^-:=[-L,-1]\), and define
the reflected zero-extension maps, where \(\mathbb C^\Gamma\) denotes
the vector space of all functions \(\Gamma\to\mathbb C\) on a finite set
\(\Gamma\):
\[
 \iota_L^\sigma:\mathbb C^{\Gamma_L^\sigma}
 \longrightarrow\ell^2(\mathbb N),\qquad
 (\iota_L^\sigma w)_j
 :=\begin{cases}
       w(\sigma j),&1\le j\le L,\\
       0,&j>L,
    \end{cases}
 \qquad \sigma\in\{+,-\}.
\]
Whenever an \(\ell^2(\mathbb N)\)-norm is used below, a finite Green
column on \(\Gamma_L^\sigma\) is identified with its image under
\(\iota_L^\sigma\).

For the central Schur complement
\eqref{eq:polynomial-central-Schur-definition}, at
\((V,x,E)=(V_0,\xi_{a,n}^0,E_n^0)\), define
\[
 \mathscr L_{a,n}^{\rm Sch}
 :=\left.D_V\mathcal S_{a,n}(V;\xi_{a,n}^0,E_n^0)
   \right|_{V=V_0},
 \qquad
 \mathfrak b_{a,n}^{\rm Sch}
 :=\left.\partial_E\mathcal S_{a,n}(V_0;\xi_{a,n}^0,E)
   \right|_{E=E_n^0}.
\]
Then
\begin{align}
 \mathscr L_{a,n}^{\rm Sch}(\delta)
 &=-\delta(\xi_{a,n}^0)
   -\sum_{j=1}^{m_n}|g_{a,n}^-(-j)|^2
      \delta(\xi_{a,n}^0-j\alpha)\notag\\
 &\hspace{24mm}
   -\sum_{j=1}^{m_n-1}|g_{a,n}^+(j)|^2
      \delta(\xi_{a,n}^0+j\alpha),
 \label{eq:polynomial-finite-Schur-profile-derivative}\\
 \mathfrak b_{a,n}^{\rm Sch}
 &=1+\|g_{a,n}^-\|_2^2+\|g_{a,n}^+\|_2^2.
 \label{eq:polynomial-finite-Schur-energy-derivative}
\end{align}
For a strict side \(\Gamma\), write
\[
 h_\Gamma=\langle e_\partial,R_\Gamma e_\partial\rangle,
 \qquad g_\Gamma=R_\Gamma e_\partial.
\]
The first resolvent identity in
\eqref{eq:polynomial-profile-resolvent-first} and
\(\partial_ER_\Gamma=-R_\Gamma^2\) give
\[
 D_Vh_\Gamma[\delta]
 =\langle g_\Gamma,\mathcal D_{\Gamma,x}(\delta)g_\Gamma\rangle,
 \qquad
 \partial_Eh_\Gamma=-\|g_\Gamma\|_2^2.
\]
Substitution in
\(\mathcal S=E-V-h^--h^+\) proves
\eqref{eq:polynomial-finite-Schur-profile-derivative} and
\eqref{eq:polynomial-finite-Schur-energy-derivative}.

\smallskip\noindent
For the fixed backward word on \(\Gamma_{a,n}^-\) and the forward word
on \([0,m_n-1]\), put
\[
\begin{aligned}
 \mathfrak D_-(V;x,E)
 &:=\det(E-H_{V,\Gamma_{a,n}^-}(x)),\\
 \mathfrak D_+(V;x,E)
 &:=\det(E-H_{V,[0,m_n-1]}(x)).
\end{aligned}
\]
Both are positive, and
after increasing \(n_{\rm resp}\) their determinant lower bounds give
\(\mathfrak D_\sigma\ge2\), hence
\(\mathfrak D_\sigma^{-2}\le1/2\).  Define
\begin{equation}\label{eq:polynomial-finite-Weyl-vectors}
\begin{aligned}
 v_-^D(V;x,E)
 &:=\frac{(1,h_{a,n}^-(V;x,E))^{\mathsf T}}
 {\sqrt{1+h_{a,n}^-(V;x,E)^2}},\\
 v_+^D(V;x,E)
 &:=\frac{(E-V(x)-h_{a,n}^+(V;x,E),-1)^{\mathsf T}}
 {\sqrt{1+(E-V(x)-h_{a,n}^+(V;x,E))^2}}.
\end{aligned}
\end{equation}
Indeed, \eqref{svd:eq:normalized-transfer-green} gives these as the
rank-one output and input boundary directions after setting the determinant
correction to zero; on the forward side one also uses
\[
 \mathfrak D_+
 =\det(E-H_{V,\Gamma_{a,n}^+}(x))
   \bigl(E-V(x)-h_{a,n}^+\bigr)>0.
\]
Interpolate the determinant correction from \(0\) to
\(\mathfrak D_\sigma^{-2}\).
\par\noindent
The gap
\eqref{svd:eq:artificial-uniform-singular-gap} uniquely continues the
displayed orientations to the actual singular directions
\(v_-^{\rm svd}\) and \(v_+^{\rm svd}\).  Hence
\begin{equation}\label{eq:polynomial-SVD-joining-as-axis-product}
 H_{a,n}(V;x,E)
 =(v_+^{\rm svd}(V;x,E))^{\mathsf T}
   v_-^{\rm svd}(V;x,E).
\end{equation}
A direct scalar product in
\eqref{eq:polynomial-finite-Weyl-vectors}, together with
\eqref{svd:eq:XiD-explicit}, gives
\[
\begin{aligned}
 H_{a,n}^D(V;x,E)
 &:=(v_+^D(V;x,E))^{\mathsf T}v_-^D(V;x,E)\\
 &=\Xi^D\bigl(h_{a,n}^-(V;x,E),h_{a,n}^+(V;x,E),
              E-V(x)\bigr)
 \mathcal S_{a,n}(V;x,E).
\end{aligned}
\]

\smallskip\noindent
We next compare the finite Green columns with the upper-edge eigenfunction.
Uniformly in \(n\) and \(a\),
\begin{equation}\label{eq:polynomial-Green-column-edge-convergence}
 \|g_{a,n}^+-u^+\|_{\ell^2}
 +\|g_{a,n}^--u^-\|_{\ell^2}
 \le Cn^{D}e^{-c\gamma m_n},
\end{equation}
where \(u^+=(u_1,u_2,\ldots)\) and
\(u^-=(u_{-1},u_{-2},\ldots)\), and finite columns are extended by zero.
Consequently,
\begin{equation}\label{eq:polynomial-finite-infinite-Schur-comparison}
\begin{aligned}
 \|\mathscr L_{a,n}^{\rm Sch}-\mathscr L_a^{\rm Sch}\|_{X^*}
 &\le Cn^{D}e^{-c\gamma m_n},\\
 |\mathfrak b_{a,n}^{\rm Sch}
 -\mathfrak b_\infty^{\rm Sch}|
 &\le Cn^{D}e^{-c\gamma m_n}.
\end{aligned}
\end{equation}
Put \(L_+=m_n-1\), \(L_-=m_n\), and, for \(\sigma\in\{+,-\}\), set
\[
\begin{aligned}
 H_{L_\sigma}^\sigma
 &:=H_{V_0,\Gamma_{a,n}^\sigma}(p_a),\\
 R_{\rm edge}^\sigma
 &:=(E_+\Id-H_{L_\sigma}^\sigma)^{-1},\\
 g_{L_\sigma}^{\sigma,{\rm edge}}
 &:=R_{\rm edge}^\sigma e_{\partial_\sigma},\\
 u_{L_\sigma}^\sigma
 &:=(u_{\sigma j})_{1\le j\le L_\sigma}.
\end{aligned}
\]
Since \(Np_a=a\in\mathbb Z\) and
\(\widehat\alpha=N\alpha\pmod1\),
\[
 V_0(p_a+j\alpha)
 =2\lambda\cos(2\pi j\widehat\alpha),\qquad j\in\mathbb Z.
\]
Thus the state in
\eqref{eq:AMO-seed-edge-state} satisfies the upper-edge equation for the
two half-line restrictions of the original model cut at \(p_a\).
The edge equation, \(u_0=1\), and \(u_{-j}=u_j\) give
\[
\begin{aligned}
 (E_+\Id-H_{L_\sigma}^\sigma)u_{L_\sigma}^\sigma
 &=e_{\partial_\sigma}
   +u_{\sigma(L_\sigma+1)}e_{\sigma L_\sigma},\\
 g_{L_\sigma}^{\sigma,{\rm edge}}-u_{L_\sigma}^\sigma
 &=-u_{\sigma(L_\sigma+1)}
   R_{\rm edge}^\sigma e_{\sigma L_\sigma}.
\end{aligned}
\]
Consequently, the side-resolvent bound and
\eqref{eq:AMO-seed-edge-state-decay} imply
\[
\begin{aligned}
 \|\iota_{L_\sigma}^\sigma
       g_{L_\sigma}^{\sigma,{\rm edge}}-u^\sigma\|_2
 &\le
 |u_{\sigma(L_\sigma+1)}|
 \|R_{\rm edge}^\sigma\|
 +\|(u_{\sigma j})_{j>L_\sigma}\|_2\\
 &\le Cn^{D}
 \exp\{-\gamma L_\sigma+CR_\tau(L_\sigma+2)\}\\
 &\le Cn^{D} e^{-c\gamma m_n}.
\end{aligned}
\]
Here the last inequality follows from \(L_\sigma\ge m_n-1\) and
\(R_\tau(m)=o_{m\to\infty}(m)\).
At the fold, let \(M_{\rm fold}^\sigma\) and
\(M_{\rm edge}^\sigma\) be the side matrices
\[
 M_{\rm fold}^\sigma
 :=E_n^0\Id-H_{V_0,\Gamma_{a,n}^\sigma}(\xi_{a,n}^0),\qquad
 M_{\rm edge}^\sigma
 :=E_+\Id-H_{V_0,\Gamma_{a,n}^\sigma}(p_a).
\]
Put \(R_{\rm fold}^\sigma:=(M_{\rm fold}^\sigma)^{-1}\).  By
\eqref{eq:polynomial-unperturbed-fold-location},
\[
\begin{aligned}
 \|M_{\rm fold}^\sigma-M_{\rm edge}^\sigma\|
 &\le |E_n^0-E_+|
   +\|V_0'\|_\infty|\xi_{a,n}^0-p_a|\\
 &\le Ce^{-c\gamma m_n}.
\end{aligned}
\]
The resolvent identity then gives
\[
 R_{\rm fold}^\sigma-R_{\rm edge}^\sigma
 =-R_{\rm fold}^\sigma
   (M_{\rm fold}^\sigma-M_{\rm edge}^\sigma)R_{\rm edge}^\sigma .
\]
The polynomial bounds for the two resolvents therefore give
\[
\begin{aligned}
 \|g_{a,n}^\sigma-g_{L_\sigma}^{\sigma,{\rm edge}}\|_2
 &\le
 \|R_{\rm fold}^\sigma\|
 \|M_{\rm fold}^\sigma-M_{\rm edge}^\sigma\|
 \|R_{\rm edge}^\sigma\|\\
 &\le Cn^{D} e^{-c\gamma m_n}.
\end{aligned}
\]
Combining this estimate with the preceding zero-extension estimate and
summing over \(\sigma\in\{+,-\}\) proves
\eqref{eq:polynomial-Green-column-edge-convergence}.  Moreover,
Cauchy--Schwarz and
\eqref{eq:polynomial-evaluation-functional-Lipschitz} give
\[
\begin{aligned}
 \sum_{\sigma=\pm}\sum_{j\ge1}
 \bigl||g_{a,n}^\sigma(\sigma j)|^2-u_j^2\bigr|
 &\le Cn^{D}e^{-c\gamma m_n},\\
 \|\mathscr L_{a,n}^{\rm Sch}-\mathscr L_a^{\rm Sch}\|_{X^*}
 &\le C|\xi_{a,n}^0-p_a|
 \bigl(1+\|g_{a,n}^-\|_2^2\\
 &\hspace{37mm}+\|g_{a,n}^+\|_2^2\bigr)
 +Cn^{D}e^{-c\gamma m_n}.
\end{aligned}
\]
Together with \eqref{eq:polynomial-finite-Schur-profile-derivative},
\eqref{eq:polynomial-finite-Schur-energy-derivative}, and
\eqref{eq:polynomial-unperturbed-fold-location}, these are precisely
\eqref{eq:polynomial-finite-infinite-Schur-comparison}.

\medskip\noindent
\emph{Step (iii).} We differentiate the comparison between the Schur complement
and the singular directions.
Use the output line for \(\sigma=-\) and the input line for
\(\sigma=+\), oriented as in \eqref{eq:polynomial-finite-Weyl-vectors}.
The unperturbed fold satisfies
\begin{equation}\label{eq:polynomial-profile-differentiated-axis-tail}
 \|v_\sigma^{\rm svd}-v_\sigma^D\|
 +\|\partial_E(v_\sigma^{\rm svd}-v_\sigma^D)\|
 +\|D_V(v_\sigma^{\rm svd}-v_\sigma^D)
   \|_{\mathcal L(X,\mathbb R^2)}
 \le Cn^{D}e^{-c\gamma m_n}.
\end{equation}
For \(\sigma\in\{-,+\}\), put
\(\varepsilon_\sigma:=\mathfrak D_\sigma^{-2}\), where
\(\mathfrak D_+\) is the complete determinant on \([0,m_n-1]\).
Let \(P_\sigma(V,t)\) be the spectral projection of this left or right
Gram matrix associated with its larger eigenvalue, with the determinant
correction replaced by \(t\).  Thus
\[
 P_\sigma(V,\varepsilon_\sigma)
 =P_\sigma^{\rm svd}(V),\qquad
 P_\sigma(V,0)=P_\sigma^D(V).
\]
The uniform singular-value gap along the interpolation,
\eqref{eq:polynomial-profile-resolvent-first}--
\eqref{eq:polynomial-profile-resolvent-second}, and Jacobi's formula
\cite{HornJohnson2013} give,
for \(i+j\le2\),
\[
\begin{gathered}
 \|D_V^i\partial_t^jP_\sigma\|
 +\|\partial_ED_V^i\partial_t^jP_\sigma\|\le Cn^{D},\\
 \|D_V^i\varepsilon_\sigma\|
 +\|\partial_ED_V^i\varepsilon_\sigma\|
 \le Cn^{D}\varepsilon_\sigma,\qquad
 D_V\varepsilon_\sigma[h]
 =-2\varepsilon_\sigma D_V\log\mathfrak D_\sigma[h].
\end{gathered}
\]
Hence the interpolation identity
\[
 P_\sigma(V,\varepsilon_\sigma)-P_\sigma(V,0)
 =\varepsilon_\sigma\int_0^1
 \partial_tP_\sigma(V,s\varepsilon_\sigma)\,ds
\]
retains the factor \(\varepsilon_\sigma\) after one profile or energy
derivative.  Let \(\mathfrak D_{\sigma,*}\) be the infimum of
\(\mathfrak D_\sigma(V;x,E)\) over the common
domain in the profile, phase, and energy variables.  Since
\eqref{eq:perturbed-principal-determinant-lower} gives
\[
 \mathfrak D_{\sigma,*}
 \ge e^{\gamma m_n-CR_\tau(m_n+2)-1},\qquad
 n^{D}\mathfrak D_{\sigma,*}^{-2}
 \le Cn^{D}e^{-c\gamma m_n},
\]
the oriented-vector selection formula
\eqref{svd:eq:signed-axis-from-projector} proves
\eqref{eq:polynomial-profile-differentiated-axis-tail}.

\smallskip\noindent
For the normalized finite-scale first variation, put
\[
 \Lambda_{a,n}
 :=\left.D_VH_{a,n}(V;\xi_{a,n}^0,E_n^0)\right|_{V=V_0}
 \in X^*,
 \qquad
 \beta_n
 :=\partial_EH_{a,n}(V_0;\xi_{a,n}^0,E_n^0).
\]
By translation covariance and
\eqref{eq:polynomial-unperturbed-fold-nondegeneracy}, \(\beta_n\) is
independent of \(a\) and bounded away from zero.  The normalized first
variations satisfy
\begin{equation}\label{eq:polynomial-normalized-response-cancellation}
 \left\|
 \frac{\Lambda_{a,n}}{\beta_n}
 -\frac{\mathscr L_a^{\rm Sch}}
        {\mathfrak b_\infty^{\rm Sch}}
 \right\|_{X^*}
 \le Cn^{D}e^{-c\gamma m_n}.
\end{equation}
Near the unperturbed fold, put
\[
\begin{aligned}
 \Xi_{a,n}(V;x,E)
 &:=\Xi^D\bigl(h_{a,n}^-(V;x,E),
 h_{a,n}^+(V;x,E),E-V(x)\bigr),\\
 \mathcal E_{a,n}^{\rm svd}(V;x,E)
 &:=H_{a,n}(V;x,E)-H_{a,n}^D(V;x,E),\\
 H_{a,n}^D(V;x,E)
 &=\Xi_{a,n}(V;x,E)\mathcal S_{a,n}(V;x,E),\\
 0<c&\le\Xi_{a,n}\le C,\\
 \|D_V\Xi_{a,n}\|_{X^*}+|\partial_E\Xi_{a,n}|
 &\le Cn^{D} .
\end{aligned}
\]
In the remaining formulas of this step, suppressed arguments are evaluated
at \((V_0;\xi_{a,n}^0,E_n^0)\).
The exact singular-direction overlap identity
\eqref{eq:polynomial-SVD-joining-as-axis-product}
and \eqref{eq:polynomial-profile-differentiated-axis-tail} imply
\[
\begin{aligned}
 \mathcal E_{a,n}^{\rm svd}
 &=(v_+^{\rm svd}-v_+^D)^{\mathsf T}v_-^{\rm svd}
 +(v_+^D)^{\mathsf T}(v_-^{\rm svd}-v_-^D),\\
 \|D_V\mathcal E_{a,n}^{\rm svd}\|_{X^*}
 +|\partial_E\mathcal E_{a,n}^{\rm svd}|
 +|\mathcal E_{a,n}^{\rm svd}|
 &\le Cn^{D}e^{-c\gamma m_n}.
\end{aligned}
\]
Since \(H_{a,n}^0=0\) at the fold,
\(|\mathcal S_{a,n}^0|\le Cn^{D}e^{-c\gamma m_n}\); differentiation of
\(H=\Xi\mathcal S+\mathcal E^{\rm svd}\) gives
\[
 \|\Lambda_{a,n}-\Xi_{a,n}\mathscr L_{a,n}^{\rm Sch}\|_{X^*}
 +|\beta_n-\Xi_{a,n}\mathfrak b_{a,n}^{\rm Sch}|
 \le Cn^{D}e^{-c\gamma m_n}.
\]
The central recurrence \(E_+-2\lambda=2u_1\), obtained from
\eqref{eq:AMO-seed-edge-state} at the origin, the Green-column limit, and
\eqref{eq:polynomial-unperturbed-fold-location} yield
\[
 \Xi_\infty:=(1+u_1^2)^{-1},\qquad
 |\Xi_{a,n}-\Xi_\infty|\le Cn^{D}e^{-c\gamma m_n}.
\]
Using \eqref{eq:polynomial-finite-infinite-Schur-comparison},
\[
 \|\Lambda_{a,n}-\Xi_\infty\mathscr L_a^{\rm Sch}\|_{X^*}
 +|\beta_n-\Xi_\infty\mathfrak b_\infty^{\rm Sch}|
 \le Cn^{D}e^{-c\gamma m_n}.
\]
Since \(\beta_n\ge c\), cancellation is quantitative:
\[
 \frac{\Lambda_{a,n}}{\beta_n}
 -\frac{\mathscr L_a^{\rm Sch}}{\mathfrak b_\infty^{\rm Sch}}
 =
 \frac{\mathfrak b_\infty^{\rm Sch}
       (\Lambda_{a,n}-\Xi_\infty\mathscr L_a^{\rm Sch})
       -(\beta_n-\Xi_\infty\mathfrak b_\infty^{\rm Sch})
        \mathscr L_a^{\rm Sch}}
      {\beta_n\mathfrak b_\infty^{\rm Sch}}.
\]
This proves \eqref{eq:polynomial-normalized-response-cancellation}.

\medskip\noindent
\emph{Step (iv).} We differentiate the fold equation.
For \(0\le s\le\eps\), write
\(Z_a(s)=(\xi_{a,n}(s),E_{a,n}(s))\) and
\[
 \mathscr F_a(Z,s)
 :=
 \binom{H_{a,n}(V_0+s\delta;Z)}
       {\partial_xH_{a,n}(V_0+s\delta;Z)}.
\]
Thus \(\mathscr F_a(Z_a(s),s)=0\).
At \((V_0+s\delta;Z_a(s))\), the Step~(i) bounds and
\eqref{eq:perturbed-central-transversality} give
\[
 \mathscr F_{a,Z}
 =\begin{pmatrix}0&H_E\\H_{xx}&H_{xE}\end{pmatrix},
 \qquad
 \mathscr F_{a,Z}^{-1}
 =\begin{pmatrix}
 -H_{xE}/(H_EH_{xx})&1/H_{xx}\\
 1/H_E&0
 \end{pmatrix},
 \qquad H_E,H_{xx}\ge c.
\]
Consequently,
\[
\begin{gathered}
 Z_a'=-\mathscr F_{a,Z}^{-1}\mathscr F_{a,s},\qquad
 E_{a,n}'=-\frac{D_VH_{a,n}[\delta]}{\partial_EH_{a,n}},\\
 E_{a,n}'(0)=-\frac{\Lambda_{a,n}(\delta)}{\beta_n}.
\end{gathered}
\]
Differentiating the implicit equation again gives
\[
 Z_a''=-\mathscr F_{a,Z}^{-1}
 \bigl(\mathscr F_{a,ss}
       +2\mathscr F_{a,Zs}Z_a'
       +\mathscr F_{a,ZZ}[Z_a',Z_a']\bigr).
\]
For some auxiliary \(C,D>0\), uniformly along these folds,
\[
 \|\mathscr F_{a,Z}^{-1}\|+\|\mathscr F_{a,s}\|
 +\|\mathscr F_{a,ss}\|+\|\mathscr F_{a,Zs}\|
 +\|\mathscr F_{a,ZZ}\|\le Cn^D.
\]
The two derivative formulas therefore imply
\[
 \sup_{0\le s\le\eps}|Z_a'(s)|\le Cn^{2D},\qquad
 \sup_{0\le s\le\eps}|Z_a''(s)|\le Cn^{6D}.
\]
Now choose \(D_5\) larger than all the finitely many polynomial
exponents obtained above, including \(6D\).  Choose \(C_6\) to
dominate the resulting coefficients in the estimates below.  These
choices depend only on \((R,r_*)\) and the fixed background data.
Take \(\eta_{\rm resp}:=C_I/(2\tau)\) and increase \(n_{\rm resp}\)
to cover all thresholds above and in
\eqref{eq:polynomial-response-tail-scale} for these choices.
Taylor's theorem and
\eqref{eq:polynomial-normalized-response-cancellation} yield
\[
 E_{a,n}^{\eps,\delta}
 =E_n^0
 -\eps\frac{\mathscr L_a^{\rm Sch}(\delta)}
                 {\mathfrak b_\infty^{\rm Sch}}
 +O_{R,r_*}\!\left(
 n^{2D_5}\eps^2
 +n^{D_5}\eps e^{-c\gamma m_n}
 \right).
\]
Subtracting this formula for \(a,b\) and using
\eqref{eq:polynomial-response-tail-scale} proves
\eqref{eq:polynomial-fold-response} and
\eqref{eq:polynomial-fold-response-error}.

Finally, \eqref{eq:polynomial-unperturbed-fold-location} and
\(|E_+^{\eps,\delta}-E_+|\le\eps\|\delta\|_\infty\) give
\[
\begin{aligned}
 |\xi_{a,n}^{\eps,\delta}-p_a|
 +|E_{a,n}^{\eps,\delta}-E_+^{\eps,\delta}|
 &\le C\int_0^\eps |Z_a'(s)|\,ds
       +Ce^{-c\gamma m_n}+R\eps\\
 &\le C_6\bigl(n^{D_5}\eps
                       +e^{-cn^{\eta_{\rm resp}}}\bigr),
\end{aligned}
\]
where \eqref{eq:polynomial-response-tail-scale} applies because
\(0<\eta_{\rm resp}<C_I/\tau\).  This is
\eqref{eq:polynomial-fold-cluster}.
\end{proof}

\subsection{Pairwise transversality and uniform initial-scale data for
multiple nondegenerate maxima}
\leavevmode\par\noindent
For the next two lemmas, fix integers \(R,q\ge1\) and a finite derivative order \(r_*\ge3\).
Let \(B_2,B_1>0\) and
\(D_3=D_3(C_I,2R,r_*+4)>0\) be supplied by a single application
of Proposition~\ref{prop:polynomial-starting-scale-perturbed-dichotomy}
with profile bound \(2R\) at order \(r_*+4\).
All lower-order data are restrictions of this same application.
Use these data in Lemmas~\ref{lem:unperturbed-tip-fold} and
\ref{lem:differentiated-tip-response} at order \(r_*+2\), and their order-\((r_*+1)\) restriction in
Corollary~\ref{cor:polynomial-multiwell-K-reinduction} at order \(r_*\).
The exponents will be chosen to satisfy
\begin{equation}\label{fr:eq:Young-transversality-exponent-budget}
\begin{aligned}
 \varkappa_{\rm per}&>D_3,\\
 B_4&>\max\{B_2,2B_1+4\},\\
 B_3&>\max\{\tau\varkappa_{\rm per}+2,B_4+2\},\\
 B_5&>\max\{B_3,B_2,B_4\}+4,\\
 B_1+1&<B_6<\tfrac12B_4-1,\\
 B_6&>\tau(\tau+2),\\
 B_7&>\max\{B_6,\tfrac12B_5\}+4,\\
 B_8&>B_7(\tau+1)+4,\\
 D_6&>
 \max\{\tau\varkappa_{\rm per}-\tfrac32B_4,
          \tfrac12B_5,0\}+1,\\
 B_9&>
 \max\{B_3,B_8+\tfrac12B_5,D_6+2\}+4.
\end{aligned}
\end{equation}

For \(0<\varepsilon<1\), choose the continued-fraction scale by
\[
 n_\eps:=\left\lfloor\varepsilon^{-1/\varkappa_{\rm per}}\right\rfloor,
 \qquad Q=q_\ell:=\max\{q_j:q_j\le n_\eps\}.
\]
The first lemma gives a uniform choice of exponents and separates the folds.

\begin{lemma}
\label{fr:lem:multiwell-separated-folds}
There are exponents satisfying
\eqref{fr:eq:Young-transversality-exponent-budget} and thresholds
\[
 0<\varepsilon_0=\varepsilon_0(R,q,r_*)<1,
 \qquad Q_0=Q_0(R,q,r_*)\ge1,
\]
depending only on \((\mathfrak D_{\rm AMO}^{(r_*+4)},R,q)\), such that, for every
\(\delta\in\mathcal P_{\mathrm{strat}}(R,q)\) and
\(0<\varepsilon<\varepsilon_0\), the following conclusions hold.
\begin{enumerate}
\item The scale selected above satisfies
\(Q\ge Q_0\) and
\[
 Q\le n_\eps<q_{\ell+1}\le\kappa^{-1}Q^\tau.
\]
Moreover,
\begin{equation}\label{fr:eq:Young-epsilon-Q-comparison}
 (\kappa/2)^{\varkappa_{\rm per}}Q^{-\tau\varkappa_{\rm per}}
 \le\varepsilon\le Q^{-\varkappa_{\rm per}}
 \le Q^{-D_3}.
\end{equation}
\item The fold data supplied by
Proposition~\ref{prop:polynomial-starting-folds}
at \(n=Q\) satisfy
\[
\begin{aligned}
 \min_{0\le a<b<N}|E_{a,Q}^{\varepsilon,\delta}
                   -E_{b,Q}^{\varepsilon,\delta}|
 &\ge Q^{-B_3},\\
 \max_{0\le a<N}\left(
 |\xi_{a,Q}^{\varepsilon,\delta}-p_a|
 +|E_{a,Q}^{\varepsilon,\delta}-E_+^{\varepsilon,\delta}|
 \right)
 &\le Q^{-B_4}.
\end{aligned}
\]
\end{enumerate}
\end{lemma}

The separated folds satisfy the common hypotheses of
Corollaries~\ref{cor:polynomial-multiwell-K-reinduction}
and~\ref{cor:polynomial-multiwell-induced-zeros}.  For each
\(1\le k\le N\), write the resulting energy interval and induced zero branches as
\[
 \mathcal I_{Q,k},\qquad \widehat c_{Q,k,s,\varsigma},
 \qquad 1\le s\le k,\quad\varsigma\in\{+,-\}.
\]
The following estimate is uniform for every distinct pair of these branches.

\begin{lemma}
\label{fr:lem:multiwell-pair-velocity-transversality}
The exponents and thresholds in
Lemma~\ref{fr:lem:multiwell-separated-folds} can be chosen so that,
for the same admissible \((\delta,\varepsilon,Q)\) and every
\(1\le k\le N\), the following conclusions hold.
\begin{enumerate}
\item All conclusions of
Corollary~\ref{cor:polynomial-multiwell-K-reinduction} and its four
subsequent results hold for the data above.
\item For distinct labels
\(\iota=(s,\varsigma),\iota'=(t,\varsigma')
\in\{1,\ldots,k\}\times\{+,-\}\), there is \(\sigma_{\iota,\iota'}\in\{+1,-1\}\), independent of \(E\),
such that
\begin{equation}\label{fr:eq:Young-long-pair-two-sided-velocity}
 2Q^{-D_6}
 \le
 \sigma_{\iota,\iota'}
 \partial_E\bigl(
  \widehat c_{Q,k,s,\varsigma}
  -\widehat c_{Q,k,t,\varsigma'}\bigr)(E)
 \le Q^{D_6},
 \qquad E\in\mathcal I_{Q,k}.
\end{equation}
\end{enumerate}
\end{lemma}

\begin{proof}[Proof of Lemmas~\ref{fr:lem:multiwell-separated-folds} and
\ref{fr:lem:multiwell-pair-velocity-transversality}]
\emph{Step (i).} We choose the exponents and estimate the first-variation remainder.
Throughout this proof, write
\[
 c_4=c_4(R,q,r_*)>0,\qquad C_7=C_7(R,q,r_*)\ge1
\]
for lower- and upper-bound coefficients.  These coefficients may change
between estimates; their only additional dependence is on the fixed
background data, and they are uniform in the admissible profiles, scales,
energies, and branch labels.  The implicit constants in \(O(\cdot)\)
and \(\asymp\) have the same dependence.
The threshold \(Q_0\) is enlarged as
needed below, and \(\varepsilon_0\) is then chosen so that the selected
denominator satisfies \(Q\ge Q_0\).  Recall
\(\eta_0\) from the setup preceding
Corollary~\ref{cor:polynomial-multiwell-K-reinduction}.  From the shared
application fixed above, retain \(D_5,\eta_{\rm resp}\), and enlarge
\(D_5\) once to dominate the centered-comparison loss below.
Choose
\begin{equation}\label{fr:eq:Young-kappa-per-choice}
\begin{aligned}
 \varkappa_{\rm per}>
 \max\left\{
 D_3,\,2D_5+2,\,
 D_5+B_2+3,\,
 D_5+2B_1+7,\right.\\
 \left.
 3D_5+7,\,
 D_5+2\tau(\tau+2)+7
 \right\},
\end{aligned}
\end{equation}
then choose
\begin{equation}\label{fr:eq:Young-edge-gap-exponent-choice}
\begin{aligned}
 \max\left\{
 \substack{B_2,\ 2B_1+4,\\
            2D_5+4,\ 2\tau(\tau+2)+4}
 \right\}
 &<B_4\\
 &<\varkappa_{\rm per}-D_5-2,\\
 B_3
 &>\max\{\tau\varkappa_{\rm per}+2,B_4+2\}.
\end{aligned}
\end{equation}
After \eqref{fr:eq:Young-edge-gap-exponent-choice}, choose
\(B_5,B_6,B_7,B_8\)
in the order given by the first two lines of
\eqref{eq:multiwell-polynomial-exponent-order}, taking in addition
\[
        B_6>\tau(\tau+2).
\]
The interval allowed for \(B_6\) is nonempty because of the last
lower bound imposed on \(B_4\).  Now fix
\begin{equation}\label{fr:eq:Young-Dvel-prior-choice}
 D_6>
 \max\left\{
 \tau\varkappa_{\rm per}-\frac32B_4,
 \frac12B_5,0\right\}+1,
\end{equation}
and only then choose \(B_9\) by the last inequality in
\eqref{fr:eq:Young-transversality-exponent-budget}.
Thus \eqref{eq:multiwell-polynomial-exponent-order} holds and
\(B_9>D_6+2\).  All exponents are fixed before
\(\varepsilon\) is decreased.

For the \(n_\eps\) and \(Q\) fixed in the statement, the
Diophantine inequality and the continued-fraction estimate give
\[
 \kappa Q^{-\tau}\le\|Q\alpha\|_{\mathbb T}<q_{\ell+1}^{-1},
 \qquad
 Q\le n_\eps<q_{\ell+1}\le\kappa^{-1}Q^\tau.
\]
For sufficiently small \(\varepsilon\),
\[
 2^{-\varkappa_{\rm per}}n_\eps^{-\varkappa_{\rm per}}
 \le\varepsilon
 \le n_\eps^{-\varkappa_{\rm per}}.
\]
Combining these bounds gives \eqref{fr:eq:Young-epsilon-Q-comparison},
so all polynomial initial-scale estimates apply at \(n=Q\).

Write \(T_a:=E_{a,Q}^{\varepsilon,\delta}\).  By
\eqref{eq:polynomial-fold-response} and
\eqref{eq:polynomial-fold-response-error},
\begin{equation}\label{fr:eq:Young-response-remainder-smallness}
\begin{aligned}
 \frac{
 |\mathcal R_{a,b,Q}^{\varepsilon,\delta}|}
 {\varepsilon/q}
 &\le C_7\left(
 Q^{2D_5}\varepsilon+
 Q^{D_5}e^{-c_4Q^{\eta_{\rm resp}}}\right)\\
 &\le C_7\left(
 Q^{2D_5-\varkappa_{\rm per}}+
 Q^{D_5}e^{-c_4Q^{\eta_{\rm resp}}}\right)
 \longrightarrow0
\end{aligned}
\end{equation}
The convergence in \eqref{fr:eq:Young-response-remainder-smallness} is
uniform in \(a<b\) and in the quantitative class.  Since
\(\sigma_{\rm Sch}(\delta)\ge q^{-1}\), for \(Q\ge Q_0\),
\begin{equation}\label{fr:eq:Young-actual-fold-gap}
 \Delta_Q^{\rm fold}
 :=\min_{0\le a<b<N}|T_a-T_b|
 \ge\frac{\varepsilon}{2q}
 \ge \frac{(\kappa/2)^{\varkappa_{\rm per}}}{2q}Q^{-\tau\varkappa_{\rm per}}
 \ge Q^{-B_3}.
\end{equation}
Moreover, \eqref{eq:polynomial-fold-cluster},
\eqref{fr:eq:Young-epsilon-Q-comparison}, and the upper restriction on
\(B_4\) give
\[
 \max_{0\le a<N}\bigl(
 |\xi_{a,Q}^{\varepsilon,\delta}-p_a|
 +|T_a-E_+^{\varepsilon,\delta}|\bigr)
 \le C_7Q^{D_5-\varkappa_{\rm per}}
      +e^{-c_4Q^{\eta_{\rm resp}}}
 \le Q^{-B_4}.
\]
Thus \eqref{eq:multiwell-fold-input} holds.
The common construction in
Corollary~\ref{cor:polynomial-multiwell-K-reinduction} through
Corollary~\ref{cor:polynomial-multiwell-induced-zeros}, with the data fixed above,
supplies \(\mathcal I_{Q,k}\), \(c_{Q,k,s,\varsigma}\), and
\(\widehat c_{Q,k,s,\varsigma}\).

\emph{Step (ii).} We estimate the root velocities for the unperturbed fold.
For every peak put
\[
 \xi_a:=\xi_{a,Q}^{\varepsilon,\delta},
 \qquad
 F_{a,Q}(y,D)
 :=H_{a,Q}^{\varepsilon,\delta}(\xi_a+y,T_a-D).
\]
Let \((\xi_{a,Q}^0,E_Q^0)\) be the unperturbed folds supplied by
Lemma~\ref{lem:unperturbed-tip-fold} at \(n=Q\).  Define the
corresponding recentered function by
\[
 F_Q^0(y,D)
 :=H_{a,Q}^{0}(\xi_{a,Q}^{0}+y,E_Q^0-D).
\]
Translation covariance makes the right-hand side independent of \(a\).
Choose once
\[
 \rho_Q\asymp Q^{-B_4}
\]
with the upper comparison constant large enough that every active depth
\(D=T_a-E\) occurring in the energy intervals \(\mathcal I_{Q,k}\) is at most
\(\rho_Q\). More explicitly, the energy-interval definitions and the fold
cluster bound give, for every active peak,
\[
 Q^{-B_5}=h_Q\le D=T_a-E
 \le\max_{0\le u,v<N}|T_u-T_v|+2h_Q+Q^{-B_9}
 \le C_7 Q^{-B_4}.
\]
Thus \(|y|\le C_7\sqrt{\rho_Q}\) and \(D\le\rho_Q\) contain all active
roots.  Since \(B_4/2>B_1+2\) and
\(B_4>B_2\), this recentered domain is contained in the common
normal-form domain.  Furthermore,
\[
 |\xi_a-\xi_{a,Q}^0|+|T_a-E_Q^0|
 \le C_7\bigl(\varepsilon Q^{D_5}
             +e^{-c_4Q^{\eta_0}}\bigr),
\]
by the perturbed and unperturbed fold-displacement estimates.  This
margin makes both terms in the telescoping comparison below well defined
on the same recentered domain.
On the common recentered domain one has the exact telescoping identity
\[
\begin{aligned}
 F_{a,Q}(y,D)-F_Q^0(y,D)
={}&
 \bigl(H_{a,Q}^{\varepsilon,\delta}-H_{a,Q}^{0}\bigr)
       (\xi_a+y,T_a-D)\\
&+H_{a,Q}^{0}(\xi_a+y,T_a-D)\\
&-H_{a,Q}^{0}(\xi_{a,Q}^{0}+y,E_Q^0-D).
\end{aligned}
\]
The stability estimate
\eqref{eq:polynomial-Schur-SVD-stability}, used at order three, and the
fold-displacement estimate
\eqref{eq:perturbed-fold-displacement-polynomial} give bounds on this
common recentered domain.  Define
\[
 \eta_Q:=\max_a\|F_{a,Q}-F_Q^0\|_{C^2},
 \qquad
 M_Q:=\max_a\bigl(\|F_{a,Q}\|_{C^3}+\|F_Q^0\|_{C^3}\bigr).
\]
Then
\begin{equation}\label{fr:eq:Young-centered-fold-comparison}
 \eta_Q\le C_7\bigl(
 \varepsilon Q^{D_5}+e^{-c_4Q^{\eta_0}}\bigr),
 \qquad
 M_Q\le C_7Q^{D_5}.
\end{equation}
Both functions and their first \(y\)-derivatives vanish at \((0,0)\).

By the choices above,
\begin{equation}\label{fr:eq:Young-fold-comparison-smallness}
 M_Q\sqrt{\rho_Q}=o_{Q\to\infty}(1),
 \qquad
 \frac{\eta_Q\rho_Q}{\Delta_Q^{\rm fold}}=o_{Q\to\infty}(1).
\end{equation}
Indeed, \eqref{fr:eq:Young-epsilon-Q-comparison} and
\eqref{fr:eq:Young-actual-fold-gap} give
\[
 \begin{aligned}
 M_Q\sqrt{\rho_Q}
 &\le C_7Q^{D_5-B_4/2}
 =o_{Q\to\infty}(1),\\
 \frac{\eta_Q\rho_Q}{\Delta_Q^{\rm fold}}
 &\le C_7\left(Q^{D_5-B_4}
 +Q^{\tau\varkappa_{\rm per}-B_4}e^{-c_4Q^{\eta_0}}\right)
 =o_{Q\to\infty}(1).
 \end{aligned}
\]

Define
\[
 a_Q:=\frac12\partial_y^2F_Q^0(0,0),
 \qquad
 b_Q:=-\partial_DF_Q^0(0,0).
\]
The transversality conclusion
\eqref{eq:perturbed-central-transversality}, applied with \(\delta=0\), gives
\[
        0<c_4\le a_Q,b_Q\le C_7.
\]
Taylor's formula at \((0,0)\), with \(D\) assigned weight two relative
to \(y\), gives the explicit normal form
\begin{equation}\label{fr:eq:Young-reference-fold-Taylor}
\begin{aligned}
 F_Q^0(y,D)
 &=a_Qy^2-b_QD+\mathcal R_Q(y,D),\\
 |\mathcal R_Q(y,D)|
 &\le C_7M_Q\bigl(
 |y|^3+|y|D+D^2\bigr).
\end{aligned}
\end{equation}
The corresponding differentiated remainder bounds are
\[
 |\partial_y\mathcal R_Q|
 \le C_7M_Q(y^2+D),
 \qquad
 |\partial_D\mathcal R_Q|
 \le C_7M_Q(|y|+D).
\]

Set \(D=t^2\), \(y=tz\), and, for \(t>0\), define
\[
 \Psi_Q(z,t):=t^{-2}F_Q^0(tz,t^2).
\]
The exact integral identity
\[
 \Psi_Q(z,t)
 =z^2\int_0^1(1-\theta)
       F_{Q,yy}^0(\theta tz,0)\,d\theta
  +\int_0^1F_{Q,D}^0(tz,\theta t^2)\,d\theta
\]
extends \(\Psi_Q\) as a \(C^1\) function to \(t=0\), where
\[
 \Psi_Q(z,0)=a_Qz^2-b_Q.
\]
More explicitly, on
\(\{|z|\le C_7,\ 0\le t\le\sqrt{\rho_Q}\}\),
\[
\begin{aligned}
 \Psi_Q(z,t)
 &=a_Qz^2-b_Q+\widetilde{\mathcal R}_Q(z,t),\\
 |\widetilde{\mathcal R}_Q|
 +|\partial_z\widetilde{\mathcal R}_Q|
 &\le C_7M_Qt,\\
 |\partial_t\widetilde{\mathcal R}_Q|
 &\le C_7M_Q.
\end{aligned}
\]
At \(t=0\), its two zeros
\(z=\sigma_\varsigma\sqrt{b_Q/a_Q}\), \(\varsigma\in\{+,-\}\), are uniformly simple,
where
\[
        \sigma_+:=-1,\qquad \sigma_-:=+1.
\]
This convention matches the global labels
\(c_{Q,k,s,+}<\xi_{a_s,Q}^{\varepsilon,\delta}
<c_{Q,k,s,-}\): the label \(+\) denotes the left root and hence the
negative recentered branch.
The quantitative implicit function theorem
\cite[Chapter~4]{Deimling1985}, together with
\(M_Q\sqrt{\rho_Q}=o_{Q\to\infty}(1)\), therefore gives two
\(C^1\) functions
\(z_{Q,\varsigma}\) on \(0\le t\le\sqrt{\rho_Q}\) such that
\begin{equation}\label{fr:eq:Young-reference-rescaled-root-asymptotics}
 z_{Q,\varsigma}(t)
 =\sigma_\varsigma\sqrt{\frac{b_Q}{a_Q}}+O(M_Qt),
 \qquad
 z_{Q,\varsigma}'(t)=O(M_Q).
\end{equation}
Thus, by \eqref{fr:eq:Young-reference-rescaled-root-asymptotics}, the
reference roots
\(y_{Q,\varsigma}^0(D):=
\sqrt D\,z_{Q,\varsigma}(\sqrt D)\) satisfy
\begin{equation}\label{fr:eq:Young-reference-root-square-root-expansion}
 y_{Q,\varsigma}^0(D)
 =\sigma_\varsigma\sqrt{\frac{b_Q}{a_Q}}\,D^{1/2}
 +O(M_QD).
\end{equation}

At the root in
\eqref{fr:eq:Young-reference-root-square-root-expansion},
\eqref{fr:eq:Young-reference-fold-Taylor} and its
differentiated form yield, with \(t=\sqrt D\),
\[
\begin{aligned}
 \partial_yF_Q^0(y_{Q,\varsigma}^0(D),D)
 &=t\bigl(2a_Qz_{Q,\varsigma}(t)+O(M_Qt)\bigr),\\
 \partial_DF_Q^0(y_{Q,\varsigma}^0(D),D)
 &=-b_Q+O(M_Qt).
\end{aligned}
\]
Define
\[
 U_\varsigma^0(D)
 :=\left(
 \frac{\partial_yF_Q^0}{\partial_DF_Q^0}
 \right)^2(y_{Q,\varsigma}^0(D),D).
\]
Then
\begin{equation}\label{fr:eq:Young-reference-inverse-square-expansion}
 U_\varsigma^0(D)
 =\frac{4a_Q}{b_Q}D+O(M_QD^{3/2}).
\end{equation}
Write \(z(t):=z_{Q,\varsigma}(t)\) and define
\[
\begin{aligned}
 A(t)&:=z(t)\int_0^1F_{Q,yy}^0(\theta tz(t),0)\,d\theta
       +t\int_0^1F_{Q,yD}^0(tz(t),\theta t^2)\,d\theta,\\
 B(t)&:=F_{Q,D}^0(tz(t),t^2),\qquad
 W_{Q,\varsigma}(t):=\bigl(A(t)/B(t)\bigr)^2.
\end{aligned}
\]
Then \(A(t)=t^{-1}F_{Q,y}^0(tz(t),t^2)\) for \(t>0\), so
\(U_\varsigma^0(t^2)=t^2W_{Q,\varsigma}(t)\). The integral formulas,
\(z'=O(M_Q)\), and \(M_Q\sqrt{\rho_Q}=o(1)\) give
\[
 |A|+|B|\le C_7,\qquad |B|\ge c_4,\qquad
 |A'|+|B'|\le C_7M_Q.
\]
Consequently,
\[
 \begin{gathered}
 W_{Q,\varsigma}(t)=\frac{4a_Q}{b_Q}+O(M_Qt),
 \qquad W_{Q,\varsigma}'(t)=O(M_Q),\\
 (U_\varsigma^0)'(D)
 =W_{Q,\varsigma}(\sqrt D)+\frac{\sqrt D}{2}W_{Q,\varsigma}'(\sqrt D)
 =\frac{4a_Q}{b_Q}+O(M_Q\sqrt D).
 \end{gathered}
\]
For \(Q\ge Q_0\), uniformly in \(Q\),
\begin{equation}\label{fr:eq:Young-reference-U-monotonicity}
 0<c_4\le(U_\varsigma^0)'(D)\le C_7,
 \qquad 0<D\le\rho_Q,\quad \varsigma\in\{+,-\}.
\end{equation}

\emph{Step (iii).} We compare the perturbed folds and separate the root velocities.
Put \(G_{a,Q}:=F_{a,Q}-F_Q^0\).  Since
\[
 G_{a,Q}(0,0)=\partial_yG_{a,Q}(0,0)=0,
\]
the \(C^2\) estimate in
\eqref{fr:eq:Young-centered-fold-comparison} and Taylor's formula give,
on the same recentered domain,
\begin{equation}\label{fr:eq:Young-centered-difference-Taylor}
\begin{aligned}
 |G_{a,Q}(y,D)|
 &\le C_7\eta_Q(D+y^2),\\
 |\partial_yG_{a,Q}(y,D)|
 &\le C_7\eta_Q(D+|y|),\\
 |\partial_DG_{a,Q}(y,D)|
 &\le C_7\eta_Q.
\end{aligned}
\end{equation}
Let \(y_{a,\varsigma}(D)\) be the corresponding root of \(F_{a,Q}\), and
define its phase coordinate by
\begin{equation}\label{fr:eq:Young-short-root-phase-coordinate}
 c_{a,\varsigma}(E)
 :=\xi_a+y_{a,\varsigma}(T_a-E).
\end{equation}
On the energy interval \(\mathcal I_{Q,k}\), if the peak \(a\) is the \(h\)-th active
peak \(a_h\), then this is exactly
\(c_{Q,k,h,\varsigma}(E)\).
For \(0<D\le\rho_Q\) in an active energy interval, at
\(y_{Q,\varsigma}^0(D)\), the first line of
\eqref{fr:eq:Young-centered-difference-Taylor} is \(O(\eta_QD)\),
whereas
\[
 |\partial_yF_Q^0(y_{Q,\varsigma}^0(D),D)|\asymp\sqrt D.
\]
Since \(\eta_Q=o_{Q\to\infty}(1)\), the same quantitative
implicit-function estimate is uniform despite the square-root slope.
Indeed, on active intervals \(D\ge h_Q>0\); throughout
\(|y-y_{Q,\varsigma}^0(D)|\le C_7\eta_Q\sqrt D\),
\eqref{fr:eq:Young-centered-difference-Taylor} and
\(M_Q\sqrt{\rho_Q}=o_{Q\to\infty}(1)\) give
\[
 |\partial_yF_{a,Q}(y,D)|\ge c_4\sqrt D,
 \qquad
 |F_{a,Q}(y_{Q,\varsigma}^0(D),D)|\le C_7\eta_QD.
\]
The mean-value theorem, together with the preserved sign change, gives
one and only one nearby root and therefore
\begin{equation}\label{fr:eq:Young-centered-root-comparison}
 |y_{a,\varsigma}(D)-y_{Q,\varsigma}^0(D)|
 \le C_7\eta_Q\sqrt D.
\end{equation}
By \eqref{fr:eq:Young-centered-difference-Taylor},
\eqref{fr:eq:Young-centered-root-comparison}, and
\(M_Q\sqrt{\rho_Q}=o(1)\),
\[
 \sup|\partial_y^2F_Q^0|
 =2a_Q+O(M_Q\sqrt{\rho_Q})=O(1).
\]
Writing \(y_a=y_{a,\varsigma}(D)\) and
\(y_0=y_{Q,\varsigma}^0(D)\), the mean-value theorem gives
\[
\begin{aligned}
 |F_{a,Q,y}(y_a,D)-F_{Q,y}^0(y_0,D)|
 &\le |G_{a,Q,y}(y_a,D)|
       +\sup|F_{Q,yy}^0|\,|y_a-y_0|\\
 &\le C_7\eta_Q\sqrt D,\\
 |F_{a,Q,D}(y_a,D)-F_{Q,D}^0(y_0,D)|
 &\le |G_{a,Q,D}(y_a,D)|
       +\sup|F_{Q,yD}^0|\,|y_a-y_0|\\
 &\le C_7\eta_Q(1+M_Q\sqrt D)
 \le C_7\eta_Q.
\end{aligned}
\]
In particular,
\[
 |\partial_yF_{a,Q}(y_{a,\varsigma}(D),D)|\asymp\sqrt D,
 \qquad
 |\partial_DF_{a,Q}(y_{a,\varsigma}(D),D)|\asymp1.
\]

By \eqref{fr:eq:Young-short-root-phase-coordinate}, since \(D=T_a-E\),
implicit differentiation of
\(F_{a,Q}(y_{a,\varsigma}(D),D)=0\) gives
\[
 \partial_Ec_{a,\varsigma}
 =\frac{\partial_DF_{a,Q}}{\partial_yF_{a,Q}}
 \quad\text{at }(y_{a,\varsigma}(D),D).
\]
Therefore, with
\[
 U_{a,\varsigma}(D):=|\partial_Ec_{a,\varsigma}|^{-2}
 =\left(\frac{\partial_yF_{a,Q}}
 {\partial_DF_{a,Q}}\right)^2(y_{a,\varsigma}(D),D),
\]
put \(A_a:=\partial_yF_{a,Q}(y_{a,\varsigma},D)\),
\(B_a:=\partial_DF_{a,Q}(y_{a,\varsigma},D)\), and define
\(A_{\rm ref},B_{\rm ref}\) analogously for the reference root.  Since
\[
 |A_a|+|A_{\rm ref}|\le C_7\sqrt D,\qquad
 |B_a|+|B_{\rm ref}|\le C_7,\qquad
 \min\{|B_a|,|B_{\rm ref}|\}\ge c_4,
\]
the elementary quotient estimate
\[
 \left|\frac{A_a^2}{B_a^2}-\frac{A_{\rm ref}^2}{B_{\rm ref}^2}\right|
 \le C_7\left(
 |A_a-A_{\rm ref}|(|A_a|+|A_{\rm ref}|)
 +D|B_a-B_{\rm ref}|\right)
 \le C_7\eta_QD
\]
gives
\begin{equation}\label{fr:eq:Young-inverse-square-comparison}
 |U_{a,\varsigma}(D)-U_\varsigma^0(D)|\le C_7\eta_QD,
 \qquad c_4D\le U_{a,\varsigma}(D)\le C_7D.
\end{equation}

Consider now two distinct peaks \(a\ne b\) and the same branch label
\(\varsigma\).  At a common energy \(E\), put
\(D_a=T_a-E\) and \(D_b=T_b-E\).  Then
\[
        D_a-D_b=T_a-T_b.
\]
On every energy interval on which both zeros are present,
\(0<D_a,D_b\le\rho_Q\).  By
\eqref{fr:eq:Young-reference-U-monotonicity},
\eqref{fr:eq:Young-inverse-square-comparison}, and
\eqref{fr:eq:Young-fold-comparison-smallness},
\[
\begin{aligned}
 |U_{a,\varsigma}(D_a)-U_{b,\varsigma}(D_b)|
 &\ge c_4|T_a-T_b|
      -C_7\eta_Q(D_a+D_b)\\
 &\ge\frac{c_4}{2}|T_a-T_b|.
\end{aligned}
\]
Since \(U_{a,\varsigma},U_{b,\varsigma}\le C_7\rho_Q\) and the two energy
derivatives have the same
sign,
\[
 \left|r^{-1/2}-s^{-1/2}\right|
 =\frac{|r-s|}
 {\sqrt{rs}(\sqrt r+\sqrt s)},
 \qquad r,s>0,
\]
gives
\begin{equation}\label{fr:eq:Young-same-sign-relative-velocity-lower}
\begin{aligned}
 |\partial_Ec_{a,\varsigma}
  -\partial_Ec_{b,\varsigma}|
 &\ge c_4|T_a-T_b|\rho_Q^{-3/2}\\
 &\ge c_4
 Q^{-\tau\varkappa_{\rm per}+\frac32B_4}.
\end{aligned}
\end{equation}
For opposite branch labels,
\[
 |\partial_Ec_{a,+}-\partial_Ec_{b,-}|
 =|\partial_Ec_{a,+}|+|\partial_Ec_{b,-}|
 \ge c_4Q^{B_4/2},
\]
including \(a=b\).  Combining this with
\eqref{fr:eq:Young-same-sign-relative-velocity-lower} and
\eqref{eq:multiwell-short-root-velocities}, every distinct pair satisfies
\[
\begin{aligned}
 |\partial_E(c_\iota-c_{\iota'})|
 &\ge c_4\min\{Q^{-\tau\varkappa_{\rm per}+3B_4/2},
                 Q^{B_4/2}\}
 \ge4Q^{-D_6},\\
 |\partial_E(c_\iota-c_{\iota'})|
 &\le2C_7Q^{B_5/2}\le\tfrac12Q^{D_6},
\end{aligned}
\]
by \eqref{fr:eq:Young-Dvel-prior-choice}, uniformly for \(Q\ge Q_0\).

It remains to pass to the distinguished induced zeros already constructed
by Corollary~\ref{cor:polynomial-multiwell-induced-zeros}.
Let \(F\) and \(\widehat F\) denote the parent and induced overlaps.
At their roots,
\[
 c_E=-\frac{F_E(c,E)}{F_x(c,E)},
 \qquad
 \widehat c_E=-\frac{\widehat F_E(\widehat c,E)}
                       {\widehat F_x(\widehat c,E)}.
\]
The endpoint and root estimates give
\[
 \|\widehat F-F\|_{C^2}+|\widehat c-c|
 \le C_7 Q^{D_4}e^{-2c_4Q^{\eta_0}},
 \qquad
 \min\{|F_x(c,E)|,|\widehat F_x(\widehat c,E)|\}
 \ge c_4Q^{-B_5/2}.
\]
Using \eqref{eq:multiwell-long-short-Cr-tail} and the polynomial derivative
bounds, we obtain, for some fixed \(D_{25}=D_{25}(R,q,r_*)>0\),
\[
\begin{aligned}
 |\widehat c_E-c_E|
 &\le
 \frac{|\widehat F_E(\widehat c,E)-F_E(c,E)|}
      {|\widehat F_x(\widehat c,E)|}\\
 &\quad+
 \frac{|F_E(c,E)|\,
       |\widehat F_x(\widehat c,E)-F_x(c,E)|}
      {|\widehat F_x(\widehat c,E)F_x(c,E)|}\\
 &\le C_7 Q^{D_{25}}\bigl(\|\widehat F-F\|_{C^2}+|\widehat c-c|\bigr)\\
 &\le C_7e^{-c_4Q^{\eta_0}}
 =o_{Q\to\infty}(Q^{-D_6}),
\end{aligned}
\]
uniformly on each \(\mathcal I_{Q,k}\).  Hence, for \(Q\ge Q_0\),
\[
 2Q^{-D_6}
 \le|\partial_E(\widehat c_\iota-\widehat c_{\iota'})|
 \le Q^{D_6}.
\]
Continuity and connectedness make its sign
\(\sigma_{\iota,\iota'}\) constant, proving
\eqref{fr:eq:Young-long-pair-two-sided-velocity}.
\end{proof}

For the next proposition, retain the fixed background data
\(N,\lambda,\rho,\alpha,\kappa,\tau\), with \(\tau>1\), and fix integers
\(r_*\ge3\) and \(R,q\ge1\).
The initial-scale parameters and \(D_{13}\) will be chosen to satisfy
\begin{equation}\label{fr:eq:multipeak-seed-propagation-budget}
\begin{aligned}
 B_{10}&>
 \max\{\tau^2+2,\ B_9+D_6+\tau+2,\
          B_{13}+2,\ D_8+D_7+4\},\\
 B_{11}&>
 B_{10}(\tau+1)+B_9+D_6+4,\\
 D_{13}&>
 (r_*+5)\bigl(D_8+D_7
                 +B_{10}+B_{11}+4\bigr).
\end{aligned}
\end{equation}

Retain the scale convention
\begin{equation}\label{fr:eq:regular-cor-starting-scale}
 n_\eps:=\left\lfloor\eps^{-1/\varkappa_{\rm per}}\right\rfloor,
 \qquad Q=q_\ell:=\max\{q_j:q_j\le n_\eps\}.
\end{equation}
The initial-scale parameters below have the form
\[
 \mathfrak D_{\rm seed}
 =(r_*+2,B,\tau,\log\lambda;\ldots),\qquad a_{\rm seed}>1,
\]
as in Definition~\ref{fr:def:Young-seed-package}.

\begin{proposition}
\label{fr:prop:prevalent-multiwell-seed}
There are constants
\[
\begin{aligned}
 \varepsilon_{\rm seed}&=\varepsilon_{\rm seed}(R,q,r_*)>0,
 &\qquad C_8&=C_8(R,q,r_*)\ge1,\\
 Q_0&=Q_0(R,q,r_*)\ge1,
\end{aligned}
\]
exponents \(\varkappa_{\rm per},D_3,B_9,D_{13}>0\),
and parameters \(\mathfrak D_{\rm seed}\) of the form above satisfying
\eqref{fr:eq:multipeak-seed-propagation-budget}, all depending only on
\((\mathfrak D_{\rm AMO}^{(r_*+6)},R,q)\), with the following property.
For every \(\delta\in\mathcal P_{\mathrm{strat}}(R,q)\) and
\(0<\eps<\varepsilon_{\rm seed}\), the following conclusions hold.
\begin{enumerate}
\item The scale
\eqref{fr:eq:regular-cor-starting-scale} satisfies \(Q\ge Q_0\) and
\begin{equation}\label{fr:eq:regular-cor-epsilon-Q-comparison}
\begin{gathered}
 Q\le n_\eps<q_{\ell+1}\le\kappa^{-1}Q^\tau,\\
 (\kappa/2)^{\varkappa_{\rm per}}Q^{-\tau\varkappa_{\rm per}}
 \le\eps\le Q^{-\varkappa_{\rm per}}\le Q^{-D_3}.
\end{gathered}
\end{equation}
\item For each \(1\le s\le N\), there is a nondegenerate compact interval
\(\mathcal J_{Q,s}^{\eps,\delta}\) with
\begin{equation}\label{fr:eq:multipeak-seed-window-measure}
 |\mathcal J_{Q,s}^{\eps,\delta}|\ge C_8^{-1}Q^{-B_9},
\end{equation}
on which the scale-\(Q\) data of \(V_{\eps,\delta}\) have \(2s\) simple
zeros and satisfy Definition~\ref{fr:def:Young-seed-package} with the
same parameters \(\mathfrak D_{\rm seed}\), uniformly in
\((\delta,\eps,s,Q)\).
\end{enumerate}
\end{proposition}

\begin{proof}[Proof of
Proposition~\ref{fr:prop:prevalent-multiwell-seed}]
\smallskip
\noindent\emph{Step (i).} We choose the starting scale and the initial-scale parameters.
We first choose the exponents.  The coefficient \(C_8=C_8(R,q,r_*)\ge1\)
will be fixed as in Step~(vi), before the final choice of \(Q_0\).
All estimates use this same coefficient.
Set \(r_\dagger:=r_*+2\).  Apply
Lemmas~\ref{fr:lem:multiwell-separated-folds} and
\ref{fr:lem:multiwell-pair-velocity-transversality} at order
\(r_\dagger\).  The derivative orders are
\[
 \begin{array}{c|c|c|c|c}
 \text{step}&\text{propagation}&\text{reinduction}
            &\text{first variation}&\text{AMO input}\\ \hline
 \text{order}&r_*&r_*+2&r_*+4&r_*+6 .
 \end{array}
\]
By \eqref{eq:fixed-background-data}, all are restrictions of one
application with profile bound \(2R\) and derivative order \(r_*+6\).
The lemmas give \(n_\eps,Q\),
\eqref{fr:eq:Young-epsilon-Q-comparison},
\eqref{fr:eq:Young-transversality-exponent-budget}, and all conclusions of
Corollary~\ref{cor:polynomial-multiwell-K-reinduction} and its four
subsequent results at order \(r_\dagger\).  After decreasing the
common amplitude threshold, these conclusions apply uniformly and give
\eqref{fr:eq:regular-cor-starting-scale}--
\eqref{fr:eq:regular-cor-epsilon-Q-comparison}.

Fix an interval index \(s\).  Take a concentric fixed-fraction subinterval
\(\mathcal J_{Q,s}^{\eps,\delta}\Subset\mathcal I_{Q,s}\).
The uniform interval-length estimate gives
\eqref{fr:eq:multipeak-seed-window-measure}, uniformly in the admissible data
and \(1\le s\le N\), by the choice of \(C_8\).
For compatible relative lifts,
\[
 \operatorname{osc}_{\mathcal J_{Q,s}^{\eps,\delta}}
   (\widetilde c_\iota-\widetilde c_{\iota'})
 \le|\mathcal J_{Q,s}^{\eps,\delta}|Q^{D_6}
 \le Q^{-B_9+D_6}<\tfrac18
\]
for sufficiently large \(Q\), since
\(B_9>D_6+2\).
Identify \(\{1,\ldots,2s\}\) with
\(\{1,\ldots,s\}\times\{+,-\}\), writing
\(a=(h,\varsigma)\).  To verify \ref{fr:def:Young-seed-package}\textup{(\ref{item:seed-factorization})--(\ref{item:seed-velocities})}, set
\[
\begin{aligned}
 \mathcal K_Q^0&:=\mathcal K_{Q,s},
 &\qquad \mathcal U_{Q,a}&:=U_{Q,s,h,\varsigma},\\
 c_{\ell,a}&:=\widehat c_{Q,s,h,\varsigma},
 & 1\le h&\le s,\quad \varsigma\in\{+,-\}.
\end{aligned}
\]
By \eqref{eq:multiwell-root-high-order-bound}, these maps belong to
\(C^{r_\dagger-1}(\mathcal J_{Q,s}^{\eps,\delta},\mathbb T)\) and are
independent of \(r\).  Choose their real lifts in the corresponding fixed
components of \eqref{eq:multiwell-K-base}; the itinerary partitions below only restrict
these global branches and lifts.  The bound for the
displacement of the zero branches
from \eqref{eq:multiwell-short-root-velocities}, the choice
\(B_9>B_7+B_5/2+2\), and the exponentially small
induced-zero displacement give
\[
 \operatorname{dist}_{\mathbb T}
 (\widehat c_{Q,s,h,\varsigma}(E),\partial U_{Q,s,h,\varsigma})
 \ge\tfrac12Q^{-B_7}
 \ge Q^{-(B_7+1)}
\]
on \(\mathcal J_{Q,s}^{\eps,\delta}\).  With \(\eta_{\rm mw}\) from
\eqref{eq:multiwell-eta-mw-explicit} and the exponent \(D_4\)
supplied by Corollary~\ref{cor:polynomial-multiwell-K-reinduction}, set
\[
 B_{13}:=B_7+1,\qquad
 a_{\rm seed}:=\frac{B_6}{2\tau}
 >\frac{\tau+2}{2}>1,
 \qquad
 \eta_{\rm seed}:=\eta_{\rm mw}.
\]
Set \(M_*^{\rm seed}:=r_\dagger B_{16}\) for this order-\(r_\dagger\)
application.  The multivariable Fa\`a di Bruno formula then converts
the relative modulus bounds into the logarithmic derivative bounds required in
the initial-scale definition.
Before choosing the radius of the smaller return set, fix
\begin{equation}\label{fr:eq:multiwell-prior-seed-exponents}
 \begin{aligned}
 D_7&>\max\{D_4,\tfrac12B_5\}+1,\\
 D_8&>
 \max\left\{D_4+2,D_7+1\right\}.
 \end{aligned}
\end{equation}
Choose \(B_{10}\) to satisfy the first line of
\eqref{fr:eq:multipeak-seed-propagation-budget} and
\(B_{10}>B_8+2\).  Next choose \(B_{11}\) and then \(D_{13}\)
to satisfy its second and third lines, respectively.
Use the smaller return sets \(K_{q_\ell}^{(r)}(E)\) of
Definition~\ref{fr:def:Young-seed-package}, with the branches and lifts
just fixed and
\[
 r\in\mathscr W_\ell=[Q^{-B_{10}},2Q^{-B_{10}}].
\]
All assertions below are uniform in \(r\in\mathscr W_\ell\).

\smallskip
\noindent\emph{Step (ii).} We construct the parameter cells and verify \ref{fr:def:Young-seed-package}\textup{(\ref{item:seed-transversality})}.
For fixed \(r\in\mathscr W_\ell\),
Lemma~\ref{B:lem:return-time-poly-two-sided}, applied to one component of
radius \(r\), gives
\(r_{K_{q_\ell}^{(r)}}^\pm\le Cr^{-\tau}\le C_8Q^{B_{10}\tau}\),
because return to the union is no later.
Equation~\eqref{eq:multiwell-long-root-velocities} gives each induced zero a
fixed derivative sign and
\(C_8^{-1}Q^{B_4/2}\le|\partial_E\widehat
c_{Q,s,h,\varsigma}|\le C_8Q^{B_5/2}\).
Since \(|\mathcal J_{Q,s}^{\eps,\delta}|\le Q^{-B_9}\), every
individual lift satisfies
\[
 \operatorname{osc}_{\mathcal J_{Q,s}^{\eps,\delta}}
 \widetilde c_{\ell,a}
 \le C_8 Q^{-B_9+B_5/2}<\tfrac18.
\]
The compatible relative lifts have oscillation less than \(1/8\) by the
pairwise estimate already imposed above.

Put \(R_Q:=\lceil4^\tau\kappa^{-1}Q^{B_{10}\tau}\rceil\) and
\[
 \mathscr E_r(E):=
 \bigcup_{|j|\le R_Q}T_\alpha^{-j}
 \left(\partial\mathcal K_Q^0\cup
 \{c_{\ell,a}(E)+\sigma r:
          1\le a\le2s,\ \sigma\in\{-1,0,1\}\}\right).
\]
The ordering of these points determines the two-sided return partition,
the intermediate \(\mathcal K_Q^0\)-itinerary, and the positions of the
zero graphs. It changes only when two points of \(\mathscr E_r(E)\)
coincide. Subtracting their translation indices gives
\(|k|\le2R_Q\). Collisions between two fixed endpoints are independent
of \(E\); every remaining collision equation has one of the forms
\[
 \widetilde c_{\ell,a}(E)+\sigma r+k\alpha=\widetilde b+m,
 \qquad
 b\in\partial\mathcal K_Q^0,\quad
 \sigma\in\{-1,0,1\},\quad m\in\mathbb Z,
\]
where \(\widetilde b\) is a lift of \(b\), or
\[
 (\widetilde c_{\ell,a}-\widetilde c_{\ell,b})(E)
 =-k\alpha+(\sigma'-\sigma)r+m,
 \qquad
 \sigma,\sigma'\in\{-1,0,1\},\quad m\in\mathbb Z.
\]
Here \(\sigma=0\) denotes a center and \(\sigma=\pm1\) an endpoint.
The fixed signs of the individual and pairwise derivatives give
\[
\begin{array}{c|c}
 \text{equation}&\text{number of energy solutions}\\ \hline
 \text{first type}&\le1\\
 \text{second type},\ a\ne b&\le1\\
 \text{second type},\ a=b&\text{independent of }E .
\end{array}
\]
The lift oscillations \(<1/8\) allow only \(O(1)\) integers \(m\).
Thus, for the finite cut set \(\mathscr C_r\),
\[
 \#\mathscr C_r\le CQ^{B_{10}\tau},\qquad
 D_{11}:=B_{10}\tau+1,\qquad
 \#\{\overline{\operatorname{comp}
   (\mathcal J_{Q,s}^{\eps,\delta}\setminus\mathscr C_r)}\}
 \le C_8Q^{D_{11}}.
\]
The resulting compact cells have disjoint interiors, and fixed-word
matrices extend continuously to each cell closure.  Collisions involving
endpoint translates lie outside the regular dynamics; center--center
cuts require no additional phase exclusions.  The restricted
global \(C^{r_\dagger-1}\) zero branches and lifts satisfy
\[
 \partial_E^j c_{\ell,a}|_{\mathcal C_-}
 =\partial_E^j c_{\ell,a}|_{\mathcal C_+}
 \quad\text{at a common endpoint},\qquad 0\le j\le r_\dagger-1.
\]

Fix one resulting parameter cell \(\mathcal C\).  For
\(a=(h,\varsigma)\) under the preceding identification, set
\[
 \nu_a:=\mathcal B_{Q,s,h,\varsigma},\qquad
 \mathscr U_{Q,a}^{\rm seed}:=\nu_a\times\mathcal C,
\]
where \(\mathcal B_{Q,s,h,\varsigma}\) is the fixed
\(\mathcal K_{Q,s}\)-return atom containing the induced-zero graph.  The
factorization
\eqref{eq:multiwell-long-transverse-factorization} restricted to this
phase--energy domain is precisely the initial simple-zero factorization
required in \ref{fr:def:Young-seed-package}\textup{(\ref{item:seed-transversality})}, with the derivative bounds in
\eqref{eq:multiwell-root-high-order-bound}.

By \eqref{eq:multiwell-return-boundary-clearance} and the exponentially
small induced-zero displacement, every
\(\widehat c_{Q,s,h,\varsigma}(E)\) stays at distance at least
\(C_8^{-1}Q^{-B_8}\) from the boundary of its distinguished
\(\mathcal K_{Q,s}\)-return atom.  Since
\(B_{10}>B_8+2\), every candidate child interval satisfies
\(I_{q_\ell,a}^{(r)}(E)\Subset
\nu_a=(\mathscr U_{Q,a}^{\rm seed})_E\) for
\((E,r)\in\mathcal C\times\mathscr W_\ell\).
This verifies the domain clause in \ref{fr:def:Young-seed-package}\textup{(\ref{item:seed-transversality})}.  Whenever a subcell is
boundary-admissible in the sense of \ref{fr:def:Young-seed-package}\textup{(\ref{item:seed-endpoint-error})}, its distinguished
first-return domains lie in \(\mathscr U_{Q,a}^{\rm seed}\).  The endpoint
estimate holds for every candidate radius; only the designation of these
domains requires boundary admissibility.

\smallskip
\noindent\emph{Step (iii).} We verify \ref{fr:def:Young-seed-package}\textup{(\ref{item:seed-factorization})}.
Decompose the first returns to
\(K_{q_\ell}^{(r)}(E)\) into consecutive \(\mathcal K_Q^0\)-returns.
The cocycle identity gives the exact chronological factorizations, while
\eqref{eq:multiwell-cross-return-bounds} and
\eqref{eq:multiwell-long-growth} give, uniformly in
\(r\in\mathscr W_\ell\), writing \(\mathsf R\) for the lattice length
of the factor,
\[
\begin{aligned}
 Q^{a_{\rm seed}}
 &\le C_8^{-1}Q^{B_6/\tau}\le\mathsf R
 \le C_8Q^{B_7\tau}\le C_8Q^{B_{13}\tau},\\
 C_8\mathsf R^{-\eta_{\rm mw}}(\log(\mathsf R+2))^B
 &\le C_8 Q^{-a_{\rm seed}\eta_{\rm mw}}(\log(Q+2))^B.
\end{aligned}
\]
The second inequality uses
\(\mathsf R\ge C_8^{-1}Q^{B_6/\tau}
=C_8^{-1}Q^{2a_{\rm seed}}\).
Together with \eqref{eq:multiwell-long-matrix-derivatives}, this proves
\ref{fr:def:Young-seed-package}\textup{(\ref{item:seed-factorization})} with
\[
 \omega_Q^{\rm seed}
 \le C_8 Q^{-a_{\rm seed}\eta_{\rm mw}}(\log(Q+2))^B.
\]
Since \(B_6/\tau>\tau+2\), these return lengths exceed
\(q_{\ell+1}\) for large \(Q\).

\smallskip
\noindent\emph{Step (iv).} We verify \ref{fr:def:Young-seed-package}\textup{(\ref{item:seed-endpoint-error})}.
Fix \(r\in\mathscr W_\ell\) and apply
\eqref{fr:eq:chain-fixed-order-endpoint} on the interior of every
fixed-itinerary domain \(\mathscr B_{\ell,\eta}\) to the two chains of
\(\mathcal K_Q^0\)-return blocks meeting at its intermediate return point.
At every internal concatenation point, either
\eqref{eq:multiwell-nondistinguished-nozero} applies or the factorization
\eqref{eq:multiwell-long-transverse-factorization}--
\eqref{eq:multiwell-long-transverse-coefficient} applies at distance at
least \(r\ge Q^{-B_{10}}\) from its zero.  Hence
\(|H_{{\rm seed},\nu}|\ge Q^{-(D_7+B_{10})}\).
For either chain, let \(L_{\rm ch}\) be its factor count, let
\(Q_{\rm ch}\) be the quantity \(Q_s\) in
Lemma~\ref{B:lem:polynomial-hyperbolic-chain} with \(s=r_\dagger\) and
norms on \(\mathscr B_{\ell,\eta}\), and let \(\Lambda_{\rm ch}\) be the
minimum there of the factor top singular values.  For a fixed \(D_{26}=D_{26}(R,q,r_*)>0\),
the return and order-\(r_\dagger\) bounds, the adjacent length bound
\(C_8^{-1}Q^{B_6/\tau}\), and \(\omega_Q^{\rm seed}\to0\) give
\[
\begin{aligned}
 \varkappa_Q&:=Q^{-(D_7+B_{10})},
 &\qquad L_{\rm ch}&\le R_Q\le C_8Q^{B_{10}\tau},\\
 Q_{\rm ch}&\le C_8Q^{D_{26}},
 & \Lambda_{\rm ch}&\ge\exp\{C_8^{-1}Q^{B_6/\tau}\}.
\end{aligned}
\]
Writing \(A:=D_0(r_\dagger)\) for the exponent in
that lemma, one obtains
\[
\begin{aligned}
 L_{\rm ch}^{A}Q_{\rm ch}^{A}
 \varkappa_Q^{-A}\Lambda_{\rm ch}^{-2}
 &\le C_8^{2A}Q^{A(B_{10}\tau+D_{26}+D_7+B_{10})}
        e^{-2Q^{B_6/\tau}/C_8}\\
 &\le10^{-3}\varkappa_Q
\end{aligned}
\]
after increasing \(Q_0\).  Thus the lemma applies to the forward and
backward chains.  Its two endpoint estimates, followed by the fixed-order
product rule for the scalar product defining the overlap, give the stated
\(C^{r_\dagger}\) tail.  Choose
\(D_9>0\) to absorb its polynomial losses.  Then
\[
 \|h_{\ell,\eta}^{(r)}
   -H_{{\rm seed},\nu(\eta)}\|
   _{C^{r_\dagger}(\mathscr B_{\ell,\eta})}
 \le C_8Q^{D_9}e^{-Q^{B_6/\tau}/C_8}
 \le C_8Q^{D_9}e^{-Q^{a_{\rm seed}}/C_8}.
\]
The fixed-word matrices and their derivatives through order
\(r_\dagger\) extend continuously to the chosen closure, so the same
estimate holds on all of \(\mathscr B_{\ell,\eta}\).

\smallskip
\noindent\emph{Step (v).} We verify \ref{fr:def:Young-seed-package}\textup{(\ref{item:seed-exterior})}.
Nondistinguished atoms satisfy
\eqref{eq:multiwell-nondistinguished-nozero}.  On a distinguished atom
outside the child interval,
\eqref{eq:multiwell-long-transverse-factorization}--
\eqref{eq:multiwell-long-transverse-coefficient} give
\[
 |\cos G_{\mathcal K_Q^0,\mathcal B_{Q,s,h,\varsigma}}(x,E)|
 \ge C_8^{-1}Q^{-B_5/2}
      \operatorname{dist}(x,\widehat c_{Q,s,h,\varsigma}(E)).
\]
The prior choice \eqref{fr:eq:multiwell-prior-seed-exponents} gives,
after increasing \(Q_0\),
\[
 C_8^{-1}Q^{-D_4}\ge Q^{-D_7},
 \qquad
 C_8^{-1}Q^{-B_5/2}\ge Q^{-D_7}.
\]
Together with the nondistinguished bound, this proves \ref{fr:def:Young-seed-package}\textup{(\ref{item:seed-exterior})} uniformly outside every candidate child
interval.  The all-atom mixed-derivative conclusion of
Proposition~\ref{prop:polynomial-multiwell-long-returns} and the choice
\(D_8>D_4+2\) give
\[
 C_8 Q^{D_4}\le\tfrac14Q^{D_8}
\]
for large \(Q\), hence both the required
\(C^{r_\dagger}\) bounds and, after enlarging the starting threshold, the
factor-\(1/4\) margin in
\eqref{fr:eq:Young-seed-angle-bootstrap-margin}.  The resulting
\(D_{11}\) and \(D_9\) are
uniform and do not enter the inequalities in
\eqref{fr:eq:multipeak-seed-propagation-budget}.

\smallskip
\noindent\emph{Step (vi).} We verify \ref{fr:def:Young-seed-package}\textup{(\ref{item:seed-velocities})} and choose uniform constants.
Lemma~\ref{fr:lem:multiwell-pair-velocity-transversality} is exactly
\ref{fr:def:Young-seed-package}\textup{(\ref{item:seed-velocities})}, including its factor-two margin; the preceding oscillation
estimate supplies the required lift domain.  After fixing the exponents,
choose \(C_8\) to dominate the finitely many upper-bound coefficients
and reciprocals of the lower-bound coefficients above.  Increase
\(Q_0\) so that all preceding estimates hold, and define
\[
\begin{aligned}
 \mathfrak D_{\rm seed}:=\bigl(&
 r_\dagger,B,\tau,\log\lambda;\,
 B_{10},B_{11},B_{13},a_{\rm seed},\eta_{\rm seed},\\
 &D_7,D_9,D_6,D_{11},
 D_8,M_*^{\rm seed},C_8\bigr).
\end{aligned}
\]
Thus the scale-\(Q\) data have \(2s\) simple zeros and satisfy
Definition~\ref{fr:def:Young-seed-package} with parameters
\(\mathfrak D_{\rm seed}\).

Since the number of energy intervals is the fixed number \(N\), all exponent and
threshold choices are simultaneous.  The minimum of the finitely many
common amplitude thresholds used above is
\(\varepsilon_{\rm seed}(R,q,r_*)\).  This proves the proposition.
\end{proof}
\section{Positive-measure selection and real--holomorphic propagation along
continued-fraction scales}
\label{sec:real-CF-propagation}

\noindent We propagate the initial-scale data along the continued-fraction
denominators, using summable energy exclusions to retain a set of
positive measure with \(2s\) simple return zeros at every scale.  At each
retained energy, we extend the real return data to complex neighborhoods
and verify the divisor estimates without further energy exclusions.
The spectral inclusion and
Theorem~\ref{thm:persistent-divisor-acceleration} then give
positive-measure spectral sets with acceleration \(s\), completing the
proof of Proposition~\ref{prop:perturbative-hierarchy-package}.

\subsection{Real selection and fiberwise holomorphic realization}

At fixed energy, we use the finite two-sided return-time partition of
Definition~\ref{fr:def:common-return-partition}, whose finiteness is
Lemma~\ref{fr:lem:finite-common-return-partition}, together with the
fixed-word extension convention preceding
\eqref{fr:eq:complete-return-joining-convention}, for \(T_\alpha\).

\begin{definition}
\label{fr:def:Young-2s-regular}
Let \(r_*\ge3\) and \(s\ge1\) be integers, let
\(V\in C^{r_*}(\mathbb T,\mathbb R)\), let
\(\alpha\notin\mathbb Q\) have continued-fraction denominators
\((q_j)_{j\ge1}\), and fix \(E\in\mathbb R\) and \(\gamma>0\).  Relative
to \(\gamma\),
a triple \((V,\alpha,E)\) is called \emph{first-return regular with
\(2s\) simple zeros through order \(r_*\)} if there exist an integer
\(j_0\ge1\), exponents
\[
 B_{10},D_{13},B_{11},M_*>0,
\]
a constant \(C_0>0\), and a growth-rate deficit
\(0<\Omega_{j_0}<\gamma/4\) for which the following conditions hold.

Set \(w_j=q_j^{-B_{10}}\).  For every \(j\ge j_0\), there exist a
radius \(\widehat w_j\) and centers
\(\xi_{j,a}(E)\in\mathbb T\), \(1\le a\le2s\), such that
\[
 w_{j_0}\le\widehat w_{j_0}\le2w_{j_0},\qquad
 w_j\le\widehat w_j\le
 w_j+\min\left\{w_j,\frac14(w_{j-1}-w_j)\right\}
 \quad(j>j_0),
\]
and the intervals
\[
 I_{j,a}(E):=[\xi_{j,a}(E)-\widehat w_j,
                  \xi_{j,a}(E)+\widehat w_j],
 \qquad
 K_{q_j}(E):=\bigcup_{a=1}^{2s}I_{j,a}(E)
\]
are pairwise disjoint.  Put
\[
 \operatorname{Reg}^{\rm ph}(E)
 :=\operatorname{Reg}\bigl(\{K_{q_j}(E):j\ge j_0\}\bigr)
 \subset\mathbb T.
\]
All return maps below are restricted to
\(K_{q_j}(E)\cap\operatorname{Reg}^{\rm ph}(E)\), and all matrix, singular-value,
and overlap-coefficient data on an atom denote their fixed-word extensions.

For every \(j\ge j_0\), the following conditions hold.

\begin{enumerate}
\item\label{item:regular-returns} The bases are nested,
\(K_{q_{j+1}}(E)\Subset K_{q_j}(E)\), and
\[
 r_{K_{q_j}(E)}^\pm(x)>q_{j+1},
 \qquad x\in K_{q_j}(E)\cap\operatorname{Reg}^{\rm ph}(E).
\]

\item\label{item:regular-expansion} On every two-sided return-time atom
\(\mathcal B\in\mathcal P_{K_{q_j}(E)}\), let
\(R_{\mathcal B}^+=p_{\mathcal B}\) and
\(R_{\mathcal B}^-=q_{\mathcal B}\).  The atomwise fixed-word blocks have
coherently oriented standard SVD data satisfying
\[
 \lambda_{K_{q_j}}^\pm(x,E)
 \ge
 \exp\!\left\{(\gamma-\Omega_{j_0})R_{\mathcal B}^\pm\right\},
 \qquad x\in\mathcal B.
\]
For \(1\le r\le r_*\),
\[
\begin{aligned}
 &|\partial_x^r\log\lambda_{K_{q_j}}^\pm|
 +|\partial_x^r u_{K_{q_j}}^\pm|
 +|\partial_x^r s_{K_{q_j}}^\pm|\\
 &\qquad +(\lambda_{K_{q_j}}^\pm)^{-1}
   \|\partial_x^r B_{K_{q_j}}^\pm\|
 \le C_0(R_{\mathcal B}^\pm+2)^{M_*},
 \qquad x\in\mathcal B.
\end{aligned}
\]

\item\label{item:regular-zeros} On each atom, use the first-return angle function
\(G_{K_{q_j},\mathcal B}\) from
\eqref{fr:eq:complete-return-joining-convention}.
There are pairwise distinct atoms
\(\mathcal B_{j,a}\in\mathcal P_{K_{q_j}(E)}\), \(1\le a\le2s\).
Recall the quotient map \(\pi_{\mathbb T}:\mathbb R\to\mathbb T\).  For
each \(a\), fix a connected component
\(\widetilde{\mathcal B}_{j,a}\) of
\(\pi_{\mathbb T}^{-1}(\mathcal B_{j,a})\).  There is a real number
\(r_{j,a}\in\widetilde{\mathcal B}_{j,a}\) such that
\[
 \pi_{\mathbb T}(r_{j,a})=\xi_{j+1,a}(E),
 \qquad 1\le a\le2s,
\]
after relabelling the components of \(K_{q_{j+1}}(E)\).  Thus every
real zero of the overlap coefficient is the center of exactly one child component.
In this lifted coordinate, the overlap coefficient has the exact factorization
\begin{equation}\label{fr:eq:Young-real-simple-factorization}
\begin{aligned}
 \cos G_{K_{q_j},\mathcal B_{j,a}}(\pi_{\mathbb T}x,E)
 &=(x-r_{j,a})\mathfrak a_{j,a}^{\mathbb R}(x,E),\\
 q_{j_0}^{-D_{13}}
 &\le|\mathfrak a_{j,a}^{\mathbb R}|
 \le q_{j_0}^{D_{13}},
 \qquad x\in\widetilde{\mathcal B}_{j,a}.
\end{aligned}
\end{equation}
The points \(\pi_{\mathbb T}(r_{j,a})\), \(1\le a\le2s\), are all the
real zeros of the atomwise overlap coefficients on \(K_{q_j}(E)\), and
each zero is simple.
Moreover,
\begin{equation}\label{fr:eq:Young-real-root-clearance}
 \min_{1\le a\le2s}
 \operatorname{dist}(r_{j,a},\partial\widetilde{\mathcal B}_{j,a})
 \ge q_j^{-B_{11}},
\end{equation}
and every atomwise overlap coefficient satisfies
\[
 \|\cos G_{K_{q_j},\mathcal B}(\,\cdot\,,E)\|_{C^{r_*}}
 \le q_{j_0}^{D_{13}}.
\]
Set
\[
 \mathcal P_{K_{q_j}}^{\rm dist}
 :=\{\mathcal B_{j,a}:1\le a\le2s\},
 \qquad
 \mathcal P_{K_{q_j}}^{\rm nd}
 :=\mathcal P_{K_{q_j}(E)}\setminus
    \mathcal P_{K_{q_j}}^{\rm dist}.
\]
For every \(\mathcal B\in\mathcal P_{K_{q_j}(E)}\),
\[
 |\cos G_{K_{q_j},\mathcal B}(x,E)|
 \ge q_{j+1}^{-D_{13}},
\]
for all \(x\in\mathcal B\) if
\(\mathcal B\in\mathcal P_{K_{q_j}}^{\rm nd}\), and for all
\(x\in\mathcal B\setminus\operatorname{int}K_{q_{j+1}}(E)\) if
\(\mathcal B\in\mathcal P_{K_{q_j}}^{\rm dist}\).
\end{enumerate}
\end{definition}
\begin{proposition}
\label{fr:prop:Young-regular-is-spectral}
Let \(r_*\ge3\) and \(s\ge1\) be integers, let
\(V\in C^{r_*}(\mathbb T,\mathbb R)\), let \(\alpha\notin\mathbb Q\), and
let \(E\in\mathbb R\).  If, for some \(\gamma>0\),
\((V,\alpha,E)\) is first-return regular with \(2s\) simple
zeros through order \(r_*\) in the sense of
Definition~\ref{fr:def:Young-2s-regular}, then
\[
 E\in\Sigma(V,\alpha),
\]
where \(\Sigma(V,\alpha)=\operatorname{spec}H_{x,V,\alpha}\) is the
common phase-independent spectrum.
\end{proposition}

\begin{proof}
Suppose for contradiction that \(E\notin\Sigma(V,\alpha)\).  By Johnson's theorem for
dynamically defined discrete Schr\"odinger operators~\cite{ZhangUH2020},
the cocycle \((\alpha,A_E^V)\) is uniformly hyperbolic.  Let
 \(E^u(x),E^s(x)\in\mathbb P^1(\mathbb R)\) be its invariant unstable
and stable lines, where \(\mathbb P^1(\mathbb R)\) is the real
projective line.  Uniform hyperbolicity and compactness give
positive constants \(C_0,\kappa_0,\delta_0\), depending on
\((V,\alpha,E)\), such that
\[
 \angle(E^u(x),E^s(x))\ge\delta_0
\]
and, for local unit representatives \(u_x\in E^u(x)\) and
\(s_x\in E^s(x)\),
\[
 \|A_n^V(x,E)u_x\|\ge C_0^{-1}e^{\kappa_0n},
 \qquad
 \|A_n^V(x,E)s_x\|\le C_0e^{-\kappa_0n}.
\]
In the estimates below,
\(C_{23}=C_{23}(V,\alpha,E)>0\) is an auxiliary constant.
Set \(Q_x=(u_x\;s_x)\).  The angle bound makes
\(Q_x\) and \(Q_x^{-1}\) uniformly bounded.  Invariance gives, for
suitable nonzero scalars \(a_n(x),b_n(x)\),
\[
 A_n^V(x,E)Q_x
 =Q_{x+n\alpha}
  \begin{pmatrix}a_n(x)&0\\0&b_n(x)\end{pmatrix},
 \qquad
 \left|\frac{b_n(x)}{a_n(x)}\right|
 \le C_{23}e^{-2\kappa_0n}.
\]
Thus
\[
 \frac{A_n^V(x,E)}{a_n(x)}
 =u_{x+n\alpha}e_1^{\mathsf T}Q_x^{-1}
 +\frac{b_n(x)}{a_n(x)}
  s_{x+n\alpha}e_2^{\mathsf T}Q_x^{-1}.
\]
The first term is rank one with its nonzero singular value bounded above
and below uniformly.  For a matrix with a simple largest singular value, denote by
\(e^{u,{\rm out}}\) and \(e^{u,{\rm in}}\) its expanding output and input
singular lines.  Applying \eqref{svd:eq:explicit-projector-formula} to the
right and left Gram matrices gives
\begin{align*}
 \operatorname{dist}_{\mathbb P^1(\mathbb R)}
 \bigl(e^{u,{\rm out}}_{A_n^V(x,E)},E^u(x+n\alpha)\bigr)
 &\le C_{23}e^{-2\kappa_0n},\\
 \operatorname{dist}_{\mathbb P^1(\mathbb R)}
 \bigl(e^{u,{\rm in}}_{A_n^V(x,E)},E^s(x)^\perp\bigr)
 &\le C_{23}e^{-2\kappa_0n}.
\end{align*}
These estimates are projective and hence independent of the local signs
of the representatives.  Since
\[
 |\langle u_x,s_x^\perp\rangle|
 =|\det(u_x,s_x)|\ge\sin\delta_0,
\]
there are \(n_0\) and \(c_7=c_7(V,\alpha,E)>0\) such that, for all \(p,q\ge n_0\),
\[
 \left|
 \left\langle e^{u,{\rm out}}_{A_q^V(x-q\alpha,E)},
                e^{u,{\rm in}}_{A_p^V(x,E)}\right\rangle
 \right|\ge c_7,
 \qquad x\in\mathbb T.
\]
Write \(s_p^+(x)\) for the contracting right-singular angle of
\(A_p^V(x,E)\) and \(u_q^-(x)\) for the expanding left-singular angle of
\(A_q^V(x-q\alpha,E)\), with arbitrary compatible local lifts.  For the
standard SVD convention this scalar product is precisely
\[
 \left|\cos\bigl(\tfrac\pi2-s_p^+(x)+u_q^-(x)\bigr)\right|.
\]

Choose \(a_j\in\{1,\ldots,2s\}\) and let \(p_j,q_j^-\) be the
forward and backward labels of \(\mathcal B_{j,a_j}\).
Its regular part is nonempty; \ref{fr:def:Young-2s-regular}\textup{(\ref{item:regular-returns})--(\ref{item:regular-zeros})} and the fixed-word
convention yield
\[
\begin{gathered}
 x_j:=\pi_{\mathbb T}(r_{j,a_j})
 \in\mathcal B_{j,a_j}\cap\operatorname{int}K_{q_{j+1}}(E),\\
 \min\{p_j,q_j^-\}>q_{j+1}\longrightarrow\infty,\\
 0=|\cos G_{K_{q_j},\mathcal B_{j,a_j}}(x_j,E)|\\
 =\left|\left\langle
 e^{u,{\rm out}}_{A_{q_j^-}^V(x_j-q_j^-\alpha,E)},
 e^{u,{\rm in}}_{A_{p_j}^V(x_j,E)}
 \right\rangle\right|\ge c_7>0
 \qquad(j\gg1),
\end{gathered}
\]
a contradiction.
\end{proof}

For the next lemma, let \(n\ge0\) be an integer and let
\(\mathcal J=\mathcal G_{-1}\supset\mathcal G_0\supset\cdots
\supset\mathcal G_n\), where \(\mathcal J\) is a compact interval and,
for \(0\le j\le n\),
\(\mathcal G_j\) is a nonempty finite union of compact intervals.  Each component
of \(\mathcal G_j\) is contained in a unique component of
\(\mathcal G_{j-1}\).  For \(-1\le j\le n\) and every component \(C\)
of \(\mathcal G_j\), let \(f_{j,C}:C\to\mathbb R\) be absolutely
continuous.  Define
\(f_j:\mathcal G_j\to\mathbb R\) by \(f_j|_C:=f_{j,C}\).
Fix \(\sigma\in\{\pm1\}\), \(\nu>0\), nonnegative numbers
\(\eta_{-1},\ldots,\eta_{n-1}\), and positive numbers
\(b_0,\ldots,b_n\).  Consider the bounds
\begin{align}
 \sigma\partial_E f_{j,C}&\ge\nu\quad\text{a.e. on }C,
 \qquad -1\le j\le n,
 \label{fr:eq:buffered-tree-slope}\\
 \|f_{j,C}-f_{j-1,P}\|_{L^\infty(C)}&\le\eta_{j-1}
 \quad(C\subset P),\qquad 0\le j\le n,
 \label{fr:eq:buffered-tree-shadow}\\
 4\sum_{t=j-1}^{n-1}\eta_t&\le\nu b_j,
 \qquad 0\le j\le n.
 \label{fr:eq:buffered-tree-margin}
\end{align}
For distinct components \(C_1,C_2\) of \(\mathcal G_j\) with the
same parent component in \(\mathcal G_{j-1}\), the separation condition is
\begin{equation}\label{fr:eq:buffered-tree-separation}
 \operatorname{dist}(C_1,C_2)\ge2b_j,\qquad 0\le j\le n.
\end{equation}

\begin{lemma}
\label{fr:lem:buffered-parameter-tree}
For the nested sets and absolutely continuous functions above, assume
\eqref{fr:eq:buffered-tree-slope}--\eqref{fr:eq:buffered-tree-margin}
and \eqref{fr:eq:buffered-tree-separation}.  Then \(f_n\)
has a continuous extension
\(\widetilde f_n:\operatorname{conv}\mathcal G_n\to\mathbb R\), equal
to \(f_n\) on \(\mathcal G_n\) and affine on every complementary gap,
with the following properties.
\begin{enumerate}
\item The extension satisfies
\begin{equation}\label{fr:eq:buffered-tree-global-slope}
 \sigma\bigl(\widetilde f_n(E_2)-\widetilde f_n(E_1)\bigr)
 \ge\frac\nu2(E_2-E_1),\qquad E_1<E_2.
\end{equation}
\item For every interval \(Z\subset\mathbb R\),
\((\widetilde f_n)^{-1}(Z)\) is an interval, whose intersection with
\(\mathcal G_n\) introduces at most two new cut points.
\item For every measurable \(Z\subset\mathbb R\),
\begin{equation}\label{fr:eq:buffered-tree-preimage}
 \bigl|\{E\in\mathcal G_n:f_n(E)\in Z\}\bigr|
 \le2\nu^{-1}|Z|.
\end{equation}
\end{enumerate}
\end{lemma}

\begin{proof}
Fix \(E_1<E_2\) in \(\mathcal G_n\).  If they lie in one level-\(n\)
component, integrate the a.e.
\eqref{fr:eq:buffered-tree-slope}.  Otherwise, let \(j\) be the first
level at which their descendants belong to different children of one
level-\((j-1)\) component \(P\).  The separation hypothesis gives
\(E_2-E_1\ge2b_j\).  By \eqref{fr:eq:buffered-tree-shadow},
\[
 |f_n(E_i)-f_{j-1,P}(E_i)|
 \le\sum_{t=j-1}^{n-1}\eta_t,
 \qquad i=1,2.
\]
Hence \eqref{fr:eq:buffered-tree-margin} gives
\[
 \sigma(f_n(E_2)-f_n(E_1))
 \ge\nu(E_2-E_1)-2\sum_{t=j-1}^{n-1}\eta_t
 \ge\frac\nu2(E_2-E_1).
\]
Set \(\widetilde f_n=f_n\) on \(\mathcal G_n\) and interpolate affinely
on each gap \((a,b)\).  Then
\[
\begin{gathered}
 \sigma\widetilde f_n'
 =\sigma\frac{f_n(b)-f_n(a)}{b-a}\ge\frac\nu2
 \quad\text{on }(a,b),\qquad
 \sigma\widetilde f_n'\ge\nu
 \quad\text{a.e. on }\mathcal G_n,\\
 \sigma\bigl(\widetilde f_n(E_2)-\widetilde f_n(E_1)\bigr)
 \ge\frac\nu2(E_2-E_1),\qquad
 \operatorname{Lip}(\widetilde f_n^{-1})\le2\nu^{-1}.
\end{gathered}
\]
Monotonicity gives the interval and cut-point claims.  More precisely,
\[
 \left|\widetilde f_n^{-1}(z_2)
       -\widetilde f_n^{-1}(z_1)\right|
 \le2\nu^{-1}|z_2-z_1|,
 \qquad z_1,z_2\in\widetilde f_n(\operatorname{conv}\mathcal G_n).
\]
The one-dimensional Lipschitz measure estimate therefore gives
\[
 \bigl|\{E\in\mathcal G_n:f_n(E)\in Z\}\bigr|
 =\bigl|\mathcal G_n\cap\widetilde f_n^{-1}(Z)\bigr|
 \le2\nu^{-1}|Z|.
\]
\end{proof}

We next consider stability of a simple-zero factorization.  Let
\(r\ge3\) be an integer, let \(\mathcal J\subset\mathbb R\) be an interval, and
let \(\mathscr A\subset\mathbb R\times\mathcal J\) have interval
fibers in the \(x\)-variable.  Suppose that \(h,h^\sharp\) are \(C^r\)
on a neighborhood of \(\mathscr A\),
\(c,c^\sharp\in C^{r-1}(\mathcal J)\), and
\[
 \{(c(E),E):E\in\mathcal J\}
 \cup\{(c^\sharp(E),E):E\in\mathcal J\}\subset\mathscr A.
\]
For \(E\in\mathcal J\), put
\(\mathscr A_E:=\{x\in\mathbb R:(x,E)\in\mathscr A\}\).
Assume further that
\(h(c(E),E)=h^\sharp(c^\sharp(E),E)=0\), and define the quotient in the
simple-zero factorization
\[
 \mathfrak a(x,E):=\int_0^1\partial_xh
 \bigl(c(E)+t(x-c(E)),E\bigr)\,dt.
\]
Put
\[
 \eta_1:=\|\partial_x(h^\sharp-h)\|_{L^\infty(\mathscr A)},\qquad
 d_0:=\|c^\sharp-c\|_{L^\infty(\mathcal J)},\qquad
 H_2:=\|\partial_x^2h\|_{L^\infty(\mathscr A)}.
\]
Set
\[
 \Lambda_r:=1+\|h\|_{C^r(\mathscr A)}
 +\|h^\sharp\|_{C^r(\mathscr A)}
 +\|c\|_{C^{r-1}(\mathcal J)}
 +\|c^\sharp\|_{C^{r-1}(\mathcal J)}.
\]
For \(0<\mu\le A<\infty\), consider the bounds
\begin{equation}\label{fr:eq:stable-transverse-hypotheses}
 \mu\le|\mathfrak a|\le A\quad\text{on }\mathscr A,
 \qquad \eta_1+H_2d_0\le\frac\mu2.
\end{equation}

\begin{lemma}
\label{fr:lem:transverse-factor-stability}
In the preceding setting, assume
\eqref{fr:eq:stable-transverse-hypotheses}.  There is \(C_{21}=C_{21}(r)>0\) such that the following hold.
\begin{enumerate}
\item The factorizations are
\begin{equation}\label{fr:eq:stable-transverse-factor}
\begin{aligned}
 h(x,E)&=(x-c(E))\mathfrak a(x,E),\\
 h^\sharp(x,E)&=(x-c^\sharp(E))\mathfrak a^\sharp(x,E),\\
 \mathfrak a^\sharp(x,E)
 &=\int_0^1\partial_xh^\sharp
 \bigl(c^\sharp(E)+t(x-c^\sharp(E)),E\bigr)\,dt.
\end{aligned}
\end{equation}
\item The factors satisfy
\begin{equation}\label{fr:eq:stable-transverse-C0}
 \|\mathfrak a^\sharp-\mathfrak a\|_{L^\infty(\mathscr A)}
 \le\eta_1+\frac12H_2d_0,
 \qquad \frac\mu2\le|\mathfrak a^\sharp|\le2A.
\end{equation}
For every \(E\in\mathcal J\), \(c^\sharp(E)\) is the only zero of
\(h^\sharp(\,\cdot\,,E)\) in \(\mathscr A_E\).
\item Their derivatives satisfy
\begin{align*}
 \|\mathfrak a^\sharp-\mathfrak a\|_{C^{r-2}(\mathscr A)}
 &\le C_{21}\Lambda_r^r
 \bigl(\|h^\sharp-h\|_{C^{r-1}(\mathscr A)}
       +\|c^\sharp-c\|_{C^{r-2}(\mathcal J)}\bigr),\\
 \|\mathfrak a^\sharp\|_{C^{r-1}(\mathscr A)}
 &\le C_{21}\|h^\sharp\|_{C^r(\mathscr A)}
 (1+\|c^\sharp\|_{C^{r-1}(\mathcal J)})^{r-1}.
\end{align*}
\end{enumerate}
\end{lemma}

\begin{proof}
Set \(p_t=c(E)+t(x-c(E))\) and
\(p_t^\sharp=c^\sharp(E)+t(x-c^\sharp(E))\).  The interval-fiber
hypothesis gives
\[
 p_t,p_t^\sharp\in\mathscr A_E,\qquad
 |p_t^\sharp-p_t|\le(1-t)d_0.
\]
The fundamental theorem of calculus gives
\eqref{fr:eq:stable-transverse-factor} and
\[
\begin{aligned}
 |\mathfrak a^\sharp-\mathfrak a|
 &\le\int_0^1\left(
 |\partial_x(h^\sharp-h)(p_t^\sharp,E)|
 +|\partial_xh(p_t^\sharp,E)-\partial_xh(p_t,E)|\right)dt\\
 &\le\eta_1+H_2d_0\int_0^1(1-t)\,dt
 =\eta_1+\tfrac12H_2d_0.
\end{aligned}
\]
This proves \eqref{fr:eq:stable-transverse-C0}.  In particular,
\[
 h^\sharp(x,E)=0
 \quad\Longleftrightarrow\quad x=c^\sharp(E).
\]
Differentiation under the integral sign and the multivariable
Fa\`a di Bruno formula give, for \(|\beta|\le r-2\),
\[
 \left|\partial^\beta(\mathfrak a^\sharp-\mathfrak a)\right|
 \le C_{21}\Lambda_r^r
 \left(\|h^\sharp-h\|_{C^{r-1}}
       +\|c^\sharp-c\|_{C^{r-2}}\right),
\]
and, for \(|\beta|\le r-1\),
\[
 |\partial^\beta\mathfrak a^\sharp|
 \le C_{21}\|h^\sharp\|_{C^r}
       (1+\|c^\sharp\|_{C^{r-1}})^{r-1}.
\]
\end{proof}

Fix an integer \(r\ge3\), numbers \(A_0,D_0>0\), and
\(\gamma_0>0\).  We describe a class of chains depending on increasing
exponents \(M_1,\ldots,M_r>1\); these exponents will be chosen in the
next lemma before the chain and its bounds \(\mathbf B_t\) are given.
Let \(\Omega\subset\mathbb R^2\) be a connected open parameter domain
with coordinates \((x,E)\),
let \(L\in\mathbb N\), \(q\ge1\), let
\(R_0,\ldots,R_{L-1}\ge1\), and let
\(\mathbf B_1,\ldots,\mathbf B_r\ge0\).
Let \(\mathsf A_0,\ldots,\mathsf A_{L-1}:\Omega\to\mathrm{SL}(2,\mathbb R)\) be a
\(C^r\) chain with coherently oriented singular value decompositions
\[
 \mathsf A_i=R_{u_i}\operatorname{diag}(\lambda_i,\lambda_i^{-1})
 R_{\pi/2-s_i},\qquad \lambda_i>1.
\]
Put
\[
 \mathcal M:=\mathsf A_{L-1}\cdots\mathsf A_0,\qquad
 h_i:=\cos(\pi/2-s_i+u_{i-1})\quad(1\le i<L).
\]
Throughout \(\Omega\), the expansion and overlap bounds are
\begin{equation}\label{fr:eq:fixed-order-chain-hypotheses}
\begin{gathered}
 q\le R_i\le q^{A_0},\quad
 \lambda_i\ge e^{\gamma_0R_i}\quad(0\le i<L),\quad L\le q^{A_0},\\
 |h_i|\ge q^{-D_0},\quad
 \|h_i\|_{C^r}\le q^{D_0}\quad(1\le i<L).
\end{gathered}
\end{equation}
For \(0\le\mathcal C_t\le\mathbf B_t\), \(1\le t\le r\), consider the
normalized derivative bounds on \(\Omega\), valid for \(0\le i<L\) and
\(\beta\in\mathbb N_0^2\), \(1\le|\beta|\le r\),
\begin{equation}\label{fr:eq:fixed-order-old-bound}
 |\partial^\beta\log\lambda_i|+|\partial^\beta u_i|
 +|\partial^\beta s_i|
 \le\mathcal C_{|\beta|}R_i^{M_{|\beta|}}.
\end{equation}
Set
\[
 R:=\sum_{0\le i<L}R_i,\qquad
 \mathfrak B_t:=1+\sum_{v=1}^t\mathbf B_v.
\]

\begin{lemma}
\label{fr:lem:level-independent-fixed-order-propagation}
There are increasing exponents \(M_1,\ldots,M_r>1\) and constants
\[
 c_6=c_6(r,A_0,D_0,\gamma_0)>0,\qquad
 C_{22}=C_{22}(r,A_0,D_0,\gamma_0)>0,
\]
and \(\Gamma=\Gamma(r,A_0,D_0,\gamma_0)>0\), with the following property.  For every choice of bounds
\(\mathbf B_1,\ldots,\mathbf B_r\ge0\), there is a threshold
\[
 q_*=q_*(r,A_0,D_0,\gamma_0,
          \mathbf B_1,\ldots,\mathbf B_r)
\]
such that every chain above satisfying
\eqref{fr:eq:fixed-order-chain-hypotheses},
\eqref{fr:eq:fixed-order-old-bound}, and \(q\ge q_*\) has expanding
left and contracting right endpoint singular directions with coherent angle lifts
\(U,S:\Omega\to\mathbb R\), uniquely normalized by
\(U-u_{L-1},S-s_0\in(-\pi/2,\pi/2)\) on \(\Omega\).
For these lifts, the following estimates hold.
\begin{enumerate}
\item For every \(\beta\in\mathbb N_0^2\) with
\(1\le|\beta|\le r\),
\begin{equation}\label{fr:eq:fixed-order-new-bound}
 |\partial^\beta\log\|\mathcal M\||+|\partial^\beta U|
 +|\partial^\beta S|
 \le\bigl(\mathcal C_{|\beta|}
 +C_{22}\mathfrak B_{|\beta|}^{\Gamma}q^{-2}\bigr)
 R^{M_{|\beta|}}.
\end{equation}
\item For every \(\beta\in\mathbb N_0^2\) with
\(1\le|\beta|\le r\),
\begin{equation}
 \|\partial^\beta\mathcal M\|
 \le C_{22}\mathfrak B_r^\Gamma(R+2)^{M_{|\beta|}}\|\mathcal M\|.
 \label{fr:eq:fixed-order-new-matrix}
\end{equation}
\item The endpoint directions satisfy
\begin{equation}
 \|U-u_{L-1}\|_{C^r}+\|S-s_0\|_{C^r}
 \le C_{22}\mathfrak B_r^\Gamma e^{-c_6q}.
 \label{fr:eq:fixed-order-endpoint-drift}
\end{equation}
\end{enumerate}
\end{lemma}

\begin{proof}
Here \(C=C(r,A_0,D_0,\gamma_0)\) denotes an auxiliary constant,
which may change between estimates.  The constants \(C_{22}\) and
\(c_6\) are chosen, respectively, sufficiently large and small to
cover the finitely many output bounds below, before \(q_*\) is fixed.
For \(1\le t\le r\), the quotient rule gives an exponent
\(K_t=K(t,D_0)>0\) such that every overlap coefficient in the statement
satisfies \(\|\log|h_i|\|_{C^t}\le C_{22}q^{K_t}\).
Choose the increasing exponents recursively so that
\[
 M_t>K_t+A_0+3,\qquad
 M_t>tM_{t-1}+A_0+3\quad(t\ge2).
\]
The case \(L=1\) now follows directly from the assumed SVD bounds, after
enlarging \(C_{22}\).  Hence assume \(L\ge2\).
Set \(g_i:=\frac\pi2-s_i+u_{i-1}\), \(H_i:=h_i\), and
\(\Theta_i:=\sin g_i\) for \(1\le i<L\), and set
\(\epsilon_i:=\lambda_i^{-2}\) for \(0\le i<L\).
Work componentwise on the parameter domains with fixed orientations, using
\(\log\epsilon_i=-2\log\lambda_i\) and \(\log|H_i|\); the sign of
\(H_i\) is constant on each such domain.
Removing the exterior rotations writes \(\mathcal M\) in the form of
Lemma~\ref{B:lem:differentiated-finite-two-state-product}, with scalar
\(\mathscr L_0=\prod_{0\le i<L}\lambda_i\).  For each \(t\le r\), the
bounds for its input derivatives through order \(t\) are at most a fixed
power of \(q\) and of
\(\mathfrak B_t\), while
\[
 \max_{0\le i<L}|\epsilon_i|\le e^{-2\gamma_0q},\qquad
 \min_{1\le i<L}|H_i|\ge q^{-D_0},\qquad L\le q^{A_0}.
\]
Consequently, for all sufficiently large \(q\), the multiplier and
endpoint-frame estimates and the affine highest-order conclusions of
that lemma and Corollary~\ref{B:cor:finite-two-state-frames} give a remainder \(\mathcal E\) such that
\begin{equation}\label{fr:eq:fixed-order-log-product}
 \log\|\mathcal M\|
 =\sum_{0\le i<L}\log\lambda_i+
   \sum_{i=1}^{L-1}\log|h_i|+\mathcal E,
 \qquad
 \|\mathcal E\|_{C^r}
 \le C_{22}\mathfrak B_r^\Gamma e^{-c_6q}.
\end{equation}
After conjugation by the two exterior rotations, the smooth map from
endpoint spectral projections to their angle coordinates converts the projection
bounds into angle bounds.  Thus, for \(|\beta|=t\le r\),
\[
 \|\partial^\beta\mathcal E\|_\infty
 +\|\partial^\beta(U-u_{L-1})\|_\infty
 +\|\partial^\beta(S-s_0)\|_\infty
 \le C_{22}\mathfrak B_t^\Gamma e^{-c_6q}.
\]
The endpoint clause gives the same estimate for \(|\beta|=0\).
The endpoint estimate immediately gives
\eqref{fr:eq:fixed-order-endpoint-drift}.  If \(|\beta|=t\), then
\[
 \sum_{0\le i<L}R_i^{M_t}\le R^{M_t},\qquad
 \sum_{i=1}^{L-1}\|\log|h_i|\|_{C^t}
 \le CLq^{K_t}\le C_{22}q^{-2}R^{M_t}.
\]
Combining \eqref{fr:eq:fixed-order-old-bound} with
\eqref{fr:eq:fixed-order-log-product} and the two endpoint corrections
therefore yields \eqref{fr:eq:fixed-order-new-bound}.  The coefficient of
\(\mathcal C_t\) remains one because
\[
\begin{aligned}
 &\sum_{0\le i<L}|\partial^\beta\log\lambda_i|
   +|\partial^\beta u_{L-1}|+|\partial^\beta s_0|\\
 &\quad\le\sum_{0\le i<L}
 (|\partial^\beta\log\lambda_i|+|\partial^\beta u_i|
 +|\partial^\beta s_i|)
 \le\mathcal C_t\sum_{0\le i<L}R_i^{M_t}.
\end{aligned}
\]
Finally, differentiate the standard singular value decomposition of
\(\mathcal M\).  The order-\(t\) terms are affine in the order-\(t\)
derivatives of \(\log\|\mathcal M\|,U,S\), and all remaining products
contain only lower-order derivatives.  The recursive choice of \(M_t\)
and \(\mathcal C_v\le\mathbf B_v\) absorb these products and give
\eqref{fr:eq:fixed-order-new-matrix}.
\end{proof}

For the continued-fraction selection, fix integers \(s\ge1\) and \(r_*\ge3\), a frequency
\(\alpha\in\DC(\kappa,\tau)\) with \(\kappa>0\) and \(\tau>1\), and a
reference rate \(\gamma>0\).  Put \(r_\dagger:=r_*+2\), let \((q_j)\)
be the continued-fraction denominators of \(\alpha\), and fix
\(c_{11},B_{12},D_{13}>0\) and initial-scale parameters
\(\mathfrak D_{\rm seed}\) whose first, third, and fourth entries are
\(r_\dagger,\tau,\gamma\), respectively.  Require its entries and
\(D_{13}\) to satisfy
\begin{equation}\label{fr:eq:CF-propagation-exponent-budget}
\begin{aligned}
 B_{10}&>
 \max\{\tau^2+2,\ B_{12}+D_6+\tau+2,\
          B_{13}+2,\ D_8+D_7+4\},\\
 B_{11}&>
 B_{10}(\tau+1)+B_{12}+D_6+4,\\
 D_{13}&>
 (r_\dagger+3)\bigl(D_8+D_7
                 +B_{10}+B_{11}+4\bigr).
\end{aligned}
\end{equation}
Set
\[
 \mathfrak D_{\rm CF}:=(
 s,r_*,\alpha,\kappa,
 c_{11},B_{12},\mathfrak D_{\rm seed},D_{13}).
\]
For a denominator \(Q=q_\ell\), let \(V_Q:\mathbb T\to\mathbb R\) be
real analytic, and fix any \(\varrho_Q>0\) such that
\(V_Q\in C^\omega_{\varrho_Q}(\mathbb T,\mathbb R)\); no lower bound on
\(\varrho_Q\) is assumed.  Let \(\mathcal J_Q\) be a nondegenerate compact
interval satisfying
\begin{equation}\label{fr:eq:CF-selection-window}
 |\mathcal J_Q|\ge c_{11}Q^{-B_{12}}.
\end{equation}
Assume that the scale-\(Q\) data of \(V_Q\) on \(\mathcal J_Q\) are
admissible initial-scale data with \(2s\) simple zeros and parameters
\(\mathfrak D_{\rm seed}\).

\begin{proposition}
\label{fr:prop:Young-CF-reinduction-exclusion}
There are
\[
 Q_*=Q_*(\mathfrak D_{\rm CF})\ge1,\qquad
 C_{11}=C_{11}(\mathfrak D_{\rm CF})\ge1,
\]
such that, for every admissible
\(Q\ge Q_*\), \(V_Q\), \(\varrho_Q\), and \(\mathcal J_Q\) above,
there are an energy-independent radius
\[
 \widehat w_\ell\in[Q^{-B_{10}},2Q^{-B_{10}}]
\]
and a compact set \(\mathcal Y_Q^{(s)}\subset\mathcal J_Q\) with
the following properties.
\begin{enumerate}
\item The measure estimates are
\begin{align*}
 |\mathcal J_Q\setminus\mathcal Y_Q^{(s)}|
 &\le C_{11}\left(
 Q^{D_6+\tau-B_{10}}
 +Q^{D_6+B_{10}\tau-B_{11}}
 \right),\\
 |\mathcal Y_Q^{(s)}|
 &\ge\frac12|\mathcal J_Q|
 \ge\frac{c_{11}}2Q^{-B_{12}}.
\end{align*}
\item For every \(E\in\mathcal Y_Q^{(s)}\), there are bases
\(K_{q_j}(E)\), \(j\ge\ell\), with
\[
 K_{q_\ell}(E)=K_{q_\ell}^{(\widehat w_\ell)}(E),
\]
for which \((V_Q,\alpha,E)\) is first-return regular with \(2s\)
simple zeros through order \(r_*\), relative to \(\gamma\), with
\(j_0=\ell\).
\end{enumerate}
The exponents \(B_{10},B_{11},D_{13}\)
are those prescribed above, and the constants \(M_*\) and \(C_0\) in
\ref{fr:def:Young-2s-regular}\textup{(\ref{item:regular-returns})--(\ref{item:regular-zeros})} depend only on \(\mathfrak D_{\rm CF}\).
\end{proposition}

The same construction gives the following bound on the loss of the
expansion rate.  Let \(\omega_Q^{\rm seed}\) be the initial deficit
in \ref{fr:def:Young-seed-package}\textup{(\ref{item:seed-factorization})}.

\begin{corollary}
\label{fr:cor:Young-CF-rate-loss}
After increasing \(C_{11}\) once, there is
\(C_{12}=C_{12}(\kappa,\tau)>0\) such that the bases constructed in
Proposition~\ref{fr:prop:Young-CF-reinduction-exclusion} have growth-rate
deficit
\begin{equation}\label{fr:eq:Young-rate-loss}
 \Omega_Q
 :=\omega_Q^{\rm seed}
 +C_{11}\sum_{j\ge\ell}
       \frac{\log(q_{j+1}+2)}{q_j}
 \le\omega_Q^{\rm seed}
   +C_{12}C_{11}\frac{\log(Q+2)}Q.
\end{equation}
After increasing the structural threshold \(Q_*\), if necessary,
\[
 \Omega_Q=o_{Q\to\infty}(1),\qquad \Omega_Q<\frac\gamma4,
\]
uniformly over all admissible choices of \(V_Q\).
\end{corollary}

We next realize these real return data on complex neighborhoods.  For
the next two results, fix \(E\in\mathcal Y_Q^{(s)}\) and use the
bases of Proposition~\ref{fr:prop:Young-CF-reinduction-exclusion} with
the deficit in Corollary~\ref{fr:cor:Young-CF-rate-loss}.  Put
\[
\begin{aligned}
 K_j&:=K_{q_j}(E),&\qquad S_j&:=S_{K_j},\\
 R_{j,\mathcal B}&:=r_{K_j}^+|_{\mathcal B},&
 m_j:=\underline R_j&:=\min_{\mathcal B\in\mathcal P_{K_j}}
 R_{j,\mathcal B},\\
 \chi&:=g_Q:=\tfrac12(\gamma-\Omega_Q)>0.
\end{aligned}
\]
Then \ref{fr:def:Young-2s-regular}\textup{(\ref{item:regular-returns})} gives \(m_j>q_{j+1}\to\infty\).  We use
\(\operatorname{DivAdm}(P;h,d,\kappa)\) for the admissibility class
in Definition~\ref{def:moving-divisors-persistent-total}.

\begin{proposition}
\label{fr:prop:Young-CF-holomorphic-realization}
For these fixed data, there are heights \(0<4h_j<\varrho_Q\), numbers
\[
 D_j\ge4,\qquad \mathscr A_j\ge2,\qquad
 d_j>0,\qquad 0<\kappa_j\le1,
\]
and level-\(j\) first-return zero-divisor data with the following
properties.
\begin{enumerate}
\item The data are admissible:
\begin{equation}\label{fr:eq:CF-holomorphic-admissibility}
 \mathscr D_j\in
 \operatorname{DivAdm}(2s;h_j,d_j,\kappa_j),\qquad j\ge\ell.
\end{equation}
Their distinguished roots are precisely those in \ref{fr:def:Young-2s-regular}\textup{(\ref{item:regular-zeros})}:
\[
 \nu_j=2s,\qquad P_{j,a}=1,\qquad
 c_{j,a,1}=r_{j,a},\qquad 1\le a\le2s.
\]
There are no other atomwise real zeros, and \(P_{\rm loc}=2s\).
\item The normal-form data satisfy
\begin{align}
 &\sup_{|y|\le h_j}\operatorname*{ess\,sup}_{x\in K_j}
 \bigl(\|U_j(x,y)\|+\|U_j(x,y)^{-1}\|
 +\|V_j(x,y)\|+\|V_j(x,y)^{-1}\|\bigr)
 \le D_j,
 \label{fr:eq:CF-holomorphic-frame-bound}\\
 &\operatorname*{ess\,inf}_{x\in K_j}|\Lambda_j(x,y)|
 \ge e^{\chi m_j},\qquad |y|\le h_j,
 \label{fr:eq:CF-holomorphic-expansion}
\end{align}
together with the three estimates in
\eqref{eq:moving-pair-logarithmic-bounds} for \(h_j\) and
\(\mathscr A_j\).
\item The constants satisfy
\begin{equation}\label{fr:eq:CF-holomorphic-subexponential-data}
 \frac{\log D_j+\log\mathscr A_j+\log^+\kappa_j^{-1}
 +\log^+d_j^{-1}+\log^+h_j^{-1}}{m_j}
 \longrightarrow0.
\end{equation}
\end{enumerate}
This construction applies to every \(E\in\mathcal Y_Q^{(s)}\), with
no further energy exclusion.  The initial complex height, the sequence
\((h_j)\), and the rate of convergence may depend on
\((Q,V_Q,\varrho_Q,E,s)\).
\end{proposition}

For the data just constructed, the preceding bounds also hold at
exponentially small imaginary shifts.

\begin{corollary}
\label{fr:cor:Young-CF-acceleration-data}
The holomorphic data have the following properties.
\begin{enumerate}
\item For every fixed \(\beta>0\), eventually
\(e^{-\beta m_j}\le h_j\), and
\begin{equation}\label{fr:eq:CF-holomorphic-vertical-variation}
 \sup_{0\le y\le e^{-\beta m_j}}
 \int_{K_j}
 \left|\partial_y\log|\Lambda_j(x+iy,E)|\right|\,dx
 =o_{j\to\infty}(1).
\end{equation}
\item With
\begin{equation}\label{fr:eq:CF-holomorphic-beta-choice}
 \beta_+=\frac{\chi}{16s},\qquad
 \beta_-=\frac{\chi}{8s},
\end{equation}
there is
\(j_{\rm hol}=j_{\rm hol}(Q,V_Q,\varrho_Q,E,s;\beta_+,\beta_-)\ge\ell\)
such that the data \((\mathscr D_j)_{j\ge j_{\rm hol}}\) satisfy
all hypotheses of Theorem~\ref{thm:persistent-divisor-acceleration}.
\end{enumerate}
The threshold and the little-\(o\) term for a fixed \(\beta\) may
depend on \((Q,V_Q,\varrho_Q,E,s,\beta)\).
\end{corollary}

We now fix the structural data for holomorphic reinduction: an integer \(r_*\ge3\), numbers
\(a_L,a_R,D>0\), \(C_{16}\ge1\), \(C_{17}>16\), and \(g>0\).  Let
\(B_{18}:=C_2(r_*)\) abbreviate the order-\(r_*\) exponent in
Lemma~\ref{B:lem:differentiated-finite-two-state-product}, and use
its smallness constant \(c_1(r_*)\).  This choice also applies to
Corollary~\ref{B:cor:finite-two-state-frames}.

The constants \(d_1,d_2,C_{13}\)
in the conditions below will be chosen in
Lemma~\ref{fr:lem:graded-holomorphic-product-inputs}, using only
these structural data and before any of the following input data.
Let \(q\ge2\), let \(I\subset\mathbb R\) be an interval, and
put \(\Omega^+:=I+i(-4h^+,4h^+)\), where \(h^+>0\).  Suppose that on a
neighborhood of \(\overline{\Omega^+}\) there is an exact chain
\begin{equation}\label{fr:eq:graded-holomorphic-chain-hypotheses}
\begin{aligned}
 B^+&=B_{L-1}\cdots B_0,
 &\qquad R&:=\sum_{t=0}^{L-1}R_t,\\
 R_t&>q,
 & L&\le C_{16}(q+2)^{a_L},\\
 \overline R&:=\max_tR_t
 & &\le C_{16}(q+2)^{a_R}.
\end{aligned}
\end{equation}
Assume that every child has a real-symmetric holomorphic normal form
\begin{equation}\label{fr:eq:graded-holomorphic-child-normal-form}
\begin{aligned}
 B_t&=U_t^+
 \begin{pmatrix}\Lambda_t&0\\0&\Lambda_t^{-1}\end{pmatrix}
 V_t^{\mathsf T},\\
 U_t^+,V_t&\in{\rm SO}(2,\mathbb C),\\
 |\Lambda_t|&\ge e^{\Gamma R_t},\qquad \Gamma\ge g,
\end{aligned}
\end{equation}
where \(\Lambda_t\) is nonvanishing and the sum of the norms of
\(U_t^+,V_t\) and their inverses is at most \(C_{17}\).  Put
\[
 \Pi_t^{\rm in}:=V_te_1e_1^{\mathsf T}V_t^{\mathsf T},
 \qquad
 \Pi_t^{\rm out}:=U_t^+e_1e_1^{\mathsf T}(U_t^+)^{\mathsf T}.
\]
For numbers \(\mathcal K_r\ge2\), assume
\begin{equation}\label{fr:eq:graded-holomorphic-child-jets}
\begin{aligned}
 \max_t\sup_{z\in\Omega^+}\Bigl(&
 \|\partial_z^r\Pi_t^{\rm in}\|
 +\|\partial_z^r\Pi_t^{\rm out}\|\\
 &+(R_t+2)^{-1}|\partial_z^r\log\Lambda_t|\Bigr)
 \le\mathcal K_r,\qquad1\le r\le r_*.
\end{aligned}
\end{equation}
For \(1\le t<L\), write
\[
 W_t:=V_t^{\mathsf T}U_{t-1}^+
 =\begin{pmatrix}H_t&-\Theta_t\\
                  \Theta_t&H_t\end{pmatrix}
 \in{\rm SO}(2,\mathbb C)
\]
and suppose that \(|H_t(x)|\ge q^{-D}\) for \(x\in I\).  The
strip-width condition is
\begin{equation}\label{fr:eq:graded-holomorphic-strip-safety-input}
 8C_{13}h^+q^D(1+\mathcal K_1)^{d_2}\le1.
\end{equation}
Let \(A_1<\cdots<A_{r_*}\), \(J_1\ge1\), and
\(\mathcal S\ge1\) satisfy
\begin{equation}\label{fr:eq:graded-holomorphic-input-ledger}
 \mathcal K_r\le e^{J_1+A_r\mathcal S},\qquad1\le r\le r_*,
\end{equation}
and the explicit product budget
\begin{equation}\label{fr:eq:graded-holomorphic-product-budget}
 d_2(J_1+A_{r_*}\mathcal S)
 +C_{13}\log(q+2)\le\Gamma q.
\end{equation}

Put
\[
 \epsilon^{\rm prod}:=e^{-2\Gamma q},\qquad
 \kappa^{\rm prod}:=\tfrac12q^{-D},\qquad
 K_0^{\rm prod}:=1+C_{17}^2,
\]
and, for \(1\le r\le r_*\),
\begin{equation}\label{fr:eq:graded-holomorphic-product-jets}
 K_r^{\rm prod}:=C_{13}(\overline R+2)
 \left(1+\sum_{\nu=1}^r\mathcal K_\nu\right)^{d_1}.
\end{equation}
Also put
\[
 \mathcal L^+:=\prod_{t=0}^{L-1}\Lambda_t
                 \prod_{t=1}^{L-1}H_t.
\]
The next three lemmas use a common collection of integers
\(d_1,d_2\ge1\), with
\begin{equation}\label{fr:eq:graded-holomorphic-degree-hierarchy}
 d_2\ge B_{18}d_1,
\end{equation}
and constants
\[
 C_{13}\ge1,\qquad c_9>0,\qquad q_{\rm hol}\ge2,
\]
and polynomials \(\mathcal P_r\),
\(1\le r\le r_*\), satisfying
\[
 \mathcal P_r\in\mathbb R_{\ge0}[X_1,\ldots,X_{r-1}],\qquad
 \deg\mathcal P_r\le d_2,
\]
where \(\mathcal P_1\) is constant.  The integer \(d_2\) bounds
the degrees after composition in
the product smallness, Riesz projection, frame, multiplier, and endpoint
conjugation estimates; \(d_1\) bounds the degrees in axis
selection and connector differentiation.  All these objects depend only on the displayed structural data; in
particular,
\[
 C_{13}=C_{13}(r_*,a_L,a_R,D,C_{16},C_{17},g).
\]
A common upper bound is used for the coefficients in the input and
output estimates and for the structural exponent in the product budget.

\begin{lemma}
\label{fr:lem:graded-holomorphic-product-inputs}
The common collection can be chosen so that this lemma and
Lemmas~\ref{fr:lem:graded-holomorphic-reinduction}
and~\ref{fr:lem:graded-holomorphic-jet-propagation} hold.
For every chain in the preceding setting satisfying
\eqref{fr:eq:graded-holomorphic-chain-hypotheses}--\eqref{fr:eq:graded-holomorphic-product-budget}
and \(q\ge q_{\rm hol}\), the following estimates hold.
\begin{enumerate}
\item The overlap coefficients satisfy
\begin{equation}\label{fr:eq:graded-holomorphic-connector-output}
 |H_t(z)|\ge\frac12q^{-D},\qquad
 z\in\Omega^+,\qquad1\le t<L.
\end{equation}
\item The normalized-product smallness bound holds with the margin
\begin{equation}\label{fr:eq:graded-holomorphic-product-smallness}
\begin{aligned}
 &\epsilon^{\rm prod}L^{B_{18}}
 (\kappa^{\rm prod})^{-B_{18}}
 \left(2+K_0^{\rm prod}+\sum_{r=1}^{r_*}K_r^{\rm prod}\right)^{B_{18}}\\
 &\qquad\le e^{-\Gamma q}\le e^{-gq}\le c_1(r_*).
\end{aligned}
\end{equation}
\end{enumerate}
The connector assertions are vacuous when \(L=1\).
\end{lemma}

Under the same hypotheses, the product admits an exact normal form.
The next two lemmas use the same frames and multiplier.

\begin{lemma}
\label{fr:lem:graded-holomorphic-reinduction}
There are real-symmetric holomorphic
\(U^+,V^+\in{\rm SO}(2,\mathbb C)\) and a nonvanishing \(\Lambda^+\)
with the following properties.
\begin{enumerate}
\item The product has the normal form
\begin{equation}\label{fr:eq:graded-holomorphic-new-normal-form}
 B^+=U^+
 \begin{pmatrix}\Lambda^+&0\\0&(\Lambda^+)^{-1}\end{pmatrix}
 (V^+)^{\mathsf T}.
\end{equation}
\item For the leading multiplier \(\mathcal L^+\),
\begin{equation}\label{fr:eq:graded-holomorphic-multiplier-error}
 \left\|\log\frac{\Lambda^+}{\mathcal L^+}\right\|_{C^{r_*}(\Omega^+)}
 \le e^{-c_9q},
\end{equation}
where the logarithm is the branch normalized by the product lemma.
\item The frames satisfy
\begin{equation}\label{fr:eq:graded-holomorphic-frame-output}
 \|U^+-U_{L-1}^+\|_{C^{r_*}(\Omega^+)}
 +\|V^+-V_0\|_{C^{r_*}(\Omega^+)}
 \le e^{-c_9q},
\end{equation}
and the same bound holds for the inverse frames.

\item With
\begin{equation}\label{fr:eq:graded-holomorphic-rate-output}
 \Gamma^+:=\Gamma-(D+2)\frac{\log(q+2)}q,
\end{equation}
the new multiplier satisfies
\begin{equation}\label{fr:eq:graded-holomorphic-expansion-output}
 |\Lambda^+(z)|\ge e^{\Gamma^+R},\qquad z\in\Omega^+.
\end{equation}
\end{enumerate}
\end{lemma}

The derivative estimates retain affine dependence on the highest-order
child jet.  For the same normal form, define
\(\Pi^{+,\rm in}\), \(\Pi^{+,\rm out}\), and \(\mathcal K_r^+\)
from \((U^+,V^+,\Lambda^+,R)\) as in \eqref{fr:eq:graded-holomorphic-child-jets}.

\begin{lemma}
\label{fr:lem:graded-holomorphic-jet-propagation}
For the common collection and the normal form of
Lemma~\ref{fr:lem:graded-holomorphic-reinduction},
\begin{equation}\label{fr:eq:graded-holomorphic-jet-output}
 \mathcal K_r^+\le(q+2)^{C_{13}}
 \left(\mathcal K_r+
 \mathcal P_r(\mathcal K_1,\ldots,\mathcal K_{r-1})\right),
 \qquad1\le r\le r_*.
\end{equation}
\end{lemma}

\begin{proof}[Proof of Lemmas~\ref{fr:lem:graded-holomorphic-product-inputs},
\ref{fr:lem:graded-holomorphic-reinduction}, and
\ref{fr:lem:graded-holomorphic-jet-propagation}]
In this proof,
\[
 C=C(r_*,a_L,a_R,D,C_{16},C_{17},g)
\]
denotes auxiliary constants in the raw derivative estimates, before
\(C_{13}\) is chosen.  These constants, the degree bounds, and the
coefficient polynomials obtained from the product and frame estimates
are independent of \(C_{13}\).  After taking a common upper bound
for this finite collection of constants, choose \(C_{13}\ge2C\);
the threshold \(q_{\rm hol}\) is chosen afterwards.

Write \(v_t\) and \(u_t\) for the normalized oriented first columns of
\(V_t\) and \(U_t^+\), respectively.  Differentiating
\(\Pi v=v\) and \(v^{\mathsf T}v=1\) gives
\(\partial_zv=(\partial_z\Pi)v\); iteration bounds every derivative of
an oriented axis by a universal polynomial in the projection derivatives.
By the choice of \(d_1\), this gives
\[
 \|\partial_zH_t\|_\infty
 \le C(1+\mathcal K_1)^{d_1}
 \le C_{13}(1+\mathcal K_1)^{d_2}.
\]
For \(z=x+iy\in\Omega^+\),
\eqref{fr:eq:graded-holomorphic-strip-safety-input} gives
\[
 \begin{aligned}
 |H_t(z)|
 &\ge |H_t(x)|-\left|\int_0^y iH_t'(x+iv)\,dv\right|\\
 &\ge q^{-D}-4h^+C_{13}(1+\mathcal K_1)^{d_2}
 \ge\tfrac12q^{-D}.
 \end{aligned}
\]

Put \(\epsilon_t:=\Lambda_t^{-2}\) and
\(D_t:=\operatorname{diag}(1,\epsilon_t)\).  The expansion in
\eqref{fr:eq:graded-holomorphic-child-normal-form} gives
\(|\epsilon_t|\le e^{-2\Gamma q}\).  The preceding axis estimates and
\(\partial_z^r\log\epsilon_t=-2\partial_z^r\log\Lambda_t\) show that
\eqref{fr:eq:graded-holomorphic-product-jets} controls the derivatives
of \(H_t,\Theta_t\), and \(\log\epsilon_t\) through order \(r_*\).
More precisely, the higher-order recursion obtained from
\(\partial_zv=(\partial_z\Pi)v\) gives polynomials \(\mathcal Q_r\),
involving only orders below \(r\), such that
\begin{equation}\label{fr:eq:graded-holomorphic-affine-input-jets}
\begin{aligned}
 &\max_t\Bigl(
 \|\partial_z^rH_t\|_\infty
 +\|\partial_z^r\Theta_t\|_\infty
 +\|\partial_z^r\log\epsilon_t\|_\infty\Bigr)\\
 &\qquad\le C(\overline R+2)
 \left(\mathcal K_r+
 \mathcal Q_r(\mathcal K_1,\ldots,\mathcal K_{r-1})\right).
\end{aligned}
\end{equation}
Thus the dependence on the order-\(r\) child jet is affine.  The
product and frame output estimates are applied with these raw bounds;
\eqref{fr:eq:graded-holomorphic-product-jets} is used only to verify
smallness.
By \eqref{fr:eq:graded-holomorphic-input-ledger},
\[
 1+\sum_{\nu=1}^{r_*}\mathcal K_\nu
 \le (r_*+1)\exp\{J_1+A_{r_*}\mathcal S\},
\]
direct substitution of the bounds for \(L,\overline R\), and
\((\kappa^{\rm prod})^{-1}=2q^D\) gives
\begin{align*}
 \log\!\left[\text{left-hand side of
 \eqref{fr:eq:graded-holomorphic-product-smallness}}\right]
 &\le-2\Gamma q+d_2(J_1+A_{r_*}\mathcal S)\\
 &\quad+C\log(q+2)+C\log(2+C_{13})\\
 &\le-2\Gamma q+d_2(J_1+A_{r_*}\mathcal S)
       +C_{13}\log(q+2),
\end{align*}
where the last inequality follows from \(C_{13}\ge2C\) after
requiring \(q_{\rm hol}\ge C_{13}+2\).
Here \eqref{fr:eq:graded-holomorphic-degree-hierarchy} absorbs the factor
\(B_{18}\) multiplying the raw exponent \(d_1\).
The degree \(d_2\) also dominates the degrees of all composed
multiplier, projector, and frame output polynomials.  Their raw
coefficients have already been included in the choice of \(C\).
By \eqref{fr:eq:graded-holomorphic-product-budget}, enlarge
\(q_{\rm hol}\) so that
\[
 \text{left-hand side of
 \eqref{fr:eq:graded-holomorphic-product-smallness}}
 \le e^{-\Gamma q}\le e^{-gq}\le c_1(r_*).
\]

Remove the two exterior frames from the exact chain.  The remaining
normalized middle product is
\[
 D_{L-1}W_{L-1}\cdots D_1W_1D_0,
\]
with the one-factor convention when \(L=1\).  With
\(\mathscr L_0:=\prod_{t=0}^{L-1}\Lambda_t\), one has
\[
 \mathscr L_0^2\prod_{t=0}^{L-1}\epsilon_t\equiv1.
\]
Thus Lemma~\ref{B:lem:differentiated-finite-two-state-product} and the
determinant-one case of Corollary~\ref{B:cor:finite-two-state-frames}
apply.  Denote the canonical middle output by
\((U_M,V_M,\Lambda_M)\), and define
\[
 U^+:=U_{L-1}^+U_M,\qquad
 V^+:=V_0V_M,\qquad
 \Lambda^+:=\Lambda_M.
\]
The leading scalar is \(\mathcal L^+\), and the specialization gives
\eqref{fr:eq:graded-holomorphic-new-normal-form} and
\eqref{fr:eq:graded-holomorphic-multiplier-error}.  Its endpoint-frame
estimates, followed by Leibniz' rule for
\(U_{L-1}^+U_M\) and \(V_0V_M\), give
\eqref{fr:eq:graded-holomorphic-frame-output}; the polynomial factors are
absorbed by the exponential margin above.  The affine highest-order
clauses of the lemma and corollary, combined with
\eqref{fr:eq:graded-holomorphic-affine-input-jets}, imply that an
order-\(r\) child jet occurs linearly, while all remaining terms contain
only derivatives of orders less than \(r\).  Since
\[
 \frac1R\sum_{t=0}^{L-1}(R_t+2)\le1+\frac2q,
 \qquad \frac{L-1}{R}<\frac1q,
\]
the multiplier identity and the two endpoint conjugations first give
\[
 \mathcal K_r^+\le(q+2)^C
 \left(\mathcal K_r+
 \mathcal P_r(\mathcal K_1,\ldots,\mathcal K_{r-1})\right).
\]
Since \(C\le C_{13}\), this proves
\eqref{fr:eq:graded-holomorphic-jet-output}.  Finally, with
\(\Gamma^+\) from \eqref{fr:eq:graded-holomorphic-rate-output},
\eqref{fr:eq:graded-holomorphic-connector-output} and
\eqref{fr:eq:graded-holomorphic-multiplier-error} yield
\[
 \log|\Lambda^+|
 \ge\Gamma R-(L-1)(D\log q+\log2)-e^{-c_9q}
 \ge\Gamma^+R,
\]
after enlarging \(q_{\rm hol}\).
\end{proof}

\begin{proof}[Proof of
Propositions~\ref{fr:prop:Young-CF-reinduction-exclusion}
and~\ref{fr:prop:Young-CF-holomorphic-realization}, and
Corollaries~\ref{fr:cor:Young-CF-rate-loss}
and~\ref{fr:cor:Young-CF-acceleration-data}]
Steps~(i)--(vii) are uniform on the retained real parameter cells.
In Steps~(viii)--(x), \(E\in\mathcal Y_Q^{(s)}\) is fixed, the same return
words are complexified, and the heights may depend on \(E\).
For the auxiliary estimates, write
\[
 c_{20}=c_{20}(\mathfrak D_{\rm CF})>0,\qquad
 C_{20}=C_{20}(\mathfrak D_{\rm CF})\ge1.
\]
Their values may change between auxiliary estimates.  Each finite
collection of bounds is fixed before increasing \(Q_*\).  The
constants selected for the conclusions and the displayed induction
budgets, including \(C_{11}\), \(C_{13}\), and \(c_5\), remain fixed.

\smallskip
\noindent\emph{Step (i).} We choose the constants and initial base.
Increase \(Q_*\) once so that \(Q=q_\ell\ge Q_*\) implies \(\ell\ge2\).
Fix such a \(Q=q_\ell\) and write
\[
 w_j:=q_j^{-B_{10}},\qquad
 D_{12}:=D_7+1.
\]
Let \(F_0=0\), \(F_1=1\), and
\(F_{n+1}=F_n+F_{n-1}\) for \(n\ge1\).  Then
\begin{equation}\label{fr:eq:CF-Fibonacci-denominator-growth}
 q_{\ell+n}\ge F_{n+1}Q,\qquad n\ge0.
\end{equation}
Fix
\[
 A_0>\max\{B_{13}\tau,B_{10}\tau\}+2,
 \qquad
 D_0>\max\{D_{13},D_7+B_{10},D_8\}+2.
\]
Invoke
Lemma~\ref{fr:lem:level-independent-fixed-order-propagation} on the
interior of each connected lifted fixed-itinerary phase--energy domain,
at order \(r_\dagger\) and with \(\gamma_0=3\gamma/4\); enlarge its
exponents, if necessary, so that \(M_t\ge M_*^{\rm seed}\).
Denote its constants at \((r_\dagger,A_0,D_0,3\gamma/4)\) by
\(c_{14},C_{18},D_{23}\), and set, for
\(1\le t\le r_\dagger\),
\[
\begin{aligned}
 C_t^{\rm seed}&:=3^{M_t}C_8,
 &\qquad \mathbf B_t&:=2\max\{1,C_t^{\rm seed}\},\\
 \mathfrak B_t&:=1+\sum_{v=1}^t\mathbf B_v,
 & M_*&:=M_{r_*}.
\end{aligned}
\]
Choose
\(c_5=c_5(\mathfrak D_{\rm CF})>0\) so that, for all sufficiently
large \(Q\) and \(q\),
\[
\begin{aligned}
 C_8Q^{D_9}e^{-C_8^{-1}Q^{a_{\rm seed}}}
 &\le e^{-4c_5\gamma Q},\\
 C_{18}\mathfrak B_{r_\dagger}^{D_{23}}e^{-c_{14}q}
 &\le e^{-4c_5\gamma q}.
\end{aligned}
\]
All structural constants below depend only on \(\mathfrak D_{\rm CF}\).

\smallskip
We first bound the polynomial losses in the zero-persistence argument.
Let \(d_{\rm IFT}=d_{\rm IFT}(r_\dagger)\) and
\(d_{\rm fac}=d_{\rm fac}(r_\dagger)\) dominate, respectively, the
degrees of the universal Fa\`a di Bruno polynomials obtained by
differentiating the quantitative implicit-function formula and the
factorization formulas in
Lemma~\ref{fr:lem:transverse-factor-stability}.  Define
\begin{align*}
 D_{17}
 &:=1+r_\dagger A_0M_{r_\dagger}+D_0,\\
 D_{18}
 &:=D_{17}
   +d_{\rm IFT}(D_{12}+D_8+1),\\
 D_{19}
 &:=1+\max\{D_{17},D_{18}\}
   +d_{\rm fac}(D_8+D_{18}+1),\\
 D_{20}&:=D_8+D_{18},
 \qquad
 D_*:=1+\max\{D_{17},D_{18},D_{19},D_{20}\}.
\end{align*}
For each overlap \(h\) and root \(c(E)\) being compared, and each
later scale \(q\ge Q\),
the endpoint product rule, implicit-function formula, and
Lemma~\ref{fr:lem:transverse-factor-stability} use
\[
 |\partial_xh(c(E),E)|^{-1}\le4Q^{D_{12}},\qquad
 \|h\|_{C^{r_\dagger}}\le Q^{D_8}.
\]
The resulting losses are
\[
\begin{array}{c|c}
 \text{comparison}&\text{polynomial loss}\\ \hline
 \text{endpoint directions and overlap}&q^{D_{17}}\\
 \text{root displacement}&q^{D_{18}}\\
 \text{nonvanishing factor}&q^{D_{19}}\\
 H_2d_0& q^{D_{20}}
\end{array}
\]
Thus \(D_*\) depends only on \(\mathfrak D_{\rm CF}\) and the already
fixed \(\mathbf B_t\).

Put
\[
\begin{aligned}
 \mathfrak e_{\rm base}&:=C_*Q^{D_*}e^{-2c_5\gamma Q},\\
 \mathfrak e_j&:=C_*q_{j+1}^{D_*}
 e^{-2c_5\gamma q_{j+1}},\qquad j\ge\ell,
\end{aligned}
\]
where the structural constant \(C_*\ge1\) dominates the sum of the
constants in the four comparisons just listed and all fixed constants
in the spatial quotient comparison.  In particular, the sum
of their \(C^{r_\dagger}\), \(C^{r_\dagger-1}\), or \(C^{r_*}\)
errors, as appropriate, is bounded by \(\mathfrak e_j\).

The continued-fraction recurrence
and the Diophantine denominator bound
give a structural constant
\[
 c_{10}=c_{10}(\alpha,\kappa,\tau,B_{10})\in(0,1]
\]
such that
\[
 w_{j-1}-w_j
 \ge c_{10}q_{j-1}^{-B_{10}-1+1/\tau},
 \qquad j\ge\ell+1.
\]
Choose
\(0<c_{12}<\min\{1/4,3c_{10}/16\}\), and set
\[
\begin{aligned}
 c_{13}&:=D_{13}2^{-D_{13}-1}\kappa^{1/\tau}>0,\\
 \Delta_{\rm sc}&:=\min\left\{
 \begin{array}{l}
 B_{10}-D_8-D_{12},\\
 D_{13}-D_{12}-B_{11},\\
 D_{13}-D_{12}-B_{10}
 \end{array}\right\}>0.
\end{aligned}
\]
The positivity follows directly from
\eqref{fr:eq:CF-propagation-exponent-budget} and
\(D_{12}=D_7+1\).  Increase \(Q_*\) once so that, for
every \(Q=q_\ell\ge Q_*\) and every \(j\ge\ell\),
\begin{equation}\label{fr:eq:CF-scale-threshold-ledger}
\begin{aligned}
\textup{(i)}\quad
 &\mathfrak e_{\rm base}\le\min\left\{
 \tfrac18Q^{-D_{12}-B_{11}},
 \frac{Q^{-D_{12}}}{8(1+Q^{D_8})},
 e^{-c_5\gamma Q},\tfrac18Q^{-D_7}\right\},\\
 &Q^{-D_7}-\mathfrak e_{\rm base}
 -\sum_{t\ge\ell}\mathfrak e_t
 -C_*Q^{D_8-B_{10}}
 \ge\tfrac14Q^{-D_{12}},\\
\textup{(ii)}\quad
 &\min\left\{\tfrac12Q^{-D_7},
 Q^{-D_7-B_{11}},
 \tfrac14Q^{-D_{12}}q_{\ell+1}^{-B_{10}}\right\}
 \ge q_{\ell+1}^{-D_{13}},\\
\textup{(iii)}\quad
 &\tfrac14Q^{-D_{12}}q_{j+2}^{-B_{10}}
 \ge2q_{j+2}^{-D_{13}},
 \qquad
 \tfrac18Q^{-D_{12}}q_{j+1}^{-B_{11}}
 \ge q_{j+2}^{-D_{13}},\\
\textup{(iv)}\quad
 &\mathfrak e_j\le\min\left\{
 \begin{aligned}
  &\tfrac18Q^{-D_{12}}q_{j+1}^{-B_{11}},
  &&\frac{Q^{-D_{12}}}{8(1+Q^{D_8})},\\
  &\tfrac12c_{13} q_{j+1}^{-D_{13}-1+1/\tau},
  &&e^{-c_5\gamma q_{j+1}}
 \end{aligned}\right\},\\
\textup{(v)}\quad
 &\sum_{t\ge\ell}\mathfrak e_t\le\tfrac18Q^{-D_{12}},
 \qquad
 C_*Q^{D_8}
 \sup_{t\ge\ell}q_t^{-B_{10}}
 \le\tfrac18Q^{-D_{12}},\\
\textup{(vi)}\quad
 &e^{-c_5\gamma q}
 \le\min\left\{q^{-B_{11}},
 \tfrac12c_{12}q^{-B_{10}-1+1/\tau}\right\}
 \qquad(q\ge Q).
\end{aligned}
\end{equation}
To verify the bounds, use
\(Q\le q_j\le q_{j+1}\le q_{j+2}\), the strict polynomial margins
measured by \(\Delta_{\rm sc}\), and
\(q_j\ge\kappa^{1/\tau}q_{j+1}^{1/\tau}\).  The last inequality and
\(1-(1+x)^{-D_{13}}\ge D_{13}2^{-D_{13}-1}x\) for
\(0\le x\le1\) give
\begin{equation}\label{fr:eq:CF-denominator-lower-bound-margin}
 q_{j+1}^{-D_{13}}-q_{j+2}^{-D_{13}}
 \ge c_{13} q_{j+1}^{-D_{13}-1+1/\tau}.
\end{equation}
The remaining verifications reduce to
\[
\begin{aligned}
 \Delta_{\rm sc}>0
 &\Longrightarrow \eqref{fr:eq:CF-scale-threshold-ledger}\textup{(ii)--(iii)},\ \eqref{fr:eq:CF-scale-threshold-ledger}\textup{(v)}_{\rm second},\\
 \sup_{q\ge Q}q^Me^{-c_{20}\gamma q}\longrightarrow0
 &\Longrightarrow \eqref{fr:eq:CF-scale-threshold-ledger}\textup{(i)}_{\rm first},\
                  \eqref{fr:eq:CF-scale-threshold-ledger}\textup{(iv)},\ \eqref{fr:eq:CF-scale-threshold-ledger}\textup{(vi)},\\
 \sum_{t\ge\ell}\mathfrak e_t\le C_{20}e^{-c_{20}\gamma Q}
 &\Longrightarrow\eqref{fr:eq:CF-scale-threshold-ledger}\textup{(v)}_{\rm first},\\
 Q^{-D_7}=Q\,Q^{-D_{12}}
 &\Longrightarrow\eqref{fr:eq:CF-scale-threshold-ledger}\textup{(i)}_{\rm second}.
\end{aligned}
\]
Here \(M\) is any fixed polynomial exponent, and the summation uses
\eqref{fr:eq:CF-Fibonacci-denominator-growth}.

\smallskip
To choose the initial radius, note that, for \(r\in[w_\ell,2w_\ell]\),
a return from a component of radius
\(r\) to itself occurs within \(Cw_\ell^{-\tau}\) iterates by
Lemma~\ref{B:lem:return-time-poly-two-sided}.  Define
\[
 \mathfrak R_\ell^{\rm rad}:=
 \bigcup_{\substack{0<|m|\le Cw_\ell^{-\tau}\\
                    \sigma\in\{\pm1\}}}
 \{r\in[w_\ell,2w_\ell]:
       \|m\alpha-\sigma r\|_{\mathbb T}<2Q^{-B_{11}}\}.
\]
After increasing \(Q_*\), the exponent assumptions give
\[
 |\mathfrak R_\ell^{\rm rad}|
 \le C_{20}Q^{B_{10}\tau-B_{11}}
 \le\frac12Q^{-B_{10}}.
\]
Choose
\(\widehat w_\ell\in[w_\ell,2w_\ell]\setminus
\mathfrak R_\ell^{\rm rad}\), independently of \(E\), and put
\(K_{q_\ell}(E)=K_{q_\ell}^{(\widehat w_\ell)}(E)\).

We construct nested finite unions of compact intervals
\[
 \mathcal G_{\ell-1}:=\mathcal J_Q\supset
 \mathcal G_\ell\supset\mathcal G_{\ell+1}\supset\cdots
\]
and, on their parameter cells, compatible real lifts of the auxiliary
simple-zero branches
\(\widetilde c_{j,a}\), \(1\le a\le2s\).  At the initial level these are
the fixed lifts in the \(\mathfrak D_{\rm seed}\)-admissible initial-scale data.  Throughout
the induction, put
\[
 c_{j,a}:=\pi_{\mathbb T}(\widetilde c_{j,a}),
 \qquad \xi_{j,a}:=c_{j,a}.
\]
Once \(\widetilde c_{j,a}\) and \(\widehat w_j\) are fixed, write
\[
 I_{j,a}(E):=\pi_{\mathbb T}
 [\widetilde c_{j,a}(E)-\widehat w_j,
  \widetilde c_{j,a}(E)+\widehat w_j],
 \qquad K_{q_j}(E):=\bigcup_{a=1}^{2s}I_{j,a}(E).
\]
For \(j>\ell\), the common radius \(\widehat w_j\) is selected
immediately after the lifts \(\widetilde c_{j,a}\) are constructed at the preceding
step and before the level-\(j\) return-itinerary partition is formed.

The finite induction at level \(j\) has the order
\[
\begin{aligned}
 (\mathcal G_{j-1},\widetilde c_{j,a},\widehat w_j,K_{q_j})
 &\longrightarrow \mathcal G_j
 \longrightarrow \text{level-}j\text{ return estimates}\\
 &\longrightarrow\widetilde c_{j+1,a}
 \longrightarrow
 \|\widetilde c_{j+1,a}-\widetilde c_{j,a}\|_{C^{r_\dagger-1}}\\
 &\longrightarrow\widehat w_{j+1}
 \longrightarrow K_{q_{j+1}} .
\end{aligned}
\]
The initial input is the initial-scale data; the first arrow removes short-return,
boundary-resonance, and buffer sets.

\smallskip
\noindent\emph{Step (ii).} We establish relative transversality and the
preimage estimate.
The inductive estimate to be established at level \(j\) is
\begin{equation}\label{fr:eq:internal-root-contraction}
 \|\widetilde c_{j+1,a}-\widetilde c_{j,a}\|_{C^{r_\dagger-1}}
 \le e^{-c_5\gamma q_j},\qquad j\ge\ell.
\end{equation}
For \(j>\ell\), continue the fixed initial lifts on each descendant cell
by the normalization
\(\|\widetilde c_{j,a}-\widetilde c_{j-1,a}\|_\infty<1/4\).
For \(j\ge\ell\) and \(a\ne b\), put
\(f_{j,a,b}:=\widetilde c_{j,a}-\widetilde c_{j,b}\).  Then
\[
\|\partial_E(f_{j,a,b}-f_{\ell,a,b})\|_\infty
 \le2\sum_{t=\ell}^{j-1}e^{-c_5\gamma q_t}
 \le2\sum_{t\ge\ell}e^{-c_5\gamma q_t}
 \le C_{20}e^{-c_5\gamma Q}.
\]
For large \(Q\), \ref{fr:def:Young-seed-package}\textup{(\ref{item:seed-velocities})} therefore implies
\begin{equation}\label{fr:eq:internal-relative-transversality}
 Q^{-D_6}
 \le\sigma_{ab}\partial_Ef_{j,a,b}
 \le2Q^{D_6}.
\end{equation}
Moreover, on every child component of \(\mathcal G_k\),
\begin{equation}\label{fr:eq:internal-relative-shadow}
 \|f_{k+1,a,b}-f_{k,a,b}\|_\infty
 \le2e^{-c_5\gamma q_k},\qquad k\ge\ell.
\end{equation}

Put \(b_j:=e^{-c_5\gamma q_{j-1}/8}\).
At level \(j\), delete the open \(b_j\)-neighborhood of every newly
created parameter boundary of the genuine level-\(j\) return partition
and of every endpoint of a newly deleted
resonance interval; at \(j=\ell\), also buffer the interior initial-cell
boundaries.  Denote the union of these buffer neighborhoods by
\(\mathcal B_j\).  Temporary refinements used to decompose a fixed
complete return word enter neither this deletion nor the cell count
below: fixed-word data and their jets glue across them, as shown in
Steps~(v) and~(vi).  Thus distinct level-\(j\) children in one
level-\((j-1)\) parent are separated by at least \(2b_j\).  Fibonacci
growth gives, for large \(Q\),
\(8\sum_{v\ge k}e^{-c_5\gamma q_v}
\le Q^{-D_6}b_k\) for \(k\ge\ell\).
For \(j=\ell\), the initial zero branches are global on \(\mathcal J_Q\); put
\(\widetilde f_{\ell,a,b}:=f_{\ell,a,b}\), and use \ref{fr:def:Young-seed-package}\textup{(\ref{item:seed-velocities})} for
the required preimage bound.
For \(j>\ell\), put \(n:=j-\ell-1\) and apply
Lemma~\ref{fr:lem:buffered-parameter-tree} with
\(\mathcal H_m:=\mathcal G_{\ell+m}\),
\(\widehat f_m:=f_{\ell+m+1,a,b}\)
\((-1\le m\le n)\),
\[
 \eta_{m-1}:=2e^{-c_5\gamma q_{\ell+m}},
 \qquad 0\le m\le n,
\]
and buffer \(b_{\ell+m}\) at level \(m\).  Equations
\eqref{fr:eq:internal-relative-transversality} and
\eqref{fr:eq:internal-relative-shadow}, the sibling separation, and the
tail bound verify its hypotheses.  Denote its affine-gap extension by
\(\widetilde f_{j,a,b}\).  Applied to each ordered pair before
the level-\(j\) deletion, it yields
\begin{equation}\label{fr:eq:internal-global-preimage}
 |\{E\in\mathcal G_{j-1}:f_{j,a,b}(E)\in Z\}|
 \le2Q^{D_6}|Z|
\end{equation}
for every measurable \(Z\subset\mathbb R\).  In particular, no factor
equal to the number of parameter cells occurs.
The oscillation clause in \ref{fr:def:Young-seed-package}\textup{(\ref{item:seed-velocities})} and
the sum of \eqref{fr:eq:internal-relative-shadow} also give
\[
 \operatorname{osc}f_{j,a,b}<\frac14.
\]
Thus every torus target below has at most two representatives in the
common real lift domain, with the same order of total length.

\smallskip
\noindent\emph{Step (iii).} We exclude short returns.
For \(a\ne b\), define
\[
 \mathcal R_{j,a,b}^{\rm ret}:=
 \bigcup_{|m|\le q_{j+1}}
 \{E\in\mathcal G_{j-1}:
   \|f_{j,a,b}(E)+m\alpha\|_{\mathbb T}<8w_j\}.
\]
Set
\[
 \mathcal R_j^{\rm ret}
 :=\bigcup_{a\ne b}\mathcal R_{j,a,b}^{\rm ret}.
\]
The term \(m=0\) only keeps the fixed neighborhoods of the auxiliary zeros disjoint; the
return-time conclusion uses \(m\ne0\).
The target union has length at most \(C_{20}q_{j+1}w_j\).  Thus
\eqref{fr:eq:internal-global-preimage} and
\(q_{j+1}\le\kappa^{-1}q_j^\tau\) give
\begin{equation}\label{fr:eq:internal-cross-return-loss}
 |\mathcal R_j^{\rm ret}|
 \le C_{11}Q^{D_6}q_j^{\tau-B_{10}}.
\end{equation}
For \(E\notin\mathcal R_j^{\rm ret}\),
\[
 \|c_{j,a}-c_{j,b}\|_{\mathbb T}\ge8w_j>2\widehat w_j
 \quad(a\ne b)
 \quad\Longrightarrow\quad I_{j,a}\cap I_{j,b}=\varnothing.
\]
A return \(x\in I_{j,a}\), \(x+m\alpha\in I_{j,b}\),
\(0<|m|\le q_{j+1}\), would give one of the contradictions
\[
\begin{array}{ll}
 a=b:&
 4q_j^{-B_{10}}
 \ge\|m\alpha\|_{\mathbb T}
 \ge\kappa q_{j+1}^{-\tau}
 \ge\kappa^{\tau+1}q_j^{-\tau^2}
 >4q_j^{-B_{10}},\\
 a\ne b:&
 8w_j\le\|f_{j,a,b}+m\alpha\|_{\mathbb T}\le4w_j.
\end{array}
\]
Here the last inequality in the first line holds for \(Q\ge Q_*\)
because \(B_{10}>\tau^2+2\).  Hence
\begin{equation}\label{fr:eq:internal-long-first-return}
 r_{K_{q_j}(E)}^\pm(x)>q_{j+1}
\end{equation}
for every regular \(x\in K_{q_j}(E)\).

\smallskip
\noindent\emph{Step (iv).} We choose the radii and obtain boundary
clearance on the retained parameter cells.
Before the level-\(j\) itinerary partition is formed, its radius is fixed.
For \(j>\ell\), it is selected in
\[
 [w_j,w_j+\delta_j],\qquad
 \delta_j:=\min\{w_j,\tfrac14(w_{j-1}-w_j)\}.
\]
The structural radius estimate recorded in Step~(i) and
\(c_{10}\le1\) give
\begin{equation*}
 \delta_j\ge\frac14c_{10}
 q_j^{-B_{10}-1+1/\tau}.
\end{equation*}
Same-component collisions with a return boundary exclude at most
\(C_{20}q_j^{B_{10}\tau}\) radius intervals of length
\(8q_j^{-B_{11}}\).
By increasing \(Q_*\), the exponent assumptions ensure, simultaneously
for every \(j\ge\ell+1\),
\[
 C_{20}q_j^{B_{10}\tau-B_{11}}\le\frac12\delta_j.
\]
Thus one common radius \(\widehat w_j\) remains; it is independent of
the parameter cell and component label.
Moreover,
\begin{equation}\label{fr:eq:internal-parent-child-radius-margin}
 \widehat w_{j-1}-\widehat w_j
 \ge\frac34(w_{j-1}-w_j)
 \ge\frac34c_{10}
 q_{j-1}^{-B_{10}-1+1/\tau}.
\end{equation}

For different components, define
\[
 \mathcal R_j^{\rm bd}:=
 \bigcup_{\substack{a\ne b,\ 0<|m|\le Cw_j^{-\tau}\\
                    \sigma\in\{\pm1\}}}
 \left\{E\in\mathcal G_{j-1}:
 \|f_{j,a,b}(E)+m\alpha-\sigma\widehat w_j\|_{\mathbb T}
 <2q_j^{-B_{11}}\right\}.
\]
By \eqref{fr:eq:internal-global-preimage},
\begin{equation}\label{fr:eq:internal-boundary-loss}
 |\mathcal R_j^{\rm bd}|
 \le C_{11}Q^{D_6}q_j^{B_{10}\tau-B_{11}}.
\end{equation}
Outside this set, every auxiliary root is at distance at least
\(2q_j^{-B_{11}}\) from the relevant return-partition boundary.
Indeed, a boundary of the two-sided return atom containing
\(c_{j,a}(E)\) can occur only if, for some
\(0<|m|\le Cw_j^{-\tau}\),
\[
 \widetilde c_{j,a}(E)+m\alpha
 \equiv \widetilde c_{j,b}(E)+\sigma\widehat w_j\pmod1,
 \qquad \sigma\in\{\pm1\}.
\]
For \(b=a\), this is a same-component collision excluded by the choice
of \(\widehat w_j\).  For \(b\ne a\), a boundary within
\(2q_j^{-B_{11}}\) is equivalent to
\[
 \|f_{j,a,b}(E)+m\alpha-\sigma\widehat w_j\|_{\mathbb T}
 <2q_j^{-B_{11}},
\]
and is excluded by \(\mathcal R_j^{\rm bd}\).  These exhaust the
possibilities because every return time is at most \(Cw_j^{-\tau}\).

Let \(\mathcal N_j\) be the number of parameter-cell boundaries after
the level-\(j\) exclusions and buffers; the initial-scale construction gives
\(\mathcal N_{\ell-1}\le C_{20}Q^{D_{11}}\).
On an old cell, a return itinerary can change only at a solution of
\[
\begin{aligned}
 \widetilde c_{j,a}(E)+\sigma\widehat w_j+m\alpha
 &\equiv
 \widetilde c_{j,b}(E)+\sigma'\widehat w_j\pmod1,\\
 |m|&\le Cw_j^{-\tau},
 &\sigma,\sigma'&\in\{-1,0,1\}.
\end{aligned}
\]
For \(a=b\) the equation is independent of \(E\), while for \(a\ne b\)
\eqref{fr:eq:internal-relative-transversality} makes its left-hand side
strictly monotone.  Hence the itinerary refinement contributes at most
\(C_{11}q_j^{B_{10}\tau}\) cuts per old cell.  For each connected real
representative of a short-return or boundary-resonance target, its inverse
under \(\widetilde f_{j,a,b}\) is one interval.  Intersecting it with
\(\mathcal G_{j-1}\) creates at most its two endpoint cuts; all intervening
gap endpoints are already old cell boundaries.  Thus the new cuts number,
respectively, at most \(C_{11}q_{j+1}\) and
\(C_{11}q_j^{B_{10}\tau}\).  Therefore
\begin{align*}
 \mathcal N_j+1
 &\le(\mathcal N_{j-1}+1)(1+C_{11}q_j^{B_{10}\tau})
      +C_{11}q_{j+1}+C_{11}q_j^{B_{10}\tau}\\
 &\le\exp\!\left\{C_{11}\sum_{t=\ell}^j\log(q_t+2)\right\}.
\end{align*}
Consequently the total level-\(j\) buffer \(\mathcal B_j\) satisfies
\begin{equation}\label{fr:eq:internal-buffer-loss}
 |\mathcal B_j|
 \le C_{20}\exp\!\left\{
 C_{11}\sum_{t=\ell}^j\log(q_t+2)
 -c_5\gamma q_{j-1}/8\right\},
 \qquad
 \sum_{j\ge\ell}|\mathcal B_j|
 \le C_{20}e^{-c_{20}\gamma Q^{1/\tau}}.
\end{equation}
Here the Diophantine bound and Fibonacci growth give
\[
 \log q_j\le C_{20}\log(q_{j-1}+2),\qquad
 \sum_{t=\ell}^j\log(q_t+2)
 \le C_{20}\bigl(\log(q_{j-1}+2)\bigr)^2,
\]
so the negative exponential in \(q_{j-1}\) dominates the cell count;
also \(q_{\ell-1}\ge c_{20}Q^{1/\tau}\).
Define \(\mathcal G_j\) by removing
\(\mathcal R_j^{\rm ret}\cup\mathcal R_j^{\rm bd}\cup\mathcal B_j\)
from \(\mathcal G_{j-1}\) and retaining the resulting closed components.
At \(j=\ell\), the separation \(2Q^{-B_{11}}\)
makes every connected component of \(\mathcal G_\ell\) a
boundary-admissible subcell of the initial parameter cell containing it, in
the sense of \ref{fr:def:Young-seed-package}\textup{(\ref{item:seed-endpoint-error})}.

\smallskip
\noindent\emph{Step (v).} We propagate expansion and the fixed-order
singular data.
We prove \eqref{fr:eq:internal-root-contraction} and
\ref{fr:def:Young-2s-regular}\textup{(\ref{item:regular-expansion})--(\ref{item:regular-zeros})} by simultaneous induction with the normalized
singular-data and overlap bounds through order \(r_\dagger=r_*+2\).
The recurrences for \(\mathcal C_{t,j}\) and the overlap coefficients
below close these additional bounds.  By \ref{fr:def:Young-seed-package}\textup{(\ref{item:seed-factorization})--(\ref{item:seed-exterior})}, the
initial overlap coefficient for the first return to the smaller set
differs from its initial-scale overlap coefficient by at most
\(C_{20}Q^{D_9}e^{-c_{20}Q^{a_{\rm seed}}}\) in
\(C^{r_\dagger}\).  After the endpoint, implicit-function,
transverse-factor, and evaluation losses are included, the base error is
at most \(\mathfrak e_{\rm base}\).  The implicit function theorem and
Lemma~\ref{fr:lem:transverse-factor-stability} give one simple real root
in each distinguished initial return-time atom, no root in any other atom, and
\[
 \|\widetilde c_{\ell+1,a}-\widetilde c_{\ell,a}\|_{C^{r_\dagger-1}}
 \le\mathfrak e_{\rm base}
 \le e^{-c_5\gamma Q},
\]
where the last inequality is
\eqref{fr:eq:CF-scale-threshold-ledger}\textup{(i)}.
The level-\(\ell\) buffers separate the retained descendants of distinct
initial parameter cells.  On each retained component,
\(\widetilde c_{\ell+1,a}\) is the unique lift within distance \(1/4\)
of the fixed initial lift.
At the initial zero the nonvanishing factor is at least
\(Q^{-D_7}\).  The second line of
\eqref{fr:eq:CF-scale-threshold-ledger}\textup{(i)} reserves, after
the base error, all later interlevel errors, and spatial evaluation,
the lower margin \(\frac14Q^{-D_{12}}\).  Hence the initial tail
estimate and these bounds give, on every distinguished atom,
\begin{equation}\label{fr:eq:internal-quotient-bootstrap}
\begin{aligned}
 \cos G_{K_{q_j},\mathcal B_{j,a}}(\pi_{\mathbb T}x,E)
 &=(x-\widetilde c_{j+1,a}(E))
   \mathfrak a_{j,a}^{\mathbb R}(x,E),\\
 \frac14Q^{-D_{12}}
 &\le|\mathfrak a_{j,a}^{\mathbb R}|
 \le Q^{D_{13}},
\end{aligned}
\end{equation}
and \(\|\mathfrak a_{j,a}^{\mathbb R}\|_{C^{r_*}}\le Q^{D_{13}}\).

Apply the radius-selection step with index \(\ell+1\) to the newly
constructed lifts \(\widetilde c_{\ell+1,a}\).  This gives a common radius
\(\widehat w_{\ell+1}\) and defines \(K_{q_{\ell+1}}(E)\) before any
level-\((\ell+1)\) return partition is used.  We also record the uniform
nonvanishing bound at the initial level.  On an atom for the first return
to the smaller set whose
central initial return-time atom is not among \(\nu_1,\ldots,\nu_{2s}\),
\ref{fr:def:Young-seed-package}\textup{(\ref{item:seed-exterior})} and the endpoint-error estimate give
\[
 |\cos G_{K_{q_\ell}}|\ge \tfrac12Q^{-D_7}.
\]
If the central initial return-time atom is \(\nu_a\), but this new atom is not
the distinguished atom containing \(c_{\ell+1,a}\), the level-\(\ell\)
boundary separation and \ref{fr:def:Young-seed-package}\textup{(\ref{item:seed-transversality})--(\ref{item:seed-endpoint-error})} give
\[
 |\cos G_{K_{q_\ell}}|
 \ge Q^{-D_7-B_{11}}.
\]
Finally, on the distinguished atom but outside
\(K_{q_{\ell+1}}(E)\), \eqref{fr:eq:internal-quotient-bootstrap} gives
\[
 |\cos G_{K_{q_\ell}}|
 \ge \tfrac14Q^{-D_{12}}\widehat w_{\ell+1}.
\]
Since \(\widehat w_{\ell+1}\ge q_{\ell+1}^{-B_{10}}\),
\eqref{fr:eq:CF-scale-threshold-ledger}\textup{(ii)} makes all three
right-hand sides at least \(q_{\ell+1}^{-D_{13}}\).  Thus the
nonvanishing clause of \ref{fr:def:Young-2s-regular}\textup{(\ref{item:regular-zeros})} holds at level \(\ell\).

Assume these assertions through level \(j\).  We first record the exact
return chain, which will be used by both the real and holomorphic estimates.
Write \(S_j:=S_{K_{q_j}}\).  On one fixed-itinerary phase--energy cell
\(\mathcal C\) for the first return to \(K_{q_{j+1}}\), put
\[
\begin{aligned}
 L&:=\min\{n\ge1:S_j^nx\in K_{q_{j+1}}\},
 &x_t&:=S_j^tx,\\
 n_t&:=\sum_{v=0}^{t-1}r_{K_{q_j}}^+(x_v),
 &n_0&:=0,
\end{aligned}
\]
and let \(\mathcal B_t\in\mathcal P_{K_{q_j}}\) be the atom containing
\(x_t\).  Write
\[
 R_{j,\mathcal B}:=r_{K_{q_j}}^+|_{\mathcal B},\qquad
 B_{j,\mathcal B}(z,E):=
 A_{R_{j,\mathcal B}}^{V_Q}(z,E),
\]
and denote by \(R_{j+1,\mathcal C}\) the fixed return length on
\(\mathcal C\), and put
\[
 B_{j+1,\mathcal C}(z,E)
 :=A_{R_{j+1,\mathcal C}}^{V_Q}(z,E).
\]
Then
the cocycle identity and fixed-word continuation give
\begin{equation}\label{fr:eq:CF-exact-reinduction}
 B_{j+1,\mathcal C}(z,E)
 =B_{j,\mathcal B_{L-1}}(z+n_{L-1}\alpha,E)
  \cdots B_{j,\mathcal B_0}(z,E).
\end{equation}
Because this is the first \(S_j\)-return to \(K_{q_{j+1}}\),
\begin{equation}\label{fr:eq:CF-internal-node-exclusion}
 x_t\in K_{q_j}\setminus\operatorname{int}K_{q_{j+1}},
 \qquad 1\le t<L.
\end{equation}
Lemma~\ref{B:lem:return-time-poly-two-sided}, applied to the component
of \(K_{q_{j+1}}\) containing the starting point, gives
\[
 R_{j+1,\mathcal C}
 \le C_{20}\widehat w_{j+1}^{-\tau}
 \le C_{20}q_{j+1}^{B_{10}\tau}.
\]
On the other hand, \eqref{fr:eq:internal-long-first-return} gives
\(R_{j,\mathcal B_t}>q_{j+1}\) for every child.  Consequently,
\begin{equation}\label{fr:eq:CF-child-count}
 R_{j+1,\mathcal C}
 =\sum_{t=0}^{L-1}R_{j,\mathcal B_t}>Lq_{j+1},
 \qquad
 L\le C_{20}q_{j+1}^{B_{10}\tau-1}.
\end{equation}
The backward chain is identical after reversing the itinerary.  For the
forward chain and \(1\le t<L\), let \(v_{j,t}^{\rm in}\) be the
coherently oriented expanding input vector of
\(B_{j,\mathcal B_t}(x_t,E)\), and let \(u_{j,t-1}^{\rm out}\) be the
expanding output vector of
\(B_{j,\mathcal B_{t-1}}(x_{t-1},E)\).  The predecessor factor is
exactly the backward-return word ending at \(x_t\).  Thus, independently
of the coherent signs chosen in \ref{fr:def:Young-2s-regular}\textup{(\ref{item:regular-expansion})},
\[
\begin{aligned}
 H_{j,t}^{\mathbb R}
 &:=(v_{j,t}^{\rm in})^{\mathsf T}u_{j,t-1}^{\rm out},\\
 |H_{j,t}^{\mathbb R}|
 &=\left|\cos\!\left(\frac\pi2-s_{K_{q_j}}^+(x_t,E)
                    +u_{K_{q_j}}^-(x_t,E)\right)\right|
 \\
 &=|\cos G_{K_{q_j},\mathcal B_t}(x_t,E)|.
\end{aligned}
\]
At every connecting point, \eqref{fr:eq:CF-internal-node-exclusion}
places the point outside \(\operatorname{int}K_{q_{j+1}}\).
Both the distinguished and nondistinguished alternatives in
\ref{fr:def:Young-2s-regular}\textup{(\ref{item:regular-zeros})} therefore give
\[
 |H_{j,t}^{\mathbb R}|\ge q_{j+1}^{-D_{13}},\qquad 1\le t<L,
\]
and the reversed chain has the same connector bound.
For either chain, write \(L_{\rm ch}\) for its factor count,
\(Q_{\rm ch}:=Q_{r_\dagger}\) for the derivative quantity in
Lemma~\ref{B:lem:polynomial-hyperbolic-chain}, and \(\Lambda_{\rm ch}\)
for its minimum factor top singular value.  Put \(q=Q\) initially and
\(q=q_{j+1}\) thereafter.  The initial and inductive bounds give, for fixed
\(D_{22}>0\),
\[
 L_{\rm ch}\le q^{A_0},\qquad
 Q_{\rm ch}\le C_{20}q^{D_{22}},\qquad
 \Lambda_{\rm ch}\ge e^{\gamma_0q}.
\]
Let \(A\) be the order-\(r_\dagger\) exponent in
Lemma~\ref{B:lem:polynomial-hyperbolic-chain}.  With
\(\varkappa:=q^{-D_{13}}\), after increasing \(Q_*\) one has
\[
 L_{\rm ch}^{A}Q_{\rm ch}^{A}
 \varkappa^{-A}\Lambda_{\rm ch}^{-2}
 \le C_{20}q^{A(A_0+D_{22}+D_{13})}e^{-2\gamma_0q}
 \le10^{-3}q^{-D_{13}}=10^{-3}\varkappa.
\]
Thus the lemma applies to all forward and backward chains.  Let
\(C_{24}=C_{24}(\mathfrak D_{\rm CF})\ge D_{13}+1\) absorb their connection loss, and define
\[
 \widetilde\Omega_j:=\omega_Q^{\rm seed}
 +2C_{24}\frac{\log(q_{\ell+1}+2)}{q_\ell}
 +2C_{24}\sum_{t=\ell}^{j-1}
       \frac{\log(q_{t+1}+2)}{q_t}.
\]
For \(j=\ell\), \ref{fr:def:Young-seed-package}\textup{(\ref{item:seed-factorization})}, \ref{fr:def:Young-seed-package}\textup{(\ref{item:seed-transversality})}, and \ref{fr:def:Young-seed-package}\textup{(\ref{item:seed-exterior})}
together with the same chain estimate give growth rate at least
\(\gamma-\widetilde\Omega_\ell\) for the complete
\(K_{q_\ell}\)-return matrices.  This initializes the rate induction.
For one subsequent concatenation, let \(R_0,\ldots,R_{L-1}\) be the
factor lengths and let
\[
 R:=\sum_{i=0}^{L-1}R_i,
\]
 be the total return length.  Since \(R_i>q_{j+1}\), one has
\(L\le R/q_{j+1}\).  With
\(\varkappa=q_{j+1}^{-D_{13}}\), the product estimate gives
\[
\begin{aligned}
 \log\lambda_{K_{q_{j+1}}}^{\pm}
 &\ge\sum_{i=0}^{L-1}\log\lambda_i
      +(L-1)\log(\varkappa/2)\\
 &\ge(\gamma-\widetilde\Omega_j)R
   -C_{20}\frac{\log(q_{j+1}+2)}{q_{j+1}}R\\
 &\ge(\gamma-\widetilde\Omega_j)R
   -C_{20}\frac{\log(q_{j+1}+2)}{q_j}R.
\end{aligned}
\]
At the initial transition, \ref{fr:def:Young-seed-package}\textup{(\ref{item:seed-factorization})} gives
\(R_i\ge Q^{a_{\rm seed}}\), while \ref{fr:def:Young-seed-package}\textup{(\ref{item:seed-transversality})--(\ref{item:seed-exterior})} give
\(\varkappa\ge Q^{-(D_7+B_{10})}\); hence its connection
loss is at most
\[
 C_{20}\frac{\log(Q+2)}{Q^{a_{\rm seed}}}R
 \le C_{20}\frac{\log(q_{\ell+1}+2)}{q_\ell}R.
\]
Thus the first correction term in \(\widetilde\Omega_j\) covers the
initial concatenation.  Hence, after taking
  \(C_{11}\ge4C_{24}\), every partial loss is bounded by \(\Omega_Q\)
in \eqref{fr:eq:Young-rate-loss}.  By
\eqref{fr:eq:CF-Fibonacci-denominator-growth} and
\(q_{j+1}\le\kappa^{-1}q_j^\tau\),
\[
\begin{aligned}
 \sum_{j\ge\ell}\frac{\log(q_{j+1}+2)}{q_j}
 &\le C_{20}\sum_{n\ge0}
 \frac{\log(Q+2)+\log(F_{n+1}+1)}{F_{n+1}Q}\\
 &\le C_{12}\frac{\log(Q+2)}Q,
 \qquad
 \sum_{n\ge0}\frac{1+\log(F_{n+1}+1)}{F_{n+1}}<\infty.
\end{aligned}
\]
Together with \ref{fr:def:Young-seed-package}\textup{(\ref{item:seed-factorization})}, this proves
\(\Omega_Q=o_{Q\to\infty}(1)\) and
\(\Omega_Q<\gamma/4\) for \(Q\ge Q_*\).

Apply the fixed-order lemma with the exponents and bounds fixed at the
start of the proof.  For \(1\le t\le r_\dagger\), let
\(\mathcal C_{t,j}\) be the least constant in the normalized
order-\(t\) bounds for
\(\log\lambda_{K_{q_j}}^\pm,u_{K_{q_j}}^\pm,s_{K_{q_j}}^\pm\), after
removing the factor \((R_{\mathcal B}^\pm)^{M_t}\).
This is the normalization in \eqref{fr:eq:fixed-order-old-bound}.
For the initial \(K_{q_\ell}\)-returns, use \(q=Q\) and
\ref{fr:def:Young-seed-package}\textup{(\ref{item:seed-factorization})}.  After enlarging \(Q_*\), the initial chains satisfy
\[
\begin{gathered}
 Q\le Q^{a_{\rm seed}}\le R_i\le Q^{A_0},\qquad L\le Q^{A_0},\\
 |h_i|\ge Q^{-(D_7+B_{10})},\qquad
 \|h_i\|_{C^{r_\dagger}}\le Q^{D_8},\qquad
 \gamma-\Omega_Q\ge\tfrac34\gamma=\gamma_0 .
\end{gathered}
\]
Since \(R_i+2\le3R_i\), the initial bounds with
\((R_i+2)^{M_t}\) are bounded in this normalization by the coefficient
\(C_t^{\rm seed}=3^{M_t}C_8\).  Thus
Lemma~\ref{fr:lem:level-independent-fixed-order-propagation}, with
\(q=Q\), gives
\[
 \mathcal C_{t,\ell}
 \le C_t^{\rm seed}
   +C_{18}\mathfrak B_t^{D_{23}}Q^{-2},
 \qquad 1\le t\le r_\dagger.
\]
At the transition from level \(j\) to level \(j+1\), apply the same
lemma with \(q=q_{j+1}\); the overlap-coefficient bound is preserved by the
endpoint estimate below.
Increase \(Q_*\) so that the lemma's threshold determined by these bounds is met
and
\[
 C_{18}\mathfrak B_t^{D_{23}}
 \left(Q^{-2}+\sum_{j\ge\ell}q_{j+1}^{-2}\right)
 \le\frac12\mathbf B_t
\]
for every \(t\).  The recurrence closes as
\[
\begin{aligned}
 \mathcal C_{t,j+1}
 &\le\mathcal C_{t,j}
   +C_{18}\mathfrak B_t^{D_{23}}q_{j+1}^{-2}\\
 &\le C_t^{\rm seed}
   +C_{18}\mathfrak B_t^{D_{23}}
       \left(Q^{-2}+\sum_{v=\ell}^jq_{v+1}^{-2}\right)\\
 &\le\tfrac12\mathbf B_t+\tfrac12\mathbf B_t=\mathbf B_t .
\end{aligned}
\]
For \(t\le r_*\), use \(M_t\le M_{r_*}=M_*\) and choose \(C_{20}\) to
dominate \(\max_{t\le r_*}\mathbf B_t\) and the matrix constant in
\eqref{fr:eq:fixed-order-new-matrix}. Since
\((R_{\mathcal B}^{\pm})^{M_t}\le(R_{\mathcal B}^{\pm}+2)^{M_*}\),
this proves \ref{fr:def:Young-2s-regular}\textup{(\ref{item:regular-expansion})}.

On each fixed itinerary,
Lemma~\ref{fr:lem:level-independent-fixed-order-propagation} compares
the forward contracting direction with that of the first child and the
backward expanding direction with that of the last child.  The chain
rule and the error bounds from Step~(i) therefore give
\begin{equation}\label{fr:eq:internal-joining-drift}
\begin{aligned}
 &\|s_{K_{q_{j+1}}}^{+}-s_{K_{q_j}}^{+}\|_{C^{r_\dagger}}
 +\|u_{K_{q_{j+1}}}^{-}-u_{K_{q_j}}^{-}\|_{C^{r_\dagger}}
 \le\mathfrak e_j,\\
 &\|\cos G_{K_{q_{j+1}}}-\cos G_{K_{q_j}}\|_{C^{r_\dagger}}
 \le\mathfrak e_j.
\end{aligned}
\end{equation}
The old branch in \eqref{fr:eq:internal-joining-drift} is the branch
selected by that fixed itinerary.  The new return matrices on the fixed
itinerary and their Riesz spectral projections are analytic across the
temporary subdomain boundaries,
so the uniform subdomainwise derivative bounds extend to their closures.
Taking suprema over the new atoms and their itinerary subdomains,
and summing the errors, \eqref{fr:eq:Young-seed-angle-bootstrap-margin}
gives
\begin{equation}\label{fr:eq:internal-all-atom-angle-bound}
 \sup_{j\ge\ell}\sup_{\mathcal B}
 \|\cos G_{K_{q_j},\mathcal B}\|_{C^{r_\dagger}}
 \le Q^{D_8}
 \le Q^{D_{13}}.
\end{equation}

\smallskip
\noindent\emph{Step (vi).} We propagate the simple roots and exclude
additional real zeros.
Fix \(1\le a\le2s\). Boundary separation places the lifted
\(q_{j+1}^{-B_{11}}\)-tube around the old zero graph inside one genuine
level-\((j+1)\) two-sided return atom and the old distinguished atom.
The fixed return words and their simple singular projections continue
across temporary itinerary cuts. Orient the new forward input and
backward output axes by positive inner product with the corresponding
old endpoint axes. Then \eqref{fr:eq:internal-joining-drift} holds on the whole tube.
There
\[
 |\partial_x\cos G_{K_{q_j},\mathcal B_{j,a}}
   (c_{j+1,a}(E),E)|
 \ge\tfrac14Q^{-D_{12}}.
\]
By \eqref{fr:eq:CF-scale-threshold-ledger}\textup{(iv)},
\[
 \mathfrak e_j
 \le\tfrac18Q^{-D_{12}}q_{j+1}^{-B_{11}}.
\]
Hence the old overlap has opposite signs at the endpoints of the lifted
tube, and these signs are preserved by the new overlap.
The quantitative implicit function theorem, applied to
\eqref{fr:eq:internal-quotient-bootstrap} and
\eqref{fr:eq:internal-joining-drift}, produces a unique local simple
root branch \(\widetilde c_{j+2,a}\) in that tube and gives
\[
 \|\widetilde c_{j+2,a}-\widetilde c_{j+1,a}\|_{C^{r_\dagger-1}}
 \le\mathfrak e_j
 \le e^{-c_5\gamma q_{j+1}}.
\]
For Lemma~\ref{fr:lem:transverse-factor-stability}, take
\(\mu:=\frac14Q^{-D_{12}}\).  The joining drift, root displacement,
and \eqref{fr:eq:internal-all-atom-angle-bound} give
\[
 \eta_1+H_2d_0
 \le(1+Q^{D_8})\mathfrak e_j
 \le\frac18Q^{-D_{12}}=\frac\mu2
\]
by \eqref{fr:eq:CF-scale-threshold-ledger}\textup{(iv)}.
Thus Lemma~\ref{fr:lem:transverse-factor-stability} applies to the old
and new simple-zero factorizations.
Denote the level-\((j+1)\) atom containing the tube by
\(\mathcal B_{j+1,a}\).
At a temporary parameter-cell boundary the fixed-word functions and
their jets coincide; uniqueness in the lifted tube therefore glues the
roots and their derivatives through order \(r_\dagger-1\), with the
unique compatible lift within distance \(1/4\) of the preceding lift.
After these zero branches are obtained, apply the preceding
radius-selection argument with index \(j+2\) to choose
\(\widehat w_{j+2}\) and hence define \(K_{q_{j+2}}(E)\).
The two margins in \eqref{fr:eq:CF-scale-threshold-ledger}\textup{(v)} control, respectively, interlevel
quotient/root-jet errors and spatial variation.  To extend the local
comparison to the whole new atom, put
\[
 h_{j+1,a}:=\cos G_{K_{q_{j+1}},\mathcal B_{j+1,a}}
\]
and, in the fixed lift of its interval fiber, define
\begin{equation}\label{fr:eq:internal-global-quotient}
 \mathfrak a_{j+1,a}^{\mathbb R}(x,E)
 :=\int_0^1\partial_xh_{j+1,a}
 \bigl(\widetilde c_{j+2,a}(E)
       +t(x-\widetilde c_{j+2,a}(E)),E\bigr)\,dt.
\end{equation}
Analyticity of the complete return word and its Riesz projections
across temporary old-itinerary boundaries makes
\eqref{fr:eq:internal-global-quotient} global on the new atom.  Hence
\[
 h_{j+1,a}(x,E)
 =(x-\widetilde c_{j+2,a}(E))
   \mathfrak a_{j+1,a}^{\mathbb R}(x,E).
\]
The quotient/root-jet error sum, the fiber radius
\(\le2q_{j+1}^{-B_{10}}\), and
\eqref{fr:eq:internal-all-atom-angle-bound} meet the slope budget in
\eqref{fr:eq:CF-scale-threshold-ledger}\textup{(i)}. Hence
\begin{equation*}
 \frac14Q^{-D_{12}}
 \le|\mathfrak a_{j+1,a}^{\mathbb R}|
 \le Q^{D_{13}},\qquad
 \|\mathfrak a_{j+1,a}^{\mathbb R}\|_{C^{r_*}}
 \le Q^{D_{13}}.
\end{equation*}
Indeed, Lemma~\ref{fr:lem:transverse-factor-stability} bounds the
interlevel differences of the quotient derivatives at the root by
\(\mathfrak e_j\).  Differentiating
\eqref{fr:eq:internal-global-quotient}, with the two additional
derivatives in \(r_\dagger=r_*+2\), bounds their variation across the
fiber.  The choice of \(D_{13}\) in
\eqref{fr:eq:CF-propagation-exponent-budget} absorbs the resulting
fixed Fa\`a di Bruno powers.  Hence
\eqref{fr:eq:internal-quotient-bootstrap} propagates on the entire
distinguished atom.

For \(x\in\mathcal B_{j+1,a}\setminus K_{q_{j+2}}(E)\),
\(|x-\widetilde c_{j+2,a}(E)|\ge\widehat w_{j+2}
\ge q_{j+2}^{-B_{10}}\).  Hence the factorization, the quotient
lower bound, and
\eqref{fr:eq:CF-scale-threshold-ledger}\textup{(iii)} give
\[
 |h_{j+1,a}(x,E)|
 \ge\tfrac14Q^{-D_{12}}q_{j+2}^{-B_{10}}
 \ge q_{j+2}^{-D_{13}}.
\]
Every remaining new atom is
therefore either attached to an old nondistinguished or exterior branch,
or is a root-free subatom of an old distinguished atom.  Fix a
fixed-itinerary subdomain of
one such atom.  If the old branch selected by this fixed itinerary was
nondistinguished, or lay outside
\(K_{q_{j+1}}(E)\), the previous nonvanishing bound,
\eqref{fr:eq:internal-joining-drift},
\eqref{fr:eq:CF-denominator-lower-bound-margin}, and
\eqref{fr:eq:CF-scale-threshold-ledger}\textup{(iv)} give
\[
 |\cos G_{K_{q_{j+1}}}|
 \ge q_{j+1}^{-D_{13}}-\mathfrak e_j
 \ge q_{j+2}^{-D_{13}}
   +\tfrac12c_{13} q_{j+1}^{-D_{13}-1+1/\tau}
 >q_{j+2}^{-D_{13}}.
\]
If instead the new atom lies
inside an old distinguished atom and does not contain its root, the
level-\((j+1)\) boundary separation and
\eqref{fr:eq:internal-quotient-bootstrap} give on the old branch
\[
 |\cos G_{K_{q_j},\mathcal B_{j,a}}|
 \ge \tfrac14Q^{-D_{12}}q_{j+1}^{-B_{11}}.
\]
Consequently,
\eqref{fr:eq:CF-scale-threshold-ledger}\textup{(iii)--(iv)} gives
\[
 |\cos G_{K_{q_{j+1}}}|
 \ge\tfrac14Q^{-D_{12}}q_{j+1}^{-B_{11}}-\mathfrak e_j
 \ge\tfrac18Q^{-D_{12}}q_{j+1}^{-B_{11}}
 \ge q_{j+2}^{-D_{13}}.
\]
Continuity across temporary itinerary boundaries extends these bounds
to every complete new atom.  Therefore
\[
\begin{gathered}
 Z_{\mathbb R}(h_{j+1,a})
 =\{\widetilde c_{j+2,a}\},\qquad
 Z_{\mathbb R}(\cos G_{K_{q_{j+1}},\mathcal B})=\varnothing
 \quad(\mathcal B\text{ nondistinguished}),\\
 \operatorname{dist}_{\mathbb T}
 (c_{j+2,a},\partial\mathcal B_{j+1,a})
 \ge2q_{j+1}^{-B_{11}}-e^{-c_5\gamma q_{j+1}}
 \ge q_{j+1}^{-B_{11}} .
\end{gathered}
\]
Here \(Z_{\mathbb R}\) denotes the zero set on the indicated real
atom; the last inequality is \eqref{fr:eq:CF-scale-threshold-ledger}\textup{(vi)}.  This proves
\eqref{fr:eq:Young-real-root-clearance} at level \(j+1\) and closes
the simultaneous induction with \eqref{fr:eq:internal-root-contraction}.

Finally, \eqref{fr:eq:internal-parent-child-radius-margin} with index
\(j+1\), together with \eqref{fr:eq:internal-root-contraction}, gives,
by \eqref{fr:eq:CF-scale-threshold-ledger}\textup{(vi)},
\[
 \|\widetilde c_{j+1,a}-\widetilde c_{j,a}\|_\infty+\widehat w_{j+1}
 \le\widehat w_j-\frac38c_{10}
 q_j^{-B_{10}-1+1/\tau},
 \qquad
 K_{q_{j+1}}(E)\Subset K_{q_j}(E).
\]
Thus \ref{fr:def:Young-2s-regular}\textup{(\ref{item:regular-returns})--(\ref{item:regular-zeros})} hold at every level on \(\mathcal G_j\).

\smallskip
\noindent\emph{Step (vii).} We sum the excluded measures and pass to the
limiting set.
The fixed choice of \(C_{11}\) includes a factor-two margin in the
summed polynomial losses.  By the same Fibonacci lower bound, equations
\eqref{fr:eq:internal-cross-return-loss},
\eqref{fr:eq:internal-boundary-loss}, and
\eqref{fr:eq:internal-buffer-loss} give
\begin{equation}\label{fr:eq:internal-total-loss}
 \sum_{j\ge\ell}|\mathcal G_{j-1}\setminus\mathcal G_j|
 \le\frac{C_{11}}2\left(
 Q^{D_6+\tau-B_{10}}
 +Q^{D_6+B_{10}\tau-B_{11}}
 \right)+C_{20}e^{-c_{20}\gamma Q^{1/\tau}}.
\end{equation}
After increasing \(Q_*\), the exponential term is at most the other
half of this polynomial bound.  This proves the first estimate in
Proposition~\ref{fr:prop:Young-CF-reinduction-exclusion}.
Since the \(\mathcal G_j\) are decreasing finite unions of compact
intervals,
\[
 \mathcal Y_Q^{(s)}:=\bigcap_{j\ge\ell}\mathcal G_j
 \quad\text{is compact},\qquad
 |\mathcal Y_Q^{(s)}|
 \ge|\mathcal J_Q|-\sum_{j\ge\ell}
       |\mathcal G_{j-1}\setminus\mathcal G_j|.
\]
By \eqref{fr:eq:CF-propagation-exponent-budget},
\[
\begin{aligned}
 Q^{B_{12}}
 \sum_{j\ge\ell}|\mathcal G_{j-1}\setminus\mathcal G_j|
 &\le C_{11}\left(
 Q^{B_{12}+D_6+\tau-B_{10}}
 +Q^{B_{12}+D_6+B_{10}\tau-B_{11}}\right)\\
 &\quad+C_{20}Q^{B_{12}}e^{-c_{20}\gamma Q^{1/\tau}}
 \longrightarrow0,\\
 |\mathcal Y_Q^{(s)}|
 &\ge|\mathcal J_Q|-\tfrac12c_{11}Q^{-B_{12}}
 \ge\tfrac12|\mathcal J_Q|.
\end{aligned}
\]
Fix \(E\in\mathcal Y_Q^{(s)}\).  Adjacent parameter-cell choices at a
common endpoint give the same fixed-word data and glued roots.  The points
\(c_{j+1,a}(E)\) lie in pairwise distinct atoms \(\mathcal B_{j,a}\).
Set \(r_{j,a}:=\widetilde c_{j+1,a}(E)\).  Then
\(\pi_{\mathbb T}(r_{j,a})=\xi_{j+1,a}(E)\), and
\eqref{fr:eq:internal-quotient-bootstrap} gives
\eqref{fr:eq:Young-real-simple-factorization}, with
\[
 \partial_x\cos G_{K_{q_j},\mathcal B_{j,a}}
   (\pi_{\mathbb T}r_{j,a},E)
 =\mathfrak a_{j,a}^{\mathbb R}(r_{j,a},E)\ne0.
\]
Since \(\frac14Q^{-D_{12}}\ge Q^{-D_{13}}\) for large \(Q\), the
quotient bound has the strength required in \ref{fr:def:Young-2s-regular}\textup{(\ref{item:regular-zeros})}.  The
nonvanishing and boundary-clearance estimates proved in the induction
exclude every other atomwise zero.  Thus \ref{fr:def:Young-2s-regular}\textup{(\ref{item:regular-returns})--(\ref{item:regular-zeros})} hold and
the first-return regularity conclusion of
Proposition~\ref{fr:prop:Young-CF-reinduction-exclusion} and the rate
bound in Corollary~\ref{fr:cor:Young-CF-rate-loss} follow.

\smallskip
\noindent\emph{Step (viii).} We initialize the holomorphic construction
at a fixed retained energy.
Fix \(E\in\mathcal Y_Q^{(s)}\), retain the notation \(K_j,S_j\)
from the setup of Proposition~\ref{fr:prop:Young-CF-holomorphic-realization},
and put
\[
 \overline R_j:=\max_{\mathcal B\in\mathcal P_{K_j}}R_{j,\mathcal B}.
\]
We suppress \(E\) when convenient.
The real objects are first defined on
\(K_j\cap\operatorname{Reg}^{\rm ph}(E)\).  Its complement in
\(K_j^{\rm reg}\) is countable; fixed-word germs extend the data
measurably across it without changing any essential supremum or integral.
The real construction and the componentwise return bound give
\begin{equation}\label{fr:eq:CF-holomorphic-return-lengths}
 q_{j+1}<m_j\le\overline R_j\le C_{20}q_j^{B_{10}\tau},
 \qquad m_j\longrightarrow\infty.
\end{equation}

For each of the finitely many initial words \(B=B_{\ell,\mathcal B}\),
put
\[
 M:=B^{\mathsf T}B,\qquad
 \Delta:=(\operatorname{tr}M)^2-4,\qquad
 \Delta(x)=(\lambda(B(x,E))^2-\lambda(B(x,E))^{-2})^2>0.
\]
Compactness gives
\[
\begin{gathered}
 h_\ell^{\rm aux}=h_\ell^{\rm aux}(Q,V_Q,\varrho_Q,E)>0,\qquad
 h_\ell^{\rm aux}\le\tfrac14,\quad4h_\ell^{\rm aux}<\varrho_Q,\\
 \Delta(z)\ne0
 \quad\text{on every }\mathcal B[4h_\ell^{\rm aux}].
\end{gathered}
\]
On each simply connected lifted atom strip, choose the real-normalized
branches
\[
 \mu^\pm=\frac{\operatorname{tr}M\pm\sqrt\Delta}{2},\qquad
 \Pi=\frac{M-\mu^-I}{\mu^+-\mu^-},\qquad
 \Pi v=v,\quad v^{\mathsf T}v=1,\quad \Lambda^2=\mu^+.
\]
The normalized axes form a two-sheeted holomorphic cover.
The coherent real axes \(v^{\mathbb R},u^{\mathbb R}\) from
\ref{fr:def:Young-2s-regular}\textup{(\ref{item:regular-expansion})} select its sheet and the sign
\(Bv^{\mathbb R}=\Lambda^{\mathbb R}u^{\mathbb R}\).
Simple connectivity and uniqueness of continuation glue the local
formulas \eqref{svd:eq:complex-signed-axis-from-projector} and make
\(v_{\ell,\mathcal B},\Lambda_{\ell,\mathcal B}\) real-symmetric.
With \(\mathcal J_0\) from
Definition~\ref{def:complex-special-orthogonal-group}, set
\[
\begin{aligned}
 u_{\ell,\mathcal B}(S_{\ell,\mathcal B}z)
 &:=\Lambda_{\ell,\mathcal B}(z)^{-1}
    B_{\ell,\mathcal B}(z)v_{\ell,\mathcal B}(z),\\
 V_{\ell,\mathcal B}(z)&:=(v_{\ell,\mathcal B}(z),
                 \mathcal J_0v_{\ell,\mathcal B}(z)),\\
 U_{\ell,\mathcal B}(S_{\ell,\mathcal B}z)&:=
 (u_{\ell,\mathcal B}(S_{\ell,\mathcal B}z),
  \mathcal J_0u_{\ell,\mathcal B}(S_{\ell,\mathcal B}z)).
\end{aligned}
\]
These are real-symmetric holomorphic \({\rm SO}(2,\mathbb C)\) frames and
\[
 B_{\ell,\mathcal B}(z,E)
 =U_{\ell,\mathcal B}(S_{\ell,\mathcal B}z,E)
  \operatorname{diag}(\Lambda_{\ell,\mathcal B}(z,E),
                      \Lambda_{\ell,\mathcal B}(z,E)^{-1})
 V_{\ell,\mathcal B}(z,E)^{\mathsf T}.
\]
For these frames, and for every descendant frame constructed below, use
the input-coordinate convention
\[
 \Pi_{j,\mathcal B}^{\rm in}(z)
 :=v_{j,\mathcal B}(z)v_{j,\mathcal B}(z)^{\mathsf T},\qquad
 \Pi_{j,\mathcal B}^{\rm out}(z)
 :=u_{j,\mathcal B}(S_{j,\mathcal B}z)
   u_{j,\mathcal B}(S_{j,\mathcal B}z)^{\mathsf T}.
\]
After shrinking the height, the two frame pairs separately satisfy
\begin{equation}\label{fr:eq:CF-holomorphic-initial-pair-frames}
\begin{aligned}
 &\|U_{\ell,\mathcal B}(S_{\ell,\mathcal B}z)\|
 +\|U_{\ell,\mathcal B}(S_{\ell,\mathcal B}z)^{-1}\|\le4,\\
 &\|V_{\ell,\mathcal B}(z)\|
 +\|V_{\ell,\mathcal B}(z)^{-1}\|\le4,
\end{aligned}
\qquad z\in\mathcal B[4h_\ell^{\rm aux}],
\end{equation}
and
\[
 |\Lambda_{\ell,\mathcal B}(z,E)|
 \ge\exp\!\left\{\frac34(\gamma-\Omega_Q)
                    R_{\ell,\mathcal B}\right\},
 \qquad z\in\mathcal B[4h_\ell^{\rm aux}].
\]
By the real derivative bounds and continuity, the height can be chosen
so that the following also holds for a fixed
\(B_{19}=B_{19}(\mathfrak D_{\rm CF})\ge1\),
\[
\begin{aligned}
 &\max_{1\le r\le r_*}\max_{\mathcal B}
 \sup_{z\in\mathcal B[4h_\ell^{\rm aux}]}\\
 &\quad
 \left(\|\partial_z^r\Pi_{\ell,\mathcal B}^{\rm in}\|
 +\|\partial_z^r\Pi_{\ell,\mathcal B}^{\rm out}\|
 +(R_{\ell,\mathcal B}+2)^{-1}
  |\partial_z^r\log\Lambda_{\ell,\mathcal B}|\right)
 \le(Q+2)^{B_{19}}.
\end{aligned}
\]
Keep the two initial costs separate:
\[
 J_1:=1+B_{19}\log(Q+2),\qquad
 J_2:=\log^+(h_\ell^{\rm aux})^{-1}.
\]
Thus
\[
 J_1=J_1(\mathfrak D_{\rm CF},Q),\qquad
 J_2=J_2(Q,V_Q,\varrho_Q,E);
\]
no uniform bound on \(J_2\) is required.

Let \(C_{16}=C_{16}(\mathfrak D_{\rm CF})\ge1\) dominate the structural constants in
\eqref{fr:eq:CF-child-count} and
\eqref{fr:eq:CF-holomorphic-return-lengths}, and fix
\(C_{17}>16\).  Apply
Lemmas~\ref{fr:lem:graded-holomorphic-product-inputs}--
\ref{fr:lem:graded-holomorphic-jet-propagation} with
\[
 a_L=B_{10}\tau-1,\qquad a_R=B_{10}\tau,\qquad
 D=D_{13},\qquad g=3\gamma/8,
\]
and with the preceding fixed value of \(C_{16}\),
and retain their common structural constants
\(d_1,\allowbreak d_2,\allowbreak C_{13},\allowbreak c_9,\allowbreak q_{\rm hol}\)
and the polynomials \(\mathcal P_r\).  These are independent of
\(Q,V_Q,\varrho_Q,E\): indeed,
\(g_Q=\frac12(\gamma-\Omega_Q)>3\gamma/8\).
Choose \(A_1<\cdots<A_{r_*}\) recursively so that
\[
 A_r>C_{13}+d_2\!\left(B_{19}+1+
          \sum_{1\le\nu<r}A_\nu\right),\qquad1\le r\le r_*,
\]
and then choose
\[
 D_{14}>d_2A_{r_*}+d_2B_{19}
              +3D_{13}+B_{10}\tau+20.
\]
Set
\begin{equation}\label{fr:eq:CF-holomorphic-strip-recursion}
 h_{j+1}^{\rm aux}:=
 \frac{h_j^{\rm aux}}{64(q_{j+1}+2)^{D_{14}}},
 \qquad
 \mathcal S_j:=1+\sum_{t=\ell+1}^j\log(q_t+2).
\end{equation}
For \(1\le r\le r_*\), put
\[
\begin{aligned}
 \mathcal K_{j,r}
 &:=\max\!\Biggl\{2,\\[-2pt]
 &\qquad\max_{\mathcal B\in\mathcal P_{K_j}}
 \sup_{z\in\mathcal B[4h_j^{\rm aux}]}\\[-2pt]
 &\qquad\left(\|\partial_z^r\Pi_{j,\mathcal B}^{\rm in}\|
 +\|\partial_z^r\Pi_{j,\mathcal B}^{\rm out}\|
 +(R_{j,\mathcal B}+2)^{-1}
  |\partial_z^r\log\Lambda_{j,\mathcal B}|\right)\Biggr\}.
\end{aligned}
\]
Define
\[
 \Gamma_{\ell}:=\frac34(\gamma-\Omega_Q),\qquad
 \Gamma_{j+1}:=\Gamma_j-
 (D_{13}+2)\frac{\log(q_{j+1}+2)}{q_{j+1}}.
\]
After enlarging \(Q_*\), summability gives
\(\Gamma_j\ge g_Q=\frac12(\gamma-\Omega_Q)\).  We maintain
simultaneously, for every \(\mathcal B\in\mathcal P_{K_j}\),
\begin{equation}\label{fr:eq:CF-holomorphic-induction-ledger}
\begin{gathered}
 \Gamma_j\ge g_Q,\qquad
 |\Lambda_{j,\mathcal B}(z,E)|
 \ge e^{\Gamma_jR_{j,\mathcal B}},
 \quad z\in\mathcal B[4h_j^{\rm aux}],\\
 \mathcal K_{j,r}\le
 \exp\{J_1+A_r\mathcal S_j\},
 \qquad1\le r\le r_*.
\end{gathered}
\end{equation}
The strip recursion has the accumulated form
\begin{equation}\label{fr:eq:CF-holomorphic-accumulated-strip}
 h_{j+1}^{\rm aux}
 =h_\ell^{\rm aux}64^{-(j+1-\ell)}
 \exp\!\left\{-D_{14}
 \sum_{t=\ell+1}^{j+1}\log(q_t+2)\right\}.
\end{equation}
Increase \(Q_*\) once more so that, with \(q:=q_{j+1}\), every
\(j\ge\ell\) satisfies the following common holomorphic threshold
bounds:
\begin{equation}\label{fr:eq:CF-holomorphic-scale-threshold-ledger}
\begin{aligned}
\textup{(i)}\quad
 &q\ge q_{\rm hol},\qquad
 8C_{13}h_{j+1}^{\rm aux}q^{D_{13}}
 \left(1+e^{J_1+A_1\mathcal S_j}\right)^{d_2}
 \le1,\\
\textup{(ii)}\quad
 &d_2(J_1+A_{r_*}\mathcal S_j)
 +C_{13}\log(q+2)\le\Gamma_jq,\\
\textup{(iii)}\quad
 &\begin{aligned}
 &(q+2)^{C_{13}}\Bigl[e^{J_1+A_r\mathcal S_j}\\
 &\qquad+\mathcal P_r\Bigl(e^{J_1+A_1\mathcal S_j},\ldots,
 e^{J_1+A_{r-1}\mathcal S_j}\Bigr)\Bigr]\\
 &\qquad\le e^{J_1+A_r\mathcal S_{j+1}},
 \qquad1\le r\le r_*,
 \end{aligned}\\
\textup{(iv)}\quad
 &4\sum_{t\ge\ell}e^{-c_9q_{t+1}}
 \le C_{17}-8.
\end{aligned}
\end{equation}
For \(r=1\), the polynomial in \eqref{fr:eq:CF-holomorphic-scale-threshold-ledger}\textup{(iii)} is understood as its
constant value.  We verify that the choice is uniform.  From
\eqref{fr:eq:CF-holomorphic-accumulated-strip},
\(h_\ell^{\rm aux}\le1/4\), and \(q\ge Q\), the logarithm of the
nonconstant left-hand side in \eqref{fr:eq:CF-holomorphic-scale-threshold-ledger}\textup{(i)} is at most
\[
 C_{20}+d_2J_1
 -(D_{14}-d_2A_1)\mathcal S_j
 -(D_{14}-D_{13})\log(q+2).
\]
Substitute \(J_1=1+B_{19}\log(Q+2)\),
\(Q\le q\), and the choice of \(D_{14}\) to obtain
\eqref{fr:eq:CF-holomorphic-scale-threshold-ledger}\textup{(i)}.  Next,
\(\mathcal S_j\le C_{20}(\log(q_j+2))^2\) gives
\[
 \sup_{j\ge\ell}
 \frac{J_1+\mathcal S_j+\log(q_{j+1}+2)}{q_{j+1}}
 \le C_{20}\frac{(\log(Q+2))^2}{Q}\longrightarrow0.
\]
Since \(\Gamma_j\ge g_Q>3\gamma/8\), this proves \eqref{fr:eq:CF-holomorphic-scale-threshold-ledger}\textup{(ii)}.
For \eqref{fr:eq:CF-holomorphic-scale-threshold-ledger}\textup{(iii)}, write \(T_{j,r}:=e^{J_1+A_r\mathcal S_j}\).
For each monomial \(a\prod_{\nu<r}T_{j,\nu}^{m_\nu}\) of
\(\mathcal P_r\), \(\sum_{\nu<r}m_\nu\le d_2\), and
\[
\begin{aligned}
 \frac{(q+2)^{C_{13}}T_{j,r}}{T_{j+1,r}}
 &=(q+2)^{C_{13}-A_r}\longrightarrow0,\\
 \log\frac{(q+2)^{C_{13}}a\prod_{\nu<r}T_{j,\nu}^{m_\nu}}
                {T_{j+1,r}}
 &\le C_{20}+(d_2-1)J_1\\
 &\quad+\left(d_2\sum_{\nu<r}A_\nu-A_r\right)\mathcal S_j
 +(C_{13}-A_r)\log(q+2)\longrightarrow-\infty .
\end{aligned}
\]
Both limits are uniform in \(j\ge\ell\) as \(Q\to\infty\), by the
recursive choice of \(A_r\) and \(Q\le q\).  Summing the finitely many
monomials proves \eqref{fr:eq:CF-holomorphic-scale-threshold-ledger}\textup{(iii)}; Fibonacci growth proves \eqref{fr:eq:CF-holomorphic-scale-threshold-ledger}\textup{(iv)}.
Choose the constants in the following order:
\[
 (d_1,d_2,C_{13})
 \longrightarrow(A_1,\ldots,A_{r_*})
 \longrightarrow D_{14}\longrightarrow Q_*.
\]
The initial multiplier and jet estimates, together with
\(\Gamma_\ell\ge g_Q\), initialize
\eqref{fr:eq:CF-holomorphic-induction-ledger}.
Simultaneously maintain the two separate bounds
\begin{equation}\label{fr:eq:CF-holomorphic-pairwise-frame-ledger}
\begin{aligned}
 &\max_{\mathcal B}\sup_{z\in\mathcal B[4h_j^{\rm aux}]}
 \bigl(\|U_{j,\mathcal B}(S_{j,\mathcal B}z)\|
 +\|U_{j,\mathcal B}(S_{j,\mathcal B}z)^{-1}\|\bigr)
 \le\frac12C_{17},\\
 &\max_{\mathcal B}\sup_{z\in\mathcal B[4h_j^{\rm aux}]}
 \bigl(\|V_{j,\mathcal B}(z)\|
 +\|V_{j,\mathcal B}(z)^{-1}\|\bigr)
 \le\frac12C_{17}.
\end{aligned}
\end{equation}
By bijectivity of \(S_j\), the assembled source \(U_j(z)\) is supplied
by the backward-return germ over \(S_j^{-1}z\).  Thus, at a common source
point, the \(U_j\)-pair and the \(V_j\)-pair are each bounded by
\(C_{17}/2\), so the four terms on the left of
\eqref{fr:eq:CF-holomorphic-frame-bound} are bounded by
\(C_{17}\).  By
\eqref{fr:eq:CF-holomorphic-initial-pair-frames}, each frame pair starts
at \(4\) and increases by at
most \(2e^{-c_9q_{j+1}}\) at one reinduction.  Therefore
\eqref{fr:eq:CF-holomorphic-scale-threshold-ledger}\textup{(iv)} gives
\[
 4+2\sum_{t\ge\ell}e^{-c_9q_{t+1}}
 \le\frac12C_{17},
\]
and Lemma~\ref{fr:lem:graded-holomorphic-reinduction} closes the two
pairwise bounds.

\smallskip
\noindent\emph{Step (ix).} We propagate the holomorphic normal forms
and their derivatives.
For \(x\in K_j^{\rm reg}\), put \(p=r_{K_j}^+(x)\) and \(x'=S_jx\).
The inverse return and cocycle identities give
\[
 r_{K_j}^-(x')=p,\qquad
 B_{K_j}^-(x')=A_p^{V_Q}(x'-p\alpha,E)
             =A_p^{V_Q}(x,E)=B_{K_j}^+(x).
\]
Choose continuous orientations of the singular axes on each connected
domain of a fixed one-sided return word, before restricting to two-sided
atoms.  A cut changing only the forward return time therefore leaves
the backward source germ unchanged.
Define the source germ of \(U_j\) by the expanding output axis of
the backward-return word.  It is the return-image germ of the forward word.
The sign of the real, nonzero multiplier is chosen by
\[
 B_{j,\mathcal B}(x,E)v_{j,\mathcal B}(x,E)
 =\Lambda_{j,\mathcal B}(x,E)u_j(S_jx,E),\qquad
 |\Lambda_{j,\mathcal B}(x,E)|=\|B_{j,\mathcal B}(x,E)\|.
\]
This identity extends holomorphically.  On temporary refinements the
fixed word is unchanged; uniqueness of the oriented continuation gives
the same germ.  Distinct genuine return atoms retain separate germs.
The real orientations are those in \ref{fr:def:Young-2s-regular}\textup{(\ref{item:regular-expansion})}, so
\begin{equation}\label{fr:eq:CF-real-complex-overlap-identity}
 H_j(x,E):=e_1^{\mathsf T}V_j(x,E)^{\mathsf T}U_j(x,E)e_1
 =\cos G_{K_{q_j},\mathcal B}(x,E),\qquad x\in\mathcal B.
\end{equation}

Use the exact chain \eqref{fr:eq:CF-exact-reinduction}.  At the fixed
energy, choose a lifted open phase fiber
\[
 I_{\mathcal C,E}:=\operatorname{int}
 \{x\in\mathbb R:(x,E)\in\mathcal C\},
 \qquad
 \Omega_{\mathcal C,E}^+
 :=I_{\mathcal C,E}+i(-4h_{j+1}^{\rm aux},4h_{j+1}^{\rm aux}).
\]
Because the itinerary is fixed on \(I_{\mathcal C,E}\) and
\(h_{j+1}^{\rm aux}<h_j^{\rm aux}\), fixed-word continuation gives
\begin{equation}\label{fr:eq:CF-holomorphic-shifted-domain-inclusion}
 z\in\Omega_{\mathcal C,E}^+
 \quad\Longrightarrow\quad
 z+n_t\alpha\in\mathcal B_t[4h_j^{\rm aux}],
 \qquad0\le t<L.
\end{equation}
For
\(1\le t<L\), let
\[
 H_{j,t}(z):=
 v_{j,\mathcal B_t}(z+n_t\alpha)^{\mathsf T}
 u_{j,\mathcal B_{t-1}}(z+n_t\alpha).
\]
Here the second vector is the return-image germ of the preceding word:
\(S_{j,\mathcal B_{t-1}}(z+n_{t-1}\alpha)=z+n_t\alpha\), so both
vectors are evaluated at the same connecting point.
The corresponding connector is
\[
 W_{j,t}(z):=
 V_{j,\mathcal B_t}(z+n_t\alpha)^{\mathsf T}
 U_{j,\mathcal B_{t-1}}(z+n_t\alpha)
 =\begin{pmatrix}H_{j,t}&-\Theta_{j,t}\\
                  \Theta_{j,t}&H_{j,t}\end{pmatrix}
 \in{\rm SO}(2,\mathbb C).
\]
Put \(q:=q_{j+1}\), and pull every child normal form back to the common
itinerary coordinate by setting, for \(0\le t<L\),
\[
\begin{aligned}
 \widetilde B_t(z)&:=B_{j,\mathcal B_t}(z+n_t\alpha,E),\\
 V_t(z)&:=V_{j,\mathcal B_t}(z+n_t\alpha,E),\\
 U_t^+(z)&:=U_{j,\mathcal B_t}(z+n_{t+1}\alpha,E),\\
 \Lambda_t(z)&:=\Lambda_{j,\mathcal B_t}(z+n_t\alpha,E).
\end{aligned}
\]
Then
\[
 \widetilde B_t=U_t^+
 \operatorname{diag}(\Lambda_t,\Lambda_t^{-1})V_t^{\mathsf T},
 \qquad V_t^{\mathsf T}U_{t-1}^+=W_{j,t}\quad(1\le t<L).
\]
By \eqref{fr:eq:CF-internal-node-exclusion} and \ref{fr:def:Young-2s-regular}\textup{(\ref{item:regular-zeros})},
\(|H_{j,t}(x)|\ge q^{-D_{13}}\) on the real itinerary cell.
Moreover, \eqref{fr:eq:internal-long-first-return} gives
\[
 R_t=R_{j,\mathcal B_t}>q_{j+1}=q,\qquad0\le t<L,
\]
and \eqref{fr:eq:CF-child-count} gives
\(L\le C_{20}(q+2)^{B_{10}\tau-1}\), while
\eqref{fr:eq:CF-holomorphic-return-lengths} gives
\(\overline R_j\le C_{20}(q+2)^{B_{10}\tau}\).  The exact chain,
\eqref{fr:eq:CF-holomorphic-shifted-domain-inclusion}, the child normal forms,
\eqref{fr:eq:CF-holomorphic-induction-ledger}, and
\eqref{fr:eq:CF-holomorphic-pairwise-frame-ledger} together with
\eqref{fr:eq:CF-holomorphic-scale-threshold-ledger}\textup{(i)--(ii)}
therefore verify all hypotheses of
Lemmas~\ref{fr:lem:graded-holomorphic-product-inputs}--
\ref{fr:lem:graded-holomorphic-jet-propagation}, with
\[
 (a_L,a_R,D,\Gamma,h^+)
 =(B_{10}\tau-1,B_{10}\tau,D_{13},
   \Gamma_j,h_{j+1}^{\rm aux}).
\]

Apply Lemmas~\ref{fr:lem:graded-holomorphic-product-inputs}--
\ref{fr:lem:graded-holomorphic-jet-propagation} on
\(\Omega_{\mathcal C,E}^+\) and extend the output to the closed strip by
continuity.  The common output defines the level-\((j+1)\) multiplier
and frames and, by
\eqref{fr:eq:graded-holomorphic-connector-output}, gives
\begin{equation}\label{fr:eq:CF-intermediate-complex-safety}
 |H_{j,t}(z)|\ge\frac12q_{j+1}^{-D_{13}},
 \qquad1\le t<L,\qquad z\in\overline{\Omega_{\mathcal C,E}^+}.
\end{equation}
Thus \eqref{fr:eq:CF-intermediate-complex-safety} gives complex
nonvanishing of the internal connectors in the reinduction.
In the original coordinates, the endpoint-frame conclusion is
\begin{equation}\label{fr:eq:CF-holomorphic-frame-drift}
\begin{aligned}
 \min_{\epsilon_U,\epsilon_V\in\{\pm1\}}\Bigl(&
 \|U_{j+1,\mathcal C}(\,\cdot+n_L\alpha)
      -\epsilon_U U_{j,\mathcal B_{L-1}}(\,\cdot+n_L\alpha)\|_{C^{r_*}}\\
 &+\|V_{j+1,\mathcal C}-\epsilon_V V_{j,\mathcal B_0}\|_{C^{r_*}}
 \Bigr)\le e^{-c_9q_{j+1}},
\end{aligned}
\end{equation}
where the norms are taken on
\(\overline{\Omega_{\mathcal C,E}^+}\);
the signs are constant on each itinerary strip, and the same estimate
holds for the inverse frames. Choose continuous real axes for each
complete return word and match the local continuations by
\[
 (U,V,\Lambda)\longmapsto
 (\epsilon_U U,\epsilon_V V,\epsilon_U\epsilon_V\Lambda).
\]
This preserves the normal form, frame norms, projections, multiplier
modulus, and positive-order logarithmic derivatives.
Across a temporary common boundary, shrink to a common complex
neighborhood and label the two descriptions by \(1,2\).  The total
fixed word and its dominant Riesz projections satisfy
\[
 B^{(1)}=B^{(2)}=B_{\rm tot},\qquad
 \Pi^{{\rm in},(1)}=\Pi^{{\rm in},(2)},\qquad
 \Pi^{{\rm out},(1)}=\Pi^{{\rm out},(2)}.
\]
The real orientation \ref{fr:def:Young-2s-regular}\textup{(\ref{item:regular-expansion})} and uniqueness of continuation give
\[
\begin{gathered}
 v^{(1)}=v^{(2)},\quad u^{(1)}=u^{(2)},\qquad
 B_{\rm tot}v^{(i)}=\Lambda^{(i)}u^{(i)}
 \ \Longrightarrow\ \Lambda^{(1)}=\Lambda^{(2)},\\
 \partial_z^r(U^{(1)},V^{(1)},\Lambda^{(1)})
 =\partial_z^r(U^{(2)},V^{(2)},\Lambda^{(2)}),
 \qquad0\le r\le r_* .
\end{gathered}
\]
This gluing also preserves a neighborhood of the full closed atom strip.
Indeed, on each closed itinerary strip,
\[
 \Delta_{\rm tot}
 :=\bigl(\operatorname{tr}(B_{\rm tot}^{\mathsf T}B_{\rm tot})\bigr)^2-4
 =\bigl(\Lambda_{\rm tot}^{2}-\Lambda_{\rm tot}^{-2}\bigr)^2\ne0,
 \qquad |\Lambda_{\rm tot}|>1.
\]
The holomorphic eigenprojections and normalized axes therefore extend
across every boundary point.  There are finitely many itinerary cells,
and their closed real fibers cover the closure of the complete return atom;
the preceding equalities glue their extensions on
\(\mathcal B[4h_{j+1}^{\rm aux}]\) and a neighborhood of it.
The same argument applies to the backward word defining the source frame.
Taking the maximum over the itinerary cells gives
\begin{equation}\label{fr:eq:CF-holomorphic-jet-recurrence}
 \mathcal K_{j+1,r}
 \le(q_{j+1}+2)^{C_{13}}
 \left(\mathcal K_{j,r}
 +\mathcal P_r(\mathcal K_{j,1},\ldots,
                       \mathcal K_{j,r-1})\right),
 \qquad1\le r\le r_*.
\end{equation}
The recurrence \eqref{fr:eq:CF-holomorphic-jet-recurrence} and
\eqref{fr:eq:CF-holomorphic-scale-threshold-ledger}\textup{(iii)} close
the graded derivative bound, while
\eqref{fr:eq:CF-holomorphic-frame-drift} and
\eqref{fr:eq:CF-holomorphic-scale-threshold-ledger}\textup{(iv)} close
the \(C_{17}\)-bound.  Finally,
\eqref{fr:eq:graded-holomorphic-expansion-output} gives
\begin{equation}\label{fr:eq:CF-holomorphic-rate-update}
\begin{aligned}
 |\Lambda_{j+1,\mathcal C}(z,E)|
 &\ge\exp\!\left\{\left(\Gamma_j
 -(D_{13}+2)\frac{\log(q_{j+1}+2)}{q_{j+1}}\right)
 R_{j+1,\mathcal C}\right\}\\
 &=e^{\Gamma_{j+1}R_{j+1,\mathcal C}},
 \qquad z\in\overline{\Omega_{\mathcal C,E}^+}.
\end{aligned}
\end{equation}
This closes the expansion induction. Since \(\Gamma_j\ge g_Q\),
\eqref{fr:eq:CF-holomorphic-induction-ledger} gives
\begin{equation}\label{fr:eq:CF-holomorphic-block-expansion}
 |\Lambda_{j,\mathcal B}(z,E)|
 \ge e^{g_QR_{j,\mathcal B}},
 \qquad z\in\mathcal B[4h_j^{\rm aux}].
\end{equation}
The bound on \(\mathcal S_j\) used in
\eqref{fr:eq:CF-holomorphic-scale-threshold-ledger} gives
\begin{equation}\label{fr:eq:CF-holomorphic-jet-subexponential}
 \log\mathcal K_{j,r}
 \le J_1+C_{20}(\log(q_j+2))^2
 =o_{j\to\infty}(m_j).
\end{equation}

\smallskip
\noindent\emph{Step (x).} We construct the complex divisors and
verify the acceleration criterion.
Writing \(H_{j,\mathcal B}\) for the atom germ of \(H_j\), define
\[
 \mathscr J_j:=2+
 \max_{\mathcal B\in\mathcal P_{K_j}}\max_{1\le r\le r_*}
 \sup_{z\in\mathcal B[4h_j^{\rm aux}]}
 |\partial_z^rH_{j,\mathcal B}(z,E)|.
\]
The frame construction and
\eqref{fr:eq:CF-holomorphic-jet-subexponential} give
\(\log\mathscr J_j=o_j(m_j)\).  Define
\begin{equation}\label{fr:eq:CF-holomorphic-divisor-height}
 h_j:=\frac1{64}\min\left\{
 h_{j+1}^{\rm aux},\ Q^{-D_{13}}\mathscr J_j^{-1},\
 q_{j+1}^{-D_{13}}\mathscr J_j^{-1}\right\}.
\end{equation}
On a distinguished lift, \ref{fr:def:Young-2s-regular}\textup{(\ref{item:regular-zeros})} and
\eqref{fr:eq:CF-real-complex-overlap-identity} give the complete real
divisor \(r_{j,a}\) and the quotient bounds
\(Q^{-D_{13}}\le|\mathfrak a_{j,a}^{\mathbb R}|\le Q^{D_{13}}\).
Apply Lemma~\ref{lem:real-cluster-complex-divisor} with
\[
 I=\overline{\widetilde{\mathcal B}_{j,a}},\quad P=1,\quad
 h_0=h_{j+1}^{\rm aux},\quad
 (\mu_-,\mu_+)=(Q^{-D_{13}},Q^{D_{13}}).
\]
Since \(P=1\) and \(r_*\ge3\), its \(J_1,J_2\) are at most
\(\mathscr J_j\), and its reduced height satisfies
\[
 \widehat h\ge\frac14\min\{h_{j+1}^{\rm aux},
          2Q^{-D_{13}}\mathscr J_j^{-1}\}\ge4h_j
\]
by \eqref{eq:real-cluster-reduced-height} and
\eqref{fr:eq:CF-holomorphic-divisor-height}.  Hence
\[
 H_j(z,E)=(z-r_{j,a})\mathfrak a_{j,a}^{\mathbb C}(z,E),
 \qquad
 \frac12Q^{-D_{13}}
 \le|\mathfrak a_{j,a}^{\mathbb C}|
 \le2Q^{D_{13}},
\]
where the quotient is holomorphic, nonvanishing, and real-symmetric.
The same lemma controls
\(\partial_z^2\log\mathfrak a_{j,a}^{\mathbb C}\) there.
On a nondistinguished atom, \ref{fr:def:Young-2s-regular}\textup{(\ref{item:regular-zeros})} and
\eqref{fr:eq:CF-holomorphic-divisor-height} give
\[
 |H_j(x+iy,E)|
 \ge q_{j+1}^{-D_{13}}-4h_j\mathscr J_j
 \ge\frac12q_{j+1}^{-D_{13}},
 \qquad |y|\le4h_j.
\]
Set \(\kappa_j:=\frac12q_{j+1}^{-D_{13}}\) when nondistinguished
atoms occur, and \(\kappa_j:=1\) otherwise; take
\[
 d_j:=\min_{1\le a\le2s}
 \operatorname{dist}(r_{j,a},\partial\widetilde{\mathcal B}_{j,a})
 \ge q_j^{-B_{11}}.
\]
Let \(D_j:=2C_{17}\).  Since \(h_j\le h_j^{\rm aux}\),
the maintained four-frame estimate gives
\eqref{fr:eq:CF-holomorphic-frame-bound} with this \(D_j\).  Likewise,
\eqref{fr:eq:CF-holomorphic-block-expansion},
\(R_{j,\mathcal B}\ge m_j\), and \(\chi=g_Q\) give
\eqref{fr:eq:CF-holomorphic-expansion}.  Fix one value of \(C_{20}\),
independent of \(j\), large enough for the divisor and logarithmic-derivative
bounds below, and set
\[
\begin{aligned}
 \mathscr A_j:=2+2Q^{D_{13}}+C_{20}\bigl(&Q^{D_{13}}\mathscr J_j
 {}+Q^{2D_{13}}\mathscr J_j^2\\
 &+\kappa_j^{-1}\mathscr J_j
 {}+\kappa_j^{-2}\mathscr J_j^2\bigr).
\end{aligned}
\]
The first line of terms controls the distinguished quotients by the
divisor lemma; the second controls
\(\partial_z^2\log H_j=H_j''/H_j-(H_j'/H_j)^2\) on nondistinguished
atoms.  In particular, \mbox{\(\log\mathscr A_j=o_j(m_j)\)}.
The factorizations above contain no further zeros on distinguished
atoms, and the nondistinguished lower bound excludes all others.  With
\(\mathcal P_j:=\mathcal P_{K_j}\), the source--image germ convention now
gives
\[
 \mathscr D_j\in\operatorname{DivAdm}(2s;h_j,d_j,\kappa_j),\qquad
 \nu_j=2s,\quad P_{j,a}=1,\quad c_{j,a,1}=r_{j,a}.
\]

Iteration of \eqref{fr:eq:CF-holomorphic-strip-recursion} gives
\[
 \log(h_j^{\rm aux})^{-1}
 \le J_2+C_{20}\mathcal S_j
 \le J_2+C_{20}(\log(q_j+2))^2.
\]
Together with \eqref{fr:eq:CF-holomorphic-return-lengths},
\eqref{fr:eq:CF-holomorphic-jet-subexponential}, and
\eqref{fr:eq:CF-holomorphic-divisor-height}, this proves
\eqref{fr:eq:CF-holomorphic-subexponential-data}; it also implies
\(e^{-\beta m_j}\le h_j\) eventually for every fixed \(\beta>0\).

Finally, real symmetry and nonvanishing give, on each atom,
\[
 F_j:=\frac{\Lambda_j'}{\Lambda_j},\qquad
 F_j(x)\in\mathbb R,\qquad
 \partial_y\log|\Lambda_j(x+iy)|=-\operatorname{Im}F_j(x+iy).
\]
Consequently, for \(0\le y\le h_j\),
\[
\begin{aligned}
 |\partial_y\log|\Lambda_j(x+iy)||
 &\le|F_j(x+iy)-F_j(x)|\\
 &\le\int_0^y|F_j'(x+it)|\,dt
 \le y(R_{j,\mathcal B}+2)\mathcal K_{j,2}.
\end{aligned}
\]
Integration gives
\[
 \int_{K_j}\left|\partial_y\log|\Lambda_j(x+iy,E)|\right|\,dx
 \le C_{20}|K_j|y(\overline R_j+2)\mathcal K_{j,2}.
\]
For every fixed \(\beta>0\), the subexponential bounds imply, eventually,
\[
\begin{aligned}
 \sup_{0\le y\le e^{-\beta m_j}}
 \int_{K_j}|\partial_y\log|\Lambda_j(x+iy,E)||\,dx
 &\le C_{20}|K_j|(\overline R_j+2)\mathcal K_{j,2}e^{-\beta m_j}\\
 &\le \exp\{-\beta m_j+o_j(m_j)\}\longrightarrow0 .
\end{aligned}
\]
The threshold and little-\(o\) may depend on \(\beta\).
For \eqref{fr:eq:CF-holomorphic-beta-choice},
\[
 2\chi-2P_{\rm loc}\beta_- -\beta_+
 =\left(\frac32-\frac1{16s}\right)\chi>0.
\]
Together with
\eqref{fr:eq:CF-holomorphic-admissibility}--
\eqref{fr:eq:CF-holomorphic-vertical-variation}, this verifies the
hypotheses of Theorem~\ref{thm:persistent-divisor-acceleration} after
removing a finite initial segment that may depend on the fixed data.
This proves Proposition~\ref{fr:prop:Young-CF-holomorphic-realization}
and Corollary~\ref{fr:cor:Young-CF-acceleration-data}, and completes
the joint proof.
\end{proof}

We apply the selection in the setting of
Proposition~\ref{fr:prop:prevalent-multiwell-seed}, with parameters
\(r_*,R,q\).  Retain its constants \(\varepsilon_{\rm seed},C_8,B_9\).

\begin{corollary}
\label{fr:cor:prevalent-multiwell-Young-regular}
With the preceding notation, there exists
\[
 \varepsilon_{\rm Y}=\varepsilon_{\rm Y}(R,q,r_*)>0
\]
such that the following holds.  Let
\[
 \delta\in\mathcal P_{\mathrm{strat}}(R,q),\qquad
 0<\varepsilon<\varepsilon_{\rm Y},
\]
and define \(Q=q_\ell\) by
\eqref{fr:eq:regular-cor-starting-scale}.  For every \(1\le s\le N\),
there is a compact set
\[
 \mathcal Y_{Q,s}^{\varepsilon,\delta}
 \subset\mathcal J_{Q,s}^{\varepsilon,\delta}
\]
such that
\begin{equation}\label{fr:eq:regular-spectral-layer-measure}
 \mathcal Y_{Q,s}^{\varepsilon,\delta}
 \subset
 \left\{E\in\Sigma(V_{\varepsilon,\delta},\alpha):
 \omega(\alpha,E;V_{\varepsilon,\delta})=s\right\},
 \qquad
 |\mathcal Y_{Q,s}^{\varepsilon,\delta}|
 \ge\frac{C_8^{-1}}2Q^{-B_9}>0.
\end{equation}
\end{corollary}

The same spectral sets carry the quantitative divisor data used above.

\begin{corollary}
\label{fr:cor:prevalent-multiwell-divisor-hierarchy}
The sets in Corollary~\ref{fr:cor:prevalent-multiwell-Young-regular}
can be chosen so that every
\(E\in\mathcal Y_{Q,s}^{\varepsilon,\delta}\) has a data-dependent index
\(j_{\rm hol}\ge\ell\)
with the following properties.
\begin{enumerate}
\item The data supplied by
Proposition~\ref{fr:prop:Young-CF-holomorphic-realization} and
Corollary~\ref{fr:cor:Young-CF-acceleration-data},
restricted to \(j\ge j_{\rm hol}\), satisfy all hypotheses of
Theorem~\ref{thm:persistent-divisor-acceleration} with
\[
 m_j=\underline R_j,\qquad
 \chi=g_Q:=\frac12(\log\lambda-\Omega_Q)>0,\qquad
 P_{\rm loc}=2s,
\]
and
\[
 \beta_+=\frac{\chi}{16s},\qquad
 \beta_-=\frac{\chi}{8s}.
\]
\item At every such level,
\(\nu_j=2s\) and \(P_{j,a}=1\), \(1\le a\le2s\).  Thus the
distinguished atoms contain exactly \(2s\) simple real zeros in total;
the atomwise factorizations are complete and there are no other local
zeros.
\end{enumerate}
\end{corollary}

\begin{proof}[Proof of
Corollaries~\ref{fr:cor:prevalent-multiwell-Young-regular}
and~\ref{fr:cor:prevalent-multiwell-divisor-hierarchy}]
Proposition~\ref{fr:prop:prevalent-multiwell-seed} supplies the intervals
\(\mathcal J_{Q,s}^{\varepsilon,\delta}\), the lower bound
\eqref{fr:eq:multipeak-seed-window-measure}, common initial-scale parameters
\(\mathfrak D_{\rm seed}\), an accompanying exponent \(D_{13}\), and
admissible initial-scale data with \(2s\) simple zeros for every
\(1\le s\le N\).
Apply Proposition~\ref{fr:prop:Young-CF-reinduction-exclusion} with
\[
 \gamma=\log\lambda,\qquad
 V_Q=V_{\varepsilon,\delta},\qquad \varrho_Q=\rho,\qquad
 (c_{11},B_{12})=(C_8^{-1},B_9).
\]
The parameter values \(r_\dagger,\tau,\gamma\) agree by construction, and
\eqref{fr:eq:multipeak-seed-propagation-budget} is precisely
\eqref{fr:eq:CF-propagation-exponent-budget}.  Since \(1\le s\le N\)
is a finite range, decreasing
\(\varepsilon_{\rm Y}\le\varepsilon_{\rm seed}\) once makes the
starting denominator exceed all corresponding thresholds.  The
proposition gives the compact sets, their measure bounds, and
first-return regularity; Corollary~\ref{fr:cor:Young-CF-rate-loss}
gives their growth-rate deficit \(\Omega_Q\).  Proposition
\ref{fr:prop:Young-regular-is-spectral} gives their spectral inclusion.

\smallskip
For the holomorphic conclusion, fix \(E\in\mathcal Y_{Q,s}^{\varepsilon,\delta}\).
Proposition~\ref{fr:prop:Young-CF-holomorphic-realization} and
Corollary~\ref{fr:cor:Young-CF-acceleration-data} give the
stated hierarchy for \(j\ge j_{\rm hol}(E)\), satisfying the hypotheses
of Theorem~\ref{thm:persistent-divisor-acceleration} with
\(P_{\rm loc}=2s\).  Hence
\[
 \omega(\alpha,E;V_{\varepsilon,\delta})
 =\frac12P_{\rm loc}=s.
\]
This proves \eqref{fr:eq:regular-spectral-layer-measure}.
\end{proof}

\begin{proof}[Proof of
Proposition~\ref{prop:perturbative-hierarchy-package}]
Let \(\mathcal P_{\mathrm{strat}}\) be defined by
\eqref{eq:explicit-Pstrat}.  It is open and prevalent by
Proposition~\ref{prop:standard-prevalence-Schur}.  Fix
\(\delta\in\mathcal P_{\mathrm{strat}}\).  By
\eqref{eq:Pstrat-quantitative-exhaustion}, choose \(R,q\ge1\) such that
\(\delta\in\mathcal P_{\mathrm{strat}}(R,q)\), and set
\[
 \varepsilon_{\mathrm{hier}}(\delta)
 :=\varepsilon_{\rm Y}(R,q,3)>0.
\]

Let \(0<\varepsilon<\varepsilon_{\mathrm{hier}}(\delta)\).  Apply
Corollaries~\ref{fr:cor:prevalent-multiwell-Young-regular}
and~\ref{fr:cor:prevalent-multiwell-divisor-hierarchy} with \(r_*=3\)
and set
\[
 \mathcal Y_s^{\varepsilon,\delta}
 :=\mathcal Y_{Q,s}^{\varepsilon,\delta},
 \qquad 1\le s\le N.
\]
Equation~\eqref{fr:eq:regular-spectral-layer-measure} gives
\eqref{eq:black-box-positive-spectral-set}.
Corollary~\ref{fr:cor:prevalent-multiwell-divisor-hierarchy} supplies, for every
\(E\in\mathcal Y_s^{\varepsilon,\delta}\), the hierarchy
\((K_j,m_j,\mathscr D_j,D_j,\mathscr A_j)_{j\ge j_{\rm hol}}\).
Taking \(j_0:=j_{\rm hol}\) and \(\chi:=g_Q\) gives the remaining
conclusions.
\end{proof}
\appendix
\section{Proof of the almost-Mathieu upper-edge estimates}
\label{app:AMO-upper-edge}

\noindent Throughout this appendix, fix the data of
Proposition~\ref{fr:prop:deterministic-AMO-seed} and write
\begin{equation}\label{eq:AMO-appendix-data}
\begin{gathered}
 \beta\in\DC(\kappa_\beta,\tau),\qquad
 H_x:=H_{\lambda,\beta,x},\qquad
 \Sigma_{\mu,\beta}:=\Sigma(2\mu\cos(2\pi\,\cdot),\beta)\quad(\mu>0),\\
 E_+:=E_+(\beta)=\sup\Sigma_{\lambda,\beta},\qquad
 A_m(x,E):=A_{m,\beta}^{2\lambda\cos(2\pi\,\cdot)}(x,E).
 \end{gathered}
\end{equation}
All return times and return matrices refer to \(T_\beta\), with the
conventions of Definition~\ref{fr:def:regular-phase-family} and
Subsection~\ref{subsec:induced-exponents}. On an atom with labels \(p,q\),
\[
 B_I^+(x,E)=A_p(x,E),\qquad
 B_I^-(x,E)=A_q(x-q\beta,E).
\]
For lattice restrictions, \(e_j\) denotes the coordinate vector at site \(j\).
Here \(c,C>0\) depend only on \((\lambda,\beta,\kappa_\beta,\tau)\).
Additional dependences are indicated explicitly. Constants in each
statement are fixed; elsewhere their values may change from one
occurrence to the next.

Fix, independently of every derivative order,
\[
 \sigma_{\rm cmp}=\frac1{2\tau},\qquad B=4,\qquad
 m^{1-1/\tau}\log^2(m+2)\le CR_\tau(m+2)=o(m).
\]
\subsection{The positive dual state and a finite-volume gap}
\label{amo:sec:dual}

For \(m\ge1\), put
\begin{equation}\label{amo:eq:P-definition}
 P_m(x,E):=\det(E\operatorname{Id}-H_x^{[0,m-1]}),
 \qquad P_0:=1,\qquad P_{-1}:=0.
\end{equation}
The Dirichlet recurrence gives
\begin{equation}\label{amo:eq:transfer-determinant-identity}
 A_m(x,E)=
 \begin{pmatrix}
 P_m(x,E)&-P_{m-1}(x+\beta,E)\\
 P_{m-1}(x,E)&-P_{m-2}(x+\beta,E)
 \end{pmatrix}.
\end{equation}

\subsubsection{Reducibility at the upper edge}
\begin{theorem}
\label{amo:thm:dual-edge-state-input}
There are a positive real analytic function \(\psi:\T\to\R\) and
an exponentially decaying real sequence \(u\) with the following properties.
\begin{enumerate}
\item The Fourier representation and normalization are
\begin{equation}\label{amo:eq:positive-dual-normalization}
 \psi(t)=\sum_{j\in\Z}u_j e^{2\pi ijt}>0,
 \qquad \int_\T\psi=1.
\end{equation}
\item The sequence \(u\) satisfies
\begin{equation}\label{amo:eq:edge-state-eq}
 H_0u=E_+u,
\end{equation}
\begin{equation}\label{amo:eq:edge-state-positive-even}
 u_0=1,\qquad 0<u_j<1\ (j\ne0),\qquad u_{-j}=u_j.
\end{equation}
For some \(C,c>0\), \(|u_j|\le Ce^{-c|j|}\).
\item The function \(\psi\) satisfies
\begin{equation}\label{amo:eq:dual-analytic-bloch}
 \lambda\bigl(\psi(t+\beta)+\psi(t-\beta)\bigr)
 +2\cos(2\pi t)\psi(t)=E_+\psi(t).
\end{equation}
\end{enumerate}
\end{theorem}

\begin{proof}
Set \(\lambda_*:=\lambda^{-1}\),
\(\widetilde E_+:=\lambda^{-1}E_+\), and
\[
 A_*(t):=\begin{pmatrix}
 \widetilde E_+-2\lambda_*\cos(2\pi t)&-1\\1&0
 \end{pmatrix}.
\]
Aubry duality and the IDS--rotation-number identity give
\[
 \Sigma_{\lambda,\beta}=\lambda\Sigma_{\lambda_*,\beta},
 \qquad N_*(\widetilde E_+)=1,
 \qquad \rho(\beta,A_*)=0
 \quad\text{\cite[\S2.3 and (2.13)]{AvilaJitomirskaya2010}}.
\]
The hypotheses of \cite[Theorems~1.4 and~4.1(3)]{AvilaJitomirskaya2010}
are satisfied because
\[
 0<\lambda_*<1=\lambda_0(2\cos(2\pi\,\cdot)),\qquad
 2\rho(\beta,A_*)=0\in\beta\Z+\Z.
\]
Thus there are analytic \(\mathcal B_0:\T\to\operatorname{PSL}(2,\R)\)
and a constant \(\mathcal C_0\) such that
\[
 \mathcal B_0(t+\beta)A_*(t)\mathcal B_0(t)^{-1}=\mathcal C_0,
 \qquad |\operatorname{tr}\mathcal C_0|\le2.
\]
The trace inequality follows from
\(\widetilde E_+\in\Sigma_{\lambda_*,\beta}\) and
\cite[\S2.2]{AvilaJitomirskaya2010}.

Let \(k=\deg\mathcal B_0\). The degree formula on
\(\operatorname{PSL}(2,\R)\) gives
\[
 2\rho(\beta,\mathcal C_0)\equiv k\beta\pmod\Z
 \quad\text{\cite[(2.9)]{AvilaJitomirskaya2010}}.
\]
If \(\mathcal C_0\) is elliptic, a constant conjugacy makes it
\(R_{2\pi\vartheta}\), where
\(2\vartheta-k\beta\in\Z\).  Consequently
\[
 \widehat{\mathcal B}(t):=R_{-\pi kt}\mathcal B_0(t),
 \qquad \deg\widehat{\mathcal B}=0,
\]
\[
 \widehat{\mathcal B}(t+\beta)A_*(t)
 \widehat{\mathcal B}(t)^{-1}
 =R_{2\pi\vartheta-\pi k\beta}
 =\operatorname{Id}\quad\text{in }\operatorname{PSL}(2,\R).
\]
If \(\mathcal C_0\) is parabolic or scalar, then
\(2\rho(\beta,\mathcal C_0)=0\pmod\Z\), hence
\(k\beta\in\Z\) and \(k=0\); take
\(\widehat{\mathcal B}=\mathcal B_0\).
Since \(\deg\widehat{\mathcal B}=0\), it has a one-periodic analytic
\(\operatorname{SL}(2,\R)\) lift \(\mathcal B\) homotopic to the identity.
The lifted conjugacy differs from a fixed lift of its constant value
by a continuous map \(\T\to\{\pm\operatorname{Id}\}\), which is constant.
Since \(\deg\mathcal B=0\),
\[
 \mathcal B(t+\beta)A_*(t)\mathcal B(t)^{-1}=\mathcal C,
 \qquad \rho(\beta,\mathcal C)=0.
\]
Thus \(\mathcal C v=v\) for some real \(v\ne0\).  Writing
\[
 \mathcal B(t)^{-1}v=
 \binom{\psi(t)}{\psi(t-\beta)}
\]
gives a nonzero real analytic solution of
\eqref{amo:eq:dual-analytic-bloch}.

Define
\[
 (\mathscr L_*f)(t):=f(t+\beta)+f(t-\beta)
                    +2\lambda_*\cos(2\pi t)f(t).
\]
Fourier transform identifies \(\mathscr L_*\) with
\(\lambda_*H_0\).  Therefore
\[
 \widetilde E_+\|\psi\|_2^2
 =\langle\psi,\mathscr L_*\psi\rangle
 \le\langle|\psi|,\mathscr L_*|\psi|\rangle
 \le\widetilde E_+\|\psi\|_2^2.
\]
Put \(Q_*:=\widetilde E_+-\mathscr L_*\ge0\).  Equality gives
\[
 \|Q_*^{1/2}|\psi|\|_2^2=0
 \quad\Longrightarrow\quad Q_*|\psi|=0.
\]
By continuity this equation holds pointwise; hence
\[
 |\psi(t_0)|=0
 \ \Longrightarrow\ |\psi(t_0\pm\beta)|=0
 \ \Longrightarrow\ \psi\equiv0.
\]
Choose the resulting constant sign so that \(\psi>0\).

Let \(a_j:=\widehat\psi_j\).  Then \(a_0>0\),
\(|a_j|\le Ce^{-c|j|}\), and \(H_0a=E_+a\).  Again,
\[
 E_+\|a\|_2^2
 =\langle a,H_0a\rangle
 \le\langle|a|,H_0|a|\rangle
 \le E_+\|a\|_2^2,
\]
Thus \(H_0|a|=E_+|a|\), and the recurrence gives
\[
 |a_j|=0\ \Longrightarrow\ |a_{j-1}|+|a_{j+1}|=0
 \ \Longrightarrow\ a_0=0,
 \qquad\text{so } |a_j|>0\quad(j\in\Z).
\]
Moreover,
\[
 0=2\sum_j\bigl(|a_j||a_{j+1}|
       -\operatorname{Re}(\overline{a_j}a_{j+1})\bigr)
 \quad\Longrightarrow\quad a_j>0\quad(j\in\Z).
\]
Reality of \(\psi\) now gives \(a_{-j}=a_j\).  For \(j\ne0\),
\[
 a_0^2-|a_j|^2
 =\frac12\int_{\T^2}\psi(s)\psi(t)
       |e^{2\pi ijs}-e^{2\pi ijt}|^2\,ds\,dt>0.
\]
Set \(u_j=a_j/a_0\) and replace \(\psi\) by \(\psi/a_0\).
\end{proof}

\subsubsection{The quadratic form and finite support}

We use the Fourier transform and quadratic form
\[
 \begin{gathered}
 (\mathcal Fv)(t):=\sum_jv_je^{2\pi ijt},\qquad
 \mathcal Q_x(v):=\langle v,(E_+-H_x)v\rangle,\\
 a(t):=\psi(t)\psi(t+\beta),\qquad a_*:=\min_\T a>0.
\end{gathered}
\]
For \(v\in\ell^2(\Z)\), write \(g=(\mathcal Fv)/\psi\).

\begin{lemma}
\label{amo:lem:magnetic-ground-form}
For every \(v\in\ell^2(\Z)\) and \(x\in\T\),
\begin{equation}\label{amo:eq:magnetic-ground-form}
 \mathcal Q_x(v)=\lambda\int_\T a(t)
       |e^{2\pi ix}g(t+\beta)-g(t)|^2\,dt,
\end{equation}
and
\begin{equation}\label{amo:eq:ground-fourier-lower}
 \mathcal Q_x(v)\ge16\lambda a_*
 \sum_{k\in\Z}\|x+k\beta\|_\T^2|\widehat g_k|^2.
\end{equation}
\end{lemma}

\begin{proof}
With \(T_\beta g(t):=g(t+\beta)\),
\[
 \mathcal FH_x\mathcal F^{-1}
 =2\cos(2\pi t)+\lambda e^{2\pi ix}T_\beta
                  +\lambda e^{-2\pi ix}T_{-\beta}.
\]
The diagonal part of the expanded square is
\[
 \begin{aligned}
 \lambda\int a(t)(|g(t+\beta)|^2+|g(t)|^2)\,dt
 &=\lambda\int\psi(t)(\psi(t-\beta)+\psi(t+\beta))|g(t)|^2\,dt\\
 &=\int(E_+-2\cos(2\pi t))|\psi(t)g(t)|^2\,dt.
 \end{aligned}
\]
The cross term is
\[
 -2\lambda\operatorname{Re}\int_\T e^{2\pi ix}
 \overline{\psi(t)g(t)}\psi(t+\beta)g(t+\beta)\,dt,
\]
which proves \eqref{amo:eq:magnetic-ground-form}. All operators involved
are bounded on \(L^2\). Finally, Parseval's identity gives
\[
 \mathcal Q_x(v)\ge\lambda a_*
 \sum_k|e^{2\pi i(x+k\beta)}-1|^2|\widehat g_k|^2,
 \qquad |e^{2\pi is}-1|\ge4\|s\|_\T.
\]
\end{proof}

\begin{lemma}
\label{amo:lem:strict-finite-edge}
For every finite integer interval \(J\) and every \(x\in\T\),
\begin{equation}\label{amo:eq:strict-finite-edge}
 E_+\operatorname{Id}-H_x^J>0.
\end{equation}
\end{lemma}

\begin{proof}
Extend \(v\in\ell^2(J)\) by zero.  If \(\mathcal Q_x(v)=0\),
\eqref{amo:eq:ground-fourier-lower} gives
\[
 \operatorname{supp}\widehat g
 \subset\{k:x+k\beta\in\Z\},
 \qquad \#\{k:x+k\beta\in\Z\}\le1.
\]
Thus either \(v=0\) or \(v_j=z u_{j-k}\) for some \(z\ne0\).
Since every \(u_j>0\), the latter is incompatible with finite support.
\end{proof}

For a finite integer interval \(J=[b_-,b_+]\cap\Z\), write
\(L=|J|\). For \(K\in\mathbb N\), we will use the following condition
on \(x\in\T\) and \(\delta>0\):
\begin{equation}\label{amo:eq:single-resonance-hypotheses}
 0<\delta\le\frac{\kappa_\beta}{4(L+2K)^\tau},
 \qquad \|x+j\beta\|_\T\ge\delta\quad(b_-<j<b_+).
\end{equation}
No restriction is imposed at the endpoints of \(J\).

\begin{lemma}
\label{amo:lem:single-resonance-gap}
There are \(K\in\mathbb N\) and \(c>0\) such that, for every
\(J,x,\delta\) satisfying \eqref{amo:eq:single-resonance-hypotheses},
\begin{equation}\label{amo:eq:single-resonance-gap}
 E_+-\max\sigma(H_x^J)\ge c\delta^2.
\end{equation}
\end{lemma}

\begin{proof}
Set
\[
 M_\psi:=\max_\T\psi,\quad U:=\|u\|_2,\quad
 C_*:=M_\psi(1+U/u_1).
\]
Since \(\min_\T\psi>0\), the reciprocal \(1/\psi\) is holomorphic in
a complex strip of positive width.  Its Fourier coefficients satisfy
\[
 b_j:=\widehat{(1/\psi)}_j,\qquad
 |b_j|\le Ce^{-c|j|},\qquad
 \mathfrak b_K:=\sum_{|j|>K}|b_j|\le Ce^{-cK}.
\]
Fix
\[
 K\ge1\quad\text{with }C_*\mathfrak b_K\le\tfrac12,
 \qquad c:=4\lambda a_*/C_*^2.
\]
For \(\|v\|_2=1\), \(\operatorname{supp}v\subset J\), let
\(\widehat g=b*v\) and \(J^K=[b_--K,b_++K]\cap\Z\).
Young's inequality gives
\begin{equation}\label{amo:eq:fixed-padding-tail}
 \|\widehat g\|_{\ell^2(\Z\setminus J^K)}
 \le\|b\mathbf1_{\{|j|>K\}}\|_1\|v\|_2=\mathfrak b_K.
\end{equation}
Let \(\mathcal R:=\{k\in J^K:\|x+k\beta\|_\T<\delta\}\).
Then \(\#\mathcal R\le1\), since two such indices give
\[
 \kappa_\beta(L+2K)^{-\tau}
 \le\|(k-\ell)\beta\|_\T
 <2\delta\le\tfrac12\kappa_\beta(L+2K)^{-\tau}.
\]
By \eqref{amo:eq:ground-fourier-lower},
\[
 \|\widehat g\|_{\ell^2(J^K\setminus\mathcal R)}
 \le\frac{\mathcal Q_x(v)^{1/2}}
                {4\sqrt{\lambda a_*}\,\delta}.
\]
When \(\mathcal R=\{k_0\}\), set \(z=\widehat g_{k_0}\) and
\(u^{(k_0)}_j=u_{j-k_0}\).  Then
\[
 v=z u^{(k_0)}+e,\qquad
 \|e\|_2\le M_\psi\left(
 \frac{\mathcal Q_x(v)^{1/2}}{4\sqrt{\lambda a_*}\,\delta}
 +\mathfrak b_K\right).
\]
The hypothesis puts \(k_0\) outside the strict interior of \(J\).
Choose a zero coordinate of \(v\) by
\[
 j_*=
 \begin{cases}
 k_0,&k_0\notin J,\\
 b_--1,&k_0=b_-,\\
 b_++1,&k_0=b_+>b_-.
 \end{cases}
 \qquad
 |z|u_1\le |z u_{j_*-k_0}|=|e_{j_*}|\le\|e\|_2.
\]
The same choice applies when \(L=1\), and hence
\[
 1=\|v\|_2\le |z|U+\|e\|_2
 \le(1+U/u_1)\|e\|_2.
\]
If \(\mathcal R=\varnothing\), then \(v=e\) and the corresponding
bound holds with \(M_\psi\) in place of \(C_*\). In either case,
\[
 1\le C_*\left(
 \frac{\mathcal Q_x(v)^{1/2}}{4\sqrt{\lambda a_*}\,\delta}
 +\mathfrak b_K\right)
 \quad\Longrightarrow\quad \mathcal Q_x(v)\ge c\delta^2.
\]
The variational principle now gives \eqref{amo:eq:single-resonance-gap}.
\end{proof}

We apply this estimate to the intervals \(J=[1,m]\) and
\(J=[-m,-1]\), where \(m\ge1\).

\begin{corollary}
\label{amo:cor:anchored-side-gap}
There is \(c>0\) such that, for every \(m\ge1\)
and either interval \(J\) above, the following hold whenever
\(|x|\le c(m+1)^{-\tau}\).
\begin{enumerate}
\item The finite-volume gap satisfies
\begin{equation}\label{amo:eq:anchored-side-gap}
 E_+-\max\sigma(H_x^J)\ge c(m+1)^{-2\tau}.
\end{equation}
\item If also \(|E-E_+|\le\tfrac12c(m+1)^{-2\tau}\), then
\begin{equation}\label{amo:eq:anchored-side-resolvent}
 E-H_x^J>0,\qquad
 \|(E-H_x^J)^{-1}\|\le2c^{-1}(m+1)^{2\tau}.
\end{equation}
\end{enumerate}
\end{corollary}

\begin{proof}
Let \(\widehat c>0\) be the gap constant in
Lemma~\ref{amo:lem:single-resonance-gap}. Choose
\[
 c=\min\left\{\frac{\kappa_\beta}{8(1+2K)^\tau},
             \frac{\widehat c\kappa_\beta^2}{16(1+2K)^{2\tau}}\right\},
 \qquad \delta=\frac{\kappa_\beta}{4(m+2K)^\tau}.
\]
For \(1\le|j|\le m\),
\[
 \|x+j\beta\|_\T
 \ge\kappa_\beta m^{-\tau}-|x|
 \ge\kappa_\beta m^{-\tau}-\frac{\kappa_\beta}{8(m+2K)^\tau}
 \ge\delta.
\]
Lemma~\ref{amo:lem:single-resonance-gap} gives
\[
 E_+-\max\sigma(H_x^J)\ge\widehat c\delta^2
 \ge c(m+1)^{-2\tau}.
\]
Hence
\[
 E-H_x^J\ge
 \bigl(c(m+1)^{-2\tau}-|E-E_+|\bigr)\operatorname{Id}.
\]
\end{proof}

\subsection{Eigenvalue counting and sharp decay}
\label{amo:sec:counting}

Let \(N\) be the IDS and, counting multiplicities, put
\[
 K_r(x,E):=\#\{\nu\in\sigma(H_x^{[0,r-1]}):\nu\le E\},
 \qquad F_r(E):=\int_\T K_r(x,E)\,dx.
\]

\begin{lemma}
\label{amo:lem:uniform-counting-discrepancy}
For every \(m\ge2\),
\begin{equation}\label{amo:eq:uniform-counting-discrepancy}
 \sup_{x\in\T,E\in\R}|K_m(x,E)-mN(E)|
 \le C m^{1-1/\tau}\log(m+2).
\end{equation}
\end{lemma}

\begin{proof}
The highest Fourier coefficient of \(P_r(\cdot,E)\) is
\[
 \widehat P_r(r,E)=(-\lambda)^r
                 e^{2\pi i\beta\sum_{j=0}^{r-1}j}\ne0.
\]
Put \(Z_r(E):=\{x\in\T:P_r(x,E)=0\}\).  Then
\[
 \deg P_r(\cdot,E)=r,\qquad
 \sum_{x\in Z_r(E)}\operatorname{ord}_xP_r(\cdot,E)\le2r.
\]
Dirichlet Jacobi eigenvalues are simple. Thus, near each
\(x_0\in Z_r(E)\), there is exactly one analytic eigenvalue branch
\(\nu(x)\) with \(\nu(x_0)=E\), and
\[
 \operatorname{ord}_{x_0}P_r(\cdot,E)
 =\operatorname{ord}_{x_0}(\nu-E).
\]
The indicator \(\mathbf1_{\{\nu\le E\}}\) contributes at most one
to the variation at a crossing and at most two at a tangency.
Since a tangency has order at least two, summing gives
\begin{equation}\label{amo:eq:counting-BV}
 \operatorname{Var}_\T K_r(\cdot,E)
 \le\sum_{x\in Z_r(E)}\operatorname{ord}_xP_r(\cdot,E)\le2r.
\end{equation}

For the direct sum of \(b\) consecutive length-\(r\) restrictions,
\[
 \operatorname{rank}\!\left(H_x^{[0,br-1]}
 -\bigoplus_{j=0}^{b-1}H_x^{[jr,(j+1)r-1]}\right)\le2(b-1).
\]
The eigenvalue counting inequality and phase invariance give
\[
 |F_{br}(E)-bF_r(E)|\le2(b-1).
\]
For every polynomial \(p\), finite propagation gives
\[
 \left|\frac1m\int_\T\operatorname{tr}p(H_x^{[0,m-1]})\,dx
          -\int p\,dN\right|\le\frac{C_{26}(p)}{m}.
\]
Polynomial approximation on the common compact spectral interval and
continuity of \(N\) \cite[Theorem~1.2]{AvilaJitomirskaya2010} give
\(F_m(E)/m\to N(E)\) for every \(E\).
Letting \(b\to\infty\) yields
\begin{equation}\label{amo:eq:averaged-counting-boundary}
 |F_r(E)-rN(E)|\le2.
\end{equation}

Let \(q_s\) be the continued fraction denominators of \(\beta\).
The Diophantine bound and the Ostrowski digits of
\(M=\sum_s b_sq_s\) give
\[
 \begin{aligned}
 q_{s+1}&\le\kappa_\beta^{-1}q_s^\tau,\\
 b_s&\le\min\{M/q_s,Cq_s^{\tau-1}\}
      \le CM^{1-1/\tau},\\
 \#\{s:b_s\ne0\}&\le C\log(M+2).
 \end{aligned}
\]
For a periodic function \(f\) of bounded variation, Denjoy--Koksma gives
\[
 \left|\sum_{a=0}^{q_s-1}f(y+a\beta)-q_s\int_\T f\right|
 \le\operatorname{Var}_\T(f).
\]
Summing over the Ostrowski blocks yields
\begin{equation}\label{amo:eq:BV-orbit-bound}
 \left|\sum_{a=0}^{M-1}f(x+a\beta)-M\int_\T f\right|
 \le C\operatorname{Var}_\T(f)M^{1-1/\tau}\log(M+2).
\end{equation}

Set \(\mathfrak d_j:=\sup_{x,E}|K_j(x,E)-jN(E)|\),
\(\mathfrak d_0:=0\).  For \(2\le r\le m/2\), write
\[
 m=t+k_tr+s_t,\quad 0\le t<r,\quad0\le s_t<r,
 \qquad M:=m-r+1.
\]
Cut into an initial length \(t\), \(k_t\) complete lengths \(r\),
and a final length \(s_t\).  Their complete starts satisfy
\[
 \bigsqcup_{t=0}^{r-1}\{t+jr:0\le j<k_t\}
 =\{0,\ldots,m-r\},\qquad\sum_{t=0}^{r-1}k_t=M.
\]
The cut errors and the two endpoint deviations yield
\[
 \begin{aligned}
 r\mathfrak d_m\le{}&
 \sup_{x,E}\left|\sum_{a=0}^{M-1}
                    (K_r(x+a\beta,E)-rN(E))\right|\\
 &+2r\max_{j<r}\mathfrak d_j+2(M+r).
 \end{aligned}
\]
By \eqref{amo:eq:counting-BV}--\eqref{amo:eq:BV-orbit-bound},
\begin{equation}\label{amo:eq:counting-recursion}
 \mathfrak d_m\le C m^{1-1/\tau}\log(m+2)
                +Cm/r+2\max_{j<r}\mathfrak d_j.
\end{equation}
Take \(r=\lceil m^{1/\tau}\rceil\) for large \(m\) and
\(\Phi(t):=(t+2)^{1-1/\tau}\log(t+2)\).  Then
\[
 m/r\le m^{1-1/\tau},\qquad
 \frac{\Phi(r)}{\Phi(m)}\longrightarrow0.
\]
Fix \(m_0\) with \(2\Phi(r)\le\Phi(m)/2\) for \(m\ge m_0\).
Writing the nonrecursive part of
\eqref{amo:eq:counting-recursion} as at most \(C\Phi(m)\), choose
\[
 \widehat C\ge\max\left\{2C,
       \max_{0\le j<m_0}\frac{\mathfrak d_j}{\Phi(j)}\right\}.
\]
Strong induction closes by
\[
 \mathfrak d_m\le C\Phi(m)+2\widehat C\Phi(r)
 \le(C+\widehat C/2)\Phi(m)\le \widehat C\Phi(m).
\]
\end{proof}

\begin{proposition}
\label{amo:prop:determinant-estimates}
For \(m\ge2\), the following estimates hold.
\begin{enumerate}
\item Uniformly in the phase,
\begin{equation}\label{amo:eq:uniform-determinant-upper}
 \sup_{x\in\T}\log P_m(x,E_+)
 \le\gamma m+C m^{1-1/\tau}\bigl(\log(m+2)\bigr)^2.
\end{equation}
\item If \(0<\eta\le\min\{1,E_+-\max\sigma(H_x^{[0,m-1]})\}\), then
\begin{equation}\label{amo:eq:gap-determinant-lower}
 \log P_m(x,E_+)
 \ge\gamma m-Cm^{1-1/\tau}\log(m+2)\log\frac C\eta.
\end{equation}
\end{enumerate}
\end{proposition}

\begin{proof}
Let \(\nu_{m,x}\) count the eigenvalues of the block, with total
mass \(m\), and set \(E_{\rm lo}:=-2\lambda-3\).  Both
\(\nu_{m,x}\) and \(m\,dN\) are supported in
\([E_{\rm lo},E_+]\), and their cumulative difference vanishes at
the endpoints.  For \(0<\eta\le1\), put
\[
 g_\eta(t):=\log\max\{E_+-t,\eta\},\qquad
 \operatorname{Var}_{[E_{\rm lo},E_+]}g_\eta\le\log(C/\eta).
\]
Stieltjes integration by parts and
Lemma~\ref{amo:lem:uniform-counting-discrepancy} give
\begin{equation}\label{amo:eq:truncated-log-comparison}
 \left|\int g_\eta\,d\nu_{m,x}-m\int g_\eta\,dN\right|
 \le C m^{1-1/\tau}\log(m+2)\log\frac C\eta.
\end{equation}
The Thouless formula \cite[(2.12)]{AvilaJitomirskaya2010} and
the AMO exponent \cite[Corollary~2]{BourgainJitomirskaya2002} give
\[
 \int\log(E_+-t)\,dN(t)=L(E_+)=\log\lambda=\gamma.
\]
By \cite[Theorem~1.2]{AvilaJitomirskaya2010},
\(1-N(E_+-s)\le C\sqrt s\).  Tonelli's theorem yields
\begin{equation}\label{amo:eq:IDS-truncation-error}
 0\le\int g_\eta\,dN-\gamma
 =\int_0^\eta\frac{1-N(E_+-s)}s\,ds
 \le C\sqrt\eta.
\end{equation}
With \(\eta=m^{-2}\),
\[
 \begin{aligned}
 \log P_m(x,E_+)
 &\le\int g_\eta\,d\nu_{m,x}\\
 &\le\gamma m+Cm\sqrt\eta
       +Cm^{1-1/\tau}\log(m+2)\log(C/\eta),
 \end{aligned}
\]
which proves \eqref{amo:eq:uniform-determinant-upper}.  Under the
gap hypothesis, truncation is inactive on the finite spectrum, so
\[
 \log P_m(x,E_+)=\int g_\eta\,d\nu_{m,x}
 \ge m\int g_\eta\,dN
       -Cm^{1-1/\tau}\log(m+2)\log(C/\eta).
\]
Since \(\int g_\eta\,dN\ge\gamma\), the lower bound follows.
\end{proof}

\subsubsection{The Wronskian and phase-zero estimates}

Put
\begin{equation}\label{amo:eq:Q-def}
 Q_j:=P_j(\beta,E_+)
     =\det(E_+\operatorname{Id}-H_0^{[1,j]}),
 \qquad Q_0=1,\qquad Q_{-1}=0.
\end{equation}

\begin{proposition}
\label{amo:prop:sharp-anchor}
\begin{enumerate}
\item For every \(j\ge0\),
\begin{align}
 e^{\gamma j-CR_\tau(j+2)}
 &\le Q_j\le e^{\gamma j+CR_\tau(j+2)},
 \label{amo:eq:sharp-Q}\\
 e^{-\gamma j-CR_\tau(j+2)}
 &\le u_j\le C e^{-\gamma j+CR_\tau(j+2)},
 \label{amo:eq:sharp-u}\\
 \|u\|_{\ell^2(\{k>j\})}
 &\le C e^{-\gamma j+CR_\tau(j+2)}.
 \label{amo:eq:sharp-u-tail}
\end{align}
\item The Wronskian identity and series representation are
\begin{equation}\label{amo:eq:W-Q-u}
 Q_ju_j-Q_{j-1}u_{j+1}=1\qquad(j\ge0),
\end{equation}
\begin{equation}\label{amo:eq:minimal-solution-series}
 u_j=Q_{j-1}\sum_{\ell=j}^{\infty}
                  \frac1{Q_{\ell-1}Q_\ell}\qquad(j\ge1).
\end{equation}
\end{enumerate}
\end{proposition}

\begin{proof}
Corollary~\ref{amo:cor:anchored-side-gap} and
Proposition~\ref{amo:prop:determinant-estimates}, with
\(\eta=\min\{1,c(j+1)^{-2\tau}\}\), give
\[
 |\log Q_j-\gamma j|
 \le Cj^{1-1/\tau}\bigl(\log(j+2)\bigr)^2
 \le CR_\tau(j+2)\qquad(j\ge2).
\]
This proves \eqref{amo:eq:sharp-Q}, after adjusting the initial cases.

Let \(\varphi\) solve \(H_0\varphi=E_+\varphi\) with
\(\varphi_0=0\), \(\varphi_1=1\).  Then
\(\varphi_{j+1}=Q_j\), and its Wronskian with \(u\) gives
\[
 Q_ju_j-Q_{j-1}u_{j+1}
 =\varphi_{j+1}u_j-\varphi_ju_{j+1}
 =\varphi_1u_0-\varphi_0u_1=1.
\]
In particular, \(u_j\ge Q_j^{-1}\).  Division and telescoping yield
\[
 \frac{u_j}{Q_{j-1}}-\frac{u_{L+1}}{Q_L}
 =\sum_{\ell=j}^L\frac1{Q_{\ell-1}Q_\ell},
 \qquad \frac{u_{L+1}}{Q_L}\longrightarrow0,
\]
because \(u\in\ell^2\) and \(Q_L\to\infty\).
This proves \eqref{amo:eq:minimal-solution-series}.

Since \(R_\tau'(t)\to0\), for the constant \(C\) in
\eqref{amo:eq:sharp-Q} and all sufficiently large \(j\),
\[
 6C\bigl(R_\tau(\ell+2)-R_\tau(j+2)\bigr)
 \le\frac\gamma2(\ell-j)\qquad(\ell\ge j).
\]
Applying \eqref{amo:eq:sharp-Q} to the series gives
\[
 \begin{aligned}
 u_j
 &\le\sum_{\ell=j}^{\infty}
   e^{-\gamma j-2\gamma(\ell-j)+3CR_\tau(\ell+2)}\\
 &\le e^{-\gamma j+3CR_\tau(j+2)}
        \sum_{h=0}^{\infty}e^{-3\gamma h/2}.
 \end{aligned}
\]
Using the same increment bound,
\[
 \sum_{k>j}u_k^2
 \le C e^{-2\gamma j+6CR_\tau(j+2)}
                \sum_{h\ge1}e^{-3\gamma h/2}.
\]
These estimates and \(u_j\ge Q_j^{-1}\) prove the remaining
bounds; enlarging \(C\) covers the finite initial range.
\end{proof}

\subsection{The central Schur complement and its second phase derivative}
\label{amo:sec:schur}

For $l,r\ge1$, set
\[
 \Lambda=[-l,r]\cap\mathbb Z,\qquad D=\Lambda\setminus\{0\},
 \qquad M=l+r+1,\qquad d=\min\{l,r\}.
\]
Write $H_x^-=H_x^{[-l,-1]}$, $H_x^+=H_x^{[1,r]}$, and define
\begin{align}
 R^\pm(x,E)&=(E-H_x^\pm)^{-1},\notag\\
 h_l^-(x,E)&=\langle e_{-1},R^-(x,E)e_{-1}\rangle,\notag\\
 h_r^+(x,E)&=\langle e_1,R^+(x,E)e_1\rangle,
 \label{amo:eq:central-weyl}\\
 S_{l,r}(x,E)&=E-2\lambda\cos(2\pi x)-h_l^-(x,E)-h_r^+(x,E).
 \label{amo:eq:central-schur-definition}
\end{align}
The vector
\begin{equation}\label{amo:eq:central-minimizer}
 f_0=1,\qquad f|_{[-l,-1]}=R^-e_{-1},\qquad
 f|_{[1,r]}=R^+e_1
\end{equation}
satisfies
\begin{equation}\label{amo:eq:central-stationarity}
 (E-H_x^\Lambda)f=S_{l,r}e_0.
\end{equation}
For real parameters with $E-H_x^\pm>0$ and vectors satisfying
$v_0=1$, $\operatorname{supp}v\subset\Lambda$, completion of the square gives
\begin{equation}\label{amo:eq:central-variational}
 \begin{aligned}
 \langle v,(E-H_x)v\rangle
 &=S_{l,r}+\langle v-f,(E-H_x^\Lambda)(v-f)\rangle,\\
 S_{l,r}&=\min_{\substack{v_0=1\\\operatorname{supp}v\subset\Lambda}}
                 \langle v,(E-H_x)v\rangle.
 \end{aligned}
\end{equation}
Vectors supported on $\Lambda$ are extended by zero to $\mathbb Z$.
For the same normalized positive dual state, use the coefficient $a$ in
\eqref{amo:eq:magnetic-ground-form} and set
\begin{equation}\label{amo:eq:effective-curvature}
 \mathcal A_\psi=\int_\T a(t)^{-1}\,dt,\qquad
 \mathcal K_\psi=\frac{8\pi^2\lambda}{\mathcal A_\psi}>0.
\end{equation}

Set $B_{14}:=\lceil8\tau+2\rceil$. For $c_*>0$, let
\begin{equation}\label{amo:eq:central-polynomial-bidisk}
 \mathcal D_M=
 \{(z,E)\in\mathbb C^2:|z|\le c_*M^{-B_{14}},\quad
                         |E-E_+|\le c_*M^{-B_{14}}\}.
\end{equation}
\begin{proposition}
\label{amo:prop:central-schur}
There are $c_*,c,C>0$ and $M_0$, depending only on the fixed data,
such that the following hold whenever $M\ge M_0$ and $d\ge M/4$.
\begin{enumerate}
\item The resolvents, Weyl values, $f$, and $S_{l,r}$ are holomorphic on a
neighborhood of $\mathcal D_M$. On $\mathcal D_M\cap\mathbb R^2$,
\begin{equation}\label{amo:eq:central-side-gap}
 E-\max\sigma(H_x^\pm)\ge cM^{-2\tau}.
\end{equation}
\item For every integer $k\ge0$ and $p+q\le k$, uniformly on $\mathcal D_M$,
\begin{equation}\label{amo:eq:central-polynomial-jets}
 \begin{split}
 &\|\partial_z^p\partial_E^qR^-\|
  +\|\partial_z^p\partial_E^qR^+\|
  +\|\partial_z^p\partial_E^qf\|_2\\
 &\qquad+|\partial_z^p\partial_E^qh_l^-|
  +|\partial_z^p\partial_E^qh_r^+|
  +|\partial_z^p\partial_E^qS_{l,r}|
  \le C_{14}(k)M^{2\tau(k+1)}.
 \end{split}
\end{equation}
The constants $c_*,M_0$ and the bidisk are independent of $k$.
\end{enumerate}
\end{proposition}

The constants in Proposition~\ref{amo:prop:central-schur} can be chosen
so that the following estimates hold on the same bidisk.

\begin{corollary}
\label{amo:cor:central-schur-curvature}
Under the hypotheses of Proposition~\ref{amo:prop:central-schur}, the following hold.
\begin{enumerate}
\item At $(0,E_+)$,
\begin{align}
 0<S_{l,r}&\le Ce^{-cd},
 &|\partial_xS_{l,r}|&\le Ce^{-cd},
 \label{amo:eq:central-small-affine}\\
 \partial_ES_{l,r}&=\|u\|_2^2+O(e^{-cd}),
 &\partial_x^2S_{l,r}&=\mathcal K_\psi+O(e^{-cd}),
 \label{amo:eq:central-explicit-limits}\\
 h_l^-&=u_1+O(e^{-cd}),
 &h_r^+&=u_1+O(e^{-cd}).
 \label{amo:eq:central-weyl-limits}
\end{align}
\item On $\mathcal D_M\cap\mathbb R^2$,
\begin{equation}\label{amo:eq:central-positive-derivatives}
 1\le\partial_ES_{l,r}\le C,\qquad
 c\le\partial_x^2S_{l,r}\le C,\qquad
 c\le h_l^-,h_r^+\le C.
\end{equation}
\end{enumerate}
\end{corollary}

We prove Proposition~\ref{amo:prop:central-schur} and
Corollary~\ref{amo:cor:central-schur-curvature} together.

\begin{proof}
We first construct a corrector for the phase derivative. Put
$\theta=2\pi x$, $Q(\theta)=E_+-H_{\theta/(2\pi)}$,
and $D_\beta h=h(\cdot+\beta)-h$;
dots denote $\theta$-derivatives at zero.
Lemma~\ref{amo:lem:magnetic-ground-form} gives
\begin{equation}\label{amo:eq:central-magnetic-form}
 \langle v,Q(\theta)v\rangle
 =\lambda\int_\T a\,|e^{i\theta}g(\cdot+\beta)-g|^2,
 \qquad \mathcal Fv=\psi g.
\end{equation}
Since $a>0$ is analytic, for some $\rho>0$ the function
\[
 b=i\bigl((\mathcal A_\psi a)^{-1}-1\bigr)
 \quad\hbox{satisfies}\quad
 \widehat b_0=0,\qquad |\widehat b_k|\le Ce^{-2\pi\rho|k|}.
\]
Taking $\rho/2$ as the smaller strip width, the Diophantine inequality yields
\[
 \widehat h_k=\frac{\widehat b_k}{e^{2\pi ik\beta}-1}\quad(k\ne0),
 \qquad |\widehat h_k|\le C|k|^\tau e^{-2\pi\rho|k|}
                  \le Ce^{-\pi\rho|k|}.
\]
Set
\[
 \widehat h_0=-\sum_{k\ne0}\widehat\psi_{-k}\widehat h_k,
 \qquad \widehat\psi_0=\int_\T\psi=1.
\]
Then
\begin{equation}\label{amo:eq:magnetic-corrector-equation}
 D_\beta h=i\bigl((\mathcal A_\psi a)^{-1}-1\bigr),
 \qquad \int_\T\psi h=0,
\end{equation}
and
\begin{equation}\label{amo:eq:magnetic-corrector-vector}
 w=\mathcal F^{-1}(\psi h),\qquad w_0=0,\qquad
 |u_j|+|w_j|\le Ce^{-c|j|}
\end{equation}
for a fixed $c>0$.
Polarization and differentiation of
\eqref{amo:eq:central-magnetic-form} give, for $k\in L^2(\T)$,
\[
 \begin{aligned}
 \langle\mathcal F^{-1}(\psi k),Q(0)w+\dot Q(0)u\rangle
 &=\lambda\int_\T a\,\overline{D_\beta k}\,(D_\beta h+i)\\
 &=\frac{i\lambda}{\mathcal A_\psi}
       \int_\T\overline{D_\beta k}=0.
 \end{aligned}
\]
Multiplication by $\psi$ is invertible on $L^2$, so
\begin{equation}\label{amo:eq:corrector-linear-equation}
 Q(0)w+\dot Q(0)u=0.
\end{equation}
Moreover,
\[
 \begin{gathered}
 \ker Q(0)
 =\mathcal F^{-1}\{\psi g:D_\beta g=0\}
 =\mathbb Cu,\\
 Q(0)(\overline w-w)=0,\qquad
 (\overline w-w)_0=0
 \quad\Longrightarrow\quad w=\overline w.
 \end{gathered}
\]
For $h_1\in L^2(\T)$, the effective quadratic coefficient follows from
\begin{equation}\label{amo:eq:corrector-effective-minimum}
 \int_\T a|i+D_\beta h_1|^2
 \ge\frac{|\int_\T(i+D_\beta h_1)|^2}{\int_\T a^{-1}}
 =\mathcal A_\psi^{-1}
 =\int_\T a|i+D_\beta h|^2.
\end{equation}
Indeed, \eqref{amo:eq:central-magnetic-form} along $u+\theta w$ and
\eqref{amo:eq:corrector-linear-equation} give
\begin{equation}\label{amo:eq:corrector-hessian}
 \begin{aligned}
 \frac{2\lambda}{\mathcal A_\psi}
 &=\left.\frac{d^2}{d\theta^2}
   \langle u+\theta w,Q(\theta)(u+\theta w)\rangle\right|_{\theta=0}\\
 &=\langle u,\ddot Q(0)u\rangle
    +4\operatorname{Re}\langle w,\dot Q(0)u\rangle
    +2\langle w,Q(0)w\rangle\\
 &=\langle u,\ddot Q(0)u\rangle
    +2\operatorname{Re}\langle w,\dot Q(0)u\rangle.
 \end{aligned}
\end{equation}

We next compare the corrector with the finite-volume minimizer.
Let $P_\Lambda,P_D$ be coordinate restriction and set
\[
 Q_\Lambda=P_\Lambda Q(0)P_\Lambda^*,\qquad
 \dot Q_\Lambda=P_\Lambda\dot Q(0)P_\Lambda^*,\qquad
 Q_D=P_DQ(0)P_D^*.
\]
Corollary~\ref{amo:cor:anchored-side-gap} gives
\begin{equation}\label{amo:eq:central-side-inverse-at-zero}
 Q_D\ge cM^{-2\tau}I,\qquad \|Q_D^{-1}\|\le CM^{2\tau}.
\end{equation}
For a sequence $v$, put
$b_\Lambda(v)=v_{-l-1}e_{-l}+v_{r+1}e_r$.
Writing $u^\Lambda=P_\Lambda u$, $w^\Lambda=P_\Lambda w$,
$f=f_{l,r}(0,E_+)$ and
$\dot f=\left.\partial_\theta f_{l,r}(\theta/(2\pi),E_+)\right|_0$,
we have
\[
 \begin{aligned}
 Q_\Lambda u^\Lambda&=b_\Lambda(u),\\
 Q_\Lambda w^\Lambda+\dot Q_\Lambda u^\Lambda&=b_\Lambda(w),\\
 Q_D(f-u^\Lambda)|_D&=-b_\Lambda(u)|_D,\\
 Q_D(\dot f-w^\Lambda)|_D
 &=-P_D\dot Q(0)(f-u^\Lambda)-b_\Lambda(w)|_D.
 \end{aligned}
\]
Here $f_0=u_0=1$, $\dot f_0=w_0=0$, and $\dot Q(0)$ is diagonal.
Thus, including the omitted tails,
\begin{equation}\label{amo:eq:central-minimizer-convergence}
 \|f-u\|_2\le CM^{2\tau}e^{-cd},
\end{equation}
\begin{equation}\label{amo:eq:central-corrector-convergence}
 \|\dot f-w\|_2
 \le CM^{2\tau}\bigl(\|f-u\|_2+e^{-cd}\bigr)
 \le CM^{4\tau}e^{-cd}.
\end{equation}
Since $M\le4d$, decreasing $c>0$ gives
\[
 \|f-u\|_2+\|\dot f-w\|_2\le Ce^{-cd}.
\]

For real parameters, stationarity and $f_0=1$ give
\begin{align}
 \partial_ES_{l,r}&=\langle f,f\rangle\ge1,
 \label{amo:eq:schur-energy-envelope}\\
 \partial_\theta S_{l,r}&=\langle f,(\partial_\theta Q)f\rangle,
 \notag\\
 \partial_\theta^2S_{l,r}|_{(0,E_+)}
 &=\langle f,\ddot Q(0)f\rangle
       +2\operatorname{Re}\langle\dot f,\dot Q(0)f\rangle.
 \label{amo:eq:schur-finite-hessian}
\end{align}
In the last two lines $E=E_+$.
Since $u_{-j}=u_j$ and $(\dot Q(0))_{jj}=2\lambda\sin(2\pi j\beta)$,
\[
 \langle u,\dot Q(0)u\rangle=0,\qquad
 E_+-2\lambda=2u_1.
\]
The boundary estimates, the boundedness of $\dot Q(0),\ddot Q(0)$,
and \eqref{amo:eq:corrector-hessian} yield
\[
 \begin{gathered}
 h_l^-=f_{-1}=u_1+O(e^{-cd}),\qquad
 h_r^+=f_1=u_1+O(e^{-cd}),\\
 S_{l,r}(0,E_+)=O(e^{-cd}),\qquad
 \partial_xS_{l,r}(0,E_+)=O(e^{-cd}),\\
 \partial_ES_{l,r}(0,E_+)=\|u\|_2^2+O(e^{-cd}),\qquad
 \partial_x^2S_{l,r}(0,E_+)
 =(2\pi)^2\frac{2\lambda}{\mathcal A_\psi}+O(e^{-cd}).
 \end{gathered}
\]
Lemma~\ref{amo:lem:strict-finite-edge} and
\eqref{amo:eq:central-variational} give $S_{l,r}(0,E_+)>0$.

It remains to extend these estimates to $\mathcal D_M$.
For $A_\pm(z,E)=E-H_z^\pm$ and $|z|\le1$, all positive-order derivatives of
$A_\pm$ have operator norm bounded independently of $M$, and
\[
 \|A_\pm(z,E)-A_\pm(0,E_+)\|
 \le C(|z|+|E-E_+|).
\]
Choose $c_*>0$ so that, on the bidisk with twice the radii of
$\mathcal D_M$,
\[
 \|A_\pm(0,E_+)^{-1}(A_\pm(z,E)-A_\pm(0,E_+))\|
 \le Cc_*M^{2\tau-B_{14}}\le\tfrac12.
\]
The Neumann series gives holomorphy and $\|R^\pm\|\le CM^{2\tau}$;
the same perturbation estimate for self-adjoint matrices gives
\eqref{amo:eq:central-side-gap}.
For the derivative estimates with fixed $k$, write $C=C(k)$ for
auxiliary constants whose values may increase. For a nonzero two-variable multi-index $\zeta$,
\[
 \partial^\zeta R^\pm
 =\sum_{\ell=1}^{|\zeta|}(-1)^\ell
   \sum_{\substack{\zeta_1+\cdots+\zeta_\ell=\zeta\\|\zeta_i|\ge1}}
   \frac{\zeta!}{\zeta_1!\cdots\zeta_\ell!}
   R^\pm(\partial^{\zeta_1}A_\pm)R^\pm\cdots
       (\partial^{\zeta_\ell}A_\pm)R^\pm.
\]
Consequently
\[
 \|\partial^\zeta R^\pm\|
 \le C\sum_{\ell=1}^{|\zeta|}M^{2\tau(\ell+1)}
 \le CM^{2\tau(k+1)},\qquad 1\le|\zeta|\le k.
\]
The defining formulas for $f,h_l^-,h_r^+,S_{l,r}$ satisfy the same
polynomial bound. Choosing a fixed $C_{14}(k)$ that dominates the sum
of these bounds proves \eqref{amo:eq:central-polynomial-jets} on the
same bidisk for every $k$.
Finally, using these bounds only through order three, the mean-value
formula gives, for real $(x,E)\in\mathcal D_M$,
\[
 \begin{aligned}
 |\partial_x^2S(x,E)-\partial_x^2S(0,E_+)|
   &\le Cc_*M^{8\tau-B_{14}},\\
 |\partial_ES(x,E)-\partial_ES(0,E_+)|
   &\le Cc_*M^{6\tau-B_{14}},\\
 |h_l^-(x,E)-h_l^-(0,E_+)|
 +|h_r^+(x,E)-h_r^+(0,E_+)|
   &\le Cc_*M^{4\tau-B_{14}}.
 \end{aligned}
\]
All three bounds tend to zero. The positive limits
$\mathcal K_\psi$, $\|u\|_2^2$, $u_1$, together with
\eqref{amo:eq:schur-energy-envelope}, prove
\eqref{amo:eq:central-positive-derivatives}.
\end{proof}
\subsection{Finite-volume singular value estimates and the overlap coefficient}
\label{amo:sec:svd}

All SVDs in this subsection are real. In the next two lemmas, the constants
\(C_{15}(s,C_0,K)\) and \(D_{21}(s,C_0,K)\) are chosen simultaneously.
In their proofs, \(C=C(s,C_0,K)\) and \(P=P(s,C_0,K)\) denote
auxiliary positive bounds that may increase from one occurrence to the next.
For the next lemma, fix
\(m\ge2\) and a real phase--energy domain \(\Omega\).  Assume on
\(\Omega\) that
\[
 E-\max\sigma(H_{\lambda,\beta,z}^{[0,m-1]})\ge\delta>0,
 \qquad \delta^{-1}\le C_0(m+2)^K,
 \qquad \|A_m(z,E)\|>2.
\]

\begin{lemma}
\label{amo:lem:svd-calculus}
Under these assumptions, for every integer \(s\ge1\),
\begin{equation}\label{amo:eq:svd-calculus}
 \frac{|\partial_{z,E}^{\zeta}\|A_m\||}{\|A_m\|}
 +|\partial_{z,E}^{\zeta}u(A_m)|
 +|\partial_{z,E}^{\zeta}s(A_m)|
 \le C_{15}(m+2)^{D_{21}},\qquad 1\le|\zeta|\le s.
\end{equation}
The angle bounds hold for every local continuous SVD lift.
\end{lemma}

\begin{proof}
Put \(J=[0,m-1]\), \(D=\det(E-H^J)>0\), and \(R=(E-H^J)^{-1}\).
The inverse-derivative formula in the proof of
Proposition~\ref{amo:prop:central-schur}, Jacobi's identity
\[
 \partial_i\log D=\operatorname{tr}\bigl(R\partial_i(E-H^J)\bigr),
\]
and \(\|\partial_z^jH^J\|\le2\lambda(2\pi)^j\) give
\begin{equation}\label{amo:eq:logdet-jets}
 \begin{aligned}
 \|\partial^\zeta R\|&\le C(1+\delta^{-1})^{s+1}
                    &&(0\le|\zeta|\le s),\\
 |\partial^\zeta\log D|&\le C m(1+\delta^{-1})^s
                    &&(1\le|\zeta|\le s),\\
 |\partial^\zeta D|&\le C(m+2)^{P}D
                    &&(1\le|\zeta|\le s).
 \end{aligned}
\end{equation}
The continuant and complementary-minor identities yield
\[
 \mathsf M:=D^{-1}A_m=
 \begin{pmatrix}
 1&-R_{0,0}\\
 R_{m-1,m-1}&-\det\bigl(R|_{\{0,m-1\}\times\{0,m-1\}}\bigr)
 \end{pmatrix}.
\]
Consequently,
\[
 \|\partial^\zeta\mathsf M\|\le C(m+2)^{P}\quad(|\zeta|\le s).
\]
Since \(D\le\|A_m\|\), Leibniz's rule gives
\begin{equation}\label{amo:eq:matrix-relative-jets}
 \|\partial^\zeta A_m\|\le C(m+2)^{P}\|A_m\|,
 \qquad 1\le|\zeta|\le s.
\end{equation}

For \(T=\mathsf M^{\mathsf T}\mathsf M\) or \(\mathsf M\mathsf M^{\mathsf T}\), set
\[
 \Delta=\sqrt{(T_{11}-T_{22})^2+4T_{12}^2},\qquad
 \nu_\pm=\tfrac12(\operatorname{tr}T\pm\Delta),\qquad
 \Pi=\frac{T-\nu_-\operatorname{Id}}{\Delta}.
\]
As \(\det A_m=1\) and \(\mathsf M_{11}=1\),
\[
 \nu_+=D^{-2}\|A_m\|^2\ge1,\qquad
 \frac{\nu_-}{\nu_+}=\|A_m\|^{-4}<\frac1{16},\qquad
 \Delta\ge\frac{15}{16}\nu_+\ge\frac{15}{16}.
\]
The projector formula and \(\|A_m\|=D\sqrt{\nu_+}\) now give
\[
 \|\partial^\zeta\Pi\|
 +\frac{|\partial^\zeta\|A_m\||}{\|A_m\|}
 \le C(m+2)^{P},\qquad 1\le|\zeta|\le s.
\]
For a local oriented unit vector \(v=(\cos\theta,\sin\theta)\) in
one of their ranges,
\[
 \Pi v=v,\quad v^{\mathsf T}\partial_i v=0
 \quad\Longrightarrow\quad
 \partial_i v=(\partial_i\Pi)v,\qquad
 \partial_i\theta=(-v_2,v_1)\partial_i v.
\]
Induction gives the same type of polynomial bound for the angles, proving
\eqref{amo:eq:svd-calculus}. Any other continuous lift differs
locally by a constant.
\end{proof}

For the next lemma, fix \(\ell,r\ge1\), put \(M=\ell+r+1\) and
\(\Lambda=[-\ell,r]\), and let \(\Omega\) be a real phase--energy
domain with \(|E|\le2\lambda+3\). Assume on \(\Omega\) that
\begin{equation}\label{amo:eq:side-svd-hypotheses}
 E-\max\sigma(H_x^{[-\ell,-1]})\ge\delta,\qquad
 E-\max\sigma(H_x^{[1,r]})\ge\delta,\qquad 0<\delta\le1,
\end{equation}
and
\[
 D_-:=\det(E-H_x^{[-\ell,-1]})\ge2,\qquad
 D_+:=\det(E-H_x^{[1,r]})\ge2.
\]
Use the Weyl values and Schur complement in
\eqref{amo:eq:central-weyl}--\eqref{amo:eq:central-schur-definition},
abbreviated by \(h^-:=h_\ell^-\), \(h^+:=h_r^+\), and \(S:=S_{\ell,r}\).
Set
\[
 \begin{gathered}
 a_0:=E-2\lambda\cos(2\pi x),\\
 L:=A_\ell(x-\ell\beta,E),\qquad R:=A_{r+1}(x,E),\\
 q_-:=\frac{(1,h^-)^{\mathsf T}}{\sqrt{1+(h^-)^2}},\qquad
 q_+:=\frac{(a_0-h^+,-1)^{\mathsf T}}{\sqrt{1+(a_0-h^+)^2}}.
 \end{gathered}
\]

\begin{lemma}
\label{amo:lem:schur-svd}
Under \eqref{amo:eq:side-svd-hypotheses} and the determinant bounds
\(D_-,D_+\ge2\) above, the following hold.
\begin{enumerate}
\item Both \(L\) and \(R\) have simple top singular values.
The expanding output vector of \(L\) and expanding input vector of
\(R\) have unique orientations \(v_-,v_+\) satisfying
\begin{equation}\label{amo:eq:weyl-positive-orientation}
 \langle v_-,q_-\rangle>0,\qquad \langle v_+,q_+\rangle>0.
\end{equation}
These vectors are real analytic.
\item Their overlap satisfies
\begin{equation}\label{amo:eq:schur-svd}
 \cos G:=\langle v_-,v_+\rangle=\Xi S+\varepsilon,\qquad
 \Xi:=\frac1{\sqrt{1+(h^-)^2}\sqrt{1+(a_0-h^+)^2}},
\end{equation}
where
\begin{equation}\label{amo:eq:schur-svd-value}
 |\varepsilon|\le2(D_-^{-2}+D_+^{-2}),\qquad c\delta^2\le\Xi\le1.
\end{equation}
\item If \(\delta^{-1}\le C_0(M+2)^K\), then, for every fixed integer
\(s\ge0\),
\begin{equation}\label{amo:eq:schur-svd-jets}
 \begin{aligned}
 |\partial_{x,E}^\zeta\varepsilon|
 &\le C_{15}(M+2)^{D_{21}}(D_-^{-2}+D_+^{-2}),\\
 |\partial_{x,E}^\zeta\Xi|&\le C_{15}(M+2)^{D_{21}},\qquad |\zeta|\le s.
 \end{aligned}
\end{equation}
\end{enumerate}
\end{lemma}

\begin{proof}
For an interval \(J=[a,b]\), write \(D=\det(E-H_x^J)\) and
\(g_a,g_b\) for its endpoint diagonal Green entries.  The continuant
identity gives
\begin{equation}\label{amo:eq:normalized-continuant}
 D^{-1}A_{|J|}(x+a\beta,E)
 =\begin{pmatrix}1&-g_a\\g_b&-g_ag_b+D^{-2}\end{pmatrix}
 =\binom1{g_b}(1,-g_a)+D^{-2}e_2e_2^{\mathsf T}.
\end{equation}
For \(|J|=1\), \(g_a=g_b=D^{-1}\), so the same formula holds.
Let \(g_-,g_+\) be the Green entries at \(-\ell,r\), respectively,
and set \(t_\pm=D_\pm^{-2}\le1/4\).  With
\(\mathsf C(a)=\left(\begin{smallmatrix}a&-1\\1&0\end{smallmatrix}\right)\),
define
\begin{align}
 M_-(t)&=\binom1{h^-}(1,-g_-)+t e_2e_2^{\mathsf T},
 \label{amo:eq:left-rank-one}\\
 M_+(t)&=\left[\binom1{g_+}(1,-h^+)+t e_2e_2^{\mathsf T}\right]
          \mathsf C(a_0).
 \label{amo:eq:right-rank-one}
\end{align}
Then
\[
 D_-^{-1}L=M_-(t_-),\qquad D_+^{-1}R=M_+(t_+).
\]
For \(0\le t\le1/4\),
\[
 \begin{gathered}
 \det M_\pm(t)=t,\\
 M_-(t)e_1=(1,h^-)^{\mathsf T},\qquad
 M_+(t)^{\mathsf T}e_1=(a_0-h^+,-1)^{\mathsf T},\\
 s_1(M_\pm)\ge1,\qquad s_2(M_\pm)=t/s_1(M_\pm)\le t,\\
 s_1(M_\pm)^2-s_2(M_\pm)^2\ge15/16.
 \end{gathered}
\]
Let \(\Pi_-(t),\Pi_+(t)\) be the upper eigenprojectors of
\(M_-M_-^{\mathsf T},M_+^{\mathsf T}M_+\).  The contracting
components of \(M_-e_1,M_+^{\mathsf T}e_1\) have norms at most
\(s_2\).  Hence
\begin{equation}\label{amo:eq:projection-margin}
 \begin{gathered}
 \|(\operatorname{Id}-\Pi_\pm(t))q_\pm\|\le t,\qquad
 \|\Pi_\pm(t)q_\pm\|\ge\sqrt{1-t^2},\\
 V_\pm(t):=\frac{\Pi_\pm(t)q_\pm}{\|\Pi_\pm(t)q_\pm\|},\qquad
 |V_\pm(t)-q_\pm|\le2t.
 \end{gathered}
\end{equation}
Since \(\|L\|\ge D_-\ge2\) and \(\|R\|\ge D_+\ge2\), their
top singular values are simple.  Thus \(v_\pm=V_\pm(t_\pm)\) are
globally defined real analytic
orientations, \(V_\pm(0)=q_\pm\), and
\[
 \langle q_-,q_+\rangle=\Xi(a_0-h^--h^+),\qquad
 |\langle v_-,v_+\rangle-\langle q_-,q_+\rangle|
 \le2(t_-+t_+).
\]
Together with \(0<h^\pm\le\delta^{-1}\) and \(|a_0|\le C\),
this proves \eqref{amo:eq:schur-svd}--\eqref{amo:eq:schur-svd-value}.
The SVD convention gives
\[
 \arg v_-=u(L),\qquad \arg v_+=s(R)-\pi/2,\qquad
 \langle v_-,v_+\rangle=\cos\bigl(\pi/2-s(R)+u(L)\bigr).
\]

For derivatives, put \(y=(x,E)\). The estimates
\eqref{amo:eq:logdet-jets}, applied through order \(s+1\) to the two side
restrictions, give polynomial bounds for the required derivatives of
\(h^\pm,g_\pm,a_0\), and
\begin{equation}\label{amo:eq:t-jets}
 |\partial_y^\zeta t_\pm|\le C(M+2)^{P}t_\pm,
 \qquad |\zeta|\le s.
\end{equation}
The explicit two-by-two projector formula in the preceding proof,
the gap \(15/16\), and \eqref{amo:eq:projection-margin} give
\[
 |\partial_y^\zeta\partial_t^j V_\pm(y,t)|
 \le C(M+2)^{P},\qquad |\zeta|+j\le s+1,\quad 0\le t\le1/4.
\]
Using the exact remainder identity
\[
 V_\pm(y,t_\pm(y))-q_\pm(y)
 =t_\pm(y)\int_0^1
       \partial_tV_\pm(y,\rho t_\pm(y))\,d\rho
\]
and \eqref{amo:eq:t-jets}, the product and chain rules give
\[
 |\partial_y^\zeta(v_\pm-q_\pm)|
 \le C(M+2)^{P}t_\pm,\qquad |\zeta|\le s.
\]
Finally,
\[
 \varepsilon=\langle v_--q_-,v_+\rangle
             +\langle q_-,v_+-q_+\rangle
\]
and the formula for \(\Xi\) give bounds of the same form.
Taking the maxima of the coefficients and exponents obtained here
and in the proof of Lemma~\ref{amo:lem:svd-calculus} at order
\(\max\{1,s\}\) gives the common
constants \(C_{15}(s,C_0,K)\), \(D_{21}(s,C_0,K)\) and proves
\eqref{amo:eq:schur-svd-jets}.
\end{proof}

\subsection{First-return estimates and completion of the proof}
\label{amo:sec:first-return-svd}

Use the windows $I_n$, $\widetilde I_n$, and
$\mathcal E_n$ of Proposition~\ref{fr:prop:deterministic-AMO-seed}.
All constants in this subsection are uniform in the return atom and $n$.
Here $K$ again denotes the fixed padding from
Lemma~\ref{amo:lem:single-resonance-gap}.

\subsubsection{Return arithmetic}

Lemma~\ref{B:lem:return-time-poly-two-sided}, applied with
\((\alpha,\kappa)=(\beta,\kappa_\beta)\), gives
\[
 c|I|^{-1/\tau}\le r_I^\pm(x)\le C|I|^{-\tau}
 \qquad(x\in I^{\rm reg}).
\]

\begin{lemma}
\label{amo:lem:reciprocal-central-atom}
For all sufficiently large $n$, $0\in I_n^{\rm reg}$ and
\begin{equation}\label{amo:eq:central-return-size}
 m_n=r_{I_n}^+(0)=r_{I_n}^-(0),\qquad
 cN_n^{1/\tau}\le m_n\le N_n.
\end{equation}
The common atom $\mathcal A_n^0$ containing $0$ has labels $m_n,m_n$ and
\begin{equation}\label{amo:eq:central-atom-width}
 [-cN_n^{-(2\tau+1)},cN_n^{-(2\tau+1)}]\subset\mathcal A_n^0.
\end{equation}
\end{lemma}

\begin{proof}
Since $\partial I_n=\{\pm N_n^{-1}\}\subset\mathbb Q/\mathbb Z$,
$j\beta\notin\partial I_n$ for all $j\in\mathbb Z$, and hence
$0\in I_n^{\rm reg}$. Symmetry, the Diophantine bound, and the
pigeonhole principle give
\[
 m_n=\min\{j\ge1:\|j\beta\|_\T\le N_n^{-1}\},\qquad
 \kappa_\beta m_n^{-\tau}\le N_n^{-1},\qquad m_n\le N_n.
\]
For $1\le |j|\le m_n\le N_n$,
\[
 \operatorname{dist}_\T(j\beta,\partial I_n)
 \ge N_n^{-1}\|jN_n\beta\|_\T
 \ge\kappa_\beta N_n^{-(2\tau+1)}.
\]
Thus $|x|\le\kappa_\beta N_n^{-(2\tau+1)}/4$ preserves every membership
relation $x+j\beta\in I_n$ for $|j|\le m_n$. Both return labels
remain $m_n$ on regular phases, proving \eqref{amo:eq:central-atom-width}.
\end{proof}

\subsubsection{First-return restrictions}

For the next proposition, let $D_1$ satisfy
\begin{equation}\label{amo:eq:Dgap-choice}
 D_1\ge C_I\tau(2\tau+1)+2.
\end{equation}
Given $n$ and $\mathcal A\in\mathcal P_{I_n}$,
let $(r,z)$ denote either first-return block,
\[
 (r,z)=(p_{\mathcal A},x)\quad\text{or}\quad
 (r,z)=(q_{\mathcal A},x-q_{\mathcal A}\beta),\qquad x\in\mathcal A.
\]

\begin{proposition}
\label{amo:prop:stopped-blocks}
There is $c>0$ such that, for all sufficiently large $n$, every
$\mathcal A\in\mathcal P_{I_n}$, and both choices of $(r,z)$ above,
the following hold.
\begin{enumerate}
\item The finite-volume gap satisfies
\begin{equation}\label{amo:eq:stopped-intrinsic-gap}
 E_+-\max\sigma(H_z^{[0,r-1]})
 \ge c(r+1)^{-2\tau}\ge cn^{-2C_I\tau^2}.
\end{equation}
\item For $|E-E_+|\le cn^{-D_1}$,
\begin{equation}\label{amo:eq:stopped-positive-growth}
 \|A_r(z,E)\|\ge P_r(z,E)
 \ge\exp\{\gamma r-CR_\tau(r+2)\}>0.
\end{equation}
\item On the same energy interval, for every fixed $s\ge1$, the SVD estimates
in Proposition~\ref{fr:prop:deterministic-AMO-seed} hold
through order $s$, with a polynomial in $r+2$ uniform in $n$ and $\mathcal A$.
\end{enumerate}
\end{proposition}

\begin{proof}
For $z\in I_n^{\rm reg}$, $r=r_{I_n}^+(z)$, and $\rho=N_n^{-1}$,
\[
 \|z+j\beta\|_\T>\rho\quad(1\le j<r),\qquad
 \kappa_\beta r^{-\tau}\le\|r\beta\|_\T\le2\rho.
\]
Thus the choice
\[
 \delta=\frac{\kappa_\beta}{8(r+2K)^\tau}
 \le\min\{\rho/4,\kappa_\beta/[4(r+2K)^\tau]\}
\]
satisfies the hypotheses of Lemma~\ref{amo:lem:single-resonance-gap}
on $[0,r-1]$.  That lemma and
Lemma~\ref{B:lem:return-time-poly-two-sided} imply
\begin{equation}\label{amo:eq:return-length-range}
 cn^{C_I/\tau}\le r\le Cn^{C_I\tau},\qquad
 g_r:=E_+-\max\sigma(H_z^{[0,r-1]})\ge c(r+1)^{-2\tau}.
\end{equation}
This proves \eqref{amo:eq:stopped-intrinsic-gap}.
Proposition~\ref{amo:prop:determinant-estimates} with
$\eta=\min\{1,c(r+1)^{-2\tau}\}$ gives
\[
 \log P_r(z,E_+)\ge\gamma r-CR_\tau(r+2).
\]
Write $\mu_1,\ldots,\mu_r$ for the eigenvalues of $H_z^{[0,r-1]}$.
For $|t|\le g_r/2$,
\[
 \begin{aligned}
 \left|\log\frac{P_r(z,E_++t)}{P_r(z,E_+)}\right|
 &=\left|\sum_{j=1}^r\log\left(1+\frac{t}{E_+-\mu_j}\right)\right|\\
 &\le\frac{2r|t|}{g_r}
 \le C(r+1)^{2\tau+1}|t|.
 \end{aligned}
\]
For $|t|\le cn^{-D_1}$, the last bound is at most
\[
 Ccn^{C_I\tau(2\tau+1)-D_1}\le Ccn^{-2},
\]
and $|t|\le g_r/2$ for large $n$. Hence
\[
 E_++t-H_z^{[0,r-1]}\ge\tfrac12g_r\operatorname{Id},\qquad
 \|A_r(z,E_++t)\|\ge P_r(z,E_++t)>2.
\]
The logarithmic estimate proves \eqref{amo:eq:stopped-positive-growth};
Lemma~\ref{amo:lem:svd-calculus} gives the derivatives.
For the preceding return block,
\[
 z=x-q_{\mathcal A}\beta\in I_n^{\rm reg},\qquad
 r_{I_n}^+(z)=q_{\mathcal A}.
\]
Density of regular phases and continuity of the fixed-length matrices
extend the estimates to every $x\in\mathcal A$.
\end{proof}

\subsubsection{The central overlap coefficient}

For the next proposition, write
\[
 L_m=A_m(x-m\beta,E),\qquad R_m=A_m(x,E).
\]

\begin{proposition}
\label{amo:prop:central-joining}
There are $B_{15},c,C>0$ such that, for every sufficiently large $m$,
the blocks $L_m,R_m$ admit the orientation in
\eqref{amo:eq:weyl-positive-orientation} on
\[
 \mathcal U_m=\{(x,E):|x|\le(2m+1)^{-B_{15}},\quad
                         |E-E_+|\le(2m+1)^{-B_{15}}\}.
\]
Writing $H_m=\cos G_m$, the following hold.
\begin{enumerate}
\item On $\mathcal U_m$,
\begin{equation}\label{amo:eq:joining-normal-form}
 H_m(x,E)=C_{1,m}(x)x^2+C_{2,m}(x,E)(E-E_+)+R_m^{\rm aff}(x),
\end{equation}
where $C_{1,m},C_{2,m}$ are real $C^2$ functions, $R_m^{\rm aff}$ is affine, and
\[
 \begin{gathered}
 c\le C_{1,m},C_{2,m}\le C,\qquad
 \|C_{1,m}\|_{C^2}+\|C_{2,m}\|_{C^2}\le C(2m+1)^{B_{15}},\\
 \|R_m^{\rm aff}\|_{C^2}\le Ce^{-cm},\qquad
 c\le\partial_EH_m,\partial_x^2H_m\le C.
 \end{gathered}
\]
\item At $(0,E_+)$,
\begin{equation}\label{amo:eq:joining-explicit-limits}
 \begin{aligned}
 \partial_EH_m(0,E_+)&=\frac{\|u\|_2^2}{1+u_1^2}+O(e^{-cm}),\\
 \partial_x^2H_m(0,E_+)&=
 \frac{\mathcal K_\psi}{1+u_1^2}+O(e^{-cm}).
 \end{aligned}
\end{equation}
\end{enumerate}
\end{proposition}

\begin{proof}
Apply Proposition~\ref{amo:prop:central-schur} and
Corollary~\ref{amo:cor:central-schur-curvature} to $[-m,m-1]$,
whose strict side lengths are $m,m-1$. On its real bidisk,
\[
 E-H_x^\pm\ge cm^{-2\tau}\operatorname{Id},\qquad
 D_-(0,E_+)=Q_m,\quad D_+(0,E_+)=Q_{m-1},
\]
and Jacobi's formula gives
\[
 |\partial_x\log D_\pm|+|\partial_E\log D_\pm|
 \le Cm^{2\tau+1}.
\]
Thus on a smaller phase--energy domain with exponent larger than $2\tau+2$,
\[
 D_-\ge e^{\gamma m-CR_\tau(m+2)},\qquad
 D_+\ge e^{\gamma(m-1)-CR_\tau(m+2)}.
\]
For some exponent $P>0$, Lemma~\ref{amo:lem:schur-svd} gives
\begin{equation}\label{amo:eq:central-comparison}
 H_m=\Xi_mS_{m,m-1}+\varepsilon_m,\qquad
 \|\varepsilon_m\|_{C^4}\le Cm^P e^{-2\gamma(m-1)+CR_\tau(m+2)}
 \le Ce^{-cm}.
\end{equation}
After increasing $P$ once, all derivatives of $\Xi_m$ and $H_m$
through order four are bounded by $Cm^P$ on this domain. At $(0,E_+)$,
\[
 \begin{gathered}
 h_m^-,h_{m-1}^+=u_1+O(e^{-cm}),\qquad
 E_+-2\lambda=2u_1,\\
 \Xi_m=(1+u_1^2)^{-1}+O(e^{-cm}).
 \end{gathered}
\]
Together with \eqref{amo:eq:central-small-affine} and
\eqref{amo:eq:central-explicit-limits}, the product rule implies
\[
 \begin{aligned}
 |H_m|+|\partial_xH_m|&\le Ce^{-cm},\\
 \partial_EH_m&=\Xi_m\partial_ES_{m,m-1}+O(e^{-cm}),\\
 \partial_x^2H_m&=\Xi_m\partial_x^2S_{m,m-1}+O(e^{-cm}).
 \end{aligned}
\]
This proves \eqref{amo:eq:joining-explicit-limits}.
Fix
\[
 B_{15}>\max\{B_{14},2\tau+2,P\}+2.
\]
Then $\mathcal U_m$ lies in the preceding domain for large $m$, and
\[
 \sup_{\mathcal U_m}|\partial^\zeta H_m-\partial^\zeta H_m(0,E_+)|
 \le Cm^{P-B_{15}}=o(1),\qquad \zeta=(0,1),(2,0).
\]
Both derivatives therefore stay between fixed positive constants.
Finally set
\[
 \begin{aligned}
 C_{1,m}(x)&=\int_0^1(1-t)\partial_x^2H_m(tx,E_+)\,dt,\\
 C_{2,m}(x,E)&=\int_0^1\partial_EH_m(x,E_++t(E-E_+))\,dt,\\
 R_m^{\rm aff}(x)&=H_m(0,E_+)+x\partial_xH_m(0,E_+).
 \end{aligned}
\]
Taylor's formula proves \eqref{amo:eq:joining-normal-form}. The
$C^4$ bound on $H_m$, established before choosing $B_{15}$, gives
\[
 \|C_{1,m}\|_{C^2}+\|C_{2,m}\|_{C^2}
 \le Cm^P\le C(2m+1)^{B_{15}}.
\]
\end{proof}

\subsubsection{Exterior separation and completion}

Choose
\begin{align}
 B_1&>\max\{C_I(2\tau+1),C_IB_{15},C_I\}+2,
 \label{amo:eq:Ccen-choice}\\
 D_{15}&=2\max\{C_I\tau^2,B_1\}.
 \label{amo:eq:Dext-choice}
\end{align}

\begin{lemma}
\label{amo:lem:exterior-return-gap}
For every atom with labels $p,q$ and
$x\in\mathcal A\cap(I_n\setminus\widetilde I_n)$,
\[
 E_+-\max\sigma(H_x^{[-q,p-1]})\ge cn^{-D_{15}}.
\]
\end{lemma}

\begin{proof}
For regular $x$, all strict interior indices $j\ne0$ satisfy
$\|x+j\beta\|_\T>N_n^{-1}$, and $\|x\|_\T\ge n^{-B_1}$.
With
\[
 \delta=c\min\{(p+q+2K)^{-\tau},N_n^{-1},n^{-B_1}\}
 \ge cn^{-D_{15}/2},
\]
Lemma~\ref{amo:lem:single-resonance-gap} gives the gap $c\delta^2$.
Continuity extends it to every stated phase.
\end{proof}

Choose the common exponent and then the energy window by
\begin{equation}\label{amo:eq:CE-choice}
 \begin{gathered}
 D_1=D_1(C_I,r_\dagger)
 \ge\max\{C_I\tau(2\tau+1)+2,\,3D_{15}+1\},\\
 B_2>D_1+2.
 \end{gathered}
\end{equation}
These choices imply \eqref{amo:eq:Dgap-choice} and
\[
 D_1\ge3D_{15}+1>C_IB_{15}+2,\qquad
 B_2>\max\{D_{15}+2,C_IB_{15}+2\}.
\]

\begin{proof}[Proof of Propositions~\ref{fr:prop:AMO-edge-state}--\ref{fr:prop:AMO-return-nondegeneracy}]
Proposition~\ref{amo:prop:sharp-anchor} proves
Proposition~\ref{fr:prop:AMO-edge-state} for the pair
$(\sigma_{\rm cmp},B)$ fixed at the beginning of the appendix.

For the return matrices and their determinants, note that, for all
sufficiently large $n$,
\[
 n^{-B_2}\le cn^{-D_1},\qquad
 D_1>2C_I\tau^2.
\]
Proposition~\ref{amo:prop:stopped-blocks}, through order $r_\dagger$,
gives the estimates in Proposition~\ref{fr:prop:deterministic-AMO-seed}
and the Dirichlet conclusions of
Proposition~\ref{fr:prop:AMO-return-nondegeneracy}.
The return norms exceed two, and their simple singular lines on the contractible domain
$\mathcal A\times\mathcal E_n$ admit continuous orientations and
real angle lifts. The common constants are chosen at the end of the proof.

For Proposition~\ref{fr:prop:AMO-central-normal-form},
Lemma~\ref{amo:lem:reciprocal-central-atom} and
\eqref{amo:eq:Ccen-choice} give
\[
 \widetilde I_n\subset\mathcal A_n^0,\qquad
 p_{\mathcal A_n^0}=q_{\mathcal A_n^0}=m_n\le N_n\le2n^{C_I},\qquad
 \widetilde I_n\times\mathcal E_n\subset\mathcal U_{m_n}.
\]
Apply Proposition~\ref{amo:prop:central-joining} with $m=m_n$ and put
\[
 \widetilde C_{1,n}=C_{1,m_n},\qquad
 \widetilde C_{2,n}=C_{2,m_n},\qquad
 \widetilde R_n=R_{m_n}^{\rm aff}.
\]
Then
\[
 C(2m_n+1)^{B_{15}}\le n^{D_1},\qquad
 Ce^{-cm_n}\le e^{-c'\gamma m_n}
\]
after increasing $n_0$ and decreasing $c'>0$.  The orientation in
\eqref{amo:eq:weyl-positive-orientation} fixes the central relative sign
and proves Proposition~\ref{fr:prop:AMO-central-normal-form}.

It remains to prove the overlap estimate in
Proposition~\ref{fr:prop:AMO-return-nondegeneracy}. On the exterior domain,
Lemma~\ref{amo:lem:exterior-return-gap} and
\eqref{amo:eq:CE-choice} imply
\begin{equation}\label{amo:eq:exterior-E-gap}
 E-H_x^{[-q,p-1]}\ge cn^{-D_{15}}\operatorname{Id},\qquad
 E\in\mathcal E_n.
\end{equation}
Thus
\[
 S_{q,p-1}
 =\langle e_0,(E-H_x^{[-q,p-1]})^{-1}e_0\rangle^{-1}
 \ge cn^{-D_{15}}.
\]
The strict sides have the same gap by interlacing. Their determinants obey
\[
 \begin{aligned}
 D_-&=P_q(x-q\beta,E)\ge e^{\gamma q-CR_\tau(q+2)},\\
 P_p(x,E)&=D_+\bigl(E-2\lambda\cos(2\pi x)-h^+\bigr),\\
 0&<E-2\lambda\cos(2\pi x)-h^+\le C,\\
 D_+&\ge C^{-1}P_p(x,E)\ge C^{-1}e^{\gamma p-CR_\tau(p+2)}.
 \end{aligned}
\]
By Lemma~\ref{amo:lem:schur-svd} and \eqref{amo:eq:return-length-range},
\[
 \begin{gathered}
 |\cos G_{I_n,\mathcal A}|=|\Xi S_{q,p-1}+\varepsilon|,\qquad
 \Xi\ge cn^{-2D_{15}},\\
 |\varepsilon|\le C(D_-^{-2}+D_+^{-2})
 \le Ce^{-c n^{C_I/\tau}},\\
 |\cos G_{I_n,\mathcal A}|
 \ge cn^{-3D_{15}}-Ce^{-c n^{C_I/\tau}}
 \ge c'n^{-D_1}.
 \end{gathered}
\]
The absolute value is independent of the relative orientation, which
completes Proposition~\ref{fr:prop:AMO-return-nondegeneracy}. Taking the
minimum of the finitely many positive constants,
and enlarging $C$ and $n_0$ as needed gives one common choice for all
four propositions.
\end{proof}
\section{Proof of the normalized product estimate}
\label{app:normalized-product-proof}

\noindent This appendix proves the finite-product estimates by graph recursions.
Steps~(i)--(ii) prove Lemma~\ref{B:lem:finite-two-state-growth};
Steps~(iii)--(iv) prove Lemma~\ref{B:lem:differentiated-finite-two-state-product};
Steps~(v)--(vi) prove Corollary~\ref{B:cor:finite-two-state-frames}.

\begin{proof}
Set
\[
 r_t:=-\frac{\Theta_t}{H_t},
 \qquad s_t:=\frac{\Theta_t}{H_t},
 \qquad
 T_t:=H_t^{-1}D_tW_t
 =\begin{pmatrix}1&r_t\\ \epsilon_ts_t&\epsilon_t\end{pmatrix}.
\]
Thus, by \eqref{B:eq:finite-two-state-product},
\(\mathscr H_L^{-1}\mathscr M_L=T_{L-1}\cdots T_1D_0\).
For \(L=1\), take \(a_L=1\), \(\zeta_L=\rho_L=0\).  Then
\(\mathscr M_1=e_1e_1^{\mathsf T}+\epsilon_0e_2e_2^{\mathsf T}\),
and all three statements follow directly.  Hence assume \(L\ge2\).

\smallskip
\noindent\emph{Step (i): Graph recursions.}
Define the forward column graph and the backward input row by
\[
\begin{aligned}
 \zeta_0&:=0,\qquad a_0:=1,\\
 \zeta_t&:=\epsilon_t
   \frac{s_t+\zeta_{t-1}}{1+r_t\zeta_{t-1}},\\
 a_t&:=a_{t-1}(1+r_t\zeta_{t-1}),
 \qquad 1\le t<L,\\
 \eta_{L-1}&:=r_{L-1},\qquad
 \eta_t:=\frac{r_t+\epsilon_t\eta_{t+1}}
              {1+\epsilon_ts_t\eta_{t+1}}
 \quad(1\le t<L-1),
 \qquad
 \rho_L:=\epsilon_0\eta_1.
\end{aligned}
\]
For \(P_L:=\mathscr H_L^{-1}\mathscr M_L\), the recursions give
\[
 P_Le_1=a_{L-1}\binom1{\zeta_{L-1}},\qquad
 e_1^{\mathsf T}P_L=b_1(1,\rho_L),\qquad
 b_1=(P_L)_{11}=a_{L-1}.
\]
Set \(a_L:=a_{L-1}\), \(\zeta_L:=\zeta_{L-1}\).  Fixing the first row
and column and then the determinant yields
\[
\begin{aligned}
 \Delta_L:=\det(\mathscr H_L^{-1}\mathscr M_L)
 &=\left(\prod_{t=0}^{L-1}\epsilon_t\right)
  \left(\prod_{t=1}^{L-1}H_t^{-2}\right),
 \\
 \mathscr H_L^{-1}\mathscr M_L
 &=a_L\binom{1}{\zeta_L}(1,\rho_L)
  +\frac{\Delta_L}{a_L}e_2e_2^{\mathsf T}.
\end{aligned}
\]
\smallskip
\noindent\emph{Step (ii): The \(C^0\) estimate.}
For Lemma~\ref{B:lem:finite-two-state-growth}, under
\eqref{B:eq:finite-two-state-input-bounds}, set
\(\delta:=\epsilon\kappa^{-2}\le10^{-4}\).  Since
\(H_t^2+\Theta_t^2=1\), one has
\(
 |r_t|=|s_t|
 =\left|H_t^{-2}-1\right|^{1/2}
 \le\sqrt2\,\kappa^{-1}
\).
The simultaneous induction closes through
\[
\begin{aligned}
 |r_t\zeta_{t-1}|,\ |\epsilon_ts_t\eta_{t+1}|&\le8\delta,
 &\frac{1+4\epsilon}{1-8\delta}&<4,\\
 |\zeta_t|&\le4\sqrt2\,\kappa^{-1}|\epsilon_t|,
 &|\eta_t|&\le4\sqrt2\,\kappa^{-1},\\
 |\rho_L|&\le4\sqrt2\,\kappa^{-1}|\epsilon_0|.
\end{aligned}
\]
All denominators therefore have modulus at least \(1-8\delta\).  Put
\(d_L:=\Delta_L/a_L^2\) and
\(\log a_L:=\sum_{t=1}^{L-1}
\operatorname{Log}(1+r_t\zeta_{t-1})\), where
\(\operatorname{Log}\) is the branch near \(1\) with
\(\operatorname{Log}1=0\).
The determinant formula gives
\[
\begin{aligned}
 \|\log a_L\|_\infty
 &\le\frac{8}{1-8\delta}(L-1)\delta
 <8.1(L-1)\delta,\\
 |a_L|&\ge(1-8\delta)^{L-1},\\
 |d_L|&\le\min\{|\epsilon_0|,|\epsilon_{L-1}|\}
 \left(\frac{\delta}{(1-8\delta)^2}\right)^{L-1}<0.0011\delta.
\end{aligned}
\]
Write
\(N_L:=a_L^{-1}\mathscr H_L^{-1}\mathscr M_L
=\binom{1}{\zeta_L}(1,\rho_L)+d_Le_2e_2^{\mathsf T}\).
Since \(|\zeta_L|,|\rho_L|\le4\sqrt2\delta\), the rank-one term has
operator norm between \(1\) and \(1+32\delta^2\).  Hence
\[
\begin{aligned}
 \bigl|\|N_L\|-1\bigr|&<0.034\delta,
 &|\log\|N_L\||&<0.1\delta,\\
 \left|
  \log\|\mathscr M_L\|-\sum_{t=1}^{L-1}\log|H_t|
 \right|
 &<8.1(L-1)\delta+0.1\delta\\
 &\le10L\delta,
\end{aligned}
\]
where the last estimate uses
\(\mathscr M_L=\mathscr H_La_LN_L\).  This proves
\eqref{B:eq:finite-two-state-uniform-log-product}.

\smallskip
\noindent\emph{Step (iii): Derivative estimates.}
In Steps~(iii)--(vi), \(C=C(m)>0\) denotes auxiliary constants,
which may increase from one estimate to the next. The fixed constants
\(C_2(m)\ge2\) and \(0<c_1(m)\le10^{-4}\) in
Lemma~\ref{B:lem:differentiated-finite-two-state-product} will be chosen
after these estimates. Interpret logarithmic derivatives as in the
statement and assume
\eqref{B:eq:finite-two-state-jet-input} and
\eqref{B:eq:finite-two-state-smallness}.  Since \(C_2(m)\ge2\),
\[
 L\delta\le c_1(m),\qquad
 \|\log a_L\|_\infty\le10L\delta\le10c_1(m)<10^{-2},\qquad
 \|a_L-1\|_\infty<\tfrac12.
\]
Put \(\mathcal Q_0:=1+K_0\) and
\(\mathcal Q_r:=1+K_r+\mathcal R_r(K_1,\ldots,K_{r-1})\), where the
polynomials \(\mathcal R_r\) have coefficients depending only on \(r\), and
set
\[
 Z_{t,r}^+:=\max_{|\beta|=r}
 \|\partial^\beta\zeta_t\|_\infty,\qquad
 Z_{t,r}^-:=\max_{|\beta|=r}
 \|\partial^\beta\eta_t\|_\infty.
\]
The graph derivatives satisfy
\[
\begin{aligned}
 \partial_\zeta\!\left[\epsilon_t
    \frac{s_t+\zeta}{1+r_t\zeta}\right]
 &=\frac{\epsilon_tH_t^{-2}}{(1+r_t\zeta)^2},\\
 \partial_\eta\!\left[
    \frac{r_t+\epsilon_t\eta}{1+\epsilon_ts_t\eta}\right]
 &=\frac{\epsilon_tH_t^{-2}}{(1+\epsilon_ts_t\eta)^2},\\
 |\partial_\zeta(\cdots)|,\ |\partial_\eta(\cdots)|
 &\le\frac{\delta}{(1-8\delta)^2}\le\frac14.
\end{aligned}
\]
Leibniz's rule, the quotient rule, and Fa\`a di Bruno's formula therefore give
\[
\begin{aligned}
 Z_{t,r}^+
 &\le\frac14Z_{t-1,r}^+
   +CL^{C}\kappa^{-C}\mathcal Q_r,\\
 Z_{t,r}^-
 &\le\frac14Z_{t+1,r}^-
   +CL^{C}\kappa^{-C}\mathcal Q_r.
\end{aligned}
\]
Geometric summation from \(\zeta_0=0\) and
\(\eta_{L-1}=r_{L-1}\) gives, uniformly in \(t\),
\[
 Z_{t,r}^++Z_{t,r}^-
 \le CL^{C}\kappa^{-C}\mathcal Q_r,\qquad 1\le t<L.
\]
To retain the endpoint factors, put
\[
 F_t:=\frac{s_t+\zeta_{t-1}}{1+r_t\zeta_{t-1}},\qquad
 \zeta_t=\epsilon_tF_t\quad(1\le t<L).
\]
The preceding bounds and the denominator bound in Step~(ii) imply
\[
 \|\partial^\beta F_t\|_\infty
 \le CL^{C}\kappa^{-C}\mathcal Q_r,
 \qquad |\beta|=r\le m.
\]
Also,
\[
 \partial^\gamma\epsilon_t
 =\epsilon_t\,\mathcal T_\gamma
   \bigl((\partial^\beta\log\epsilon_t)_{1\le|\beta|\le|\gamma|}\bigr),
\]
where \(\mathcal T_\gamma\) is the universal derivative polynomial
and \(\mathcal T_0=1\).  Leibniz's rule therefore gives
\[
 \begin{aligned}
 \|\partial^\beta\zeta_t\|_\infty
 &\le C\|\epsilon_t\|_\infty
       L^{C}\kappa^{-C}\mathcal Q_r,\qquad 1\le t<L,\\
 \|\partial^\beta\rho_L\|_\infty
 &=\|\partial^\beta(\epsilon_0\eta_1)\|_\infty
 \le C\epsilon_{\rm in}L^{C}\kappa^{-C}\mathcal Q_r.
 \end{aligned}
\]
In particular, the first estimate at \(t=L-1\) contains
\(\epsilon_{\rm out}\).  Differentiating the formula for \(\log a_L\) in Step~(ii) gives
\[
\begin{aligned}
 \|\partial^\beta\log a_L\|_\infty
 &\le C\epsilon L^{C}\kappa^{-C}\mathcal Q_r,\\
 \partial^\beta\log\Delta_L
 &=\sum_{t=0}^{L-1}\partial^\beta\log\epsilon_t
   -2\sum_{t=1}^{L-1}\partial^\beta\log H_t
   \qquad(|\beta|\ge1),
\end{aligned}
\]
Indeed, each term in a positive-order derivative of
\(\operatorname{Log}(1+r_t\zeta_{t-1})\) contains a factor
\(\partial^\gamma\zeta_{t-1}\), with \(0\le|\gamma|\le|\beta|\);
the term with \(t=1\) vanishes.  Finally,
\(d_L=\Delta_La_L^{-2}\), so
\[
 \begin{aligned}
 |\partial^\beta d_L|
 &\le C|d_L|L^{C}\kappa^{-C}\mathcal Q_r,\\
 \|\partial^\beta d_L\|_\infty
 &\le C\min\{\epsilon_{\rm in},\epsilon_{\rm out}\}
       L^{C}\kappa^{-C}\mathcal Q_r.
 \end{aligned}
\]
Here we used the bound for \(d_L\) in Step~(ii).  The same step gives
the zeroth-order bounds.  For the highest derivatives,
write \(F\) for the vector of local input coordinates
\((\log\epsilon_t,H_t,\Theta_t)\), and \(\Phi\) for one of the
analytic expressions above. If \(|\beta|=r\ge1\),
\[
 \partial^\beta[\Phi(F)]
 =D\Phi(F)[\partial^\beta F]
 +\sum_{p=2}^{r}\ 
  \sum_{\substack{\beta_1+\cdots+\beta_p=\beta\\|\beta_i|\ge1}}
  c_{\beta_1,\ldots,\beta_p}
  D^p\Phi(F)[\partial^{\beta_1}F,\ldots,\partial^{\beta_p}F].
\]
The constants in the sum are the multivariate chain-rule coefficients.
Every \(|\beta_i|\le r-1\) in the second term; thus the dependence on
\(\partial^\beta F\) is affine. Applying the preceding estimates
separately to the two terms proves the asserted bounds.

\smallskip
\noindent\emph{Step (iv): Spectral projections.}
Put \(\mathbf x_L:=(1,\zeta_L)^{\mathsf T}\),
\(\mathbf y_L:=(1,\rho_L)^{\mathsf T}\), and
\(E_{22}:=e_2e_2^{\mathsf T}\). Then
\(N_L=\mathbf x_L\mathbf y_L^{\mathsf T}+d_LE_{22}\).
Define
\[
\begin{aligned}
 G_{\rm in}&:=N_L^{\mathsf T}N_L,&
 G_{{\rm in},0}&:=
  (\mathbf x_L^{\mathsf T}\mathbf x_L)
  \mathbf y_L\mathbf y_L^{\mathsf T},\\
 G_{\rm out}&:=N_LN_L^{\mathsf T},&
 G_{{\rm out},0}&:=
  (\mathbf y_L^{\mathsf T}\mathbf y_L)
  \mathbf x_L\mathbf x_L^{\mathsf T}.
\end{aligned}
\]
For \(d_L=0\), \(G_{\rm in}=G_{{\rm in},0}\) and
\(G_{\rm out}=G_{{\rm out},0}\); their common nonzero eigenvalue stays
near \(1\), and their Riesz projections are respectively
\[
 \frac{\mathbf y_L\mathbf y_L^{\mathsf T}}
      {\mathbf y_L^{\mathsf T}\mathbf y_L},
 \qquad
 \frac{\mathbf x_L\mathbf x_L^{\mathsf T}}
      {\mathbf x_L^{\mathsf T}\mathbf x_L}.
\]
Direct multiplication gives
\[
\begin{aligned}
 G_{\rm in}-G_{{\rm in},0}
 &=d_L\bigl(
   \mathbf y_L\mathbf x_L^{\mathsf T}E_{22}
  +E_{22}\mathbf x_L\mathbf y_L^{\mathsf T}\bigr)
  +d_L^2E_{22},\\
 G_{\rm out}-G_{{\rm out},0}
 &=d_L\bigl(
   \mathbf x_L\mathbf y_L^{\mathsf T}E_{22}
  +E_{22}\mathbf y_L\mathbf x_L^{\mathsf T}\bigr)
  +d_L^2E_{22}.
\end{aligned}
\]
Also \(\widehat M=a_LN_L\).  Hence the preceding bounds and
\(L\delta\le c_1(m)\le 10^{-4}\) give
\[
\begin{aligned}
 \|\widehat M-e_1e_1^{\mathsf T}\|_\infty
 &\le \|a_L-1\|_\infty\\
 &\quad+\|a_L\|_\infty
 \|N_L-e_1e_1^{\mathsf T}\|_\infty\\
 &\le20L\delta\le20c_1(m)\le2\cdot10^{-3}<\frac1{10}.
\end{aligned}
\]
Criterion~\ref{crit:rank-one-spectral-separation} defines the spectral
data throughout \(\Omega\) and fixes the
branches normalized at \(1\).  Multiplication of a Gram matrix by the
nonzero scalar \(a_L^2\) does not change its spectral subspaces, so the
Riesz projections of \(G_{\rm in}\) and \(G_{\rm out}\) are precisely
\(\Pi_M^{\rm in}\) and \(\Pi_M^{\rm out}\).

Choose one positively oriented circle \(\Gamma_1\) separating the disk
about \(1\) from the disk about \(0\) for all four Gram matrices above.
All resolvents on \(\Gamma_1\) are uniformly bounded.  If \(G\) and
\(G_0\) denote either corresponding pair, the Riesz identity is
\[
 \Pi(G)-\Pi(G_0)
 =\frac1{2\pi i}\int_{\Gamma_1}
  (\xi I_2-G)^{-1}(G-G_0)
  (\xi I_2-G_0)^{-1}\,d\xi .
\]
Differentiation and the displayed Gram expansions give, for
\(|\beta|=r\le m\),
\[
 \|\partial^\beta(\Pi(G)-\Pi(G_0))\|_\infty
 \le C\min\{\epsilon_{\rm in},\epsilon_{\rm out}\}
       L^{C}\kappa^{-C}\mathcal Q_r.
\]
The rank-one input and output projectors depend only on \(\rho_L\)
and \(\zeta_L\), respectively.  Their bounds from Step~(iii) prove the
first two lines of \eqref{B:eq:finite-two-state-projector-output},
including the affine dependence on the highest derivatives.
Choose \(\mathcal P_r\) so that \(\mathcal Q_r\le\mathcal B_r\).

Finally,
\[
 \nu_M=\operatorname{tr}
 \bigl(\widehat M^{\mathsf T}\widehat M\Pi_M^{\rm in}\bigr),
 \qquad
 \log\frac{\Lambda_M}{\mathscr L_0\mathscr H_L}
 =\frac12\log\nu_M.
\]
Since \(\widehat M=a_LN_L\), differentiating these identities
and using the graph and Riesz bounds gives the multiplier estimate in the
third line of \eqref{B:eq:finite-two-state-projector-output}.
Step~(iii) also gives the affine dependence and separate bounds for the
linear and constant terms.  After reducing \(c_1(m)\), the zeroth-order
bound gives
\[
 |e_1^{\mathsf T}\Pi_M^{\rm in}e_1-1|<\tfrac12,
\]
as required in Corollary~\ref{B:cor:finite-two-state-frames}.

\smallskip
\noindent\emph{Step (v): Canonical frames.}
The projection bound defines \(v_M\), and
\[
\begin{aligned}
 v_M^{\mathsf T}v_M&=1,&
 u_M&=\Lambda_M^{-1}Mv_M=\nu_M^{-1/2}\widehat Mv_M,\\
 u_M^{\mathsf T}u_M
 &=\Lambda_M^{-2}v_M^{\mathsf T}M^{\mathsf T}Mv_M=1,&
 u_Mu_M^{\mathsf T}&=\Pi_M^{\rm out}.
\end{aligned}
\]
Thus \(U_M,V_M\in{\rm SO}(2,\mathbb C)\).  Reducing \(c_1(m)\) gives
\(\|u_M-e_1\|<1/2\), which selects the normalized branch
\[
 u_M=\frac{\Pi_M^{\rm out}e_1}
 {(e_1^{\mathsf T}\Pi_M^{\rm out}e_1)^{1/2}},
\]
with the square root normalized at \(1\).  Differentiating this formula
and its input analogue for \(v_M\) gives the frame bounds, with factors
\(\epsilon_{\rm out}\) and \(\epsilon_{\rm in}\), respectively.
The affine dependence follows as in Step~(iii).

\smallskip
\noindent\emph{Step (vi): The normal form.}
The defining identities give
\[
 Mv_M=\Lambda_Mu_M,\quad M^{\mathsf T}u_M=\Lambda_Mv_M
 \quad\Longrightarrow\quad
 U_M^{\mathsf T}MV_M=
 \begin{pmatrix}\Lambda_M&0\\0&\det(M)\Lambda_M^{-1}\end{pmatrix}.
\]
This is \eqref{B:eq:finite-two-state-exact-bilinear-normal-form}.

Finally, \eqref{B:eq:finite-two-state-exact-bilinear-normal-form} and
\(\det M=\mathscr L_0^2\prod_{t=0}^{L-1}\epsilon_t\) give the
determinant-one specialization.  For real data the Gram eigenvalues are
nonnegative, and the two separation disks show that the selected eigenvalue
near \(1\) is the larger one.  Therefore \(|\Lambda_M|=\|M\|\), while
\(\mathscr L_0\mathscr H_L\) fixes the orientation.  In the real-symmetric
setting, uniqueness of the Riesz projections and normalized branches gives
the asserted symmetry.

For each \(0\le r\le m\), the preceding argument produces only finitely
many coefficients and powers, all depending only on \(m\). Choose a single
\(C_2(m)\ge2\) dominating all of them, and choose \(c_1(m)>0\) small
enough for the finitely many smallness requirements above. Since
\(L\ge1\), \(\kappa\le1\), and
\(2+K_0+\sum_{r=1}^mK_r\ge2\), enlarging \(C_2(m)\) strengthens
\eqref{B:eq:finite-two-state-smallness} and enlarges the upper bounds in
\eqref{B:eq:finite-two-state-projector-output} and
\eqref{B:eq:finite-two-state-frame-output}. These fixed constants therefore
work simultaneously at every order \(0\le r\le m\), completing the proof.
\end{proof}

\end{document}